\documentclass[a4paper,12pt,reqno,oneside]{amsart} 

\usepackage{tikz}
\usetikzlibrary{decorations.pathmorphing}
\usetikzlibrary{cd}

\usepackage{setspace} 
\usepackage{typearea} 
\usepackage{amscd,amssymb,verbatim} 
\usepackage
{graphicx} \setkeys{Gin}{width=\linewidth}
\usepackage{enumitem}
\usepackage{bbold}

\usepackage{xr-hyper}
\usepackage{cite}
\usepackage[colorlinks,backref,final,bookmarksnumbered,bookmarks]{hyperref}
\usepackage[backrefs]{amsrefs}

\usepackage{ifthen}
\newcommand{\arxiv}[2][]{\ifthenelse{\equal{#1}{}}
{\href{http://arxiv.org/abs/#2}{\tt arXiv:#2}}
{\href{http://arxiv.org/abs/math/#2}{\tt arXiv:math.#1/#2}}}

\usepackage[T1,T2A]{fontenc}
\usepackage[utf8]{inputenc}

\makeatletter
\renewcommand\subsection{\@startsection
{subsection}{2}{0cm} 
{-\baselineskip}     
{0.5\baselineskip}   
{\sffamily}} 
\makeatother

\theoremstyle{plain}
\newtheorem{theorem}{Theorem}[section]
\newtheorem{lemma}[theorem]{Lemma}
\newtheorem{corollary}[theorem]{Corollary}
\newtheorem{proposition}[theorem]{Proposition}

\theoremstyle{definition}
\newtheorem{example}[theorem]{Example}
\newtheorem{convention}[theorem]{Convention}

\newtheoremstyle{remark}
{}{}{}{}{\itshape}{}{ }{\thmname{#1}\thmnumber{ \itshape #2.}}
\theoremstyle{remark}
\newtheorem{remark}[theorem]{Remark}

\newtheoremstyle{concise}
{}{}{}{}{\bfseries}{}{ }{\thmnumber{#2.}\thmnote{ #3.}}
\theoremstyle{concise}

\DeclareFontEncoding{LS1}{}{}
\DeclareFontSubstitution{LS1}{stix}{m}{n}
\DeclareFontEncoding{LS2}{}{}
\DeclareFontSubstitution{LS2}{stix}{m}{n}

\DeclareMathVersion{language}
\DeclareMathVersion{metalanguage}

\def\formulas{\mathversion{language}}

\def\metameta{\mathversion{normal}}

\DeclareMathAlphabet{\mf}{OML}{zplm}{m}{it} 
\DeclareMathAlphabet{\ts}{OT1}{cmss}{m}{sl} 

\DeclareMathAlphabet{\fm}{U}{eur}{m}{n} 
\DeclareMathAlphabet{\tr}{OT1}{cmss}{m}{n} 
\DeclareMathAlphabet{\mm}{OML}{cmm}{m}{it} 

\def\FM#1{\[\fm{#1}\]}

\DeclareMathAlphabet{\script}{LS1}{stixscr}{m}{n}
\DeclareMathAlphabet{\mathbbb}{U}{bbold}{m}{n}

\SetSymbolFont{letters}{language}{U}{eur}{m}{n}
\SetSymbolFont{letters}{metalanguage}{OML}{zplm}{m}{it}

\newcommand{\newsym}[5]{\fontfamily{#2}\fontencoding{#1}\fontseries{#3}\fontshape{#4}\selectfont\char#5}
\makeatletter
\newcommand{\newmathsymbol}[6]{#1{\@Pimathsymbol{#2}{#3}{#4}{#5}{#6}}}
\def\@Pimathsymbol#1#2#3#4#5{\mathchoice
  {\@Pim@thsymbol{#1}{#2}{#3}{#4}{#5}\tf@size}
  {\@Pim@thsymbol{#1}{#2}{#3}{#4}{#5}\tf@size}
  {\@Pim@thsymbol{#1}{#2}{#3}{#4}{#5}\sf@size}
  {\@Pim@thsymbol{#1}{#2}{#3}{#4}{#5}\ssf@size}}
\def\@Pim@thsymbol#1#2#3#4#5#6{\mbox{\fontsize{#6}{#6}\newsym{#1}{#2}{#3}{#4}{#5}}}
\makeatother

\def\too{\newmathsymbol{\mathrel}{LS1}{stixsf}{m}{n}{"99}}
\newcommand{\Dfrac}[2]{\genfrac{}{}{1.2pt}0{#1}{#2}}

\def\lgroup{\newmathsymbol{\mathopen}{LS2}{stixex}{m}{n}{"DC}}
\def\rgroup{\newmathsymbol{\mathclose}{LS2}{stixex}{m}{n}{"DD}}
\newcommand{\mq}[1]{\lgroup #1\rgroup\,}
\newcommand{\eq}[1]{\left<#1\right>}
\def\mc#1{\mq{\oldstylenums{#1}}}
\def\ec#1{\eq{\oldstylenums{#1}}}
\def\prin{\boldsymbol\cdot\hskip1.5pt}

\def\cltop{\top}
\def\clbot{\bot}
\def\ab{{\mathchoice
  {\mbox{\larger[1]$\times$}}
  {\mbox{\larger[1]$\times$}}
  {\mbox{\larger[-2]$\times$}}
  {\mbox{\larger[-4]$\times$}}
}}
\def\triv{\checkmark}
\def\Bot{\newmathsymbol{\mathord}{U}{cmllr}{m}{n}{13}}
\def\Top{\raisebox{8pt}{\scalebox{1}[-1]{$\Bot$}}}

\def\KK{\hskip1pt\tikz[baseline=0pt,decoration={snake,amplitude=0.6pt,segment length=3pt}]
\draw[decorate] (0pt,-3pt) -- (0pt,9pt);\hskip1pt}

\def\from{\leftarrow}
\def\tofrom{\leftrightarrow}
\def\To{\ \longrightarrow\ }
\def\Tofrom{\ \longleftrightarrow\ }
\def\imp{\Rightarrow}
\def\when{\Leftarrow}
\def\iff{\Leftrightarrow}
\def\Imp{\ \Longrightarrow\ }
\def\Iff{\ \Longleftrightarrow\ }

\def\turnstile{\vdash}
\def\Turnstile{\vDash}

\DeclareMathSymbol{\impord}{\mathord}{symbols}{41}
\DeclareMathSymbol{\mand}{\mathbin}{operators}{`\&}
\DeclareMathOperator{\varmand}{\,\tilde\&\,}
\DeclareMathOperator{\varmetabot}{\,\tilde\metabot\,}
\DeclareMathOperator{\varmor}{\,\tilde\parallel\,}
\newcommand{\vareq}[1]{\widetilde{\left<#1\right>}}

\DeclareMathSymbol{\oc}{\mathord}{operators}{`!}
\DeclareMathSymbol{\wn}{\mathord}{operators}{`?}

\def\ocf{\newmathsymbol{\mathord}{OT1}{cmr}{m}{it}{"21}}
\def\wnf{\newmathsymbol{\mathord}{OT1}{cmr}{m}{it}{"3F}}
\def\metabot{\newmathsymbol{\mathord}{LS2}{stixtt}{m}{n}{"A4}}

\DeclareMathSymbol{\alpha}     {\mathalpha}{letters}{"0B}
\DeclareMathSymbol{\beta}      {\mathalpha}{letters}{"0C}
\DeclareMathSymbol{\gamma}     {\mathalpha}{letters}{"0D}
\DeclareMathSymbol{\delta}     {\mathalpha}{letters}{"0E}
\DeclareMathSymbol{\epsilon}   {\mathalpha}{letters}{"0F}
\DeclareMathSymbol{\zeta}      {\mathalpha}{letters}{"10}
\DeclareMathSymbol{\eta}       {\mathalpha}{letters}{"11}
\DeclareMathSymbol{\theta}     {\mathalpha}{letters}{"12}
\DeclareMathSymbol{\iota}      {\mathalpha}{letters}{"13}
\DeclareMathSymbol{\kappa}     {\mathalpha}{letters}{"14}
\DeclareMathSymbol{\lambda}    {\mathalpha}{letters}{"15}
\DeclareMathSymbol{\mu}        {\mathalpha}{letters}{"16}
\DeclareMathSymbol{\nu}        {\mathalpha}{letters}{"17}
\DeclareMathSymbol{\xi}        {\mathalpha}{letters}{"18}
\DeclareMathSymbol{\pi}        {\mathalpha}{letters}{"19}
\DeclareMathSymbol{\rho}       {\mathalpha}{letters}{"1A}
\DeclareMathSymbol{\sigma}     {\mathalpha}{letters}{"1B}
\DeclareMathSymbol{\tau}       {\mathalpha}{letters}{"1C}
\DeclareMathSymbol{\upsilon}   {\mathalpha}{letters}{"1D}
\DeclareMathSymbol{\phi}       {\mathalpha}{letters}{"1E}
\DeclareMathSymbol{\chi}       {\mathalpha}{letters}{"1F}
\DeclareMathSymbol{\psi}       {\mathalpha}{letters}{"20}
\DeclareMathSymbol{\omega}     {\mathalpha}{letters}{"21}
\DeclareMathSymbol{\varepsilon}{\mathalpha}{letters}{"22}
\DeclareMathSymbol{\vartheta}  {\mathalpha}{letters}{"23}
\DeclareMathSymbol{\varpi}     {\mathalpha}{letters}{"24}
\DeclareMathSymbol{\varrho}    {\mathalpha}{letters}{"25}
\DeclareMathSymbol{\varsigma}  {\mathalpha}{letters}{"26}
\DeclareMathSymbol{\varphi}    {\mathalpha}{letters}{"27}

\DeclareMathSymbol{\bbomega}{\mathalpha}{letters}{"7F}

\def\N{\mathbbb{N}}
\def\R{\mathbbb{R}}
\def\Z{\mathbbb{Z}}
\def\Q{\mathbbb{Q}}
\def\C{\mathbbb{C}}
\def\D{\script{D}}
\def\Ds{\mathcal{D}}
\def\E{\script{E}}
\def\F{\mathcal{F}}
\def\G{\mathcal{G}}
\def\H{\mathcal{H}}

\def\I{\script{i}}
\def\J{\script{j}}
\def\K{\script{k}}
\def\L{\script{L}}

\def\O{\script{O}}

\def\Ru{\mathcal{R}}
\def\Su{\mathcal{S}}
\def\Fm{\script{F}}
\def\Qm{\script{Q}}
\def\Rt{\mathbbb{\omega}}
\def\T{\script{T}}
\def\Th{\mathcal{T}}
\def\0{\mathbbb{0}}
\def\1{\mathbbb{1}}
\def\m{\mathbbb{\mu}}
\def\q{\mathbbb{q}}
\def\iass{\script{u}}
\def\pval{\script{v}}
\def\preass{\script{U}}
\def\preval{\script{V}}

\def\var{\mathbbb{x}}
\def\pr{\mathbbb{p}}
\def\inc{\mathbbb{i}}

\def\x{\times}
\def\but{\setminus}
\def\emb{\hookrightarrow}

\def\eps{\varepsilon}
\def\phi{\varphi}
\def\emptyset{\varnothing}
\def\xr#1{\xrightarrow{#1}}
\def\xl#1{\xleftarrow{#1}}
\renewcommand{\:}{\colon}
\def\normal#1{{\textstyle #1}} 

\DeclareMathOperator{\id}{id}
\DeclareMathOperator{\Int}{Int}
\DeclareMathOperator{\Cl}{Cl}
\DeclareMathOperator{\Fr}{Fr}

\DeclareMathOperator{\Hom}{Hom}
\DeclareMathOperator{\Supp}{Supp}
\DeclareMathOperator{\Char}{Char}
\DeclareMathOperator{\Dom}{Dom}
\DeclareMathOperator{\solve}{\wr}

\DeclareMathOperator{\Bew}{Bew}
\DeclareMathOperator{\Prov}{Prov}
\DeclareMathOperator{\Consis}{Consis}

\def\bGamma{\Gamma}
\def\bDelta{\Delta}

\def\bTheta{\Theta}
\def\bP{P}
\def\bQ{Q}
\def\bR{R}

\def\CH{\text{{\rm CH}}}

\def\ZFC{\text{{\rm ZFC}}}

\begin{document}
\title{Mathematical semantics of intuitionistic logic}
\author{Sergey A. Melikhov}
\address{Steklov Mathematical Institute of the Russian Academy of Sciences,
ul.\ Gubkina 8, Moscow, 119991 Russia}
\email{melikhov@mi.ras.ru}
\thanks{Supported by Russian Foundation for Basic Research Grant No.\ 15-01-06302}

\begin{abstract}
This work is a mathematician's attempt to understand intuitionistic logic.
It can be read in two ways: as a research paper interspersed with lengthy digressions into rethinking 
of standard material; or as an elementary (but highly unconventional) introduction to first-order 
intuitionistic logic.
For the latter purpose, no training in formal logic is required, but a modest literacy in mathematics, 
such as topological spaces and posets, is assumed.

The main theme of this work is the search for a formal semantics adequate to Kolmogorov's informal 
interpretation of intuitionistic logic (whose simplest part is more or less the same as the so-called
BHK interpretation).
This search goes beyond the usual model theory, based on Tarski's notion of semantic consequence, 
and beyond the usual formalism of first-order logic, based on schemata.
Thus we study formal semantics of a simplified version of Paulson's meta-logic, used in 
the {\tt Isabelle} prover.

By interpreting the meta-logical connectives and quantifiers constructively, we get a generalized
model theory, which covers, in particular, realizability-type interpretations of intuitionistic logic.
On the other hand, by analyzing Kolmogorov's notion of semantic consequence (which is an alternative 
to Tarski's standard notion), we get an alternative model theory.
By using an extension of the meta-logic, we further get a generalized alternative model theory, which
suffices to formalize Kolmogorov's semantics.

On the other hand, we also formulate a modification of Kolmogorov's interpretation, which is compatible
with the usual, Tarski-style model theory.
Namely, it can be formalized by means of sheaf-valued models, which turn out to be a special case of Palmgren's
categorical models; intuitionistic logic is complete with respect to this semantics. 
\end{abstract}

\maketitle
\tableofcontents

\section{Introduction}

\subsection{Formal semantics}

Much of this long treatise is devoted to understanding in rigorous mathematical terms A. Kolmogorov's 
1932 short note ``On the interpretation of intuitionistic logic'' \cite{Kol}.
What the logical community has grasped of this text is contained in the so-called BHK interpretation,
often considered to be ``the intended interpretation'' of intuitionistic logic and very extensively
analyzed in the literature (see \S\S\ref{BHK}--\ref{about-bhk} below). 
However, as observed in the classic textbook by Troelstra and van Dalen \cite{TV}, ``the BHK-interpretation 
in itself has no `explanatory power'\,'' since ``on a very `classical' interpretation of [the informal 
notions of] construction and mapping,'' its six clauses ``justify the principles of two-valued (classical) logic.''
But as discussed in \S\S\ref{confusion}--\ref{Medvedev-Skvortsov} below, this arguably does not apply to 
Kolmogorov's original interpretation.

The problem with Kolmogorov's original interpretation is that it is incompatible with the usual model theory, 
based on Tarski's notion of semantic consequence; worse yet, it is an interpretation of something that 
does not exist in the usual formalism of first-order logic, based on schemata.

Instead of this usual formalism we use a simplified (and, eventually, slightly extended) version of 
L. Paulson's 1989 meta-logic \cite{Pau1}, which is the basis of the {\tt Isabelle} prover (see 
\S\S\ref{formal-intro}--\ref{logics}).
(In fact, it builds not on the standard modern formalism but rather on the early tradition of 
first-order logic, as in the textbooks by Hilbert--Bernays \cite{HB} and Hilbert--Ackermann \cite{HA}.)
The meta-logical connectives and quantifiers can be used to express, for instance, the horizontal bar 
and comma in inference rules (in fact, the entire deductive system of a first-order logic is 
a single meta-formula); derivability of a formula or a rule; implication between principles 
(in the usual sense, which has nothing to do with ordinary implication between formulas or schemata); 
variables fixed in a derivation (in the sense of Kleene), etc. 
In this approach, principles and rules, and even implications between them become finite entities within
the formal language (without any meta-variables --- even in situations where the standard formalism would 
have side conditions such as ``where $x$ is not free in $F$'' or ``where $t$ is free for $x$ in $F(x)$'')
and so one can speak of their interpretation by mathematical objects.

Formal semantics of the meta-logic of first-order logics seems to have never been developed, however; 
thus we have to undertake this task (see \S\S\ref{models0}--\ref{classical models} and
\S\ref{meta-semantics}).
In usual model theory, regarded as a particular semantics of the meta-logic, the meta-logical connectives and 
quantifiers are interpreted classically (and moreover in the two-valued way) --- even though they are constructive.
By interpreting them constructively, we get a generalized model theory 
(see \S\ref{generalized Tarski-style}).
It suffices to formalize, in particular, realizability-type interpretations of intuitionistic logic,
which do not fit in the usual model theory (see \S\ref{uniform realizability}).
However, to formalize Kolmogorov's interpretation of inference rules (which is a substantial part of 
his interpretation of intuitionistic logic, and seems to be uniquely possible in his approach), we need 
an alternative generalized model theory (see \S\ref{Frege-style2}).
It is nothing new for formulas and principles, but an entirely different interpretation of rules 
(and more complex meta-logical constructs) already in the case where the meta-logical connectives 
and quantifiers have a classical, two-valued interpretation (see \S\ref{Frege-style} concerning this case). 

We also formulate an alternative interpretation of intuitionistic logic, called ``modified 
BHK-interpretation'', which is compatible with the usual, Tarski-style model theory.
A specific class of usual models of intuitionistic logic that fully meets the expectations
of the modified BHK-interpretation is described.
These are sheaf-valued, or ``proof-relevant topological'' models (\S\ref{sheaves}) of intuitionistic logic
--- not to be confused with the usual open set-valued ``sheaf models'' (see \cite{AB00} and references there).
As one could expect, sheaf-valued models turned out to be a special case of something well-known: 
the ``categorical semantics'' (cf.\ \cite{LS}), and more specifically Palmgren's models of intuitionistic 
logic in locally cartesian closed categories with finite sums \cite{Pa1}.
But unexpectedly, this very simple and natural special case with obvious relevance for the informal semantics 
of intuitionistic logic does not seem to have been considered per se (apart from what amounts to 
the $\land,\to$ fragment \cite{Aw00}).
We prove completeness of intuitionistic logic with respect to sheaf-valued models by showing that in
a certain rather special case (sheaves over zero-dimensional separable metrizable spaces) they interpret
the usual topological models over the same space.
Normally (e.g.\ for sheaves over Euclidean spaces or Alexandroff spaces) this is far from being so.
On the other hand, our sheaf-valued models also interpret Medvedev--Skvortsov models
(\S\ref{Medvedev-Skvortsov}).

\subsection{Informal semantics}

We adopt and develop Kolmogorov's understanding of intuitionistic logic as the logic of schemes of
solutions of mathematical problems (understood as tasks).
In this approach, intuitionistic logic is viewed as an extension package that upgrades classical logic without
removing it --- in contrast to the standard conception of Brouwer and Heyting, which regards intuitionistic logic
as an alternative to classical logic that criminalizes some of its principles.
The main purpose of the upgrade comes, for us, from Hilbert's idea of equivalence between proofs of a given theorem,
and from the intuition of this equivalence relation as capable of being nontrivial.
This idea of ``proof relevance'' amounts to a categorification; thus, in particular, our sheaf-valued semantics
of intuitionistic logic can be viewed as a categorification of the familiar Leibniz--Euler--Venn subset-valued
semantics of classical logic.

Traditional ways to understand intuitionistic logic (semantically) have been rooted either
in philosophy --- with emphasis on the development of knowledge (Brouwer, Heyting, Kripke)
or in computer science --- with emphasis on effective operations (Kleene, Markov, Curry--Howard).
Kolmogorov's approach emphasizes the right order of quantifiers, and
is rooted in the mathematical tradition, dating back to the Antiquity, of solving of problems
(such as geometric construction problems) and in the mathematical idea of a method (of solution or proof).
The approach of Kolmogorov has enjoyed some limited appreciation in Europe in the 1930s,
and had a deeper and more lasting impact on the Moscow school of logic%
\footnote{In the introduction to her survey of logic in the USSR up to 1957, S. A. Yanovskaya emphasizes
``The difference of the viewpoints of the `Moscow' school of students and followers of P.~S.~Novikov and
A.~N.~Kolmogorov and the `Leningrad' school of students of A.~A.~Markov on the issue of the meaning of
non-constructive objects and methods'', where the former viewpoint ``consists in allowing in one's work
not only constructive, but also classical methods of mathematics and mathematical logic'' \cite{Ya}.
In the West, the principles of the `Leningrad' school are better known as ``Russian constructivism''
(cf.\ \cite{TV}).}
--- but nowadays is widely seen as a mere historical curiosity.
One important reason for that%
\footnote{Besides the lack of an English translation of Kolmogorov's paper \cite{Kol} until the 1990s,
and the lack of a word in English to express the notion of a ``mathematical task''
(German: Aufgabe, Russian: задача) as distinguished from the notion of an ``open problem''
(German: Problem, Russian: проблема).
Let us note, in particular, that Kolmogorov himself, when writing in French, took care to mention
the German word and so spoke of ``probl\`emes (Aufgaben)'' \cite{Kol2}.}
could have been the lack of a sufficiently clear formulation and of any comprehensive exposition
of intuitionistic logic based on this approach.
The present treatise attempts to remedy this deficiency.

Our introduction to some basic themes of intuitionistic logic (\S\ref{topology}) is in the language of its usual 
topological models, which were discovered independently by Stone, Tang and Tarski in the 1930s but largely 
went out of fashion since the 60s.
(However, in the zero-order case these models include, in disguise, Kripke models, which have had far better 
luck with fashion.)
We present a motivation of topological models through a model of a classical first-order theory extracted from 
the clauses of the BHK interpretation (\S\ref{weak BHK}).

\subsection*{Disclaimers}
(i) This unconventional introduction to intuitionistic logic should suffice for the purposes of
the author's papers \cite{M1}, \cite{M2}, \cite{M3}, but as such it does not pretend to be complete or
even finished.
It is hoped that a future version will cover certain additional topics and will be accompanied by
an unconventional introduction to homotopy type theory.

(ii) Most translations quoted in this paper have been edited by the present author in order to improve
syntactic and semantic fidelity.
When emphasis is present in quoted text, it is always original.

\subsection*{Acknowledgements}
I am indebted to A. Bauer, L. Beklemishev, M. Bezem, K. Do\v sen, and N. Polyakov
for important remarks, corrections, and improving my understanding of some basics.
I'm also grateful to T. Coquand, D. Grayson, H. Gylterud, S. Kuznetsov, E. V. Melikhova, K. Mineshima, 
F. Pakhomov, D. Saveliev, G. Schlaepfer, O. K. Sheinman, V. B. Shehtman, D. Skvortsov and 
G. Tolokonnikov for stimulating discussions and useful remarks.
A part of this work was carried out while enjoying the hospitality and stimulating atmosphere of
the Institute for Advanced Study and its Univalent Foundations program in Spring 2013.
There I gradually realized that in order to achieve my goals related to Univalent Foundations I first need
a better understanding of intuitionistic logic; and O. Viro further convinced me that in order
to truly understand the latter, I have no choice but to recreate it in a way more compatible with
the mentality of a ``working mathematician''.
I'm also grateful D. Botin and D. Cenzer, who taught me some bits of logic many years ago.

\section{Motivation: homotopies between proofs}\label{motivation}

\subsection{Hilbert's 24th Problem}
In 1900, Hilbert published (and partly presented in his lecture at the
International Congress of Mathematicians in Paris) a list of 23 open
problems in mathematics \cite{Hi1}.
Shortly before the 100th anniversary of the list, the historian
R\"udiger Thiele discovered a would-be 24th problem in
Hilbert's notebooks \cite{Th} (for a photocopy see \cite{Ko}
and \cite{Th2}).
We don't know why Hilbert refrained from publicizing it, but one
obvious possibility is that he simply did not see how to state it
precisely enough.

\smallskip
\begin{center}
\parbox{14.7cm}{\small
``The 24th problem in my Paris lecture was to be: Criteria of simplicity,
or proof of the greatest simplicity of certain proofs. Develop a theory
of the method of proof in mathematics in general. Under a given set of
conditions there can be but one simplest proof. Quite generally, if
there are two proofs for a theorem, you must keep going until you have
derived each from the other, or until it becomes quite evident what
variant conditions (and aids) have been used in the two proofs. Given
two routes, it is not right to take either of these two or to look for
a third; it is necessary to investigate the area lying between the two
routes. Attempts at judging the simplicity of a proof are in my
examination of syzygies and syzygies between syzygies (see
[\cite{Hi}], Lectures XXXII--XXXIX).
The use or the knowledge of a syzygy simplifies in an essential way
a proof that a certain identity is true. Because any process of
addition [is] an application of the commutative law of addition etc.
[and because] this always corresponds to geometric theorems or logical
conclusions, one can count these [processes], and, for instance, in
proving certain theorems of elementary geometry (the Pythagoras theorem,
[theorems] on remarkable points of triangles), one can very well decide
which of the proofs is the simplest.''\footnotemark
}
\end{center}
\smallskip
\footnotetext{Translation by Thiele, who also comments that ``a part
of the last sentence is not only barely legible in Hilbert's notebook
but also grammatically incorrect'', and points out that ``corrections
and insertions that Hilbert made in this entry show that he wrote down
the problem in haste''.}

A rather narrow understanding of the problem, as merely asking for
``a criterion of simplicity in mathematical proofs and the development
of a proof theory with the power to prove that a given proof is
the simplest possible''\footnote{Quoted from Wikipedia's
``Hilbert's 24th problem'' (English) as of 11/11/11} has became
widespread in the literature (see e.g.\ \cite{BV}*{p.\ 38},
\cite{HR}, \cite{Th}).
However, Hilbert's original text above is unambiguously posing a totally
different problem as well (in the two sentences starting with ``Quite
generally''): to develop a proof theory that can tell whether
two given proofs of a theorem are ``essentially same'', or
{\it homotopic}, in the sense that one can {\it derive each from
the other} (whatever that means).
It is this aspect of the 24th problem that we shall be concerned with.

Moreover, Hilbert's reference to syzygies between syzygies is obviously%
\footnote{Let us also note the connection between (higher) syzygies
and (higher) Martin-L\"of identity types.
In Voevodsky's model of Martin-L\"of's type theory, types are
interpreted as Kan complexes (see \cite{HoTT}), so one kind of a type
is an abelian simplicial group, which by the Dold--Kan correspondence
(see \cite{GJ}) can be identified with a pointed chain complex $C$ of
abelian groups:
$b\in C_{-1}\xl{\partial_0} C_0\xl{\partial_1} C_1\xl{\partial_2}
C_2\xl{\partial_3}\dots$.
Write $t:C$ to mean $t\in\partial_0^{-1}(b)$, and $t=u:C$ to mean
that $t:C$ and $u:C$ are homologous relative to
$\partial_0(t)=\partial_0(u)$, that is, $t-u\in\partial_1(C_1)$.
Then the identity type $Id_C(t,u)$ will be the following pointed chain
complex:
$t-u\in C_0\xl{\partial_1}C_1\xl{\partial_2}
C_2\xl{\partial_3}\dots$.}%
\,\footnote{Perhaps we should also note the connection between (higher)
identity types and (higher) ``homotopical syzygies'' (see \cite{Lo}),
which goes back to J. H. C. Whitehead's classical work on identities among
relations (see \cite{2DH}).
The role of the Dold--Kan correspondence is played here by the
not so well understood connection (see \cite{MP}) between extended group
presentations of \cite{Br}, where higher relators correspond to higher
cells of a CW complex, and free simplicial groups, such as the Kan loop
group (see \cite{GJ}) of a simplicial set.}
provoking one to think of a hypothetical space of proofs of a given
theorem and of whether it should be acyclic or contractible
(so that proofs can be simplified in a canonical way) under a given set
of conditions.


Certainly, Hilbert was not alone in being concerned with (informal)
questions of this kind.
For instance, a sample of attempts by mathematicians other than
proof theorists at judging ``essential sameness'' of certain proofs
(of the irrationality of $\sqrt 2$ and of Gauss' quadratic
reciprocity law) can be found in T. Gowers' blog \cite{Go}.
A more systematic study of the question of identity of proofs
has taken place in intuitionistic proof theory, starting in the 70s
(Prawitz \cite{Pr}, Kreisel \cite{Kr}, Martin-L\"of \cite{ML1};
see a brief survey in \cite{D1}*{\S2}) and continued in the
context of higher categories and higher $\lambda$-calculus
\cite{LS}, \cite{Se}, \cite{Hel}, \cite{Hil}.
The more recent approach of homotopy type theory \cite{HoTT} is to
a large extent a development of this proof-theoretic tradition.
We should also point out some very different approaches to somewhat
similar questions, which cannot be discussed here in detail:
\cite{La} (and references there), \cite{Sa}, \cite{CT}, \cite{BV},
\cite{Mos}*{\S8.2}, \cite{BDG}.

So how to judge if two proofs of a theorem are ``essentially same''?
Hilbert's idea that there might be the {\it simplest} proof of
a theorem $\Theta$ and his mention of the {\it use} of variant
conditions (and aids) in proving $\Theta$ establish a context where
it seems natural to identify a proof $p$ of $\Theta$ with every its
{\it simplification}, that is, a proof of $\Theta$ obtained from $p$
by crossing out some redundant steps.
Of course, the notion of a ``redundant step'' (or a step that is not
``essentially used'' in the proof) needs an accurate definition;
we will return to this issue later.

On the other hand, Hilbert's words (``keep {\it going} until you
have {\it derived} each from the other'') may seem to hint at a formal
theory of proofs containing a number of specified elementary
transformations between proofs of the same assertion.
In this case the relation of {\it homotopy} between proofs of the same
assertion can be defined as the equivalence relation generated by
elementary transformations.

Such formal theories endowed with elementary transformations between
proofs are best known in the case of constructive proofs
(in particular, the natural deduction system of Gentzen--Prawitz
and the corresponding $\lambda$-calculus, see \cite{Gi}).
Before turning to constructive proofs, let us review the situation
with usual, non-constructive ones.

\subsection{The collapse of non-constructive proofs}\label{Lafont}
A standard theory of derivations based on classical logic is
Gentzen's sequent calculus \cite{Ge} (see \cite{Gi}).
In particular, the elementary inferences

\medskip
($Cut$) from $A\lor C$ and $C\to B$ infer $A\lor B$;

\smallskip
($C_R$) from $A\lor A$ infer $A$;

\smallskip
($W_{R1}$) from $A$ infer $A\lor B$;

\smallskip
($W_{R2}$) from $A$ infer $B\lor A$;

\smallskip
($W_L$) from $B$ infer $C\to B$

\medskip

\noindent
can be interpreted as instances of the inference rules of sequent
calculus.%
\footnote{Namely, ($Cut$) can be interpreted as an instance of
the cut rule, ($C_R$) as an instance of the right contraction rule,
($W_{R1}$) and ($W_L$) as instances of the right and left weakening
rules, and ($W_{R2}$) as ($W_{R1}$) followed by an instance of
the exchange rule.}
A number of elementary transformations between derivations of the
same proposition in sequent calculus can be found in the proof of
Gentzen's Hauptsatz on cut elimination (see \cite{Gi}, \cite{Ge}).
These {\it $\Gamma$-transformations} include, in particular,
the following ($\Gamma_L$) and ($\Gamma_R$).

\medskip
($\Gamma_L$) Suppose that $q$ is a proof of $C\to B$. Then the following
fragment of a proof:

\begin{itemize}
\item $A$ implies $A\lor C$ by ($W_{R1}$);
\item $C\to B$ by $q$;
\item hence $A\lor B$ by ($Cut$)
\end{itemize}

transforms into

\begin{itemize}
\item $A$ implies $A\lor B$ by ($W_{R1}$).
\end{itemize}
\medskip

\noindent
We note that ($\Gamma_L$) can be seen as a simplifying transformation:
it eliminates the proof $q$ of the redundant lemma $C\to B$ at the cost
of replacing one weakening ($A$ implies $A\lor C$) by another
($A$ implies $A\lor B$).

\medskip
($\Gamma_R$) Suppose that $p$ is a proof of $A\lor C$. Then the following
fragment of a proof:

\begin{itemize}
\item $B$ implies $C\to B$ by ($W_L$);
\item $A\lor C$ holds by $p$;
\item hence $A\lor B$ by ($Cut$)
\end{itemize}

transforms into

\begin{itemize}
\item $B$ implies $A\lor B$ by ($W_{R2}$).
\end{itemize}
\medskip

\noindent
The symmetry between ($\Gamma_L$) and ($\Gamma_R$) is highlighted
in Gentzen's original notation \cite{Ge}*{III.3.113.1-2}.
We can see it explicitly by rewriting $C\to B$ and $A\lor C$
respectively as $B\lor\neg C$ and $\neg C\to A$, using tautologies
of classical logic.

Let us also consider two additional {\it $\Lambda$-transformations}:

\medskip
($\Lambda_1$), ($\Lambda_2$): The following fragment of a proof can
be eliminated:
\begin{itemize}
\item $A$ implies $A\lor A$ by ($W_{R1}$) or by ($W_{R2}$);
\item $A\lor A$ implies $A$ by ($C_R$).
\end{itemize}

\begin{theorem}[Lafont \cite{Gi}*{Appendix B.1}; see also
\cite{D1}*{\S7}]
All proofs of any given theorem are equivalent modulo $\Gamma$- and
$\Lambda$-transformations.
\end{theorem}

Lafont's own summary of his result is that ``classical logic is
inconsistent, not from the logical viewpoint ([falsity] is not
provable) but from an algorithmic one.''

\begin{proof}[Proof]
Let $p$ and $q$ be proofs of a theorem $\Theta$.
They can be augmented to yield a proof $p+W_{R1}$ of
$\Theta\lor\Phi$, where $\Phi$ is some fixed assertion, and a proof
$q+W_L$ of the implication $\Phi\to\Theta$.
Combining the latter two proofs, we obtain a proof
$(p+W_{R1})+(q+W_L)+Cut$ of the assertion $\Theta\lor\Theta$.
Hence we get a third proof, $(p+W_{R1})+(q+W_L)+Cut+C_R$, of $\Theta$,
which we will denote $p\oplus q$ for brevity.%
\footnote{A version of this $\oplus$ is also found in Artemov's Logic of Proofs (see \cite{JdeJ}).}

It remains to show that $p\oplus q$ can be reduced to $p$, as well as
to $q$ by $\Gamma$- and $\Lambda$-transformations.
Indeed, $\Gamma_L$ reduces $p\oplus q$ to the proof $p+W_{R1}+C_R$ of
$\Theta$ (which augments $p$ by saying that $\Theta$ implies
$\Theta\vee\Theta$, which in turn implies $\Theta$).
The latter proof is reduced to $p$ by $\Lambda_1$.
Similarly, $\Gamma_R$ reduces $p\oplus q$ to the proof $q+W_{R2}+C_R$
of $\Theta$, which in turn is reduced to $q$ by $\Lambda_2$.
\end{proof}

It should be noted that, despite Lafont's theorem, some authors have been able to
find inequivalent proofs in the framework of classical logic.
See, in particular, Do\v sen \cite{D1}
(see details in \cite{DP1}; see also \cite{DP} and references there), Kuznets \cite{Kuz} and
references there; and Guglielmi \cite{Gu1}, \cite{Gu2} (see also \cite{Gui} and references there)

\begin{remark}
One issue with $\Gamma$-transformations is that ($\Gamma_L$)
does not merely eliminate $q$, but only does so at the expense
of changing a weakening.
So it is not immediately clear that the lemma $C\to B$ is redundant in
the initial proof of ($\Gamma_L$), in the strict sense of not being
used at all in that proof.

To address this concern, let us revisit Lafont's construction.%
\footnote{What follows is only an informal discussion, perhaps
raising more issues than it attempts to settle.}
Recall that we were free to choose any assertion as $\Phi$.
We can set $\Phi$, for instance, to be the negation $\neg\Theta$
(any weaker assertion, such as $\Theta\lor\neg\Theta$,
would also do).
Then $\Theta\lor\Phi$ holds in classical logic regardless
of validity of $\Theta$.
Therefore the step $W_{R1}$ of the proof $p\oplus q$ replaces
the proved theorem $\Theta$ with the tautology $\Theta\lor\Phi$.
The effect of this step is that we forget all that we have
learned from the proof $p$.
Thus $p\oplus q$ does not essentially use $p$, so $p$ must be redundant
even in the stricter sense.

On the other hand, we could take $\Phi$ to be $\Theta$ itself
(any stronger assertion, for instance
$\Theta\land\neg\Theta$, would also do).
Then the step $W_L$ replaces the proved theorem $\Theta$ by
the tautology $\Phi\to\Theta$.
The effect of this step is that we forget all that we have
learned from $q$, and we similarly conclude that $q$ must be
redundant.
However, it appears that the meaning of the proof $p\oplus q$ should not
vary depending on the strength of $\Phi$, since for the purposes of
this proof we did not need to know anything at all about $\Phi$.

Our conclusion, that neither $p$ nor $q$ is ``essentially
used'' in $p\oplus q$, shows that the notion of a proof ``essentially
using'' its part cannot be sensibly formalized in the framework of
classical logic.
It turns out, however, that this notion can be consistently defined
once we abandon classical logic in favor of intuitionistic.
\end{remark}

\subsection{About constructive proofs}
In contrast to classical logic, where one can sometimes prove that a certain object exists without giving 
its explicit construction, proofs in intuitionistic logic can be interpreted, at least very roughly,
as algorithms (the Curry--Howard correspondence, see \cite{Gi}, \cite{SU}).
For example, if $S$ is a set, an intuitionistic proof of the assertion ``$S$ is nonempty'' can be roughly 
thought of as a program (executed on a Turing machine, say) that returns a specific element of $S$.
(These issues will be discussed in more detail in \S\ref{PEM} and \S\ref{uniform realizability}.)

In writing a computer program, it is not enough to know that some
subsidiary computation is possible in principle; for the program
to run, every step must be actually implemented.
In fact, the value of every particular variable can be chased through
all dependent subroutines using a debugger.
This is in sharp contrast with classical mathematics, where the
particular way that a lemma is proved has absolutely no effect on
the proof of the theorem modulo that lemma.
In intuitionistic logic, there can be no dependence between statements
other than a dependence between their specific proofs.

With the above in mind, although $\Theta\lor\Phi$ is intuitionistically
a tautology as long $\Phi$ is (for instance, $\Phi$ can be taken
to be $\Theta\to\Theta$),
the intuitionistic step $W_{R1}$, replacing the proved theorem $\Theta$
with the tautology $\Theta\lor\Phi$ does not have the effect of
forgetting how $\Theta$ was proved.
On the contrary, it has the effect of specifying that $\Theta\lor\Phi$
is to be proved ``via $\Theta$'' and not ``via $\Phi$'', and extending
a given proof of $\Theta$ to a specific proof of $\Theta\lor\Phi$.
Since the latter proof does not depend on any proof of $\Phi$, the
intuitionistic step $Cut$ promotes it to a proof of $\Theta\lor\Theta$
that does not depend on any proof assuming $\Phi$ as a premise.

It follows that if $p$ and $q$ are constructive, then $p\oplus q$,
understood constructively, essentially uses $p$ and does not essentially
use $q$.
In agreement with this, the intuitionistic version of the sequent
calculus (see \cite{Gi}) admits only asymmetric interpretations of
$p\oplus q$, which can be reduced by the intuitionistic analogues
of $\Gamma$- and $\Lambda$-transformations to $p$, but not to $q$.

\section{What is intuitionistic logic?}\label{intro}

\subsection{Semantics of classical logic}\label{Tarski truth}
In classical logic, the meaning of propositions is determined by
answering the question: {\it When is a given proposition true?}
A standard answer to this question is given by Tarski's definition
of truth.
The definition is by induction on the number of logical connectives
and quantifiers:
\begin{itemize}
\item which primitive propositions are true is assumed to be known
from context.
\end{itemize}
For instance, in the context of arithmetic, a proposition of
the form $t=s$ (with $t$ and $s$ being arithmetical expressions) is true
whenever $t$ and $s$ rewrite to the same numeral.
\begin{itemize}
\item $\bP\lor \bQ$ is true if $\bP$ is true or $\bQ$ is true;
\item $\neg \bP$ is true if $\bP$ is not true;
\item $\exists x\, \bR(x)$ is true if $\bR(x_0)$ is true for some $x_0\in\D$.
\end{itemize}
Here $\bP$, $\bQ$ are propositions and $\bR$ is a unary predicate, i.e., a proposition with one 
parameter.
In first-order logic, all parameters are understood to run over a fixed set $\D$, which is called 
the {\it domain of discourse}.
Thus $\bR(x_0)$ is a proposition for any $x_0\in\D$

The other classical connectives and quantifiers $\top$, $\bot$ $\land$, $\to$, $\forall$ do not need
to be mentioned since they are expressible in terms of $\lor$, $\neg$ and $\exists$.
For the record, we get for them, as a consequence of the above:
\begin{itemize}
\item $\top$ is true;
\item $\bot$ is false;
\item $\bP\land \bQ$ is true if $\bP$ is true and $\bQ$ is true;
\item $\bP\to \bQ$ is true if $\bQ$ is true or $\bP$ is not true;
\item $\forall x\, \bR(x)$ is true if $\bR(x_0)$ is true for all $x_0\in\D$.
\end{itemize}
We can also understand the latter six clauses the other way round, as an
``explanation'' of classical connectives and quantifiers in terms of truth of propositions.

\begin{remark}
In this chapter we always assume our meta-logic to be classical; that is, when discussing intuitionistic 
or any other logic, we always mean {\it this discussion} to take place in classical logic.
It the next chapter, some of that discussion will be formalized (in fact, more of it than in any standard
text on first-order logic), so we will speak of a formal meta-logic, which in fact will turn out to be
not exactly classical (but for many purposes it can assumed to be classical); but then 
the meta-meta-logic, in which all discussions of that formal meta-logic take place, will still 
be classical.
\end{remark}

Intuitionistic logic can be approached in several ways, of which we will discuss the two earliest ones:
the Brouwer--Heyting approach (briefly), and Kolmogorov's approach (in more detail).

\subsection{Platonism vs.\ Verificationism}\label{truth}
The tradition initiated by Brouwer in 1907 and continued by his school in Amsterdam (Heyting,
Troelstra, et al.) and elsewhere till the present day --- perhaps more successfully in computer
science than in mathematics --- presupposes a deep rethink (along the lines of Logical Positivism and
prior development of verificationist thinking; see Martin-L\"of \cite{ML12}, \cite{ML87}*{pp.\ 414--416}),
of the conventional understanding of what mathematics is and what it is not.
In the words of Troelstra and van Dalen \cite{TV}*{p.\ 4}:

\smallskip
\begin{center}
\parbox{14.7cm}{\small
``It does not make sense to think of truth or falsity of a mathematical
statement independently of our knowledge concerning the statement.
A statement is {\it true} if we have [a constructive] proof of it, and
{\it false} if we can show that the assumption that there is
a [constructive] proof for the statement leads to a contradiction.
For an arbitrary statement we can therefore not assert that it is
either true or false.''
}
\end{center}
\medskip

Thus, in the tradition of Brouwer and Heyting, propositions are considered
to have a different --- intuitionistic --- meaning, which is determined by
answering the question: {\it When do we know a constructive proof of
a given proposition?}
We will return to this question in \S\ref{BHK}, once it will have been
restated in \S\ref{Kolmogorov} so as to avoid the terminological clash with
the classical meaning of propositions.

Whether the classical meaning of propositions ``makes sense'' or not is a philosophical question
which we will try to avoid --- so as to focus on answering mathematical questions.
But at the very least, the author is determined to keep good relations with the 99\%
of mathematicians who are convinced that the classical meaning of propositions does
make sense; and so considers the {\it terminology} of Brouwer and Heyting, which is in
direct conflict with the convictions of the 99\% of mathematicians, to be unfortunate.
Indeed, a useful terminology would provoke mathematical, and not philosophical questions.

The depth of the rethink of mathematics promoted by the battlefield terminology of Brouwer
and Heyting should not be underestimated.
In the above quote, we are asked to abstain not only from using the principle of excluded middle
(in the form asserting that every statement is either true or false), but also from
understanding mathematical objects with the customary mental aid of Platonism (i.e.\ Plato's
vision of mathematical objects as ideal entities existing independently
of our knowledge about them).
To be sure, Heyting has stated it boldly enough \cite{He1}:
\smallskip
\begin{center}
\parbox{14.7cm}{\small
``Faith in transcendental existence, unsupported by concepts, must be rejected as a means of mathematical
proof, [... which] is the reason for doubting the law of excluded middle.''
}
\end{center}
\medskip

The present author is no proponent of Platonism, but, on the contrary, feels that properly understood
mathematics should not depend on any particular philosophy (for reasons related to Occam's Razor).
Heyting certainly has a point in that when we mathematicians speak of ``existence'' of mathematical
objects and ``truth'' of mathematical theorems, it is not at all clear what we mean by this;
the implicit hypothesis that this talk ``makes sense'' at all can certainly be considered as
a religious dogma.
On the other hand, basic concepts of Intuitionism such as ``construction'', ``proof'' and
``contradiction'' are just as obscure as ``existence'' and ``truth''; and the intuitionists'
conviction that these concepts (as well as their ``explanation'' by clauses of the BHK interpretation,
cf.\ \S\ref{BHK} below) ``make sense'' is certainly no less of a religious dogma.

Thus the claims of greater meaningfulness and lesser dogmatism that Intuitionism pretends to have
over classical mathematics appear to be highly controversial.
Can we, nevertheless, {\it understand} what is intuionistic mathematics?

\subsection{The issue of understanding}

From the viewpoint of Intuitionism, it is trivial to understand classical mathematics: one
just needs to explicitly mention the principle of excluded middle in the hypothesis of every classical
theorem, thus transforming it into a conditional assertion.

Moreover, as shown and emphasized by Kolmogorov \cite{Kol0}, the use of this principle alone cannot be
the source of a contradiction (see \S\ref{negneg}); nor it can lead to a contradiction when combined
with the axioms of first-order arithmetic (G\"odel--Gentzen, cf.\ \cite{ML07}).
The same is true of the second-order arithmetic (Kreisel, cf.\ \cite{Coq}) and, with some
additional effort, also of the set theory ZF (without the axiom of choice) \cite{Fr}, \cite{Po} ---
albeit this requires impredicative intuitionism, which is something that not all constructivists accept
\cite{ML07}, \cite{Coq}.
However, in the presence of the countable axiom of choice (which suffices for classical analysis) in its
standard form, interpreted intuitionistically according to Kolmogorov, the addition of the principle
of excluded middle generally increases consistency strength even with respect to impredicative intuitionism
\cite{BBC} (see also \cite{OS}).
In fact, the issue of precise formulation of the axiom of choice is important here; see
\cite{Coq}*{2.2.4} and \cite{ML08}.

This is entirely parallel to the situation within classical mathematics.
Some principles of classical set theory, including the continuum hypothesis and the axiom of choice,
are internally provable to be consistent with ZF; in particular, their addition to ZF does not increase
its consistency strength --- or in other words, their use alone cannot be the source of a contradiction.
Other principles, like large cardinal axioms or the Axiom of Determinacy, are known or conjectured to
increase consistency strength when added to ZF; yet there are various good reasons to believe that
resulting systems are still consistent, perhaps not significantly different from the reasons for one's
belief in the consistency of ZF.
Thus classical mathematics with or without the axiom of choice should be, from the viewpoint of
an intuitionist, on a par with areas of classical mathematics that depend on the said principles
as seen by a classical mathematician who does not find these principles to be intuitively justified.
Conversely, it would be beneficial to {\it understand} intuitionistic logic
and its relations with classical logic in a way that is compatible
with the plain Platonic intuition arguably shared by most ``working
mathematicians''.

First attempts at meaning explanations of intuitionistic logic in terms
of ordinary practices of classical mathematics can be found in
a 1928 paper by Orlov \cite{Or}*{\S\S6,7} (see also \cite{D92})
and in the 1930 and 1931 papers by Heyting \cite{He0}, \cite{He1}.
Both Orlov and Heyting gave explanations (see \cite{M1}*{\S\ref{g1:letters1}} for a detailed review) ---
in terms of provability of propositions (``propositions'' and ``provability'' being
understood in a sense compatible with classical logic) which directly
anticipate G\"odel's 1933 provability translation --- to be reviewed in \S\ref{provability}.
G\"odel's translation certainly yields an accurate explanation in arguably classical terms (albeit not 
exactly in terms of classical logic) of the standard formalization of intuitionistic logic.
But nevertheless it arguably misses the essence of intuitionistic logic, in that it does not
explain how one proposition can have essentially different proofs (cf.\ \S\ref{motivation}).

\subsection{Solutions of problems}\label{Kolmogorov}
The second approach to intuitionistic logic, initiated by Kolmogorov in 1932 and continued to some extent 
by his school in Moscow aimed at being less demanding of one's philosophical convictions.
In the words of Kolmogorov \cite{Kol}:

\smallskip
\begin{center}
\parbox{14.7cm}{\small
``On a par with theoretical logic, which systematizes schemes of
proofs of theoretical truths, one can systematize schemes of
solutions of problems --- for example, of geometric construction
problems.
For instance, similarly to the principle of syllogism we have
the following principle here:
{\it If we can reduce the task of solving $b$ to that of $a$, and
the task of solving $c$ to that of $b$, then we can also reduce
the task of solving of $c$ to that of $a$.}

Upon introducing appropriate notation, one can specify the rules of
a formal calculus that yield a symbolic construction of a system of
such problem solving schemes.
Thus, in addition to theoretical logic, a certain new
{\it calculus of problems} arises.
In this setting there is no need for any special, e.g.\ intuitionistic,
epistemological presuppositions.

The following striking fact holds: {\it The calculus of problems
coincides in form with Brouwer's intuitionistic logic,
as recently formalized by Mr.\ Heyting}.

In the second section we undertake a critical analysis of intuitionistic logic, accepting
general intuitionistic presuppositions;
and observe that intuitionistic logic should be replaced with the calculus of problems,
since its objects are in reality not theoretical propositions but
rather problems.''
}
\end{center}
\medskip

In this paper, we use the word {\it problem} solely in the sense of a request (or desire) to find
a construction meeting specified criteria on output and permitted means (as in ``chess problem'',
``geometric construction problem'', ``initial value problem'').
With {\it problem} understood in this sense, Kolmogorov’s interpretation is, in the words of
Martin-L\"of \cite{ML}*{p.\ 6},
\smallskip
\begin{center}
\parbox{14.7cm}{\small
``very close to programming.
`$a$ is a method [of solving the problem (doing the task) $A$]
can be read as `$a$ is a program ...'.
Since programming languages have a formal notation for the program $a$,
but not for $A$, we complete the sentence with `... which meets
the specification $A$'. In Kolmogorov's interpretation, the word {\it problem}
refers to something to be done and the word {\it program}\footnotemark\
to how to do it.''
}
\footnotetext{Kolmogorov's original interpretation does not mention any ``programs'', so this is likely 
a misprint, and the phrase is to be read as ``the word {\it solution}''.}
\end{center}
\medskip
A key difference between problems and propositions is the lack of a direct analogue of the Platonic concept
of truth in the case of problems.
In the words of Kolmogorov \cite{Kol}:

\smallskip
\begin{center}
\parbox{14.7cm}{\small ``It is, however, essential that every proved proposition is {\it correct};
for problems one has no concept corresponding to this correctness.''
}
\end{center}
\medskip

As a matter of convention, we will always understand a {\it solution} of a problem positively,
in that a proof that a problem $\bGamma$ has no solutions is not considered to be a solution of $\bGamma$.
Instead, we consider it to be a solution of a different problem, denoted $\neg\bGamma$; and
the problem of either finding a solution of $\bGamma$ or proving that one does not exist is
denoted $\bGamma\lor\neg\bGamma$.
We say that $\bGamma$ is {\it solvable} if $\bGamma$ has a solution, and {\it decidable} if
$\bGamma\lor\neg\bGamma$ has a solution.

\subsection{The Hilbert--Brouwer controversy} \label{controversy}
A problem of the form $\bGamma\lor\neg\bGamma$ may in general be very hard.
There is nothing surprising here, of course, since the principle of excluded middle applies
to propositions, whereas $\bGamma$ is a problem.
Nevertheless, the principle of decidability of every mathematical problem has been stressed forcefully
by Hilbert in the preface to his famous list of problems \cite{Hi1}:
\smallskip
\begin{center}
\parbox{14.7cm}{\small
``Occasionally it happens that we seek the solution [of a mathematical problem] under insufficient
hypotheses or in an incorrect sense, and for this reason do not succeed. The problem then
arises: to show the impossibility of the solution under the given hypotheses, or in
the sense contemplated. Such proofs of impossibility were effected by the ancients,
for instance when they showed that the ratio of the hypotenuse to the side of
an isosceles right triangle is irrational. In later mathematics, the question as to
the impossibility of certain solutions plays a preeminent part, and we perceive
in this way that old and difficult problems, such as the proof of the axiom of
parallels, the squaring of the circle, or the solution of equations of the fifth degree
by radicals have finally found fully satisfactory and rigorous solutions, although
in another sense than that originally intended. It is probably this important fact
along with other philosophical reasons that gives rise to the conviction (which
every mathematician shares, but which no one has as yet supported by a proof)
that every definite mathematical problem must necessarily be susceptible of an
exact settlement, either in the form of an actual answer to the question asked,
or by the proof of the impossibility of its solution and therewith the necessary
failure of all attempts. Take any definite unsolved problem, such as the question
of the irrationality of the Euler--Mascheroni constant $C$, or of the existence of an
infinite number of prime numbers of the form $2n+1$. However unapproachable
these problems may seem to us and however helpless we stand before them, we
have, nevertheless, the firm conviction that their solution must follow by a finite
number of purely logical processes.''
}
\end{center}
\medskip

A counterexample to Hilbert's ascription of his conviction to all mathematicians was soon provided
by Brouwer, who himself fiercely opposed it \cite{Br07}*{Statement XXI}, \cite{Br08}, \cite{Br28}.
Yet one can hardly deny that the task of {\it settlement}, in the sense of Hilbert, of mathematical
problems (including, indeed, those in Hilbert's list), i.e., the task of solving problems of
the form $\bGamma\lor\neg\bGamma$, is something that we mathematicians all constantly undertake,
with great efforts --- including those 99\% of us who consider the principle of excluded middle to be
a triviality, which can be used freely whenever the need arises.

We thus come to distinguish the {\it principle of excluded middle} (for propositions) and
the {\it principle of decidability} (for problems).
Yet the two have been systematically conflated by Heyting (see \cite{He1} and \cite{Tr99}*{p.\ 235})
and especially by Brouwer; in fact, their conflation was at the very root of Brouwer's opposition
to the principle of excluded middle, whose first published record \cite{Br08} reads:

\smallskip
\begin{center}
\parbox{14.7cm}{\small
Firstly, the {\it syllogism} infers from the embedding [inpassing]\footnotemark\
of a system [systeem] $b$ into a system $c$, along with
the embedding of a system $a$ into the system $b$, a direct embedding of the system $a$ into
the system $c$. This is nothing more than a tautology.

Likewise the principle of {\it non-contradiction} is indisputable: The accomplishment of
the embedding of a system $a$ into a system $b$ in a prescribed manner, and the
obstruction showing the impossibility of such an embedding, exclude each other.

Now consider the principle of {\it excluded middle}: It claims that every hypothesis is
either true or false; in mathematics this means that for every hypothetical embedding
of one system in another, satisfying certain given conditions, we can either accomplish
such an embedding by a construction, or we can construct the obstruction showing its
impossibility. It follows that
the question of the validity of the principle of excluded middle is equivalent to the
question whether undecidable mathematical problems can exist.
There is no indication of a justification for the occasionally pronounced (see [Hilbert's problem list
\cite{Hi1} ...]) conviction that there exist no undecidable mathematical problems.}
\footnotetext{This is the first occurrence of this term in the text, but it was explained
in Brouwer's dissertation \cite{Br07}*{p.\ 77}, where he apparently rejects the notion of
a subset defined by a property of its elements in favor of an explicitly constructed subset:
``Often it is quite easy to construct inside such a structure, independently of
how it originated, new structures, whose elements are taken from the elements of the
original structure or systems of these, arranged in a new way, but bearing in mind
their original arrangement. The so-called `properties' of a system express the
possibility of constructing such new systems having a certain connection with
the given system. And it is exactly this {\it embedding} of new systems {\it in
a given system} that plays an important part in building up mathematics, often in
the form of an inquiry into the possibility or impossibility of an embedding satisfying
certain conditions, and in the case of possibility into the ways in which it is possible.''
}
\end{center}
\medskip

This all is remarkably similar to what one could say about the {\it problem} of ``embedding of one
system in another'' (whatever that means), except for the alleged connection with
the principle of excluded middle.
Twenty years later, Brouwer has even complained that after his repeated declarations that the principle of
excluded middle is the same as the principle of decidability, Hilbert still fails to recognize their
alleged identity \cite{Br28}.
Unfortunately, neither this complaint nor the repeated declarations cited in it seem to bring in any
arguments as to why the two principles should be equated (apart from the insistence that they
obviously should).
In fact, in the opinion of Kolmogorov \cite{Kol29},
\smallskip
\begin{center}
\parbox{14.7cm}{\small ``The principle of excluded middle is, according to Brouwer, inapplicable
only to judgements of a special kind, in which a theoretical proposition is inseparably linked
with construction of the object of the proposition. Therefore, one may suppose that Brouwer’s
ideas are actually not at all in contradiction with the traditional logic, which has in fact never
dealt with judgements of this kind.''
}
\end{center}
\medskip

\subsection{Problems vs.\ conjectures}\label{conjectures}
Obviously, Brouwer committed what most mathematicians still do today: he conflated a problem (understood
as a request for finding, for instance, a proof of a certain assertion, according to a certain specification
of what counts as a proof) with a conjecture (understood as a judgement that a certain assertion is true).
Indeed, most of Hilbert's problems consist in proving a conjecture.
Let us take, for instance, his first problem, which asks to prove the Continuum Hypothesis ($\CH$):
{\it Every infinite subset of $\R$ is of the same cardinality as $\R$ or $\N$.}

Is it the case that $\CH$ is either true or false? Yes: this is a theorem of ZFC --- a direct
consequence of the principle of excluded middle.

So, is $\CH$ true? Not really: $\CH$ cannot be proved in ZFC, as shown by P. Cohen.

So, is $\CH$ false? Not really: $\CH$ cannot be disproved in ZFC, as shown by G\"odel.

This is why in modern classical logic, {\it truth} is considered to be a semantic notion, depending on
the choice of a two-valued model of the theory (cf.\ \S\ref{Tarski truth}; {\it two-valued} means that
every formula without parameters is either true or false).
For instance, $\CH$ is true in G\"odel's constructible universe, and false in Cohen's models of ZFC.

In contrast, the following problem ($\oc\CH$): {\it Prove $\CH$ in $\ZFC$} --- has no solutions
(by Cohen); and the problem ($\oc\neg\CH$): {\it Disprove $\CH$ in $\ZFC$} --- also has no solutions
(by G\"odel).
And we do not need to fix a model of ZFC for this to make sense.
Thus we should not equate the problem $\oc\CH$ with any question pertaining to {\it truth}.
Then, if we want to be precise, Hilbert's first problem should indeed be seen as a {\it problem} ---
arguably, best formalized as $\oc\CH$ --- and not a yes/no {\it question}.%
\footnote{This is even reasonably consistent with Hilbert's original wording \cite{Hi1}:
``The investigations of Cantor on such assemblages of points suggest
a very plausible theorem, which nevertheless, in spite of the most strenuous efforts, no one has
succeeded in proving. This is the theorem: Every system of infinitely many real numbers, i.e., every
assemblage of numbers (or points), is either equivalent to the assemblage of natural integers,
$1,2,3,\dots$ or to the assemblage of all real numbers''.
Alongside the Continuum Hypothesis, Hilbert also poses the problem of well-ordering of the continuum:
``It appears to me most desirable to obtain a direct proof of this remarkable statement of Cantor's,
perhaps by actually giving an arrangement of numbers such that in every partial system a first number
can be pointed out.''}
But had Hilbert used a {\it whether} question in his first problem, it could as well be formalized as
the problem $\oc\CH\lor\oc\neg\CH$, i.e., {\it Prove or disprove $\CH$ in $\ZFC$}.
Note, incidentally, that insolubility of the latter problem (which follows from G\"odel {\it and} Cohen)
does not yield
a counterexample to Hilbert's principle of decidability of all mathematical problems.
Indeed, $\oc\CH$ is decidable (by Cohen), $\oc\neg\CH$ is decidable (by G\"odel), and even
$\oc\CH\lor\oc\neg\CH$ is decidable (by G\"odel and Cohen).
But the decidability of these problems is not, of course, merely a consequence of the principle of excluded middle.

In general, a problem --- as understood in the present paper --- should not be conflated with
any yes/no question.
In particular, the problem {\it Prove $\bP$}, where $\bP$ is a proposition, is closely related to two
distinct questions: {\it Is $\bP$ true?} and {\it Does there exist a proof of $\bP$?}
Truth should not be conflated with provability, since by G\"odel's incompleteness theorem
most theories are not complete with respect to any model.
On the other hand, if $G$ and $H$ are groups, the question {\it Is $G$ isomorphic to $H$?} is closely
related to two distinct problems: {\it Prove that $G$ is isomorphic to $H$} and {\it Find
an isomorphism between $G$ and $H$}.
Proofs should not be conflated with constructions of isomorphisms: one proof might correspond to several
isomorphisms or not correspond to any specific isomorphism.

\subsection{The BHK interpretation}\label{BHK}
The mathematical meaning of problems is determined by answering the question:
{\it What does it mean to find a solution to a given problem?}
A standard but informal answer to this question was given essentially in Kolmogorov's 1932 paper 
\cite{Kol} and can also be seen as an interpretation of the intuitionistic connectives, and to some 
extent also quantifiers%
\footnote{An interpretation of $\exists$ is discussed in a slightly different language in \S2 
of Kolmogorov's paper.}
in terms of solutions of problems.
Independently%
\footnote{A note inserted at proof-reading in Kolmogorov's paper
\cite{Kol} acknowledged the connection with Heyting's \cite{He1}.
A remark found in the body of \cite{Kol}: ``This explanation of
the meaning of the sign $\turnstile$ is very different from that
by Heyting, although it leads to the same formal calculus''
must refer to Heyting's syntactic definition of $\turnstile$
and indicates that at the time of writing, Kolmogorov was unaware
of Heyting's semantic interpretation of $\turnstile$ in \cite{He1},
which is rather similar to Kolmogorov's.}
of Kolmogorov, Heyting gave rather similar interpretations for some intuitionistic connectives in terms of 
proofs of propositions ($\neg$ and $\lor$ in \cite{He1} and $\to$ in a 1930 letter to Freudenthal 
\cite{vD}, \cite{vA} and in print in a 1934 book \cite{He2}).
Kolmogorov went somewhat further than this in that he described a systematic interpretation of zero-order 
formulas of intuitionistic logic.
But much more importantly, Kolmogorov's paper also offers what in essence is an interpretation of 
the meta-logical connectives and quantifiers (as discussed in \S\ref{clarified} and \S\ref{K-consequence} below); 
to the best of the author's knowledge, this aspect of Kolmogorov's interpretation has been completely 
ignored in all of the previous literature dealing with the ``BHK-interpretation''.

A formal answer to the question {\it What does it mean to find a solution to a given problem?} (or 
{\it What does it mean to give a constructive proof to a given proposition?}) was given, in particular,
by Gentzen's 1935 natural deduction calculus \cite{Ge}.
(An independent but essentially equivalent answer is given by the $\lambda$-calculus, see 
\cite{Gi}, \cite{Ho}, \cite{Lae}, \cite{Coq}, \cite{SU}*{\S4}, \cite{TV}*{\S10.8}.)
This answer can of course be read syntactically, on a par with, say, the sequent calculus for classical logic.
Yet the introduction rules of the natural deduction calculus can be also read as a semantic interpretation of 
the intuitionistic connectives and quantifiers, which turns out to be closely related to Kolmogorov's and 
Heyting's interpretations (see \cite{FD}, \cite{ML87}*{pp.\ 410--412}).

Initially, Heyting has switched to Kolmogorov's language; in particular, the introduction into 
intuitionistic logic found in his 1934 book \cite{He2} is based on Kolmogorov's understanding of 
the connectives and quantifiers (but not the meta-quantifiers, which Heyting ignores in \cite{He2}).
More influential, however, was his later return to the ideologically loaded language of Brouwer, 
most notably in his 1956 book \cite{He3}.
This in turn inspired the modern tradition of the so-called ``BHK interpretation'', named so 
by Troelstra after Brouwer, Heyting and Kreisel \cite{Tr0}, with ``K'' later reassigned to Kolmogorov 
by Troelstra and van Dalen \cite{TV}.
It is also known as the ``standard interpretation'', and even
the ``intended interpretation'' of intuitionistic logic
(see \cite{Tr2}, \cite{SU}, \cite{St}, \cite{Gi}, \cite{vDh}, \cite{Ru}, \cite{We}).
Compared to Kolmogorov's original interpretation, the BHK interpretation
(i) omits Kolmogorov's interpretation of the meta-connectives and meta-quantifiers;
(ii) treats the quantifiers in a more systematic way;
(iii) gives a more detailed interpretation of $\to$ due to Heyting
(which is similar to Kolmogorov's interpretation of $\forall$).

Most presentations of the BHK interpretation in the literature are in Brouwer's language, but here is one 
in the language of Kolmogorov.

\medskip
\begin{itemize}
\item What are solutions of primitive problems is assumed to be known
from context.
\end{itemize}
For instance, Euclid's first three postulates are the following primitive problems:%
\footnote{See \cite{M3}*{\S\ref{g3:intro}, \S\ref{g3:postulates vs axioms}} for a thorough
discussion of postulates vs.\ axioms.}

(1) draw a straight line segment from a given point to a given point;

(2) extend any given straight line segment continuously to a longer one;

(3) draw a circle with a given center and a given radius.

\noindent
We may thus stipulate that each of the primitive problems (1) and (3) has a unique solution, and
give some description of the possible solutions of (2).

\begin{itemize}
\item A solution of $\bGamma\land\bDelta$ consists of a solution
of $\bGamma$ and a solution of $\bDelta$.

\item A solution of $\bGamma\lor\bDelta$ consists of an explicit choice
between $\bGamma$ and $\bDelta$ along with a solution of the chosen
problem.

\item A solution of $\bGamma\to\bDelta$ is a {\it reduction} of $\bDelta$ to
$\bGamma$; that is, a general method of solving $\bDelta$ on the basis
of any given solution of $\bGamma$.

\item The {\it absurdity} $\ab$ has no solutions;
$\neg\bGamma$ is an abbreviation for $\bGamma\to\ab$, and the {\it triviality} $\triv$ is an
abbreviation for $\neg\ab$.%
\footnote{The notation $\ab$, $\triv$ is non-standard; the usual one is $\bot$, $\top$ 
--- same as in classical logic.
We will eventually need to differentiate notationally between classical and intuitionistic logics, and
to this end we use $\top$, $\bot$ and Roman letters in classical contexts; and $\triv$, $\ab$ and Greek
letters in intuitionistic contexts.
This suffices to differentiate classical formulas from intuitionistic ones, and so we can afford to
keep the other connectives, and both quantifiers, identical in appearance in classical and intuitionistic 
contexts.}

\item A solution of $\exists x\, \Theta(x)$ is a solution of $\Theta(x_0)$
for some explicitly chosen $x_0\in\D$.

\item A solution of $\forall x\, \Theta(x)$ is a general method
of solving $\Theta(x_0)$ for all $x_0\in\D$.
\end{itemize}
\medskip

Kolmogorov's notion of a ``general method'' roughly corresponds to
the notion of a ``construction'' advocated by Brouwer and Heyting, but
is perhaps less rhetorical in that it puts a central emphasis on
the tangible issue of the right order of quantifiers.
According to Kolmogorov, if $\bGamma(\script X)$ is a problem depending on
the parameter $\script X$ ``of any sort'', then ``to present a general method
of solving $\bGamma(\script X)$ for every particular value of $\script X$''
should be understood as ``to be able to solve $\bGamma(\script X_0)$
for every given specific value of $\script X_0$ of the variable $\script X$
by a finite sequence of steps, known in advance (i.e.\ before the choice
of $\script X_0$)''.
This does not pretend to be a fully unambiguous definition, but
``hopefully cannot lead to confusion in specific areas of
mathematics''.

An important clarification to the BHK interpretation, emphasized by Kreisel (in somewhat different terms;
see details and further discussion in \cite{DT} and \cite{Be}*{\S5.1}), is that a solution of a problem
must not only consist of a construction of a prescribed type, but also include a proof that the proposed
construction does solve the problem.
In fact, this stipulation is clearly expressed already in Proclus' ancient commentary on Euclid's Elements
(see \cite{M3}), and so may be called the {\it Proclus--Kreisel principle}.
If we parse the clauses of the BHK interpretation, assuming that solutions of primitive problems come with
correctness proofs, we see that these easily yield correctness proofs for solutions of problems in the
$\land$, $\lor$, and $\exists$ clauses; thus in the context of the BHK interpretation, the Proclus--Kreisel
principle essentially amounts to the convention that every general method (such as those in the $\to$
and $\forall$ clauses) is supplied with a proof that it actually does what it is meant to do.%
\footnote{A well-known alternative to the Proclus--Kreisel principle is to deal with bare constructions,
and defer their correctness verifications until the solubility of the entire problem (built out 
of primitive problems with the logical connectives) is asserted.
This approach necessitates a clear distinction of ``propositions'' (or problems) vs.\ ``judgements''
and is used in Martin-L\"of's type theory (see \cite{Be}*{\S5.1.2}).
In fact, according to \cite{Gol}*{p.\ 36}, this alternative approach is implicit already in Frege's work.}

Finally, two more precautions on interpreting the BHK interpretation:
\begin{itemize}
\item The domain $\D$ must be a ``simple domain'' such as the set $\N$ of natural numbers; a subset of
$\N$ defined with aid of quantifiers would require a more elaborate version of the BHK interpretation
(see \cite{DT}).
In fact, Martin-L\"of's type theory \cite{ML} is based on extending the BHK interpretation to more
general domains.
\item The status of all primitive problems (i.e.\ if they are considered to have a solution) is
supposed to be fixed before the logical connectives can be explained, and cannot be updated as we
determine the meaning of the connectives (Prawitz; see \cite{SP}).
\end{itemize}

\subsection{Understanding the connectives}\label{about-bhk}

\subsubsection{$\lor$ and $\exists$}
The seemingly innocuous interpretations of $\lor$ and $\exists$ are
already in sharp contrast with classical logic.
Let us consider the following classical proof that $x^y$
can be rational for irrational $x$ and $y$: if $\sqrt2^{\sqrt2}$ is
rational, then $x=y=\sqrt2$ will do; else let $x=\sqrt 2^{\sqrt2}$ and
$y=\sqrt2$.
Although we have proved that the proposition $\bP(x,y)=${\it ``$x^y$ is rational''}
holds either for $(x,y)=(\sqrt2,\sqrt2)$ or for
$(x,y)=(\sqrt2^{\sqrt2},\sqrt2)$, this method is not going to prove it
for any specific pair.
Thus if $\oc\bQ$ denotes the problem {\it Prove the proposition $\bQ$,} then we have solved the problem
$\oc\big(\bP(\sqrt2,\sqrt2)\lor\bP(\sqrt2^{\sqrt2},\sqrt2)\big)$, and hence also the problem
$\oc\exists x\exists y\, \bP(x,y)$ (with $x$ and $y$ ranging over all irrational numbers),
but we have no clue about solving either $\exists x\exists y\,\oc \bP(x,y)$ or
$\oc\bP(\sqrt2,\sqrt2)\lor\oc\bP(\sqrt2^{\sqrt2},\sqrt2)$.

On the other hand, with classical meta-logic we may wonder if for a given particular problem $\bGamma$
there exists a solution of $\bGamma$.
According to the BHK interpretation (as presented above), the set of solutions of a problem
$\bGamma\lor\bDelta$ is the disjoint union of those of $\bGamma$ and $\bDelta$.
The disjoint union of two sets is nonempty if and only if at least one of them is nonempty.
Thus we conclude that $\bGamma\lor\bDelta$ is soluble if and only if
either $\bGamma$ is soluble or $\bDelta$ is soluble.
In symbols, if $\wn\bGamma$ denotes the proposition {\it ``The problem $\bGamma$ is soluble''},
then $\wn(\bGamma\lor\bDelta)$ is equivalent to $\wn\bGamma\lor\wn\bDelta$; for similar reasons,
the proposition $\wn\exists x\,\bGamma(x)$ is equivalent to $\exists x\,\wn\bGamma(x)$.

Let us note that the explicit choice between $\bGamma$ and $\bDelta$ that is present in any
solution of $\bGamma\lor\bDelta$ is forgotten (and cannot be recovered in general) when only
the existence of such a solution is asserted.
For instance, the proposition $\wn\big(\oc\bP(\sqrt2,\sqrt2)\lor\oc\bP(\sqrt2^{\sqrt2},\sqrt2)\big)$
might follow from a proof-theoretic result asserting that rationality of real numbers of a certain
form can always be either proved or disproved (in ZFC, say); this would still give us no clue which
of the two numbers is rational (unless the proof-theoretic result is proved in a constructive way).%
\footnote{In fact, irrationality of $\sqrt2^{\sqrt2}$ follows from the Gelfond--Schneider theorem
(which answered a part of Hilbert's 7th problem): $a^b$ is transcendental if $a$ and $b$ are algebraic,
$a\ne 0,1$, and $b$ is not rational.}

\subsubsection{$\forall$ and $\land$}\label{BHK-forall}
The significance of the Proclus--Kreisel principle is clear from the following {\it Schwichtenberg's paradox}.
Let us consider the problem $\oc\bP(x,y,z,n)$ of proving the following proposition $\bP(x,y,z,n)$:
$$x^n+y^n=z^n\ \ \to\ n\le 2.$$
Of course, it is trivial to devise a general method $m$ to verify the inequality $x^n+y^n\ne z^n$ and thus
to solve the problem $\oc\bP(x,y,z,n)$ for every particular choice of positive integers $x,y,z$ and $n>2$
for which it can be solved.
What is hard is to prove that $m$ actually succeeds on all inputs.
Thus it is only due to Kreisel's clarification that the problem
$\forall x,y,z,n\,\oc\bP(x,y,z,n)$ (with $x,y,z,n$ ranging over all positive integers) is nontrivial.
Yet this problem is easily seen to be equivalent to the problem $\oc\forall x,y,z,n\,\bP(x,y,z,n)$ of
proving Fermat's last theorem.
Indeed, any solution of the problem $\forall x\,\oc\bQ(x)$ is clearly a proof of the proposition
$\forall x\,\bQ(x)$; conversely, any proof of the latter proposition yields by specialization a proof of
every its instance $\bQ(x)$, and this specialization procedure is a general method with respect to
the input $x$, whose success is guaranteed by the original proof of $\forall x\,\bQ(x)$.
Let us note that, due to the existence of the general method $m$, the problem of proving Fermat's last theorem
is also equivalent to the problem $\oc\neg\wn\oc\neg\forall x\,,y,z,n\,\bP(x,y,z,n)$
of proving that Fermat's last theorem cannot be disproved.%
\footnote{A similar remark about Goldbach's conjecture is found in Kolmogorov's letter to
Heyting (see \cite{M1}*{\S\ref{g1:letters2}}), which indicates that he did implicitly
assume general methods to include verification of their own correctness
and seems to have been aware of the significance of this provision.
In his paper \cite{Kol}, he also emphasized universal acceptability of
the validity of solutions as an inherent property of logical and
mathematical problems.}

If there exists a general method of solving a problem $\bGamma(x)$ for every $x$, then, trivially,
$\bGamma(x)$ is soluble for each $x$; in symbols, $\wn\forall x\,\bGamma(x)\to\forall x\,\wn\bGamma(x)$.
This implication cannot be reversed in general.
For example, if $\bGamma(f)$ asks to find a complex root of an algebraic equation $f(x)=0$ of degree $>0$,
then a solution of $\bGamma(f)$ exists by the fundamental theorem of algebra; yet a general method of
solving all such problems $\bGamma(f)$ could be interpreted, for instance, as a formula providing
a solution in radicals.
Similarly, if $\bGamma(f)$ is an initial value problem of the form $y'=f(x,y)$, $y(x_0)=y_0$,
$f\in C^1$, then a solution of $\bGamma(f)$ exists in some neighborhood of $x_0$ by the well-known theorem;
yet a general method of solving all such problems could be interpreted, for instance, as
a formula providing a solution in terms of elementary functions.

The intuitionistic conjunction per se does not behave very differently compared to the classical conjunction.
In particular, one can see that the proposition $\wn(\bGamma\land\bDelta)$ is equivalent to
$\wn\bGamma\land\wn\bDelta$ (similarly to the above discussion of $\lor$) and that the problem
$\oc(\bP\land\bQ)$ is equivalent to $\oc\bP\land\oc\bQ$ (similarly to the above discussion of $\forall$).

\subsubsection{$\to$}\label{BHK-to}
By the {\it modus ponens} rule $$\frac{\bP,\,\bP\to\bQ}{\bQ}$$
any proof of the proposition $\bP\to\bQ$ yields
a general method of converting proofs of $\bP$ into proofs of $\bQ$, along with a proof that this method
indeed effects such conversion.
In other words, any solution of the problem $\oc(\bP\to\bQ)$ yields a solution of $\oc\bP\to\oc\bQ$.
But not conversely.
For instance, a general method of converting ZFC-proofs of the Continuum Hypothesis $\CH$ into ZFC-proofs
of $0=1$ (where $0=\emptyset$ and $1=\{\emptyset\}$, as usual) is given by simply keeping the given
ZFC-proof; and a verification of the correctness of this method is given by Cohen's proof that ZFC-proofs
of $\CH$ do not exist.
Yet we cannot possibly obtain from this a ZFC-proof that $\CH$ implies
$0=1$, since no such ZFC-proofs exist, as proved by G\"odel.
Of course, what this argument really shows is that the problems $\oc\neg\bP$ and $\neg\oc\bP$ are
inequivalent for any independent statement $\bP$.%
\footnote{\label{CH} Some issues need to be clarified here.
For the sake of definiteness, we may demand that the verifications of correctness of general methods take
place in some theory $T$.
If all atomic problems request to prove something in a certain theory, a natural choice for $T$
would be that theory.
It might seem that identifying $\ab$ with $\oc0=1$ would then get us in trouble with the second
incompleteness theorem, which says that $T$ itself does not prove $\Bew(\#0=1)\to 0=1$, where
$\Bew$ is a formula expressing G\"odel's provability predicate for $T$ (see \cite{Sm}).
However, we are dealing here with solutions rather solubility, but the incompleteness theorems arise
only because of the existential quantifier.
Namely, let us recall that $\Bew(\#\bP)$ is defined as $\exists n\Prov(n,\#\bP)$, where $\Prov(n,m)$
is a formula expressing the predicate that $n$ is the G\"odel number of a proof of the formula with
G\"odel number $m$.
$T$ proves $\Prov(n,\#\bP)$ for some numeral $n$ if and only if $T$ proves $\bP$; and
$T$ proves $\neg\Prov(n,\#\bP)$ for every numeral $n$ if and only if $T$ does not prove $\bP$
(see \cite{Sm}*{3.2.4}).
In particular, $T$ proves that $\Prov(n,\#0=1)\to 0=1$.
Thus from a proof of $0=1$ in $T$ we can certainly get a contradiction.
Also, returning to our example (where $T=$ ZFC), since ZFC does not prove either $\CH$ or $\neg\CH$ by
G\"odel and Cohen, ZFC proves $\neg\Prov(n,\CH)$ and $\neg\Prov(n,\neg\CH)$.}

By the $\to$ clause of the BHK interpretation, or by the intuitionistic {\it modus ponens} rule
$$\frac{\bGamma,\,\bGamma\to\bDelta}{\bDelta}$$
if there exists a reduction of the problem $\bDelta$ to the problem $\bGamma$, then the solubility of
$\bGamma$ implies the solubility of $\bDelta$; in symbols, $\wn(\bGamma\to\bDelta)\to(\wn\bGamma\to\wn\bDelta)$.
The converse implication does not hold in general.
For example, let $\bGamma$ be the problem $\oc\wn\bDelta$ of proving the solubility of $\bDelta$.
Clearly, if $\bGamma$ is soluble, i.e., there exists a proof of the solubility of $\bDelta$,
then $\bDelta$ is soluble.
On the other hand, a proof of the solubility of $\bDelta$ does not necessarily yield an actual solution of
$\bDelta$.%
\footnote{According to a certain formal logic QHC of $\wn$'s and $\oc$'s \cite{M1}, for this to be possible, 
$\wn\bDelta$ must be independent from the formal theory being considered 
\cite{M2}*{Remark \ref{g2:refuting Hilbert}}.
QHC also places further limitations on potential examples of the above type.
Although the implication $\neg\wn\bGamma\to\wn\neg\bGamma$ does not hold in general, the rule
$\neg\wn\bGamma\,/\,\wn\neg\bGamma$ is valid; thus one cannot prove $\neg\wn\bGamma$ without
also proving $\wn\neg\bGamma$ \cite{M1}*{\ref{g1:weak Hilbert}}.
Also, there are the rule $\wn\oc\bP\to\wn\oc\bQ\,/\,\wn(\oc\bP\to\oc\bQ)$
and the implication
$(\oc\wn\bGamma\to\oc\wn\bDelta)\to\oc(\wn\bGamma\to\wn\bDelta)$ \cite{M1}*{\ref{g1:move-oc-wn}}.}

\subsubsection{Absurdity}\label{understanding bot}
Let us now discuss the meaning of the negation, $\neg\bGamma$.
Using $\oc$ and $\wn$ as before, and writing $\neg$ also for the classical negation
of propositions, we get the problem $\oc\neg\wn\bGamma$, which reads:
{\it Prove that $\bGamma$ has no solutions}.
(As before, this refers to classical proofs, not constructive ones.)
We will now argue that this problem is equivalent to $\neg\bGamma$ on the BHK interpretation
(with our classical meta-logic), by describing solutions of the problems
$\neg\bGamma\,\to\,\oc\neg\wn\bGamma$ and $\oc\neg\wn\bGamma\,\to\,\neg\bGamma$.

Indeed, suppose that we have a solution of $\neg\bGamma$, that is, a general method $m$ of
obtaining a solution of $\ab$ based on any given solution of $\bGamma$, including a proof that
this method works.
But since $\ab$ has no solutions, we get the following proof $p_m$ that $\bGamma$ has no solutions:
``If $s$ is a solution of $\bGamma$, then $m$ applied to $s$ yields a non-existent object, which is
a contradiction; thus $\bGamma$ has no solutions.''
Moreover, it is clear that the procedure associating $p_m$ to $m$ is a general method, which
works regardless of the nature of $m$.

Conversely, let $p$ be a proof that $\bGamma$ has no solutions.
If $s$ is a solution of $\bGamma$, then $p$ yields a proof that $s$ is {\it itself} also
a solution of $\ab$; indeed, $p$ proves that $s$ does not exist, but everything is true of
a non-existent object.
Thus, given a $p$, we get a general method $m_p$ of obtaining a solution of $\ab$ on the basis of
any given solution $s$ of $\bGamma$; namely, $m_p$ returns $s$ itself and includes a verification
that $s$ is a solution of $\ab$ (which is given by $p$).
Thus $m_p$ is a solution of $\neg\bGamma$.
Moreover, it is clear that the procedure associating $m_p$ to $p$ is a general method, which
works regardless of the nature of $p$.

Curiously, this analysis is at odds with Kolmogorov's remark \cite{Kol}:
\smallskip
\begin{center}
\parbox{14.7cm}{\small
``Let us note that $\neg a$ should not be understood as the problem
`prove the unsolvability of $a$'.
In general, if one considers `unsolvability of $a$' as
a well-defined notion, then one only obtains the theorem that
$\neg a$ implies the unsolvability of $a$, but not vice versa.
If, for example, it were proved that the well-ordering of the continuum
surpasses our capabilities, one could still not claim that
the existence of such a well-ordering leads to a contradiction.''
}
\end{center}
\medskip
\noindent
Could one make any sense out of the last sentence, mathematically?
It could be referring, for instance, to the fact that the existence
of a well-ordering of the continuum might be (i) not provable in some
formal theory, say ZF, and yet (ii) consistent with this theory.
(For example, this is how Kolmogorov's words are interpreted by Coquand \cite{Coq}.)

To interpret (i) as the unsolvability of some problem $\bGamma_1$, this $\bGamma_1$ could be the problem of
deriving the existence of the well-ordering of the continuum (i.e., a special case of the axiom of choice)
from the axioms of ZF; a solution of $\bGamma_1$ would be such a derivation, and the proof of
the non-existence of such derivation would presumably be either in ZF or in some stronger theory.
But in this case, $\neg\bGamma_1$ was solved by Cohen: given such a derivation, he does get a contradiction
(by a method whose correctness is verified in ZF; cf.\ footnote \ref{CH} above).

And to interpret (ii) as the unsolvability of $\neg\bGamma_2$ for some problem $\bGamma_2$, we could understand
it syntactically or semantically.
On the syntactic reading of (ii), $\neg\bGamma_2$ could be the problem of deriving a contradiction from
the axioms of ZF and the hypothesis of the well-orderability of the continuum; but this does not seem to be
of the form $\neg\bGamma_2$ for any problem $\bGamma_2$.
On the semantic reading of (ii), $\bGamma_2$ could be the problem of constructing a model of ZF and
a well-ordering of the continuum in this model; indeed, the unsolvability of this $\neg\bGamma_2$ amounts to
the (non-constructive) existence of a model of ZF where the continuum can be well-ordered
(compare \cite{M1}*{\S\ref{g1:diamond}}).
But in this case, $\bGamma_2$ was solved by G\"odel, who constructed a certain model of ZF (G\"odel's
constructible universe), and a well-ordering of the continuum in that model.

Thus it is possible that Kolmogorov might have simply conflated $\bGamma_1$ and $\bGamma_2$, thinking of
a single informal problem, ``Find a well-ordering of the continuum''.

On the other hand, Kolmogorov's wording (``$\neg a$ implies the {\it unsolvability} of $a$'';
``the {\it existence} of such a well-ordering leads to a contradiction'') certainly conflates
implications between problems with implications between propositions; this makes us suspect that
in writing of what must have been the former, he might have actually been thinking of the latter
(in this particular remark about the well-ordering of the continuum).
Indeed, if $\bP$ is the proposition {\it The continuum is well-orderable}, then the {\it unsolvability} of
the problem of proving it, or $\neg\wn\oc\bP$ in the above notation, certainly does not imply $\neg\bP$,
the proposition that the {\it existence} of a well-ordering of the continuum leads to a contradiction.
In this connection, let us also not forget Kolmogorov's initial rejection of
the explosion principle $\ab\to\alpha$ \cite{Kol0}, \cite{vD}.

It should be noted that Heyting, while using Kolmogorov's terminology
(and so being aware of Kolmogorov's remark quoted above) has explicitly
identified $\neg\bGamma$ with $\oc\neg\wn\bGamma$ \cite{He2}:
\smallskip
\begin{center}
\parbox{14.7cm}{\small
``It is appropriate to interpret the notion of `reduction' in such
a way that the proof of the impossibility of solving $a$ at
the same time reduces the solution of any problem whatsoever to that of
$a$.''
}
\end{center}
\medskip

\subsection{Something is missing here}\label{confusion}
Troelstra and van Dalen note of the BHK interpretation \cite{TV}*{p.\ 9 and p.\ 33}:

\smallskip
\begin{center}
\parbox{14.7cm}{\small
``This explanation is quite informal and rests itself on our understanding
of the notion of construction and, implicitly, the notion of mapping;
it is not hard to show that, on a very `classical' interpretation of
construction and mapping, H1--6 [the six clauses of the BHK interpretation] justify
the principles of two-valued (classical) logic.''

``This exercise [\cite{TV}*{Exercise 1.3.4}] shows that the BHK-interpretation in itself 
has no `explanatory power': the possibility of recognizing a classically valid logical schema as
being constructively unacceptable depends entirely on our interpretation
of `construction', `function', `operation'.''
}
\end{center}
\medskip

Indeed (see also \cite{Sato}), let $\KK\bGamma\KK$ denote the set of solutions of the problem $\bGamma$.
Then the BHK interpretation guarantees that:
\begin{itemize}
\item $\KK\bGamma\land\bDelta\KK$ is the product $\KK\bGamma\KK\x\KK\bDelta\KK$;
\item $\KK\bGamma\lor\bDelta\KK$ is the disjoint union $\KK\bGamma\KK\sqcup \KK\bDelta\KK$;
\item there is a map
$\script F\:\KK\bGamma\to\bDelta\KK\to\Hom(\KK\bGamma\KK,\KK\bDelta\KK)$ into the set of all maps;
\item $\KK\ab\KK=\emptyset$;
\item $\KK\exists x\,\bGamma(x)\KK$ is the disjoint union
$\bigsqcup_{d\in\D} \KK\bGamma(d)\KK$;
\item there is a map
$\script G\:\KK\forall x\,\bGamma(x)\KK\to\prod_{d\in\D} \KK\bGamma(d)\KK$
into the product.
\end{itemize}
If we force $\script F$ to be the identity map, we obtain
what might be called the ``classical BHK''.
Indeed, $\KK\bGamma\lor\neg\bGamma\KK=\KK\bGamma\KK\sqcup\Hom(\KK\bGamma\KK,\emptyset)$
is never empty; thus $\bGamma\lor\neg\bGamma$ has a solution for each
problem $\bGamma$.
Note that we get the same result as long as $\script F$ is surjective.

Thus the six clauses of the BHK interpretation, as presented either above or in Troelstra--van Dalen \cite{TV}, 
fail to capture the essence of intuitionistic logic.
The way to deal with this issue in the intuitionistic tradition has been, unfortunately, just to sweep it 
under the carpet.%
\footnote{And it is arguably this particular sweeping under the carpet that is
largely responsible for the fact that most mathematicians are aware of intuitionistic logic,
yet fail to understand it.}
Thus, elsewhere in the same book by Troelstra and van Dalen \cite{TV}*{p.\ 10}, we find the assertion
\smallskip
\begin{center}
\parbox{14.7cm}{\small
``Even if the explanations H1-6 [the six clauses of the BHK interpretation] leave a lot of questions open, they
suffice to show that certain logical principles should be generally acceptable from a constructive point 
of view, while some other principles from classical logic are not acceptable.''
}
\end{center}
\medskip
This is followed by an analysis of ``the principle of excluded middle'' (i.e.\ the principle of decidability 
in our terminology), which ends with a striking conclusion \cite{TV}*{p.\ 11}:
\smallskip
\begin{center}
\parbox{14.7cm}{\small
 ``Thus we cannot accept PEM [the Principle of Excluded Middle] as a universally 
valid principle on the BHK-interpretation.''
}
\end{center}
\medskip
Don't the latter two quotations directly contradict the previous two quotations from the same book, in which 
the same BHK interpretation was found to ``justify the principles of two-valued (classical) logic''?

The only sign of hope for understanding what is going on here is the clause {\it ``as a universally valid 
principle''}, whose meaning is not really explained in Troelstra--van Dalen \cite{TV}.
We are told, however, that 
\smallskip
\begin{center}
\parbox{14.7cm}{\small
``Constructively, accepting PEM [the Principle of Excluded Middle] as a general principle means that we have 
a universal method for obtaining, for any proposition $A$, [a proof of $A\lor\neg A$, that is,] either a proof 
of $A$ or a proof of $\neg A$''}
\end{center}
\medskip
This is parallel to Heyting's original assertion \cite{He1}*{p.\ 59}:
\smallskip
\begin{center}
\parbox{14.7cm}{\small
``a proof that the law of excluded middle is a general law must consist in giving a method by which, when
given an arbitrary proposition, one could always prove either the proposition itself or it negation.''
}
\end{center}
\medskip
Yet these demands for a {\it ``method''} appear out of the blue both in Heyting \cite{He1} and in 
Troelstra--van Dalen \cite{TV}, without any explanations.
Of course, what they could have said to settle the issue is that (i) it is precisely their demands 
for a ``method'' that are lacking to give ``explanatory power'' to the six clauses of the BHK-interpretation; 
(ii) thus these demands should really be included in an additional clause of the BHK-interpretation, 
in order to guarantee that it does explain intuitionistic rather than classical logic.

In fact, Kolmogorov did explicitly incorporate a similar demand (for a ``method'') in his original 
version of the BHK interpretation \cite{Kol}.
But before we are ready to formulate and discuss it in a clear language, we will need to understand what 
could have prevented Troelstra and van Dalen, as well as all other writers in the intuitionistic tradition 
(to the best of the author's knowledge) to add such a stipulation to the BHK interpretation.

One factor, which could be deciding in the case of Heyting's early work, was the attitude that, in view of 
Heyting's derivation system, intuitionistic logic needs to be explained philosophically, rather than 
mathematically.
It is obviously due to this attitude that Heyting did not attempt a systematic contentual interpretation of 
intuitionistic logic prior to Kolmogorov's paper, even though he had clearly gathered the necessary bits 
and pieces. 
But this attitude also meant that one does not even aim at delineating exactly what is included in 
the interpretation.

Another, much more serious factor was the widespread acceptance of first-order logic with its usual model theory
and a particular tradition of trying to express notions that are not expressible in it (including 
the ``principle of excluded middle'') as the default language in foundations of mathematics.
Like any language (natural or artificial), it comes with a plethora of implicit assumptions that on 
the one hand facilitate communication and thought, but on the other hand subjugate them to invisible limitations.
We will now proceed to review and overcome some of these limitations.

\subsection{Classical logic revisited}\label{universalism}

A detailed treatment of first-order formal logic will be given in the next chapter (\S\ref{formal}).
Here we include a quick glimpse of it, just enough to proceed with clarification of the BHK interpretation
in the next section (\S\ref{clarified}).

\subsubsection{Syntax}

Formal classical logic deals with formulas, such as
\[\fm{p\lor\neg q}\]
and 
\[\fm{q\lor\neg p}.\]
Here $\fm p$ and $\fm q$ are {\it propositional variables} and $\lor$, $\neg$ are {\it (formal) connectives}; 
all of these are some of the ingredients of the formal language of classical logic.
In classical logic, one cannot (formally) derive any of these two formulas from the other one (as we will  
discuss shortly).
Similarly, none of the formulas
\[\fm{p\lor\neg p}\]
and 
\[\fm{q\lor\neg q}\]
is derivable from each other without using the law of excluded middle.

In modern treatments of first order classical logic, the law of excluded middle is usually expressed by 
the {\it schema}
\[f\lor\neg f\]
in which $\lor$ and $\neg$ have the same meaning as before, but $f$ is a {\it metavariable} or schematic 
variable (not to be confused with the propositional variable $\fm f$), which is absent from the formal language 
of classical logic, but is understood to denote a formula.
Formulas like $\fm{p\lor\neg p}$, $\fm{q\lor\neg q}$ and $\fm{(p\lor q)\lor\neg(p\lor q)}$, obtained by
replacing each meta-variable with some formula are {\it instances} 
of this schema, just like e.g.\ $3^2-2^2=(3+2)(3-2)$ is an instance of the identity $x^2-y^2=(x+y)(x-y)$.

The schema itself is sometimes identified with the set of its instances.
But in any case it does not belong to the formal language --- because infinite sets of formulas, like metavariables, 
exist only in meta-language and not within the formal language.
(A finite set of formulas can often be identified with their conjunction, which is a single formula, but 
infinite conjunctions do not exist in traditional classical logic.)
In particular, when the ingredients of the formal language are interpreted in some way, we get interpretations
of formulas, but no interpretations of schemata.

Thus, even though the law of excluded middle is constantly used in classical logic, it is absent as an entity 
from both syntax and semantics in the standard modern formalism (like all other laws and inference rules
that define classical logic).
One could argue that ``only specific instances of the law of excluded middle are really used'', but this is only 
a part of the story, since to recognize some formula as an instance of the schema one does refer to the actual 
schema.

This recognition step can actually be bypassed.
There is a one-to-one correspondence between formulas and schemata given by replacing predicate variables by 
metavariables and conversely (at least if they are drawn from the same alphabet, meaning either a finite human
alphabet, for practical purposes, or a countably infinite mathematical alphabet, for theoretical ones).
Then instead of instances of a schema one can speak of {\it substitution instances} of the corresponding formula
(which are the same set of formulas).

Moreover, classical logic (and all other logics that we will consider) satisfies the substitution theorem:
if some formula is derivable in classical logic, then also all its substitution instances are derivable.
Because of this, it suffices for practical purposes to speak only of derivability of formulas, and so often
there is no practical need for a separate notion of ``derivability of schemata''; for example, all instances 
of the schema $\Big(\big((f\to g)\to f\big)\to f\Big)$ are derivable if and only if the formula
$\fm{\Big(\big((f\to g)\to f\big)\to f\Big)}$ is derivable.
(They are indeed derivable in classical logic, but this does not concern us here.)
But as one goes deeper into the modern formalism of first-order logic there appear more and more of such 
surrogates, i.e.\ writing one thing while actually thinking about another one.%
\footnote{Suffice it to say that in intuitionistic logic, the formula $\fm{\phi\lor\neg\phi}$ is not
a consequence of the formula $\fm{\neg\neg\phi\to\phi}$; but nevertheless people say and write, and for 
a good reason, that ``the principle of excluded middle'', which they express by the schema $\phi\lor\neg\phi$, is 
a consequence of ``the principle of double negation'', which express by the schema $\neg\neg\phi\to\phi$ 
(see Example \ref{ibp} and Remark \ref{ibp2} below for the details).
And people obviously do think in these terms of ``principles'', which are made illegitimate by the standard 
formalism.
Worse yet, the modern tradition of dealing with schemata, according to which ``principles'' are expressed
by schemata, is in itself a source of confusion.
The relation between the meta-language and the language is akin to the relation between syntax and 
semantics. 
Just like the arithmetical formula $\tr x+\tr y=\tr y+\tr x$ denotes any of its interpretations $0+1=1+0$, 
$2+3=3+2$, etc. (see also \S\ref{meta-quantifier} concerning the semantics of non-closed formulas), 
the schema $f\lor\neg f$ denotes any of its instances $\fm{p\lor\neg p}$, $\fm{(p\lor q)\lor\neg(p\lor q)}$, etc.
But just like the formula $\forall\tr x\,\forall \tr y\ \tr x+\tr y=\tr y+\tr x$ denotes effectively the infinite
conjunction $0+1=1+1\land 2+3=3+2\land\dots$, it would be more accurate to say that the ``principle of 
excluded middle'' is expressed not by the schema $\phi\lor\neg\phi$, but rather by the infinite 
conjunction of all its instances, which might be written symbolically as $\forall\phi\ \phi\lor\neg\phi$
(and similarly for the ``principle of double negation'').}
So it can be argued that to sacrifice the notion of ``derivability of schemata'' is to step on a slippery slope 
from a conceptual viewpoint.

\begin{remark}
Let us pause here to note that the substitution theorem does not say that each substitution instance $F'$ of 
a formula $F$ is derivable from $F$ in classical logic (and it is generally not derivable!). 
It only says that the {\it judgement} of derivability of $F$ in classical logic, denoted $\turnstile F$, implies 
the judgement $\turnstile F'$.
The difference is that both these judgements belong to the meta-theory, and so the implication between them 
is ``contentual'' rather than formal (in other words, one is derived from the other not in formal classical logic,
but in textbooks on formal classical logic).
\end{remark}

It not hard to formally include a device that amounts to the use schemata into the formal language.
Let us add a new symbol ``$\prin$'' to the formal language, and write it in front of 
a formula whenever we want to ``read it as a schema''; such expressions of the extended formal language,
consisting of a formula prefixed by ``$\prin$'', will be called {\it principles}.
(Thus the notion of a formula is in effect extended to include principles; however, to avoid confusion, we will 
keep the old meaning for the word ``formula'', using it to refer to formulas of the original language, 
without ``$\prin$''.)
That is, instead of the schema 
\[f\lor\neg f\]
(which does not belong to the formal language) or rather the infinite conjunction of its instances, 
which could be written symbolically as 
\[\forall f\ f\lor\neg f\]
(but anyway does not belong to the formal language) we will write the principle
\[\prin\fm{f\lor\neg f}\]
(which now belongs to the formal language).
Of course, the principle
\[\prin\fm{p\lor\neg p}\]
will then serve the same purpose --- but this is not a big deal, as the formula
\[\forall\tr x\,\fm r(\tr x)\]
already serves the same purpose as%
\footnote{We will see in \S\ref{formal} that in both cases this is a matter of replacing an expression of 
$\lambda$-calculus with an $\alpha$-equivalent expression, because ``$\prin$'' is really just an abbreviation
for the ``meta-quantifier closure'' (whatever it means).}
\[\forall\tr y\,\fm r(\tr y).\]
Here we encounter some further ingredients of the formal language of classical logic: the {\it quantifier} 
$\forall$, the unary {\it predicate variable} $\fm r$ and the {\it individual variables} $\tr x$, $\tr y$.
It should be noted in this connection that instances of the schema $p\lor\neg p$ also include formulas like
$\forall\tr x\,\fm r(\tr x)\lor\neg\forall\tr x\,\fm r(\tr x)$ and $\fm r(\tr x)\lor\neg\fm r(\tr x)$;
and instantiation and substitution become a bit more complicated to define precisely in the presence of 
individual variables (this will be discussed in \S\ref{formal}).

But there are more substantial questions to worry about:
\begin{enumerate}
\item Can one derive arbitrary instances of a schema from the corresponding principle, without 
resorting to meta-language?
\item Can one reasonably formalize derivability of principles?   
\item Does ``$\prin$'' have any natural semantics?
\end{enumerate}
The answers are affirmative, and will be briefly discussed here and in more detail in \S\ref{formal}.

Our way to answer the first two questions is along the lines of the well-known proof assistant {\tt Isabelle},
whose approach is in many ways closer to the early tradition of first-order logic, as e.g.\ in the textbooks by 
Hilbert--Ackermann \cite{HA} and Hilbert--Bernays \cite{HB} than to the modern formalism.

For our present purposes, we can identify the judgement of derivability of a principle $\prin F$, denoted
$\turnstile\prin F$, with the judgement $\turnstile F$.
A more reasonable definition is in \S\ref{formal}.

Now let us turn to the third question.

\subsubsection{Formal semantics}

As discussed in \S\ref{formalism-universalism} above, the elements of the formal language of 
classical logic (or any other formal language) are themselves devoid of any meaning but can be given 
meaning in an interpretation.
For example, one can interpret
\begin{itemize}
\item each propositional variable by a subset of the set $\R$ of real numbers 
(for example, $|\fm p|$ could consist of all integers, and $|\fm q|$ of all positive reals);
\item each individual variable by an element of the set $\N$ of natural numbers 
(for example, $|\tr x|=3$, $|\tr y|=5$); 
\item each $k$-ary predicate variable by an $\N^k$-indexed family of subsets of $\R$ 
(for example, $|\fm r|(n)=(-\infty,n]$); and 
\item each connective or quantifier by an appropriate set-theoretic operation; in particular, $\lor$ by taking
the union of the two given sets, $\neg$ by taking the complement of the given set in $\R$, and $\forall$ by 
taking the intersection of all members of the given $\N$-indexed family of sets.
\end{itemize}
Then each formula also gets an interpretation, defined by a straightforward recursion on its subformulas.
In our example, $\fm{|p\lor\neg q|}$ consists of all integers and all nonnegative reals, whereas 
$\fm{|q\lor\neg p|}$ consists of all positive reals and all non-integer reals; incidentally, these two sets
would have to be equal if the formulas $\fm{p\lor\neg q}$ and $\fm{q\lor\neg p}$ were interderivable.

The interpretation of individual variables, which is called an {\it individual assignment}, is needed
to interpret formulas with {\it free occurrences} of individual variables, such as 
$\fm r(\tr x)\lor\forall\tr x\,\big(\fm r(\tr x)\lor\fm r(\tr y)\big)$.
Here the first occurrence of $\tr x$ and the only occurrence of $\tr y$ are free, but the second occurrence 
of $\tr x$ is not free, because it is ``bound'' by the quantifier.
But we obviously do not need an individual assignment to interpret {\it closed} formulas, that is, those in which
no individual variable occurs freely.
For example, $|\forall\tr x\,\fm r(\tr x)|=\bigcap_{n\in\N}(-\infty,n]=(-\infty,0]$.

It is also not hard to see that all instances of the law of excluded middle, including
$\fm{p\lor\neg p}$, $\fm{q\lor\neg q}$, $\forall\tr x\,\fm r(\tr x)\lor\neg\forall\tr x\,\fm r(\tr x)$ and 
$\fm r(\tr x)\lor\neg\fm r(\tr x)$ have the same interpretation, $\R$.
Moreover, this remains true regardless of exactly how propositional, predicate and individual variables 
are interpreted.

The latter fact suggests that the principle $\fm{\prin p\lor\neg p}$ could also be assigned a definite 
interpretation: ``(semantic) truth'', denoted $\Top$.
In general, 
\begin{itemize}
\item every principle $\prin F$ is assigned an interpretation $|{\prin F}|\in\{\Top,\Bot\}$, which is 
set to be $\Top$ if for every interpretation of propositional and predicate variables that occur in $F$
and individual variables that occur freely in $F$, the interpretation of $F$ is the entire set $\R$; 
and otherwise, $\Bot$.
\end{itemize}
Thus $|\fm{\prin p\lor\neg q}|=\Bot$ (which is called ``(semantic) falsity'') since e.g.\ 
$|\fm{q\lor\neg p}|\ne\R$.

The interpretation of propositional and predicate variables, called a {\it predicate valuation}, is
needed to interpret formulas, but is not needed to interpret principles.
By an {\it interpretation of the language} of classical logic we will mean only an interpretation of connectives
and quantifiers (which includes the choices of $\N$ and $\R$).
This suffices to interpret principles.

If $|{\prin F}|=\Top$, we also say that the principle $\prin F$ is {\it valid} in our interpretation of 
the language of classical logic.
The judgement that $\prin F$ is valid in the interpretation being considered is denoted by $\Turnstile\prin F$.
It is well-known that $\Turnstile\prin F$ is equivalent to $\turnstile\prin F$ in the present example 
(see \S\ref{Euler}, where this example is discussed in more detail).
Validity of formulas with respect to a fixed predicate valuation and/or a fixed individual assignment will 
be discussed in \S\ref{formal}, but it is not needed for our present purposes.

\begin{remark} An interpretation of the language of a theory $T$ over classical logic (such as Peano Arithmetic, 
Tarski's elementary geometry or Zermalo--Fraenkel set theory) extends an interpretation of the language of 
classical logic by interpreting additional ingredients in the language of $T$: function symbols (such as those 
for the binary operations of addition and multiplication and for numeric constants) and predicate constants
(such as those for the binary predicates of equality and membership and the ternary predicate of betweenness 
for points on a line).

Connectives, quantifiers, function symbols and predicate constants are constants; whereas individual variables 
and predicate variables are variables (in the sense of $\lambda$-calculus, see \S\ref{expressions} below).
This distinction is mathematically significant, as predicate constants and predicate variables do behave 
differently even in classical logic (see examples in \S\ref{two-valued} and \S\ref{Euler}), and was clear 
in classic texts such as those of Hilbert--Ackermann \cite{HA}, Hilbert--Bernays \cite{HB}, Church \cite{Ch} 
and Novikov \cite{N}, \cite{N3}.

Nevertheless the modern tradition of first-order logic effectively conflates predicate variables and predicate 
constants under the guise of ``predicate symbols'' (or ``predicate letters'').%
\footnote{They are usually introduced on a par with function symbols, and are indeed treated as constants in 
the context of theories --- in which case the substitution theorem does not apply (unless all axioms are
given by schemata).
But when dealing with a bare logic and its models, ``predicate symbols'' can often be substituted with
formulas, which one would not normally do to constants; and one often assumes a countable list of 
``predicate symbols'' of each arity, none of them standing for anything specific (as e.g.\ in the texts 
of Kleene \cite{Kl} and Troelstra--van Dalen \cite{TV}), which is also a strange way of dealing with constants.
But it should be noted that in the case of propositional logic, ``propositional symbols'' are clearly 
distinguished from the (syntactic) propositional constants $\top$ and $\bot$; and in logic with equality, 
the equality is clearly distinguished from ``predicate symbols''.
The difference is that $\top$, $\bot$ and $=$ are not supposed to be substitutable by formulas.}
Consequently, the distinction between the language of a logic and the language of a theory is also blurred,
and hence an interpretation of such language is assumed to include an interpretation of ``predicate symbols'' in 
modern treatments of first-order logic --- even though modern treatments of propositional (=zero-order) logic
usually clearly separate valuation of propositional variables from interpretation of connectives.  
\end{remark}

\subsubsection{Informal semantics} \label{contentual}

Let us now look at another way of interpreting the ingredients of the language of classical logic:
\begin{itemize}
\item each propositional variable is interpreted by a contentful (e.g.\ mathematical) proposition; 
\item each individual variable by an element of a set $\D$, called the {\it domain of discourse};
\item each $k$-ary predicate variable by a $k$-ary contentful predicate on $\D$;
\item each formal connective or quantifier is interpreted according to the usual truth tables 
(see \S\ref{Tarski truth}), or in other words by the corresponding contentual%
\footnote{``Contentual'' (or rather its more widely used German equivalent {\it inhaltlich}) is Hilbert's 
term, the opposite of ``formal'', which he used to explain the Formalist program.}
connective or quantifier ($\lor$ by ``or'', $\neg$ by ``not'', $\forall$ by ``for all'', etc.);
\item a principle $\prin F$ is interpreted by the {\it meta-proposition} (that is, a judgement about 
propositions, itself treated as a higher-level proposition) asserting that for every interpretation 
of propositional and predicate variables that occur in $F$ and individual variables that occur freely in $F$, 
the contentful proposition that interprets $F$ is true;%
\footnote{In more detail, let $\vec x=(x_1,\dots,x_n)$ be the tuple of all individual variables that occur 
freely in $F$, let $\vec p=(p_1,\dots,p_m)$ be the tuple of all propositional and predicate variables that 
occur in $F$, and let $\vec r=(r_1,\dots,r_m)$ be the tuple of their arities.
Then $\prin F$ is interpreted by the meta-proposition
$\forall\vec P\,\forall\vec X\,(\vec p,\vec x\mapsto\Phi)(\vec P,\vec X)$
asserting the truth of all the contentful propositions 
$(\vec p,\vec x\mapsto F)(\vec P,\vec X)$ obtained from $F$ by substituting $\vec x$ by
an $n$-tuple $\vec X$ of elements of $\D$ and $\vec p$ by an $m$-tuple $\vec P$ of contentful predicates 
of arities $r_1,\dots,r_m$ on $\D$, and then applying the usual truth tables to interpret the connectives 
and quantifiers.}
\item $\Turnstile\prin F$ asserts that the meta-proposition that interprets $\prin F$ is true.
\end{itemize}
This is not really a definite interpretation but rather an informal sketch of what could possibly be made
into a definite interpretation.
To be sure, this can be done; for instance:
\begin{itemize}
\item each propositional variable is interpreted by a closed formula in the language of Peano Arithmetic (PA);
\item each individual variable by a closed term of PA (such as $(1+1)*(1+1)+1$);
\item each $k$-ary predicate variable by a formula $F(\tr x_1,\dots,\tr x_k)$ in the language of PA such that 
$\forall\tr x_1\dots\forall\tr x_k\, F(\tr x_1,\dots,\tr x_k)$ is a closed formula;
\item each formal connective or quantifier by the namesake connective or quantifier in the language 
of PA, or rather of its underlying (classical) logic;
\item a principle $\prin F$ is interpreted by the formula of the language of PA which expresses, using
the G\"odel numbering, the assertion  that for every interpretation of propositional and predicate variables 
that occur in $F$ and individual variables that occur freely in $F$, the formula that interprets $F$ 
is derivable in PA;
\item $\Turnstile\prin F$ asserts that the formula that interprets $\prin F$ is true in the standard model $\N$
of PA.
\end{itemize}

In this example, if $|\fm p|$ is a formula expressing that there exist infinitely many twin primes and 
$|\fm q|$ is G\"odel's formula, ``asserting its own non-derivability'', then $\fm{|p\lor\neg p|}$ and 
$\fm{|q\lor\neg q|}$ are different formulas (in contrast to the previous example), though both are 
derivable in PA.
Incidentally, that neither $|\fm q|$ nor $|\neg\fm q|$ is derivable in PA is nothing special, because
already in the previous example neither $|\fm q|$ nor $|\neg\fm q|$ was the entire $\R$.
Also, $|\fm{\prin p\lor\neg p}|$ and $|\fm{\prin \neg\neg p\to p}|$ are different formulas (in contrast
to the previous example), though both are true in $\N$ (and derivable in PA, though this is less obvious 
and does not matter for the purposes of our interpretation).

Admittedly, this example is silly enough in that it can be said to interpret formal classical logic by means of
formal classical logic --- except that the new ingredient ``$\prin$'' is interpreted in a more meaningful way.
Indeed, derivability of formulas in PA is no longer the business of PA and its underlying logic, but rather
of the meta-theory of PA and its underlying (meta-)logic --- which has been formalized by G\"odel within PA
and its underlying logic.
Thus the interpretation of $\prin$ takes us from ``propositions'' to ``meta-propositions'' (which are realized 
here as formalized judgements about formalized propositions).
But this is not all.
Validity of arithmetical formulas in the standard model of PA is also no longer the business of PA.
Thus the the definition of validity takes us further from meta-propositions to judgements about 
meta-propositions (which are this time left at an informal level).

Returning to the informal sketch above, let us emphasize that an informal interpretation like above is
in general not necessarily formalizable in a reasonable way; but if it happens to be formalizable, there 
ought to some flexibility in how it can be formalized.
(In our example, we could look at a non-standard model of PA, or at Tarski's plane geometry along with 
its various models, or at the set theory ZFC along with its various models.)
Also, the informal sketch above seems to be adding nonzero value to its particular formalization above, 
just like reading an informal sketch/idea of proof often turns out to be more helpful than going through 
the details of a carefully written, but unmotivated formal proof. 
Because of this, in what follows we will occasionally still resort to informal ``contentual'' interpretations 
on a par with the sketch above; but it should always be clear what is supposed to be rigorous and what is not.

Let us note that according to the informal sketch above, the principle of excluded middle,
\[\fm{\prin p\lor\neg p},\]
is interpreted by the meta-proposition asserting that {\it for each contentful proposition $p$, the proposition
$p\lor\neg p$ is true}.
The judgement 
\[\Turnstile\fm{\prin p\lor\neg p}\] 
asserts that this meta-proposition is true.
If we ignore the subtleties of different levels of reflection, which are anyway far from clear in such 
vague formulations, we can roughly express both the informal interpretation of the principle of excluded
middle and the judgement of its validity as
\[\forall P\ P\lor\neg P,\]
where $P$ is understood to run over contentful (e.g.\ mathematical) propositions.

\subsection{Clarified BHK interpretation}\label{clarified}

Let us now try to do the same with intuitionistic logic and with the principle of decidability,
\[\prin\fm{\gamma\lor\neg\gamma}.\]
Here $\fm\gamma$ is a propositional variable in the formal language of intuitionistic logic, which we denote 
by a Greek letter in order to avoid confusion with the language of classical logic.
In fact, in view of our previous discussions, it makes sense to
call $\fm\gamma$ a {\it problem variable}, because intuitionistic logic is supposed to be about problems 
rather than propositions (at least, certainly not about propositions in the usual sense of classical 
mathematics).

At this point, we are completely ignorant about the laws and inference rules of intuitionistic logic, and 
know little about its semantics.
But at least we could hope that an informal interpretation of the principle of decidability parallel to 
the above informal interpretation of the principle of excluded middle can be roughly expressed as 
\[\forall\Gamma\ \Gamma\lor\neg\Gamma,\]
where $\Gamma$ is understood to run over contentful (e.g.\ mathematical) problems.
But what does this expression mean? 
If we read it by analogy with the $\forall$ clause in the BHK interpretation, then it expresses 
{\it the (meta-)problem of finding a general method of solving the problem $\Gamma\lor\neg\Gamma$ for 
every contentful problem $\Gamma$.}
And this is precisely Kolmogorov's original interpretation of the principle of decidability \cite{Kol},
except that his use of the symbol ``$\turnstile$'' where we use the symbol ``$\prin$'' raises some issues 
(which we will discuss in a moment).

If we now turn to the judgement 
\[\Turnstile\prin\fm{\gamma\lor\neg\gamma}\]
that the principle of decidability is valid in our informal interpretation, the natural way to express it
is that {\it there exists} a solution of the said meta-problem; that is, {\it there exists a general method of 
solving the problem $\Gamma\lor\neg\Gamma$ for every contentful problem $\bGamma$.}
Using the notation of \S\ref{about-bhk}, this can be roughly written as
\[\wn\forall\Gamma\ \Gamma\lor\neg\Gamma.\]
And this is reminiscent of Brouwer's {\it ``Third insight.} The identification of the principle of
excluded middle with the principle of decidability of every mathematical problem.''%
\footnote{Brouwer's four ``insights'' intended to summarize his work in intuitionism; the third
is the shortest, and essentially reiterates after 20 years the identification discussed in
\S\ref{conjectures} --- except that this ``Third insight'' taken alone can be seen as a {\it definition}
of the principle of excluded middle.
The ``principle of decidability of every mathematical problem'', originally formulated by Hilbert 
(see \S\ref{Kolmogorov}), hardly intended to assert a {\it general method} for settlement of all 
mathematical problems.
However, it is likely that Brouwer, with his philosophy that all mathematical statements should be
read constructively, did actually mean this principle to assert a {\it general method} --- at least
in the later paper \cite{Br28}.
See, for instance, footnote 6 in \cite{Br28}, which guarantees this sort of reading for another
principle.}
\cite{Br28}*{\S1}.

\begin{remark} Let us now discuss Kolmogorov's original approach \cite{Kol}.
It is conceivable that he did not really mean his semantic definition of $\turnstile$ (which coincides with
our semantic interpretation of $\prin$) to serve as an interpretation of Heyting's syntactic definition 
of $\turnstile$ (``To indicate that a formula is included in the list of `correct formulas' the sign
$\turnstile$ will be put in front of it'' \cite{He4}).
This reading is supported by the fact Kolmogorov separately lists ``{\it rules} of our calculus of problems'',
which are themselves not problems.
In this case our interpretation above can be seen simply as a slight extension of Kolmogorov's.

However, Kolmogorov's words ``This explanation of the meaning of the sign $\turnstile$ is quite different 
from that of Heyting, even though it leads to the same rules of the calculus'' \cite{Kol} suggest that
he could have actually meant one to interpret the other.
In this case our interpretation above is rather a ``correction'' of Kolmogorov's, since in our approach
the judgement $\turnstile\prin\fm{\gamma\lor\neg\gamma}$ is naturally interpreted by the judgement
$\Turnstile\prin\fm{\gamma\lor\neg\gamma}$, and not by the meta-problem that interprets $\prin\fm{\gamma\lor\neg\gamma}$.
It should be noted here that Heyting, in his review of Kolmogorov's problem interpretation \cite{He2}, 
already chose to interpret a formula prefixed by the $\turnstile$ symbol as a judgement and not a problem 
(though without any mention of a general method).

The ``correction'', if it indeed occurs, may have some philosophical implications.
Kolmogorov has claimed that in his calculus of problems, ``there is no need for any special, e.g.\ 
intuitionistic, epistemological presuppositions'' (cf.\ his quote in \S\ref{Kolmogorov}).
But he arguably did not quite achieve this stated goal, because he appears to implicitly refer to 
the same epistemological presuppositions as Heyting.%
\footnote{According to Kolmogorov \cite{Kol}, ``the formula [$\turnstile a\lor\neg a$]
reads as follows: to give a general method that for every problem $a$
either finds a solution of $a$ or deduces a contradiction on the basis of
such a solution!
In particular, when the problem $a$ consists in proving a proposition,
one must possess a general method to either prove or reduce to
a contradiction any proposition.
If the reader does not consider himself to be omniscient, he will probably
determine that [$\turnstile a\lor\neg a$] cannot be on the list of
problems solved by him.''
Let us note that this effectively appeals to an assumption regarding the reader's knowledge.
Elsewhere in \cite{Kol} Kolmogorov introduces a list of axioms of intuitinistic logic
by stating that
``we must assume that the solutions of some elementary problems are already known.
We take as postulates that we already have solutions to the following groups A and B
of problems. The subsequent presentation is directed only to the reader who has
solved all these problems''.
Here Kolmogorov again clearly assumes something regarding the reader's knowledge.
In fact, his assumption that solutions of certain problems beginning with $\turnstile$
{\it are already known} is not so different from Heyting's epistemological definition of
$\turnstile$ \cite{He0}:
``To satisfy the intuitionistic demands, the assertion must be
the observation of an empirical fact, that is, of the realization of
the expectation expressed by the proposition $p$.
Here, then, is the {\it Brouwerian assertion} of $p$: {\it
It is known how to prove $p$.} We will denote this by $\turnstile p$. The words `to prove' must be taken
in the sense of `to prove by construction'.''}
It is also arguable that our choice of the judgement of existence instead of Kolmogorov's implicit or
Heyting's explicit judgements of knowledge does provide independence from ``epistemological presuppositions'', 
even if gained by the price of what Heyting called ``faith in transcendental existence''.

On a more practical level, the switch from knowledge to existence corresponds to the use of classical, rather
constructive meta-logic (or meta-meta-logic in the setting of \S\ref{formal}).
\end{remark}

Let us now extend the above to the following informal interpretation of the ingredients of the language of 
intuitionistic logic:
\begin{itemize}
\item each individual variable is interpreted by an element of a set $\D$;
\item each $k$-ary problem variable is interpreted by a contentful (e.g.\ mathematical) problem
with $k$ parameters, each running over $\D$; 
\item each formal connective or quantifier is interpreted according to the usual BHK interpretation 
(see \S\ref{BHK});
\item a principle $\prin\Phi$ is interpreted by the meta-problem of finding a general method of solving all 
the contentful problems that result from all possible interpretations of all problem variables that occur 
in $\Phi$ and all individual variables that occur freely in $\Phi$;%
\footnote{In more detail, let $\vec x=(x_1,\dots,x_n)$ be the tuple of all individual variables that occur 
freely in $\Phi$, let $\vec\gamma=(\gamma_1,\dots,\gamma_m)$ be the tuple of all problem variables that 
occur in $\Phi$, and let $\vec r=(r_1,\dots,r_m)$ be the tuple of their arities.
Then $\prin\Phi$ is interpreted by the meta-problem
$\forall\vec\Gamma\,\forall\vec X\,(\vec\gamma,\vec x\mapsto\Phi)(\vec\Gamma,\vec X)$
of finding a general method of solving all the contentful problems 
$(\vec\gamma,\vec x\mapsto\Phi)(\vec\Gamma,\vec X)$ obtained from $\Phi$ by substituting $\vec x$ by
an $n$-tuple $\vec X$ of elements of $\D$ and $\vec\gamma$ by an $m$-tuple $\vec\Gamma=(\Gamma_1,\dots,\Gamma_m)$ 
of contentful problems, where each $\Gamma_i$ has $r_i$ parameters running over $\D$, and then applying 
the usual BHK interpretation to interpret the connectives and quantifiers.}
\item $\Turnstile\prin\Phi$ asserts that the meta-problem that interprets $\prin\Phi$ has a solution.
\end{itemize}

To fix a name, we will call this the {\it clarified BHK} (or briefly cBHK) interpretation of 
intuitionistic logic.

\subsection{The principle of decidability and its relatives}\label{PEM}

So is the principle of decidability valid on the clarified BHK interpretation?
That is, does there exist a general method%
\footnote{One can also wonder about Hilbert's conviction, that $\bGamma\lor\neg\bGamma$ admits
a (possibly {\it ad hoc}) solution for every contentful problem $\bGamma$.
This judgement, which can be expressed as $\forall\Gamma\,\wn(\Gamma\lor\neg\Gamma)$ in the above notation, 
falls outside of the scope of intuitionistic logic on the clarified BHK interpretation,.
A certain extension QHC of both intuitionistic and classical logics, which includes additional connectives $\wn$ 
and $\oc$ (and no further connectives) is studied in the author's papers \cite{M1}, \cite{M2}, \cite{M3}.
Hilbert's conviction is naturally seen as an interpretation of the principle 
$\prin\wn(\fm\gamma\lor\neg\fm\gamma)$ in the language of QHC, which turns out to be an independent principle 
with respect to QHC.}
of solving the problem $\bGamma\lor\neg\bGamma$ for every contentful problem $\bGamma$?

If $G_p$ is a group given by a finite presentation $p$ (i.e.\ a finite list of generators and
relations), let $\bGamma_p$ be the problem, {\it Find an isomorphism between $G_p$ and
the trivial group $1$}.
Then from any solution of $\bGamma_p$ we can extract a proof that $G_p\simeq 1$ (by Kreisel's
clarification), whereas from any solution of $\neg\bGamma_p$ we can extract a proof that
$G_p\not\simeq 1$ (by our solution of $\neg\bGamma_p\,\to\,\oc\neg\wn\bGamma_p$).
Hence any solution of $\bGamma_p\lor\neg\bGamma_p$ would tell us,
in particular, whether $G_p$ is isomorphic to $1$ or not (since any
solution of $\bGamma\lor\bDelta$ involves, in the first place,
an explicit choice between $\bGamma$ and $\bDelta$).

Now if we have a general method of solving all problems of the form $\bGamma\lor\neg\bGamma$,
then in particular we have a general method of solving the problem $\bGamma_p\lor\neg\bGamma_p$
for all values of $p$.
If the latter general method is interpreted as an {\it algorithm} (in the sense of Turing machines)
with input $p$ (which is perfectly consistent with Kolmogorov's explanation of a general method),
then such a general method would yield an algorithm deciding whether $G_p$ is isomorphic to $1$.
But it is well-known that there exists no such algorithm.%
\footnote{The geometrically-minded reader might prefer the modification of this example based of
S. P. Novikov's theorem
(improving on an earlier result of A. A. Markov, Jr.), that for a certain sequence of finite
simplicial complexes $K_1,K_2,\dots$ there exists no algorithm to decide, for any positive integer
input $n$, whether $K_n$ is piecewise-linearly homeomorphic to the $5$-dimensional sphere
(see \cite{CL}).}
This shows that the principle of decidability is not valid under the clarified BHK interpretation:
$$\not\Turnstile\prin\fm\gamma\lor\neg\fm\gamma.$$
In the next section (\S\ref{Medvedev-Skvortsov}) we will review a much more elementary example, which 
shows incidentally that it is not really necessary to interpret general methods constructively, 
as algorithms.

Let us pause to note that we have {\it not} established validity, on the clarified
BHK interpretation, of the negated principle of decidability, $\prin\neg\fm\gamma\lor\neg\fm\gamma$.
For that we would need a solution of $\forall\Gamma\,\oc\neg\wn(\Gamma\lor\neg\Gamma)$, that is, a 
general proof (=a general method to prove) that each contentful problem of the form
$\bGamma\lor\neg\bGamma$ has no solutions.
But there exists no such general proof, since, in fact, any problem of the form
$(\bDelta\to\bDelta)\lor\neg(\bDelta\to\bDelta)$ does have a solution.
Thus we see, quite trivially, that
$$\not\Turnstile\prin\neg(\fm\gamma\lor\neg\fm\gamma)$$ on the clarified BHK interpretation.
In fact, a stronger assertion turns out to hold:
$$\Turnstile\prin\neg\neg(\fm\gamma\lor\neg\fm\gamma);$$
thus on the clarified BHK interpretation, there exists a general proof (we will describe it explicitly
in \S\ref{tautologies}, (\ref{not-not-LEM})) that for each problem $\bGamma$ there exists no proof that
$\bGamma\lor\neg\bGamma$ has no solutions.
Yet on the other hand, we have just seen a proof that the problem
$\forall p\,(\bGamma_p\lor\neg\bGamma_p)$ has no solutions.
In particular, our former argument works also to establish
$$\not\Turnstile\prin\neg\neg\forall\tr x\,\big(\fm\gamma(\tr x)\lor\neg\fm\gamma(\tr x)\big)$$
on the clarified BHK interpretation; this was originally observed by Brouwer, who presented
``counterexamples to the freedom from contradiction of the Multiple Principle of Excluded Middle of
the second kind'' \cite{Br28}*{\S2}, and now is better known as independence of the Double Negation
Shift principle (see \S\ref{shift}).

\subsection{Medvedev--Skvortsov problems}\label{Medvedev-Skvortsov}
Let us fix a set $X$, and consider the class of problems $\bGamma_f$ of the form:
{\it Solve the equation $f(x)=0$}, where $f\:X\to\{0,1\}$ is an arbitrary (set-theoretic) function.
Thus a solution of $\bGamma_f$ is any $x\in X$ such that $f(x)=0$.
Such problems will be called {\it Medvedev--Skvortsov problems} with domain $X$.
Of course, $f$ is determined by the pair of sets $\big(X,f^{-1}(0)\big)$.
Thus we may write a pair of sets, $(X,Y)$, to encode the problem $\bGamma_f$, where $f\:X\to\{0,1\}$ is 
such that $f^{-1}(0)=Y$.
It is this language of pairs that was used by Medvedev \cite{Me1}, \cite{Me2}, Skvortsov \cite{Skv}
and L\"auchli \cite{Lae} (see also \cite{ShSk}).
If $\D$ is a fixed set (the domain of discourse) and $k\in\N$, we may also consider $k$-parameter families of
Medvedev--Skvortsov problems with a fixed domain $X$.
Every such family is encoded by a $k$-parameter family of pairs of sets $\big(X,Y(x_1,\dots,x_k)\big)$, 
where $x_1,\dots,x_k$ run over $\D$.
Here $Y$ is a function from $\D^k$ to the set $2^X$ of all subsets of $X$.

Let us consider the following interpretation of the ingredients of the language of intuitionistic logic.

\begin{itemize}
\item each nullary problem variable is interpreted by a pair of sets $(X,Y)$;
\item each individual variable is interpreted by an element of $\D$;
\item each $k$-ary problem variable is interpreted by a $k$-parameter family of pairs sets of the form
$\big(X,Y(t_1,\dots,t_k)\big)$, where each $t_i$ runs over $\D$;
\item each formal connective or quantifier is interpreted as follows:
\begin{itemize}
\item $(X,Y)\,|{\land}|\,(X',Y')=(X\x X',\,Y\x Y')$;
\item $(X,Y)\,|{\lor}|\,(X',Y')=(X\sqcup X',\,Y\sqcup Y')$;
\item $(X,Y)\,|{\to}|\,(X',Y')=\big(\Hom(X,X'),\,\{f\:X\to X'\mid f(Y)\subset Y'\}\big)$;
\item $|{\ab}|=(\{\emptyset\},\emptyset)$;
\item $|{\exists}| t\, \big(X,Y(t)\big)=\big(\D\x X,\,\{(d,x)\in\D\x X\mid x\in Y(d)\}\big)$;
\item $|{\forall}| t\, \big(X,Y(t)\big)=\big(\Hom(\D,X),\,\{f\:\D\to X\mid \forall d\in\D\ f(d)\in Y(d)\}\big)$.
\end{itemize}
\item a principle $\prin\Phi$ that contains occurrences of precisely $m$ problem variables is interpreted
by a function $|{\prin\Phi}|$ that assigns to every $m$-tuple of sets $X_1,\dots,X_m$ a pair of sets 
$(S,S')$, defined as follows.
Let $x_1,\dots,x_n$ be all individual variables that occur freely in $\Phi$, let $\gamma_1,\dots,\gamma_m$ 
be all problem variables that occur in $\Phi$, and let $r_1,\dots,r_m$ be their arities.
Upon substituting each $x_i$ by some $d_i\in\D$, and each $\gamma_i$ by an $r_i$-parameter family
of pairs sets of the form $\big(X_i,Y_i(t_1,\dots,t_{r_i})\big)$, and applying the above interpretation of 
the connectives and quantifiers, we obtain an interpretation of $\Phi$ by a pair or sets $(S,T)$.
Here $S$ depends only on $\vec X=(X_1,\dots,X_m)$ and $T=T(\vec d,\vec Y)$ is a subset of $S$ that depends 
additionally on $\vec d=(d_1,\dots,d_n)$ and $\vec Y=(Y_1,\dots,Y_m)$.
The desired pair $|{\prin\Phi}|(\vec X)$ is $(S,S')$, where $S'$ is the intersection of the sets
$T(\vec d,\vec Y)$ over all $(\vec d,\vec Y)\in\D^n\x\Hom(\D^{r_1},2^{X_1})\x\dots\x\Hom(\D^{r_m},2^{X_m})$.
\item $\Turnstile\prin\Phi$ asserts that the function $|{\prin\Phi}|$ assigns to every tuple of sets
a pair $(S,S')$ with nonempty $S'$.
\end{itemize}

This is very much in the spirit of the clarified BHK interpretation.
Indeed, let us recall that a pair of sets $(X,Y)$ is actually supposed to encode a problem whose set of 
solutions is $Y$.
Thus let us write $\KK(X,Y)\KK=Y$.
If we denote pairs of sets by Greek letters, then we obtain, just like in \S\ref{confusion}: 
\begin{itemize}
\item $\KK\bGamma\,|{\land}|\,\bDelta\KK=\KK\bGamma\KK\x\KK\bDelta\KK$;
\item $\KK\bGamma\,|{\lor}|\,\bDelta\KK=\KK\bGamma\KK\sqcup \KK\bDelta\KK$;
\item there is a map
$\script F\:\KK\bGamma\,|{\to}|\,\bDelta\KK\to\Hom(\KK\bGamma\KK,\KK\bDelta\KK)$
(actually, a surjection);
\item $\KK|{\ab}|\KK=\emptyset$;
\item $\KK|{\exists}| t\,\bGamma(t)\KK=\bigsqcup_{d\in\D} \KK\bGamma(d)\KK$;
\item there is a map
$\script G\:\KK|{\forall}| t\,\bGamma(t)\KK\to\prod_{d\in\D} \KK\bGamma(d)\KK$
(actually, a bijection).
\end{itemize}

Moreover, a ``general method'' of solving the problems encoded by $\big(S,T(\vec d,\vec Y)\big)$ for all 
$(\vec d,\vec Y)$ is naturally interpreted as a common solution to all these problems, i.e., an element 
in the intersection of all the sets $T(\vec d,\vec Y)$.
And then of course such a general method exists if and only if the intersection is nonempty.
Thus the Medvedev--Skvortsov interpretation of intuitionistic logic is arguably a formalization of 
the clarified BHK interpretation.

Let us now consider the interpretation of the principle of decidability, $|{\prin\fm\gamma\lor\neg\fm\gamma}|$.
It is a function assigning to a set $X$ the pair of sets $(S,\bigcap_{Y\subset X} T(Y))$, where
\[\big(S,T(Y)\big)=(X,Y)\,|{\lor}|\,|{\neg}|(X,Y).\]
We have $S=X\sqcup\Hom(X,\{\emptyset\})$ and $T(Y)=Y\sqcup\{f\:X\to\{\emptyset\}\mid f(Y)\subset\emptyset\}$,
that is, 
\[T(Y)=
\begin{cases}
Y\sqcup\emptyset,&\text{if }Y\ne\emptyset;\\
\emptyset\sqcup\Hom(X,\{\emptyset\}),&\text{if }Y=\emptyset.
\end{cases}\]
Thus $\bigcap_{Y\subset X} T(Y)=\emptyset$.
So the principle of decidability is not valid under the Medvedev--Skvortsov interpretation.

\subsection{Some intuitionistic validities}\label{tautologies}

The point of the clarified BHK interpretation is that it provides a reasonable explanation of intuitionistic
validities, without relying on any formal system of axioms and inference rules.
Although this explanation is highly informal, it works --- and is often more helpful than
any formal calculus if one needs to quickly verify whether a given formula is an intuitionistic
validity.%
\footnote{Why do these informal BHK-arguments work?
One possible answer is that they do just because they are in fact sketches of proofs that the principles
in question are satisfied in the sheaf-valued models of \S\ref{sheaves}.
Intuitionistic logic is shown to be complete with respect to this class of models in \S\ref{sheaves},
so by formalizing these sketches one would indeed
establish that the principles in question are intuitionistic validities.
Alternatively, one can interpret the informal BHK-arguments as textual representations of $\lambda$-terms
(see \cite{Ho}, \cite{SU}*{\S4}, \cite{TV}*{\S10.8}) or, equivalently, as sketches of proofs that
the principles in question are satisfied in L\"auchli's models \cite{Lae}.
In fact, there seems to be no significant difference between the two alternatives (see \cite{Aw00}).}

The order of precedence of connectives, quantifiers and the symbol $\prin$ is (in groups of equal priority, 
starting with higher precedence/stronger binding): 1) $\neg$, $\exists$ and $\forall$;
2) $\land$ and $\lor$; 3) $\to$ and $\tofrom$; 4) $\prin$.

Another convention (perhaps obvious): we do not omit parameters of problem variables, e.g.\ in the formula
$\fm\forall\tr x\,\big(\fm\alpha\lor\fm\beta(\tr y)\lor\fm\gamma(\tr x,\tr y)\big)$,
$\fm\alpha$ is a nullary problem variable, $\fm\beta$ is a unary one and $\fm\gamma$ is a binary one.
(This agrees with the notation of \S\ref{formal}.)

\subsubsection{Basic validities}

\begin{enumerate}
\item\label{double negation} $\Turnstile\prin\fm{\gamma\to\neg\neg \gamma}$
\end{enumerate}
Given a problem $\bGamma$ and a solution $s$ of $\bGamma$,
we need to produce a solution of $\neg\neg\bGamma$ by a general method.
Indeed, given a solution of $\neg\bGamma$, that is, a method turning
solutions of $\bGamma$ into solutions of $\ab$, we simply apply this
method to $s$ and get a contradiction.

\begin{enumerate}[resume]
\item\label{contrapositive} (contrapositive)
$\Turnstile\prin\fm{ (\gamma\to \delta)\To(\neg \delta\to\neg \gamma)}$
\end{enumerate}
Indeed, from $\bGamma\to\bDelta$ and $\bDelta\to\ab$ we infer
$\bGamma\to\ab$.

\begin{enumerate}[resume]
\item\label{triple negation} $\Turnstile\prin\fm{\neg \gamma\Tofrom\neg\neg\neg \gamma}$
\end{enumerate}
This follows from (\ref{double negation}): the ``$\to$'' implication by substitution,
and the ``$\from$'' implication via (\ref{contrapositive}).

\begin{enumerate}[resume]
\item $\Turnstile\prin\fm{\neg\neg\ab\Tofrom\ab}$
\end{enumerate}

Here the ``$\from$'' implication is a special case of (\ref{double negation}), and the ``$\to$'' implication
can be rewritten as $\Turnstile\prin\fm{\neg\neg\neg\ab}$, which by (\ref{triple negation}) is equivalent to
$\Turnstile\prin\fm{\neg\ab}$.
The latter is a special case of $\Turnstile\prin\fm{\gamma\to\gamma}$ (where the problem $\bGamma\to\bGamma$
has an obvious solution).

\begin{enumerate}[resume]
\item\label{explosion0} (explosion)
$\Turnstile\prin\fm{\ab \to \gamma}$
\end{enumerate}
If $s$ is a solution of $\ab$, then $s$ does not exist; in particular, $s$ is itself also
a solution of any given problem $\bGamma$.

\begin{enumerate}[resume]
\item\label{implication0}
$\Turnstile\prin\fm{\neg\delta\lor\gamma\To(\delta\to\gamma)}$
\end{enumerate}
Given a solution of $\neg\bDelta\lor\bGamma$ and a solution of $\bDelta$,
we get a solution of $\ab\lor\bGamma$, hence a solution of $\bGamma\lor\bGamma$.
This yields a solution of $\bGamma$ by considering two cases.

\begin{enumerate}[resume]
\item\label{implication0'}
$\Turnstile\prin\fm{\beta\lor\gamma\To(\neg\beta\to\gamma)}$
\end{enumerate}
This follows from (\ref{implication0}) with $\fm\delta$ substituted by $\neg\fm\beta$, 
using (\ref{double negation}).

\begin{enumerate}[resume]
\item\label{decidable-stable}(decidability implies stability)
$\Turnstile\prin\fm{\gamma\lor\neg \gamma\To(\neg\neg\gamma\to \gamma)}$
\end{enumerate}
This follows from (\ref{implication0}) with $\fm\delta$ substituted by $\neg\neg\fm\gamma$,
using (\ref{triple negation}).

\subsubsection{Quantifiers}

\begin{enumerate}[resume]
\item\label{quantifier interchange}
$\Turnstile\prin\fm{\exists \tr x\,\forall \tr y\,\gamma(\tr x,\tr y)\To
\forall \tr y\,\exists \tr x\,\gamma(\tr x,\tr y)}$
\end{enumerate}
Indeed, if we found an $x_0$ and a method to turn every $y$ into
a solution of $\bGamma(x_0,y)$, then we have a method to produce for
every $y$ an $x$ and a solution of $\bGamma(x,y)$.

Here is a good place to note that, like in classical logic, $\exists$
and $\forall$ can be thought of as generalizations of $\lor$ and $\land$.
Thus similarly to (\ref{quantifier interchange}) we get
\begin{enumerate}[resume]
\item $\Turnstile\prin\fm{\exists \tr x\, \big(\gamma(\tr x)\land\delta(\tr x)\big)\To
\exists \tr x\,\gamma(\tr x)\land\exists \tr x\,\delta(\tr x)}$
\item\label{q1} $\Turnstile\prin\fm{\forall \tr y\, \gamma(\tr y)\lor\forall \tr y\,\delta(\tr y)\To
\forall \tr y\,\big(\gamma(\tr y)\lor\delta(\tr y)\big)}$
\end{enumerate}
Intuitionistic logic features an additional connection: $\exists$ and
$\forall$ behave as if they were generalizations of $\land$, and $\to$,
respectively.%
\footnote{This connection is made precise in dependent type theory, where no distinction is made between 
the domain of a variable, like $\D$ in the BHK clauses for $\exists$ and $\forall$, and the set of
solutions of a problem, like that of $\bGamma$ in the BHK clauses for $\land$ and $\to$.}
Thus similarly to (\ref{quantifier interchange}) we also get
\begin{enumerate}[resume]
\item\label{q2}
$\Turnstile\prin\fm{\exists \tr x\,\big(\theta\to\gamma(\tr x)\big)\To
\big(\theta\to\exists \tr x\,\gamma(\tr x)\big)}$%
\footnote{The converse implication is known as the Principle of Independence
of Premise; see \S\ref{Harrop1}.}
\end{enumerate}
Moreover, similarly to the obvious validities:
$\Turnstile\prin\fm{\forall \tr x\,\forall \tr y\,\gamma(\tr x,\tr y)\Tofrom
\forall \tr y\,\forall \tr x\,\gamma(\tr x,\tr y)}$ and
$\Turnstile\prin\fm{\exists \tr y\,\exists \tr x\,\gamma(\tr x,\tr y)\Tofrom
\exists \tr x\,\exists \tr y\,\gamma(\tr x,\tr y)}$
we get something otherwise not so obvious:
\begin{enumerate}[resume]
\item\label{q3}
$\Turnstile\prin\fm{\forall \tr x\,\big(\theta\to\gamma(\tr x)\big)\Tofrom
\big(\theta\to\forall \tr x\,\gamma(\tr x)\big)}$
\item\label{q4}
$\Turnstile\prin\fm{\theta\land\exists \tr x\,\gamma(\tr x)\Tofrom
\exists \tr x\,\big(\theta\land\gamma(\tr x)\big)}$
\end{enumerate}
In addition, by specializing (\ref{q1}) we get
\begin{enumerate}[resume]
\item\label{q5}
$\Turnstile\prin\fm{\theta\lor\forall \tr x\,\gamma(\tr x)\To\forall \tr x\,\big(\theta\lor\gamma(\tr x)\big)}$%
\footnote{The converse implication is known as the Constant Domain Principle; see 
\S\ref{constant domain} below.}
\end{enumerate}

There are two more validities of the above sort.
Firstly,
\begin{enumerate}[resume]
\item\label{gq1}
$\Turnstile\prin\fm{\forall \tr x\,\big(\gamma(\tr x)\to\theta\big)\Tofrom
\big(\exists \tr x\, \gamma(\tr x)\to\theta\big)}$
\end{enumerate}
Indeed, a solution of $\forall x\,\big(\bGamma(x)\to\bTheta\big)$
turns each $x$ into a solution of $\bGamma(x)\to\bTheta$.
A solution of $\exists x\, \bGamma(x)\to\bTheta$ turns
any given $x_0$ and solution of $\bGamma(x_0)$ into a solution of $\bTheta$.
These are obviously reducible to each other.

Secondly,
\begin{enumerate}[resume]
\item\label{gq2}
$\Turnstile\prin\fm{\exists \tr x\,\big(\gamma(\tr x)\to\theta\big)\To
\big(\forall \tr x\, \gamma(\tr x)\to\theta\big)}$%
\footnote{The converse implication is an independent principle; see \S\ref{Harrop1}.}
\end{enumerate}
Indeed, given an $x_0$ and a solution of $\bGamma(x_0)\to\bTheta$,
and assuming a method of solving $\bGamma(x)$ for every $x$,
we apply this method to $x=x_0$ to get a solution of
$\bGamma(x_0)$ and consequently a solution of $\bTheta$.

\begin{enumerate}[resume]
\item\label{quantifiers1}
$\Turnstile\prin\fm{\neg\exists \tr x\,\gamma(\tr x)\Tofrom \forall \tr x\,\neg \gamma(\tr x)}$
\item\label{quantifiers2}
$\Turnstile\prin\fm{\exists \tr x\,\neg \gamma(\tr x)\To\neg\forall \tr x\,\gamma(\tr x)}$%
\footnote{The converse implication is known as the Generalized Markov Principle;
see \S\ref{Markov}.}
\end{enumerate}
These are special cases of (\ref{gq1}) and (\ref{gq2}).

\subsubsection{Distributivity}

\begin{enumerate}[resume]
\item\label{morgan1} $\Turnstile\prin\fm{\theta\land(\gamma\lor\delta)\Tofrom
(\theta\land\gamma)\lor(\theta\land\delta)}$
\item\label{morgan2} $\Turnstile\prin\fm{\theta\lor(\gamma\land\delta)\Tofrom
(\theta\lor\gamma)\land(\theta\lor\delta)}$
\end{enumerate}
Here (\ref{morgan1}) and the ``$\to$'' implication in (\ref{morgan2}) follow
similarly to (\ref{q4}) and (\ref{q5}).

The remaining implication says that if we have (i) either a solution
of $\bTheta$, or a solution of $\bGamma$, and (ii) either a solution of
$\bTheta$, or a solution of $\bDelta$, then we can produce, by a general
method, (iii) either a solution of $\bTheta$, or solutions of $\bGamma$
and $\bDelta$.
Clearly, there is such a general method which proceeds by a finite
analysis of cases; in fact, there are two such distinct methods $m_i$,
$i=1,2$, which in the case that we have two solutions of $\bTheta$
select the $i$th one.

\begin{enumerate}[resume]
\item\label{anti-dm1} $\Turnstile\prin\fm{ (\theta\to\gamma)\land(\theta\to\delta)\Tofrom
\big(\theta\to(\gamma\land\delta)\big)}$
\item\label{anti-dm2} $\Turnstile\prin\fm{ (\theta\to\gamma)\lor(\theta\to\delta)\To
\big(\theta\to(\gamma\lor\delta)\big)}$%
\footnote{The converse implication is equivalent to the G\"odel--Dummett principle
$\fm{\prin(\gamma\to\delta)\lor(\delta\to\gamma)}$; see \S\ref{Harrop1}.}
\end{enumerate}
These follow similarly to (\ref{q3}) and (\ref{q2}).

\begin{enumerate}[resume]
\item\label{dm1} $\Turnstile\prin\fm{ (\gamma\to\theta)\land(\delta\to\theta)\Tofrom
\big((\gamma\lor\delta)\to\theta\big)}$
\item\label{dm2} $\Turnstile\prin\fm{ (\gamma\to\theta)\lor(\delta\to\theta)\To
\big((\gamma\land\delta)\to\theta\big)}$%
\footnote{The converse implication, Skolem's principle, is equivalent to
the G\"odel--Dummett principle; see \S\ref{Harrop1}.}
\end{enumerate}
These follow similarly to (\ref{gq1}) and (\ref{gq2}).

\begin{enumerate}[resume]
\item\label{exponential} (exponential law)
$\Turnstile\prin\fm{ (\alpha\land \beta\to \gamma)\Tofrom\big(\alpha\to (\beta\to \gamma)\big)}$
\item\label{co-exponential}
$\Turnstile\prin\fm{\alpha\land (\beta\to \gamma)\To\big((\alpha\to \beta)\to \gamma\big)}$
\end{enumerate}
These are again similar to (\ref{gq1}) and (\ref{gq2}).

\begin{enumerate}[resume]
\item\label{deMorgan1} (de Morgan law)
$\Turnstile\prin\fm{\neg \gamma\land\neg \delta\Tofrom\neg (\gamma\lor \delta)}$
\item\label{deMorgan2} (de Morgan law)
$\Turnstile\prin\fm{\neg \gamma\lor\neg \delta\To\neg (\gamma\land \delta)}$%
\footnote{The converse implication is equivalent to Jankov's principle $\neg\alpha\lor\neg\neg\alpha$,
see Proposition \ref{Jankov's logic} below.}
\end{enumerate}
These are special cases of (\ref{dm1}) and (\ref{dm2}).

\begin{enumerate}[resume]
\item\label{implication-exp}
$\Turnstile\prin\fm{\neg(\gamma\land \delta)\Tofrom (\gamma\to\neg \delta)}$
\item\label{implication-coexp}
$\Turnstile\prin\fm{\gamma\land\neg \delta\To\neg(\gamma\to \delta)}$
\end{enumerate}
These follow from (\ref{exponential}) and (\ref{co-exponential}).

\subsubsection{Consequences and refinements}

\begin{enumerate}[resume]
\item\label{not-not-LEM}
$\Turnstile\prin\fm{\neg\neg(\gamma\lor\neg\gamma)}$
\end{enumerate}
Indeed, a solution of $\neg(\bGamma\lor\neg \bGamma)$ yields by
(\ref{deMorgan1}) a solution of $\neg \bGamma\land\neg\neg \bGamma$,
that is, a solution of $\neg \bGamma$ together with a method to turn
such a solution into a contradiction.

\begin{enumerate}[resume]
\item\label{implication-shift}
$\Turnstile\prin\fm{ (\gamma\to\neg \delta)\Tofrom (\delta\to\neg \gamma)}$
\end{enumerate}
This follows from (\ref{implication-exp}).

\begin{enumerate}[resume]
\item\label{double contrapositive}
$\Turnstile\prin\fm{ (\neg\neg \gamma\to\neg\neg \delta)\Tofrom (\neg \delta\to\neg \gamma)}$
\end{enumerate}
This ``idempotence of contrapositive'' follows from
(\ref{triple negation}) and (\ref{implication-shift}).

\begin{enumerate}[resume]
\item\label{neg-neg-drop}
$\Turnstile\prin\fm{ (\neg\neg \gamma\to\neg\delta)\Tofrom (\gamma\to\neg\delta)}$
\end{enumerate}
This follows from (\ref{implication-shift}) and (\ref{triple negation}),
or alternatively from (\ref{double negation}),
(\ref{contrapositive}) and (\ref{triple negation}).

\begin{enumerate}[resume]
\item\label{neg-neg-drop2}
$\Turnstile\prin\fm{\neg(\gamma\land\neg\neg \delta)\Tofrom \neg(\gamma\land\delta)}$
\item\label{neg-neg-drop3}
$\Turnstile\prin\fm{\neg(\gamma\lor\neg\neg \delta)\Tofrom \neg(\gamma\lor\delta)}$
\end{enumerate}
These follow from (\ref{implication-exp}) and (\ref{deMorgan1}) respectively,
using (\ref{triple negation}).

Surprisingly, (\ref{implication-coexp}) is a consequence of a reversible implication:
\begin{enumerate}[resume]
\item\label{implication1}
$\Turnstile\prin\fm{\neg\neg \gamma\land\neg \delta\Tofrom\neg(\gamma\to \delta)}$
\end{enumerate}
Here the ``$\from$'' implication follows from (\ref{deMorgan1}) and
the contrapositive of (\ref{implication0}).

Conversely, suppose we are given a solution of
$\neg\neg \bGamma\land\neg \bDelta$
and a solution of $\bGamma\to \bDelta$.
Then we have solutions of $\neg \bDelta$, of $\neg \bGamma\to\ab$, and
of $\neg \bDelta\to\neg \bGamma$ (the contrapositive).
Applying the latter to the solution of $\neg \bDelta$, we get a solution of
$\neg \bGamma$, and hence a contradiction.

\begin{enumerate}[resume]
\item\label{implication2}
$\Turnstile\prin\fm{\neg(\neg\delta\lor\gamma)\Tofrom\neg(\delta\to \gamma)}$
\end{enumerate}
This improvement on the contrapositive of (\ref{implication0})
follows from (\ref{implication1}) and (\ref{deMorgan1}).

\begin{enumerate}[resume]
\item\label{implication0''}
$\Turnstile\prin\fm{\neg(\beta\lor\gamma)\Tofrom\neg(\neg\beta\to\gamma)}$
\end{enumerate}
This improvement on the contrapositive of (\ref{implication0'})
follows from (\ref{implication2}) and (\ref{neg-neg-drop3}).

\begin{enumerate}[resume]
\item\label{deMorgan3}
$\Turnstile\prin\fm{\neg(\neg \gamma\lor\neg \delta)\Tofrom\neg\neg(\gamma\land \delta)}$
\end{enumerate}
This improvement on the contrapositive of (\ref{deMorgan2})
follows from (\ref{implication2}) and (\ref{implication-exp}).

\begin{enumerate}[resume]
\item\label{implication-coexp'}
$\Turnstile\prin\fm{\neg(\gamma\land\neg\delta)\Tofrom\neg\neg(\gamma\to \delta)}$
\end{enumerate}
This improvement on the contrapositive of (\ref{implication-coexp})
follows from (\ref{implication1}) and (\ref{neg-neg-drop2}).

\begin{enumerate}[resume]
\item\label{neg-neg-and}
$\Turnstile\prin\fm{\neg\neg\gamma\land\neg\neg\delta\Tofrom\neg\neg(\gamma\land \delta)}$
\end{enumerate}
This follows from (\ref{deMorgan3}) and (\ref{deMorgan1}), or alternatively from
(\ref{implication1}) and (\ref{implication-exp}).

\begin{enumerate}[resume]
\item\label{neg-neg-imp}
$\Turnstile\prin\fm{ (\neg\neg\gamma\to\neg\neg\delta)\Tofrom\neg\neg(\gamma\to\delta)}$
\end{enumerate}
This also follows from (\ref{implication1}) and (\ref{implication-exp}).

\begin{enumerate}[resume]
\item\label{neg-neg-or}
$\Turnstile\prin\fm{\neg\neg\gamma\lor\neg\neg\delta\To\neg\neg(\gamma\lor\delta)}$%
\footnote{The converse implication is equivalent to Jankov's principle $\neg\alpha\lor\neg\neg\alpha$,
see Proposition \ref{Jankov's logic} below.}
\end{enumerate}
This follows from (\ref{deMorgan2}) using (\ref{deMorgan1}).

\begin{enumerate}[resume]
\item\label{DNS-converse}
$\Turnstile\prin\fm{\neg\neg\forall \tr x\,\gamma(\tr x)\To\forall \tr x\,\neg\neg \gamma(\tr x)}$%
\footnote{The converse implication is known as the principle of Double Negation Shift;
see \S\ref{shift}.}
\item\label{SMP-converse}
$\Turnstile\prin\fm{\exists \tr x\,\neg\neg\gamma(\tr x)\To\neg\neg\exists \tr x\,\gamma(\tr x)}$%
\footnote{The converse implication is known as the Strong Markov Principle; see \S\ref{Markov}.}
\end{enumerate}
Each of these follows from (\ref{quantifiers2}) using (\ref{quantifiers1}).

\section{What is a logic, formally?}\label{formal}

This chapter provides a rather unconventional introduction to basic first-order logic.

On the one hand, it aims at being relatively readable by including just enough detail to make it
possible for the interested reader to reconstruct the rest.
(For instance, we use parentheses without discussing their official treatment in
the formal language.)

On the other hand, we are much more careful about meta-level reasoning than is usually done.
Some reasons for doing this are discussed in the following Introduction.

\subsection{Introduction}\label{formal-intro}

This Introduction is not formally used in the sequel; it is addressed to the reader who has some acquaintance
with logic textbooks and is wondering what is the point of yet another treatment of basic first-order logic.

\subsubsection{Two conceptions of a logic}
We identify a logic with the collection of its derivable rules (including derivable principles,
i.e.\ derivable rules with no premisses).
This is equivalent to identifying a logic with its consequence relation.
In the terminology of W\'ojcicki \cite{Wo}*{\S1.6} and Blok--Pigozzi \cite{BP} this is the ``inferential''
conception of logic, as opposed to the ``formulaic'' one, in which a logic is identified with the collection
of its derivable principles (or, equivalently, derivable formulas); these authors as well as Avron \cite{Av}
argue in favor of the inferential conception.

In particular, as noted by W\'ojcicki \cite{Wo}*{\S1.6.0}, if the purpose of a logic is to elucidate
schemes of reasoning, then logical schemes of inferences found in a mathematical practice that is considered
to be in agreement with a logic $L$ are captured by the derivable rules of $L$; in general, these cannot be
reconstructed from the derivable principles of $L$ alone, which capture only direct use of logical validities
in the said mathematical practice.

The two conceptions of a logic ascribe somewhat different meanings to the notion of an inference rule, which
is related to the difference between derivable and admissible rules.
(A rule is called admissible if adding it as a new inference rule to the derivation system does not affect
the collection of derivable principles.
Thus, the notion of admissibility is ``formulaic'' in that it only depends on the collection of derivable
principles.)
In classical logic, the difference between derivable and admissible rules exists but is rather insignificant
(see Examples \ref{cbot} and \ref{omitting}, Remark \ref{Dzik} and Proposition \ref{der-adm} below).
So the entire issue of two different conceptions of a logic only really manifests itself in intuitionistic
and other non-classical logics.

\subsubsection{What is a rule?}
Another variety of phenomena that one cannot fully appreciate by only looking at classical logic is
independent principles, as well as consequences between them --- and more generally, consequences between
rules (see, however, Examples \ref{principle implication} and \ref{cbot2} and
Proposition \ref{rule implication} concerning the situation with these in classical logic).
Analyzing such consequences is going to be quite helpful in the next chapter to get some feel of how
intuitionistic logic works.

Now, even to give a definition of derivable rules or consequences between principles, some precise definition
of a rule --- and, in particular, of a principle --- is needed.
Indeed, it is not so clear exactly what they are, since the formal language provides no syntax for their
usual side conditions such as ``provided that $x$ is not free in $\alpha$'' or ``provided that $t$ is free
for $x$ in $\alpha(x)$'', and {\it a fortiori} no formal rules for dealing with these side conditions.
Often principles and rules are treated as arbitrary infinitary collections of formulas (or of tuples of
formulas).
But this is clearly at odds with the fact that in essence they are finite objects (written with finitely many
symbols in logic textbooks).
The well-known workaround \cite{PP} (see also \cite{Dz}): a rule is called structural if it is preserved by
all substitutions without anonymous variables (in the sense of Kleene \cite{Kl}) --- is clearly still far 
from being a satisfactory solution.

The usual way to understand (and represent in writing) principles and rules as finite objects is by
means of schemata, which belong to the meta-language.
While it is not difficult to make this meta-language just as formal as the language itself 
(for example, by regarding the set of terms and the sets of $\alpha$-equivalence classes of $n$-formulas
of the language as semantic objects, and creating an obvious syntax to describe them), this approach 
does not seem to be technically advantageous (see, however, \cite{SSKI}, \cite{Ne}, \cite{GM} concerning 
formal calculi of meta-variables, and \cite{Ry}*{\S3.7}, \cite{Ie3} concerning schematic rules).
And, as it turns out, there is a good reason for not taking this path: the extra level of linguistic
abstraction (where syntax is regarded as an abstraction of semantics) is actually an overkill, because
the desired effect can already be achieved by means of the usual $\lambda$-abstraction (where a function 
is regarded as an abstraction of the dependence of a prescribed expression on prescribed variables).

\subsubsection{Meta-logic}
To get rid of side conditions of the form ``provided that $x$ is not free in $\alpha$'', we will use
a {\it meta-quantifier}.
In the author's initial approach, the desired effect was achieved by means of a device that is like
a usual quantifier in that it binds an (otherwise free) metavariable in a schema, but is invisible
on the level of formulas, since the binding is only used to specify which formulas are considered to be
instances of the schema.
Thus it was natural to call this device a ``meta-quantifier''.

However, by googling for the word {\it meta-quantifier}, the author discovered that a similar, but much
more convenient device has been used for the same purpose (of getting rid of the side conditions) in
the {\tt Isabelle} theorem prover \cite{NPW}*{\S5.9}, \cite{Pau1}.
It has also been called the ``meta-quantifier'' (e.g.\ in the index of the old manual \cite{Pau2}) but
for a somewhat different reason \cite{Pau1}*{\S2.1}: it is the (universal) quantifier of {\tt Isabelle}'s
meta-logic, which serves the purpose of a common ground for a number of different logics supported by
{\tt Isabelle}.

This view of {\tt Isabelle} and its creator Paulson that the meta-quantifier is really a quantifier of
the meta-logic turns out to be very beneficial.
The side conditions do not disappear, they are only relegated from rules to meta-rules, where they can be seen
to belong naturally.
No schemata are needed whatsoever, because the meta-quantifier operates on formulas.
In this respect Paulson's approach builds on the early tradition of first-order logic, found in
the textbooks by Hilbert--Ackermann \cite{HA}, Hilbert--Bernays \cite{HB}, Church \cite{Ch} 
(alongside the modern schematic approach) and P. S. Novikov \cite{N}, \cite{N3}.
In this tradition, one does not speak of any schemata, but only of formulas; instead, the derivation system
includes a substitution rule (see also Church \cite{Ch}
and Feferman \cite{Fe} for a comparison and the history of the two approaches).
However, inference rules were anyway stated in schematic form; and the substitution rule may be seen
as problematic because it is not structural.
In Paulson's approach, inference rules involve no meta-variables, and the substitution rule effectively
becomes a meta-rule.

Another principal ingredient of {\tt Isabelle}'s metalogic is the meta-implication, which is closely related,
in particular, to the entailment (the horizontal bar in rules) and to the consequence ($\turnstile $) --- not
to be conflated with the implication ($\to$) --- in object logics.
These relations will be clarified below; the former is described implicitly in Paulson's example
\cite{Pau1}*{\S4.4, subsection ``Hypothetical rules''}, whereas of the latter Paulson only notes
that $\turnstile $ is not the same as the meta-implication \cite{Pau1}*{Remark at the end of \S3}.

To summarize, meta-logic enables one to ``explain'' many standard notions in logic that one just gets
used to rather than understands in the traditional approach, as well as to avoid some traditional ambiguities,
which can easily lead to errors through mechanical application of the official definitions.

\subsubsection{Two conceptions of consequence}\label{meta-quantifier}
In mathematics, there are two ways to interpret formulas with parameters.
The following example is given in Kleene's textbook \cite{Kl}*{\S32}: the arithmetical formula
$(x+y)^2=x^2+2xy+y^2$ begs to be understood as an identity (valid for all natural numbers $x$),
whereas the arithmetical formula $x^2+2=3x$ begs to be understood as an equation (i.e., as a condition
on $x$).
There is no difference in syntax to reflect this obvious distinction in meaning.
But the latter is not illusory: it is reflected in use.
For, as noted by Avron \cite{Av3}, when ``dealing with identities [...] the substitution rule is available,
and one may infer $\sin x=2\sin\frac x2\cos\frac x2$ from the identity $\sin 2x=2\sin x\cos x$.
In contrast, [...] substituting $\frac x2$ for $x$ everywhere in an equation is an error''.

In first-order logic, similar phenomena in the use of substitution are normally associated with
the distinction between free and bound variables: it is only bound variables that can be
harmlessly renamed.%
\footnote{In the words of Paulson, ``The purpose of the proviso `$x$ not free in ...' is to ensure that
the premise may not make assumptions about the value of $x$, and therefore holds for all $x$.''
\cite{Pau2}*{\S1.2.2}.}
However, in all of the above examples, $x$ is intended (by Kleene and Avron) to signify a free variable
(in the sense of first-order logic).
Logicians often speak of ``fresh free variables'', which can be seen to behave like variables in
identities rather than equations.
While this idea suffices for practical applications (cf.\ \cite{Pau1}*{\S7}), it does not seem to be
appropriate for theoretical/foundational purposes.

Instead, the way that traditional first-order logic accounts for the semantic distinction between
``identities'' and ``equations'' is by distinguishing two notions of syntactic consequence and two
corresponding notions of semantic consequence.
Thus Kleene \cite{Kl}*{\S22, \S32} speaks of variables being {\it varied} or {\it held constant} in
a deduction (effectively distinguishing two variants of syntactic consequence); whereas Avron
\cite{Av}, \cite{Av2}, \cite{Av3} speaks of {\it truth} and {\it validity} semantic consequence relations
(both are mentioned also e.g.\ in the Kolmogorov--Dragalin textbook \cite{KD}).
The same sort of distinction is also thought to exist in propositional modal logic \cite{Av},
\cite{Av2} (see also \cite{HN} and references there).

However, the view that first-order logic (classical or intuitionistic) must include two distinct consequence
relations at the same time has never gained wide acceptance, presumably due to its technical inconvenience and
philosophical awkwardness.
Most authors appear to believe that there is a unique, ``true'' consequence relation (either syntactic
or semantic), but then they disagree with each other on which one it is.
Thus, the ``truth''/``fixed variables'' approach, corresponding to the equational interpretation
of free variables, is chosen e.g.\ by Church \cite{Ch}, Enderton \cite{End}, Troelstra \cite{Tr3} and 
van Dalen \cite{vDx}; whereas the ``validity''/``varied variables'' approach, corresponding to 
the identity interpretation of free variables, is chosen e.g.\ by Shoenfield \cite{Sho} and, in 
the syntactic part, by Mendelson \cite{Men}.

Each of the two approaches has its advantages.
In particular, with the ``validity''/``varied variables'' approach, the unrestricted Hilbert-style
generalization rule holds (and so the notion of a {\it rule} does not have to be very complicated).
With the ``truth''/``fixed variables'' approach, the unrestricted deduction theorem holds (and so
proof-theoretic formalisms such as sequent calculus and natural deduction apply in their usual form).
The two approaches are discussed and compared --- with the opposite conclusions reached ---
by Blok--Pigozzi \cite{BP} and Hakli--Negri \cite{HN}.

The meta-quantifier offers a somewhat different outlook on these issues, which finally clarifies
the picture.
There is only one consequence relation (coming from the meta-logic), from which both traditional ones 
can be easily recovered.
An appropriate form of the generalization rule holds without restrictions (and so the notion of a rule
does not have to be complicated; in fact, as discussed above, it is simpler than in either of 
the traditional approaches).
The deduction theorem holds without restrictions --- or rather under the automatic restriction that
the meta-quantifier cannot be used inside a formula (because logical connectives and quantifiers cannot
operate on expressions that involve meta-logical ones).

\subsection{Typed expressions} \label{expressions}

This section is mostly standard material (cf.\ \cite{SU}, \cite{Gi}, \cite{LS}), apart from some deviations
in notation and terminology, which opt for a more mathematical (rather than computer science) style,
and some additional conventions in \S\ref{simultaneous}, \S\ref{free substitution}.

\subsubsection{Simply typed $\lambda$-calculus} \label{lambda}

One fixes a collection of {\it base types} (we will encounter only three base types: of terms,
of formulas, and of meta-formulas).
Arbitrary {\it types} (or in more detail {\it types of $\lambda$-expressions})%
\footnote{These are also called ``arities'' following Martin-L\"of (see \cite{NPS}*{\S3.6}), but we will use
the word ``arity'' in its traditional sense.}
are defined as follows:
\begin{enumerate}
\item base types are types;
\item if $\Gamma$ and $\Delta$ are types, then so is $\Gamma\too\Delta$;
\item if $\Gamma$ and $\Delta$ are types, then so is $\Gamma\x\Delta$.
\end{enumerate}

For each type $\Delta$, one may fix a collection of {\it constants} of type $\Delta$.
For each type $\Delta$, one always has a countable ordered collection of {\it variables} of type $\Delta$,
denoted $\var_i^\Delta$, $i=1,2,\dots$.
(Thus, by saying ``$x$ is a variable of type $\Delta$'' we mean that $x=\var_i^\Delta$ for some $i$.)
Arbitrary {\it expressions}%
\footnote{In $\lambda$-calculus, these are usually called ``terms'', but we reserve this word for
expressions of one specific type, in accordance with the terminology of first-order logic.}
{\it of type} $\Delta$ are defined as follows, where $T:\Delta$ denotes the relation ``$T$ is an expression of
type $\Delta$'':
\begin{enumerate}
\item a constant of type $\Delta$ is an expression of type $\Delta$;
\item a variable of type $\Delta$ is an expression of type $\Delta$;
\item ($\lambda$-abstraction) if $T:\Delta$ on the assumption that $x$ is a variable of type $\Gamma$, then
$x\mapsto T:\Gamma\too\Delta$;%
\footnote{In $\lambda$-calculus, $x\mapsto T$ is usually written as $\lambda x.T$, but we will use the 
traditional mathematical notation.}
\item (function application) if $F:\Delta\too\Gamma$, and $T:\Delta$, then $F(T):\Gamma$;%
\footnote{In $\lambda$-calculus, $F(T)$ is usually written as $FT$, but we will normally use the more customary
mathematical notation. 
In some cases we do resort to $FT$, regarding it as an abbreviation for $F(T)$.}
\item (tuples) if $S:\Gamma$ and $T:\Delta$, then $(S,T):\Gamma\x\Delta$.
\item (projections) if $\Gamma$, $\Delta$ are types, then $\pr_1:\Gamma\x\Delta\too\Gamma$ and
$\pr_2:\Gamma\x\Delta\too\Delta$.
\end{enumerate}
We will abbreviate $F\big((S,T)\big)$ by $F(S,T)$, and usually also $\big(F(T)\big)(S)$ by $F(T)(S)$
and $\pr_i(T)$ by $\pr_i T$.

\begin{example} Let $\R$ be a type with one constant for each real number.
Then addition of real numbers $(x,y)\mapsto x+y$ is of type $\R\x\R\too\R$;
integration of functions $(f,(a,b))\mapsto\int_a^b f(x)\,dx$ is of type $(\R\too\R)\x(\R\x\R)\too\R$;
and the differential $f\mapsto df$ is of type $(\R\too\R)\too(\R\too(\R\too\R))$.
(This example ignores some complications such as non-differentiable functions, etc.)
\end{example}

\begin{example} \label{composition}
Given $F:\Theta\too\Gamma$ and $G:\Gamma\too\Delta$, we can define their {\it composition}
$G\circ F:\Theta\too\Delta$ by $r\mapsto G(F(r))$.
Thus $\circ:(\Theta\too\Gamma)\x(\Gamma\too\Delta)\too(\Theta\too\Delta)$ is the following expression:
$q\mapsto\Big(r\mapsto\pr_1(q)\big(\pr_2(q)(r)\big)\Big)$.
\end{example}

\begin{example} \label{quantifier example} 
In first-order logic as treated e.g.\ in {\tt Isabelle}, we have a type of terms, 
which will be denoted below by $\0$, and a type of formulas, which will be denoted below by $\1$.
Then a {\it quantifier} is a constant of type $(\0\too\1)\too\1$.
If $\frak q$ is a quantifier, the expression $\frak q(x\mapsto F)$ is customarily denoted by 
$\frak q x\, F$. 
(So e.g.\ for the universal quantifier $\forall:(\0\too\1)\too\1$ this takes the usual form $\forall x\,F$.)
\end{example}

\begin{remark}
Throughout this section we will continue to denote variables by lowercase Roman letters, constants by
lowercase Fraktur letters, unknown terms by uppercase Roman letters, compound terms that are specific
(at least up to $\alpha$-equivalence, which is defined below) by uppercase Fraktur letters; and unknown types
by uppercase Greek letters.
These letters all represent {\it metavariables} in that they do not themselves belong to $\lambda$-calculus,
as opposed to the variables $\var^\Delta_i$, the constants $\pr_i$, specific base types and their specific
constants (these will introduced in the next section).
It should also be noted that the indices of the variables $\var^\Delta_i$ are plain numbers (i.e., sequences
of bars, such as $|||||$, abbreviated as usual by means of decimal notation) and not numeric expressions; 
thus $\var_{10^{10^{10}}}^\Delta$ is not an expression of type $\Delta$, but only an informal meta-expression 
that may be taken as a euphemism referring to an actual expression of type $\Delta$ (which is unprintable).
\end{remark}

Abstraction {\it binds} variables; variables that are not bound by an abstraction are called {\it free}.
More precisely, whether a variable $x$ of some type {\it occurs freely} in an expression $T$ of some type
is determined by the following recursive definition (recursion over the number of symbols in $T$).
Let $y$ be a variable distinct from $x$ (i.e.\ $x=\var_i^\Gamma$ and $y=\var_j^\Delta$ where either
$\Gamma\ne\Delta$ or $i\ne j$), let $\frak c$ be a constant of some type and $F$, $R$ and $S$
be expressions of some types.
Then:
\begin{enumerate}
\item $x$ occurs freely in $x$;
\item $x$ does not occur freely in $y$;
\item $x$ does not occur freely in $\frak c$;
\item $x$ occurs freely in $(R,S)$ if it occurs freely in $R$ or in $S$;
\item $x$ occurs freely in $F(R)$ if it occurs freely in $F$ or in $R$;
\item $x$ does not occur freely in $x\mapsto R$;
\item $x$ occurs freely in $y\mapsto R$ if it occurs freely in $R$.
\end{enumerate}
For example, $x$ does not occur freely in $x\mapsto x$; but $y$ occurs freely in $(x\mapsto x)(y)$.
Also, $f$ and $g$ occur freely in $(f\mapsto f)\Big(x\mapsto f\big(g(x)\big)\Big)$.

An expression of some type is {\it closed} if no variable of any type occurs freely in it.

\begin{example} \label{meta-map1} A closed expression
$\frak E:\big(\Lambda\too(\Delta\too\Gamma)\big)\too\big((\Lambda\too\Delta)\too(\Lambda\too\Gamma)\big)$
can be defined by $f\mapsto\bigg(g\mapsto\Big(l\mapsto f(l)\big(g(l)\big)\Big)\bigg)$.
\end{example}

\begin{example} \label{meta-map4} A closed expression $\frak D:\big((\Delta\too\Gamma)\too\Theta\big)\too
\Big(\big(\Delta\too(\Lambda\too\Gamma)\big)\too(\Lambda\too\Theta)\Big)$
can be defined by
$f\mapsto\bigg(g\mapsto\Big(l\mapsto f\big(t\mapsto g(t)(l)\big)\Big)\bigg)$.
\end{example}

\subsubsection{Substitution}
Let $T:\Delta$, and let $x$ be a variable of type $\Delta$.
The {\it (capture-avoiding, strict) substitution} $S|_{x:=T}$, where $S:\Gamma$, is, when defined,
an expression of type $\Gamma$, which is determined recursively as follows.
Let $y$ be a variable distinct from $x$, let $\frak c$ be a constant of some type and $F$, $R$ and $S$
be expressions of some types.
Then:
\begin{enumerate}
\item $x|_{x:=T}=T$;
\item $y|_{x:=T}=y$;
\item $\frak c|_{x:=T}=\frak c$;
\item $(R,S)|_{x:=T}=(R|_{x:=T},\,S|_{x:=T})$;
\item $F(R)|_{x:=T}=(F|_{x:=T})(R|_{x:=T})$;
\item $x\mapsto R|_{x:=T}=x\mapsto R$;
\item $y\mapsto R|_{x:=T}=\begin{cases}
\text{undefined,} & \text{if $x$ occurs freely in $R$ and $y$ occurs freely in $T$;}\\
y\mapsto(R|_{x:=T})& \text{otherwise.}
\end{cases}$
\end{enumerate}
Of course, the recursion halts once it reaches the ``undefined'' case, and in this event $S|_{x:=T}$ is set
to be undefined.

\begin{example} In the notation of Example \ref{quantifier example}, suppose that $>$ is a constant
of type $\0\x\0\too\1$ and $+$ is a constant of type $\0\x\0\to\0$, and let us write $x>y$ to mean
${>}(x,y)$ and $x+y$ to mean ${+}(x,y)$.
Then 
\begin{itemize}
\item $\forall x\ x>y|_{y:=y+z}=\forall x\ x>y+z$;
\item $\forall x\ x>y|_{y:=x+z}$ is undefined;
\item $\forall x\ x>y|_{x:=y+z}=\forall x\ x>y$.
\end{itemize}
\end{example}

The following auxiliary definition will be needed on several occasions.
A relation $\succ$ on expressions is called {\it compatible} if it satisfies the following conditions,
where $x$ is a variable of some type, and $F,G,R,S,T$ are expressions of some types.
\begin{enumerate}
\item if $S\succ T$, then $x\mapsto S\succ x\mapsto T$;
\item if $S\succ T$, then $F(S)\succ F(T)$;
\item if $F\succ G$, then $F(T)\succ G(T)$;
\item if $S\succ T$, then $(R,S)\succ (R,T)$ and $(S,R)\succ (T,R)$;
\item if $S\succ T$, then $\pr_1 S\succ\pr_1 T$ and $\pr_2 S\succ\pr_2 T$.
\end{enumerate}

Two expressions are considered to be the same for all practical purposes if one can be obtained
from another by renaming of bound variables.
In more detail, the {\it $\alpha$-equivalence} relation is the least compatible equivalence relation on
expressions of arbitrary types such that whenever $x,y$ are variables of some type and $T$ is an expression
of some type,
\begin{itemize}
\item if $y$ does not occur freely in $T$ and $T|_{x:=y}$ is defined, then
$x\mapsto T\overset\alpha=y\mapsto (T|_{x:=y})$.
\end{itemize}
Clearly, $\alpha$-equivalent expressions are always of the same type.
Two $\alpha$-equivalent expressions are also said to be obtained from one another by {\it $\alpha$-conversion}.

\begin{example} $x\mapsto\big(y\mapsto F(x,y)\big)$ is $\alpha$-equivalent to
$a\mapsto\big(b\mapsto F(a,b)\big)$ and to $y\mapsto\big(x\mapsto F(y,x)\big)$.
However, $x\mapsto (y\mapsto x)$ is not $\alpha$-equivalent to $y\mapsto (x\mapsto y)$ (because of
the free variable).

Also, $x\mapsto(x\mapsto x)$ is $\alpha$-equivalent to $x\mapsto (y\mapsto y)$ and to
$y\mapsto (x\mapsto x)$ but is not $\alpha$-equivalent to $y\mapsto (x\mapsto y)$.
Informally speaking, since bound variables can be renamed, they are essentially only nameless placeholders.
In particular, once a variable gets bound (like in $x\mapsto x$), it loses its name to an outside observer;
this name may then be recycled (like in $x\mapsto (x\mapsto x)$), should one need to do so.
\end{example}

The point of $\alpha$-equivalence that if a substitution $S|_{x:=T}$ is undefined, then we can
replace $S$ by an $\alpha$-equivalent term $S'$ so that $S'|_{x:=T}$ is defined.
(For example, $x\mapsto f(y)|_{y:=g(x)}$ is undefined, but becomes defined upon renaming $x$ in $x\mapsto f(y)$
into $z$.)
Although $S'$ can be chosen in many ways, the result will be well-defined up to $\alpha$-equivalence.

Let us note that the $\alpha$-equivalence class $S|^\alpha_{x:=T}$ includes the results (when they
are defined) of all substitutions of the form $S'|_{x:=T'}$, where $S'$ is $\alpha$-equivalent
to $S$ and $T'$ is $\alpha$-equivalent to $T$.
This is useful, for instance, since one can always make a double substitution of the form 
$(S|_{x:=T})|_{y:=U}$ defined by replacing $S$ and $T$ by $\alpha$-equivalent terms (see \cite{SU}).

\subsubsection{$\beta\eta$-equivalence}

The relation of {\it $\beta$-reduction} is the least compatible relation such that
\begin{enumerate}
\item every expression of the form $(x\mapsto S)(T)$ $\beta$-reduces to $S|_{x:=T}$;
\item every expression of the form $\pr_i(T_1,T_2)$ $\beta$-reduces to $T_i$ for $i=1,2$;
\item $\alpha$-equivalent expressions $\beta$-reduce to each other.
\end{enumerate}
For example, $(x\mapsto x)(T)$ $\beta$-reduces to $T$; and $(x\mapsto T)(x)$ $\beta$-reduces to $T$.

\begin{example} \label{meta-map5} In Example \ref{meta-map1},
$\frak E(F)(G)(L)$ $\beta$-reduces $F(L)\big(G(L)\big)$ in three steps.
In Example \ref{meta-map4},
$\frak D(F)(G)(L)$ $\beta$-reduces to $F(t\mapsto G(t)(L)$ in three steps.
\end{example}

\begin{example} The expression $\big(f\mapsto f(z)\big)\big(x\mapsto(z\mapsto x)\big)$ $\beta$-reduces to
$f(z)|_{f:=x\mapsto(z\mapsto x)}$.
This substitution is defined, with the result $\big(x\mapsto (z\mapsto x)\big)(z)$.
But the latter $\beta$-reduces to $z\mapsto x|_{x:=z}$, which is not defined.
Up to $\alpha$-equivalence, this is the same as $y\mapsto x|_{x:=z}$, which evaluates to $y\mapsto z$.
\end{example}

The relation of {\it $\eta$-reduction} is the least compatible relation such that
\begin{enumerate}
\item every expression of the form $x\mapsto F(x)$, where $x$ does not occur freely in $F$,
$\eta$-reduces to $F$;
\item every expression of the form $\big(\pr_1T,\pr_2T\big)$ $\eta$-reduces to $T$;
\item $\alpha$-equivalent expressions $\eta$-reduce to each other.
\end{enumerate}
Let us note that $\eta$-reduction converts between functions that return the same output on
the same input (possibly by employing different algorithms).

The inverse of a $\beta$- or $\eta$-reduction is called a $\beta$- or $\eta$-{\it expansion}.

The {\it $\beta\eta$-equivalence} is the least compatible equivalence relation on expressions
of arbitrary types generated by $\beta$-reduction and $\eta$-reduction.
Clearly, $\beta\eta$-equivalent expressions are always of the same type.
We will denote $\beta\eta$-equivalence by the symbol $=$.

\begin{example} \label{exp-lambda} Let us construct closed expressions
$\frak F:(\Theta\x\Gamma\too\Delta)\too\big(\Theta\too(\Gamma\too\Delta)\big)$
and $\frak G:\big(\Theta\too(\Gamma\too\Delta)\big)\too(\Theta\x\Gamma\too\Delta)$ such that
$\frak F\big(\frak G(G)\big)=G$ and $\frak G\big(\frak F(F)\big)=F$.

We define $\frak F$ by $f\mapsto\Big(r\mapsto\big(s\mapsto f(r,s)\big)\Big)$ and $\frak G$ by
$g\mapsto\Big(q\mapsto g\big(\pr_1q\big)\big(\pr_2q\big)\Big)$.

\begin{multline*}
\frak F\big(\frak G(G)\big)\overset\beta=\frak F\Big(q\mapsto G\big(\pr_1q\big)\big(\pr_2q\big)\Big)\overset\beta=
r\mapsto\big(s\mapsto f(r,s)\big)\Big|_{f:=q\mapsto G(\pr_1q)(\pr_2q)}\\
=r\mapsto\bigg(s\mapsto\Big(q\mapsto G\big(\pr_1q\big)\big(\pr_2q\big)\Big)(r,s)\bigg)
\overset\beta= r\mapsto\bigg(s\mapsto \Big(G\big(\pr_1q\big)\big(\pr_2q\big)\big|_{q:=(r,s)}\Big)\bigg)\\
=r\mapsto\Big(s\mapsto G\big(\pr_1(r,s)\big)\big(\pr_2(r,s)\big)\Big)
\overset\beta= r\mapsto\Big(s\mapsto G(r)(s)\Big)\overset\eta=r\mapsto G(r)\overset\eta=G.
\end{multline*}
Also,
\begin{multline*}
\frak G\big(\frak F(F)\big)\overset\beta=\frak G\Big(r\mapsto\big(s\mapsto F(r,s)\big)\Big)
\overset\beta=q\mapsto g\big(\pr_1q\big)\big(\pr_2q\big)\Big|_{g:=r\mapsto\normal(s\mapsto F(r,s)\normal)}\\
=q\mapsto\Big(r\mapsto\big(s\mapsto F(r,s)\big)\Big)\big(\pr_1q\big)\big(\pr_2q\big)
\overset\beta=q\mapsto\big(s\mapsto F(r,s)\big|_{r:=\pr_1q}\big)(\pr_2q)\\
=q\mapsto\Big(s\mapsto F\big(\pr_1q,s\big)\Big)\big(\pr_2q\big)
\overset\beta=q\mapsto\Big(F\big(\pr_1q,s\big)|_{s:=\pr_2q}\big)
=q\mapsto F\big(\pr_1q,\pr_2q\big)\\
\overset\eta=q\mapsto F(q)\overset\eta=F.
\end{multline*}
\end{example}

\begin{remark}\label{exp-lambda2}
Here is an alternative argument, which uses (in addition to $\beta$- and
$\eta$-reductions) also $\eta$-expansions (but is perhaps conceptually clearer).

We have $\frak F(F)\overset\beta=r\mapsto\big(s\mapsto F(r,s)\big)$, and consequently
$\frak F(F)(R)\overset\beta=s\mapsto F(R,s)$, and hence also $\frak F(F)(R)(S)\overset\beta=F(R,S)$.

Also, $\frak G(G)\overset\beta=q\mapsto G\big(\pr_1q\big)\big(\pr_2q\big)$, and
therefore $\frak G(G)(Q)\overset\beta=G\big(\pr_1Q\big)\big(\pr_2Q\big)$, and hence also
$\frak G(G)(R,S)\overset\beta=G(R)(S)$.

Thus $\frak F\big(\frak G(G)\big)(R)(S)=\frak G(G)(R,S)=G(R)(S)$.

Also, $\frak G\big(\frak F(F)\big)(R,S)=\frak F(F)(R)(S)=F(R,S)$.

Finally,
\begin{multline*}
\frak G\big(\frak F(F)\big)\overset\eta=q\mapsto\frak G\big(\frak F(F)\big)(q)\overset\eta=
q\mapsto\frak G\big(\frak F(F)\big)(\pr_1q,\pr_2q)=q\mapsto F(\pr_1q,\pr_2q)\\
\overset\eta=q\mapsto F(q)\overset\eta=F.
\end{multline*}
Also,
\begin{multline*}
\frak F\big(\frak G(G)\big)\overset\eta=r\mapsto\frak F\big(\frak G(G)\big)(r)\overset\eta=
r\mapsto\Big(s\mapsto\frak F\big(\frak G(G)\big)(r)(s)\Big)=r\mapsto\Big(s\mapsto G(r)(s)\Big)\\
\overset\eta=r\mapsto G(r)\overset\eta=G.
\end{multline*}
\end{remark}

\begin{example} \label{retraction}
Let us construct closed expressions $\frak A:\Gamma\x\Delta\too\Gamma\x(\Gamma\too\Delta)$
and $\frak B:\Gamma\x(\Gamma\too\Delta)\too\Gamma\x\Delta$ such that $\frak B\big(\frak A(Q)\big)=Q$, but in general
$\frak A\big(\frak B(Q)\big)\ne Q$.

We set $\frak A=q\mapsto\big(\pr_1q,\,s\mapsto\pr_2q\big)$ and
$\frak B=q\mapsto\Big(\pr_1q,\,\pr_2(q)\big(\pr_1q\big)\Big)$.

Then $\frak A(S,T)\overset\beta=(S,s\mapsto T)$ and $\frak B(S,F)\overset\beta=\big(S,F(S)\big)$.

Hence $\frak B\big(\frak A(S,T)\big)=\frak B(S,s\mapsto T)\overset\beta=(S,T)$.
It follows that $\frak B\big(C(Q)\big)=Q$.

However, $\frak A\big(\frak B(S,F)\big)=\frak A\big(S,F(S)\big)=\big(S,s\mapsto F(S)\big)\neq\big(S,s\mapsto F(s)\big)
\overset\eta=(S,F)$.
\end{example}

\subsubsection{Simultaneous substitution} \label{simultaneous}

\begin{convention} \label{iterated products}
For $n\ge 2$ let us write $\Delta_1\x\dots\x\Delta_{n+1}$ to mean $(\Delta_1\x\dots\x\Delta_n)\x\Delta_{n+1}$
and $(T_1,\dots,T_{n+1})$ to mean $\big((T_1,\dots,T_n),T_{n+1}\big)$.
We also write $\pr_i:\Delta_1\x\dots\x\Delta_{n+1}\to\Delta_i$ to mean
$\pr_i\circ\pr_1$ for $i\le n$ and to mean $\pr_2$ for $i=n+1$.

We abbreviate $F\big((T_1,\dots,T_n)\big)$ by $F(T_1,\dots,T_n)$.
We also write $\Delta^n$ for the $n$-tuple product $\Delta\x\dots\x\Delta$.

We will further identify $\Delta_1\x\Delta_2\x\Delta_3$ with $\Delta_1\x(\Delta_2\x\Delta_3)$
via the ``function'' of type $\Delta_1\x\Delta_2\x\Delta_3\too\Delta_1\x(\Delta_2\x\Delta_3)$, defined by
$q\mapsto\Big(\pr_1q,\big(\pr_2q,\pr_3q\big)\Big)$, and the ``function'' of type
$\Delta_1\x(\Delta_2\x\Delta_3)\too\Delta_1\x\Delta_2\x\Delta_3$, defined by
$q\mapsto\Big(\pr_1q,\pr_1\big(\pr_2q\big),\pr_2\big(\pr_2q\big)\Big)$, which are easily seen
to be mutually inverse.

By an iterated use of this identification we also identify $\Delta_1\x\dots\x\Delta_n$ with any iterated
product of $\Delta_1,\dots,\Delta_n$.
\end{convention}

For each $i=1,\dots,n$ let $T_i:\Delta_i$, and let $x_i$ be a variable of type $\Delta_i$.
Let us write $\vec x:=(x_1,\dots,x_n)$ and $\vec T:=(T_1,\dots,T_n)$, and also
$\vec x|\hat k$ and $\vec T|\hat k$ for the same tuples with $x_k$ and $T_k$ omitted.
The {\it (capture-avoiding, strict) simultaneous substitution} $S|_{\vec x:=\vec T}$, where $S:\Gamma$,
is, when defined, an expression of type $\Gamma$, which is determined recursively as follows.
Let $y$ be a variable distinct from each $x_i$, let $\frak c$ be a constant of some type and $F$, $R$ and $S$
be expressions of some types.
Then:
\begin{enumerate}
\item $x_k|_{\vec x:=\vec T}=T_k$;
\item $y|_{\vec x:=\vec T}=y$;
\item $\frak c|_{\vec x:=\vec T}=\frak c$;
\item $(R,S)|_{\vec x:=\vec T}=(R|_{\vec x:=\vec T},\,S|_{\vec x:=\vec T})$;
\item $F(R)|_{\vec x:=\vec T}=(F|_{\vec x:=\vec T})(R|_{\vec x:=\vec T})$;
\item $x_k\mapsto R|_{\vec x:=\vec T}=x_k\mapsto (R|_{\vec x|\hat k:=\vec T|\hat k})$;
\item $y\mapsto R|_{\vec x:=\vec T}=\begin{cases}
\text{undefined,} & \text{if for some $i$,} \\
&\quad\text{$x_i$ occurs freely in $R$ and $y$ occurs freely in $T_i$}\\
y\mapsto(R|_{\vec x:=\vec T})& \text{otherwise.}
\end{cases}$
\end{enumerate}

\begin{example}  $\big((x,y)|_{x:=y}\big)\big|_{y:=x}=(y,y)$ and $\big((x,y)|_{y:=x}\big)\big|_{x:=y}=(x,x)$.
However, $(x,y)|_{x:=y,\, y:=x}=(y,x)$.
\end{example}

\begin{example} Somewhat unexpectedly,
$\frak G\big(x\mapsto(y\mapsto T)\big)=\frak G\big(y\mapsto(x\mapsto T)\big)$.
Indeed, $\frak G\big(x\mapsto(y\mapsto T)\big)(R)(S)=(y\mapsto T|_{x:=R})(S)$,
where the substitution $|_{x:=R}$ is undefined if $y$ occurs freely in $R$.
If $z$ is a variable that does not occur freely in $R$ and $T$, then up to $\alpha$-equivalence we have
\[(y\mapsto T|_{x:=R})(S)=\big(z\mapsto (T|_{y:=z})|_{x:=R})(S)=\big((T|_{y:=z})|_{x:=R}\big)\big|_{z:=S}=
T|_{x:=R,\,y:=S}.\]
Similarly, $\frak G\big(y\mapsto(x\mapsto T)\big)(R)(S)=T|_{x:=R,\,y:=S}$, and the assertion follows.
\end{example}

\begin{convention} \label{prod-convention}
If $x_0,\dots,x_n$ are variables of types $\Delta_0,\dots,\Delta_n$, and $T:\Delta$, an expression
\[x_0,\dots,x_n\mapsto T\] of type $\Delta_0\x\dots\x\Delta_n\too\Delta$ is defined recursively as
$\frak G\big(x_0\mapsto(x_1,\dots,x_n\mapsto T)\big)$ (see Convention \ref{iterated products}).
Thus, for instance, $\big(x,y\mapsto (x,y)\big)(y,x)=(y,x)$.

Up to $\alpha$-equivalence, $x_1,\dots,x_n\mapsto T$ is the same as
$q\mapsto T|_{x_1:=\pr_1q,\,\dots,\,x_n:=\pr_nq}$, where $q$ is
a variable of type $\Delta_1\x\dots\x\Delta_n$ that does not occur freely in $T$.
The new expression is defined without recursion, but is well-defined only up to $\alpha$-equivalence
(because of the variable $q$).
Let us note that the substitutions $|_{x_i:=\pr_iq}$ commute, since no $x_j$ occurs in
the expression $\pr_iq$.
So it does not matter if we make them simultaneously or in some order.
\end{convention}

\begin{remark}\label{confluence}
If $\vec x$ denotes an $n$-tuple of variables of some types with $n>0$, one may use the notation
$\vec x\mapsto T$ as if it were an expression of the appropriate type --- and not just an abbreviation
for such an expression.
In particular, we have a version of the $\beta$-reduction: $(\vec x\mapsto S)(\vec T)=S|_{\vec x:=\vec T}$,
and a version of the $\eta$-reduction: $\vec x\mapsto F(\vec x)=F$.

One may also use the notation $\vec T$ as if it were an expression of the appropriate type --- and not just
an abbreviation for such an expression.
We have a version of the $\beta$-reduction: $\pr_i(T_1,\dots,T_n)=T_i$, and a version of
the $\eta$-reduction: $(\pr_1 T,\dots,\pr_n T)=T$.
However, one must be careful not to assume that this all works also for $n=0$, since it does not.
Something much more involved does \cite{CdC} (see also \cite{dCK}).
\end{remark}

\begin{convention} We will regard $\Delta^0\x\Gamma$, $\Gamma\x\Delta^0$ and $\Delta^0\to\Gamma$ as abbreviations
for $\Gamma$ (without attempting to define $\Delta^0$).
\end{convention}

\subsubsection{Free substitution} \label{free substitution}

It is well-known (see e.g.\ \cite{SU}, \cite{LS}) that for every expression $T$ of any type $\Delta$ there
exists an expression $T_0:\Delta$ such that $T$ $\beta\eta$-reduces to $T_0$, and if $T$ $\beta\eta$-reduces
to $T_1$, then $T_1$ $\beta\eta$-reduces to $T_0$; moreover, every $T_0'$ satisfying the same property
is $\alpha$-equivalent to $T_0$.
Such a $T_0$ is called a {\it normal form} of $T$; and if $T$ is its own normal form,
it is said to be {\it in normal form}.
(For our purposes the normal form could be replaced in what follows by the $\beta$-normal form, which is defined
similarly using $\beta$-reduction.)

If $T_0$ is a normal form of $T$, and some $\beta\eta$-reduction of $T$ to $T_0$
involves no $\alpha$-conversions, then $T_0$ will be called a {\it strict normal form} of $T$, and $T$
will be said to {\it have a strict normal form}.
It is not hard to see that when a strict normal form of a $\lambda$-expression $T$ exists, it is unique (as
a $\lambda$-expression, and not just up to $\alpha$-equivalence).
Clearly, $T$ is its own strict normal form if and only if it is in normal form.

If $S$ is in normal form, and the expression $(\vec x\mapsto S)(\vec T)$ has a strict normal form,
where $\vec x=(x_1,\dots,x_n)$, $\vec T=(T_1,\dots,T_n)$, each $x_i$ is a variable of some type $\Delta_i$,
and each $T_i:\Delta_i$, then this strict normal form of $(\vec x\mapsto S)(\vec T)$ will be denoted by
$S[\vec x/\vec T]$, or in more detail by $S[x_1/T_1,\dots,x_n/T_n]$, and $\vec T$ will be called
{\it free for (substitution for) $\vec x$ in $S$,} or in more detail $T_1,\dots,T_n$ will be called
{\it free for (simultaneous substitution for) $x_1,\dots,x_n$ in $S$.}

\subsection{Languages and structures}\label{metavariables}

\subsubsection{Terms and $n$-terms}\label{terms}
A {\it first-order language} $\L$ provides, in particular, formal means to speak about {\it individuals},
which are the elements of a certain nonempty set $\D$, called the {\it domain of discourse}.
For instance, if $\L$ is the standard language of arithmetic, a good choice for $\D$ would be the set of
natural numbers.
Now $\D$ and its elements are regarded as ``semantic'' objects, whereas $\L$ itself is about ``syntax''
only.
This means, in particular, that $\L$ does not refer on its own to the elements of any specific set $\D$,
but is capable of doing it in multiple ways.
One particular way is specified by an {\it $\L$-structure} $\I$, also known as an {\it interpretation}
of $\L$, which includes a choice of a set $\D$ along with a number of other choices.
The formal definitions of $\L$ and $\I$ will be summarized at the end of \S\ref{formulas}, once we will
have introduced all of their ingredients.

The ``individual'' fragment of $\L$ is can be formulated in terms of a base type, which we denote $\0$.
The variables $\var_1^\0,\var_2^\0,\dots$ of this type are the {\it individual variables} of the language
$\L$.
For reasons of readability we will also use the alternative spelling $\tr{a,b,c,\dots,x,y,z}$ for the first
26 ones, reserving an upright sans-serif font for this purpose.
These are abbreviations, and not metavariables, as they are determined by the specific identifications
$\tr a=\var_1^\0$, $\tr b=\var_2^\0$, etc.
Metavariables are used, for instance, in a phrase like ``let $a$ and $b$ be individual variables'',
which amounts to saying ``let $a=\var_i^\0$ and $b=\var_j^\0$ for some unspecified $i$ and $j$
(possibly $i=j$)''.

$\L$ may also contain {\it function symbols} (such as $+$ and $*$), each standing for an operation on $\D$
(such as addition and multiplication of natural numbers) that inputs an $n$-tuple of individuals (for some
specified $n\ge 0$) and outputs one individual.
Thus an $n$-ary function symbol of $\L$ is a constant $\frak f$ of type $\0^n\too\0$, and its interpretation
$\frak f_\I$ is an $n$-ary operation $\D^n\to\D$.
Nullary function symbols, which are constants of type $\0$ (such as $\tr 0$ and $\tr 1$), are called
{\it individual constants} of $\L$ and stand for specific individuals (such as the natural numbers $0$ and $1$).

{\it Terms} of $\L$ are defined inductively, as built out of individual variables using the function
symbols.
A sample term in the case of the language of arithmetic is $\tr{(x+1)*y+x}$.
Thus terms of $\L$ are those expressions of type $\0$ that involve only variables of type $\0$, tuples,
function application, and a prescribed collection of constants of types $\0^n\too\0$.
(In particular, terms of $\L$ do not involve $\lambda$-abstraction.)

If $T$ is a {\it closed term} (i.e.\ a term that is a closed expression), its interpretation $|T|_\I$
is a specific element of $\D$.
To interpret an arbitrary term $T$ as an element of $\D$, one additionally needs a {\it variable assignment}
$\iass$, which assigns an individual $\iass(x)$ to each individual variable $x$.
Indeed, $\iass$ extends in the obvious way to a function $|\cdot|_\I^\iass\:\T\to\D$, where $\T=\T(\L)$ is
the set of terms of $\L$.
Thus for every term $T$ we have $|T|_\I^\iass\in\D$.

It is often convenient to consider expressions of the form $x_1,\dots,x_n\mapsto T$, where $T$ is a term
and the $x_i$ are pairwise distinct individual variables (in other words, each $x_i=\var_{i_n}^\0$, where
$i_1,\dots,i_n$ are pairwise distinct).
We will call these {\it $n$-terms}; of course, they are of type $\0^n\too\0$.
Let us note that the arithmetical 1-terms $\tr{x\mapsto x+y}$, $\tr{y\mapsto x+y}$ and $\tr{z\mapsto x+y}$
should not be conflated, since application works differently for them: $\tr{(x\mapsto x+y)(t)=t+y}$,
$\tr{(y\mapsto x+y)(t)=x+t}$ and $\tr{(z\mapsto x+y)(t)=x+y}$.
One can define $n$-terms inductively, as built out of {\it atomic $k$-terms}, which are individual variables
and expressions of the form $x\mapsto x$, where $x$ is an individual variable, by using appropriate lifts
of the function symbols.
For instance, using the lifted sum, from the arithmetical $2$-terms $\tr{x,y}\mapsto\tr{x*y}$ and
$\tr{y,z}\mapsto\tr{3*y}$ one obtains the arithmetical $3$-term $\tr{x,y,z}\mapsto\tr{x*y+3*y}$.
(However, one does not obtain anything in this fashion from the arithmetical $1$-terms
$\tr x\mapsto\tr{x*y}$ and $\tr y\mapsto\tr y$.)

In more detail, every $k$-term can be lifted to a $(k+l)$-term by introducing $l$ dummy
variables via $\frak F_l:(\0^k\too\0)\too(\0^k\x\0^l\too\0)$, defined by
$f\mapsto\Big(q\mapsto f\big(\pr_1q\big)\Big)$; and every function symbol, $\frak f:\0^n\too\0$,
can be lifted to $\frak f_m:(\0^m\too\0)^n\too(\0^m\too\0)$
via $\frak G_m:(\0^n\too\0)\too\big((\0^m\too\0)^n\too(\0^m\too\0)\big)$,
$c\mapsto\bigg(r\mapsto\Big(q\mapsto c\big(\pr_1(r)(q),\dots,\pr_n(r)(q)\big)\Big)\bigg)$.
Thus $\frak f_m(F_1,\dots,F_n)=\frak f\circ (F_1\x\dots\x F_n)$, where
$F_1\x\dots\x F_n:\0^m\too\0^n$ is defined by $q\mapsto\big(F_1(q),\dots,F_n(q)\big)$.
Let us note that an interpretation of a function symbol induces, via the analogue of $\frak G_m$,
an interpretation of its lifted version, which is by a function $\Hom(\D^m,\D)^n\to\Hom(\D^m,\D)$,
where $\Hom(X,Y)$ denotes the set of all maps $X\to Y$.

Using the inductive construction of $n$-terms, every {\it closed $n$-term} $F$ (i.e.\ an $n$-term that is
a closed expression) is interpreted by a specific function $|F|_\I\:\D^n\to\D$.
Also, given a variable assignment $\iass$, any $n$-term $F$ is interpreted by a function
$|F|_\I^\iass\:\D^n\to\D$.

Since $\lambda$-abstraction is traditionally not included in first-order languages, $n$-terms
are, strictly speaking, not included in the language for $n>0$.
This is not a serious issue, since an $n$-term uniquely corresponds to a pair consisting
of a term and an $n$-tuple of pairwise distinct positive integers.
However, there is no such simple connection between terms and $\alpha$-equivalence classes of $n$-terms.

\subsubsection{Formulas} \label{formulas}
In fact, the language $\L$ is not supposed to be a means to speak directly of individuals, but rather of
functions of individuals such as propositions or problems about individuals or tuples of individuals.
The case of $0$-tuples amounts to propositions/problems without parameters, which are really
the subject of zero-order logic --- but are also included in first-order logic.

Instead of dealing with actual human-language sentences expressing propositions or problems,
one fixes a set $\O$ of mathematical objects, which can be thought of as Platonic ``ideas'' (or rather
Fregean ``senses'') of propositions or problems (thus an element of $\O$ is supposed to encode all
propositions or problems without parameters that express the same ``idea'' in different words).
A function $\D^m\to\O$ can then be thought of as encoding a proposition/problem with $m$ parameters.

Of course, on assuming that a proposition has no ``sense'' beyond its ``reference'' (that is, beyond
its truth value), one must have $\O=\{0,1\}$.
However, even with classical logic, there are more interesting options, discovered by Leibniz and Euler and
reviewed in \S\ref{Euler} below.

Now these all are ``semantic'' objects, whereas $\L$ itself is about ``syntax'' only.
Thus, a choice of $\O$ is included not in the language $\L$,
but in its interpretation $\I$, which will still include a number of further choices.

Now, just like elements of $\D$ are represented by terms and other expressions of type $\0$, elements of
$\O$ will be represented by expressions of a new base type, $\1$.
$\L$ may contain constants of type $\0^n\too\1$, which are called {\it $n$-ary predicate constants}.
Thus an $n$-ary predicate constant $\frak p$ is interpreted by a predicate $|\frak p|_\I\:\D^n\to\O$.
Examples include the arithmetical binary predicates $=$ and $<$ and the geometric ternary predicate
of betweenness (for points on a line).

The variables $\var_1^{\0^n\too\1},\var_2^{\0^n\too\1},\dots$ of type $\0^n\too\1$ for each $n=0,1,2\dots$
are called the {\it predicate variables} of $\L$.
In dealing with classical logic (or modal logics based on classical logic), we will also call nullary
predicate variables {\it propositional variables}, and use the nicknames $\fm {a,b,c,\dots,x,y,z}$ for
the first 26 predicate variables of each arity, reserving a fancy (Euler) upright serif font for this
purpose.
To indicate the arity of a predicate variable $f$ (here $f$ is a metavaraible, in contrast to
the abbreviations $\fm{a,b,\dots,f,\dots}$), we may write the appropriate $\eta$-expansion of $f$
such as $\tr x\mapsto f(\tr x)$ or $\tr{x,y}\mapsto f(\tr{x,y})$ instead of $f$.
(See Convention \ref{prod-convention}.)
In dealing with intuitionistic logic (and its non-classical extensions), the predicate variables can be
called {\it problem variables}; in this case, we will also write
$\fm {\alpha,\beta,\gamma,\dots,\chi,\psi,\omega}$ for the first 24 ones of each arity, reserving also
the Greek part of the font.

The language $\L$ may further contain constants of type $\1^n\too\1$, which stand for logical operations,
and are called {\it $n$-ary connectives}.
Thus for every $n$-ary connective $\frak f$, its interpretation $|\frak f|_\I$ is a logical operation
$\O^n\to\O$.

It should be noted that nullary connectives are the same kind of $\lambda$-expressions as
nullary predicate constants (all of them are constants of type $\1$).
It is nevertheless convenient to distinguish them, since connectives belong to the language of a logic,
whereas predicate constants only belong to the language of a theory over a logic.

$\L$ may also contain constants of type $(\0\too\1)\too\1$, which are called {\it quantifiers}.
A quantifier $\frak q$ is interpreted by a function $|\frak q|_\I\:\Hom(\D,\O)\to\O$.
If $\frak q$ is a quantifier, $F$ is an expression of type $\1$, and $x$ is an individual variable,
then $\frak qx\, F$ abbreviates the expression $\frak q(x\mapsto F)$.
Here the individual variable $x$ is bound by the $\lambda$-abstraction $\mapsto$; but traditionally
it is also said to be {\it bound} by the quantifier.
Also, $F$ is called the {\it scope} of (the occurrence of) the quantifier.
Moreover, $\frak qx_0,\dots,x_n\ F$ abbreviates $\frak qx_0\,(\frak qx_1,\dots,x_n\ F)$.

\begin{example}
The language of intuitionistic logic contains a nullary connective $\ab$ (``absurdity''), which stands for
a problem that has no solutions, binary connectives $\land$ (``conjunction''), $\lor$ (``disjunction'') and
$\to$ (``implication''; not to be confused with the type constructor $\too$), and quantifiers $\exists$
(``existential'') and $\forall$ (``universal'').
There is no redundancy between these ingredients of the language (see \S\ref{independence} below).

It is also convenient to add, as is often done, some ``syntactic sugar'' to the language of intuitionistic
logic: a logical constant $\triv$ (``triviality'') is defined as $\ab\to\ab$; a unary connective $\neg$ 
(``negation'') is defined as $\alpha\mapsto\alpha\to\ab$ (thus $\triv=\neg\ab$), and a binary connective 
$\tofrom$ (``equivalence'') is defined as $\alpha,\beta\mapsto(\alpha\to\beta)\land(\beta\to\alpha)$.
\end{example}

\begin{example}
The language of classical logic can be taken to be the same as that of intuitionistic logic, except for
the logical constants, which we will write as $\cltop$ (``truth'') and $\clbot$ (``falsity'').
(This complements our convention of denoting predicate
variables by Greek letters in the intuitionistic case and by Roman letters in the classical case).
There is additional redundancy here: the classical $\land$ can be defined as
$p,q\mapsto\neg(p\to\neg q)$, the classical $\lor$ as $p,q\mapsto\neg p\to q$, and the classical $\exists$
as $f\mapsto\neg\forall x\,\neg f(x)$.
\end{example}

\begin{example} The language of a theory over classical or intuitionistic logic (for instance, the ``empty theory'',
which has no axioms) may additionally contain finitely many function symbols and/or predicate constants.
\end{example}

An {\it atomic formula} of $\L$ is an expression of type $\1$ obtained by applying either an $n$-ary predicate
constant or an $n$-ary predicate variable to an $n$-tuple of terms.
A {\it formula} of $\L$ is an expression built out of atomic formulas using the connectives and
the quantifiers.
In other words, formulas are expressions of type $\1$ involving only variables of types $\0$ and $\0^n\too\1$,
tuples, function application, a prescribed collection of constants of types $\0^n\too\0$, $\0^n\too\1$ and
$\1^n\too\1$, and a prescribed collection of constants of type $(\0\too\1)\too\1$ used in conjunction with
certain specific $\lambda$-abstractions.
(No other uses of $\lambda$-abstraction are allowed.)

A formula is called {\it closed} (or ``closed as a formula'') if individual variables do not occur freely
in it, and {\it $\lambda$-closed} (or ``closed as a $\lambda$-expression'') if no variables occur freely in it.
The arithmetical formula
$\fm{\forall \tr{x,y}\ \big(\alpha(\tr{x,y})\land\alpha(\tr{y,x})\to \tr x=\tr y}\big)$ is closed, but not
$\lambda$-closed.
The arithmetical formula $\fm{\forall \tr{x,y,z}\ \big(\tr{x+z}=\tr{y+z}\to\tr x=\tr y}\big)$
is $\lambda$-closed.

This completes our definition of a first-order language $\L$.
Namely, it consists of the following sets of typed $\lambda$-expressions (variables and constants only),
where $n$ ranges over $\N=\{0,1,2,\dots\}$:
\begin{itemize}
\item[(1$_n$)] of $n$-ary function symbols (of type $\0^n\too\0$);
\item[(2$_n$)] of $n$-ary predicate constants (of type $\0^n\too\1$);
\item[(3$_n$)] of $n$-ary connectives (of type $\1^n\too\1$);
\item[(4)] of quantifiers (of type $(\0\too\1)\too\1$);
\item[(5)] of individual variables (of type $\0$);
\item[(6$_n$)] of $n$-ary predicate variables (of type $\0^n\too\1$);
\end{itemize}
Each of the sets (5), (6$_n$) is countably infinite, and each of the remaining sets is required to be finite.
The sets (3$_n$)--(6$_n$) are called the {\it purely logical part} of the first-order language $\L$.
The set of terms $\T(\L)$ and the set of $\alpha$-equivalence classes of formulas $\Fm_0(\L)$ can be
reconstructed from the sets (1)--(6), so they do not have to be included in $\L$.

The definition of an $\L$-structure $\I$ is also complete: it consists of the sets $\D$ and $\O$ and of
the following functions, where $n$ ranges over $\N$:
\begin{itemize}
\item[(1$_n'$)] from the set of $n$-ary function symbols to $\Hom(\D^n,\D)$;
\item[(2$_n'$)] from the set of $n$-ary predicate constants to $\Hom(\D^n,\O)$;
\item[(3$_n'$)] from the set of $n$-ary connectives to $\Hom(\O^n,\O)$;
\item[(4$'$)] from the set of quantifiers to $\Hom(\Hom(\D,\O),\O)$.
\end{itemize}

We have not yet defined interpretation of formulas in an $\L$-structure; this will be done next, with
the aid of $n$-formulas.

\subsubsection{$n$-Formulas and $(n,\vec r)$-formulas} \label{n-formulas}

By an {\it $n$-formula} we mean an expression of the form
$x_1,\dots,x_n\mapsto F$, where the $x_i$ are pairwise distinct individual variables and $F$ is a formula.
Thus $n$-formulas are of type $\0^n\too\1$; Kleene \cite{Kl}*{\S34} called them ``name forms''.
An $n$-formula is {\it closed} if individual variables do not occur freely in it, and {\it $\lambda$-closed}
if no variables occur freely in it.
Let us note that quantifiers do not operate on formulas, but rather transform a 1-formula
into a formula; thus in some sense $n$-formulas are inherent in first-order languages --- even though
traditionally these do not include $n$-formulas.

It should be noted that an arbitrary formula can be viewed as a closed $n$-formula applied
to an $n$-tuple of individual variables; thus closed $n$-formulas can be viewed as more fundamental
than formulas.
Moreover, an arbitrary $k$-formula can be viewed as a closed $(n+k)$-formula pre-composed with
an expression of the form $x_1,\dots,x_n\mapsto (x_1,\dots,x_n,y_1,\dots,y_k)$, where the $x_i$ and
the $y_j$ are pairwise distinct individual variables.
Thus closed $n$-formulas can be viewed as more fundamental even than arbitrary $k$-formulas.

Using $k$-formulas, one can give an alternative recursive description of formulas (and $n$-formulas), which
is particularly suitable for closed formulas (and closed $n$-formulas).
An {\it atomic (closed) $n$-formula} is an $(n+k)$-ary predicate variable or constant
pre-composed with a (closed) expression of the form $x_1,\dots,x_n\mapsto (T_1,\dots,T_{n+k})$, where the
$x_i$ are pairwise distinct individual variables and each $T_j$ is a term.
In particular, each $x_1,\dots,x_n\mapsto T_j$ is a (closed) $n$-term.
A {\it (closed) $n$-formula} is an expression built out of atomic (closed) $k$-formulas using
appropriate lifts of the connectives and the quantifiers.
For instance, using the lifted conjunction, from closed $2$-formulas $\tr{x,y}\mapsto\fm\phi(\tr{x,y})$ and
$\tr{y,z}\mapsto\fm\psi(\tr{y,z})$ one obtains the closed $3$-formula
$\tr{x,y,z}\mapsto\fm\phi(\tr{x,y})\land\fm\psi(\tr{y,z})$; and using the lifted existential
quantifier, from the latter one obtains the closed $2$-formula
$\tr{x,z}\mapsto\exists\tr y\,\fm\phi(\tr{x,y})\land\fm\psi(\tr{y,z})$.

In more detail, every $k$-formula can be made into a $(k+l)$-formula by introducing $l$ dummy
variables via $\frak F_l:(\0^k\too\1)\too(\0^k\x\0^l\too\1)$, defined by
$f\mapsto\Big(q\mapsto f\big(\pr_1 q\big)\Big)$; and every connective on formulas, $\frak c:\1^n\too\1$,
can be made into a connective on $m$-formulas, $\frak c_m:(\0^m\too\1)^n\too(\0^m\too\1)$,
via $\frak G_m:(\1^n\too\1)\too\big((\0^m\too\1)^n\too(\0^m\too\1)\big)$,
defined by $c\mapsto\bigg(r\mapsto\Big(q\mapsto c\big(\pr_1(r)(q),\dots,\pr_n(r)(q)\big)\Big)\bigg)$.
Thus $\frak c_m(F_1,\dots,F_n)=\frak c\circ (F_1\x\dots\x F_n)$, where $F_1\x\dots\x F_n:\0^m\too\1^n$
is defined by $q\mapsto\big(F_1(q),\dots, F_n(q)\big)$.
Also, every predicate variable or predicate constant, $\phi:\0^n\too\1$, can be lifted to an expression
$\phi_m:(\0^m\too\0)^n\too(\0^m\too\1)$, so as to produce $m$-formulas out of $m$-terms,
via $\frak H_m:(\0^n\too\1)\too\big((\0^m\too\0)^n\too(\0^m\too\1)\big)$, which is defined similarly to
$\frak G_m$.
Finally, every formula quantifier, $\frak q:(\0\too\1)\too\1$, can be made into an $m$-formula quantifier
$\frak q_m:(\0^{m+1}\too\1)\too(\0^m\too\1)$ via
$\frak K_m:\big((\0\too\1)\too\1\big)\too\big((\0^m\x\0\too\1)\too(\0^m\too\1)\big)$, defined by
$q\mapsto\bigg(f\mapsto\Big(r\mapsto q\big(\frak F(f)(r)\big)\Big)\bigg)$,
where $\frak F$ is as in Example \ref{exp-lambda}.
Thus $\frak q_m(y_1,\dots,y_m,y_{m+1}\mapsto F)=y_1,\dots,y_m\mapsto \frak q(y_{m+1}\mapsto F)$.
Let us note that interpretations of formula connectives and quantifiers induce, by using the analogues
of $\frak G_m$, $\frak H_m$ and $\frak K_m$, interpretations of $m$-formula connectives, predicate
constants/variables and quantifiers, which are, respectively, by functions $\Hom(\D^m,\O)^n\to\Hom(\D^m,\O)$,
$\Hom(\D^m,\D)^n\to\Hom(\D^m,\O)$ and $\Hom(\D^{m+1},\O)\to\Hom(\D^m,\O)$.

A {\it predicate valuation} $\pval$ assigns a function $\pval(\phi)\:\D^m\to\O$ to each $m$-ary predicate
variable $\phi$, for $m=0,1,2,\dots$.
(This will also be called simply ``valuation'', especially in the context of intuitionistic logic, where
the ``predicate variables'' are really problem variables.)
Traditionally valuations are found only in zero-order logic, since modern treatments of first-order logic
do not include predicate variables in the language.

Given a predicate valuation $\pval$, for every closed formula $F$ we get a specific element
$|F|^\pval_\I\in\O$; and, more generally, for every closed $n$-formula $F$ a specific function
$|F|^\pval_\I\:\D^n\to\O$.
This function is defined straightforwardly using the previous recursive description of closed $n$-formulas,
and it is clear from this definition that it depends only on the $\alpha$-equivalence class of $F$.
Let us note that if $F$ is $\lambda$-closed, then $|F|_\I:=|F|^\pval_\I$ does not depend on $\pval$.

Given a variable assignment $\iass$, every predicate valuation $\pval$ also yields, for each $n=0,1,\dots$,
a function $|\cdot|^{\iass\pval}_\I\:\Fm_n\to\Hom(\D^n,\O)$, where $\Fm_n=\Fm_n(\L)$ is the set of
$\alpha$-equivalence classes of arbitrary $n$-formulas of $\L$.
Thus for every formula $F$ we get a specific element $|F|^{\iass\pval}_\I\in\O$, well-defined up
to $\alpha$-equivalence of $F$.

It is often convenient to consider expressions of the form $\phi_1,\dots,\phi_l\mapsto F$, where $F$
is an $n$-formula and the $\phi_i$ are pairwise distinct predicate variables of (not necessarily
distinct) arities $r_1,\dots,r_m$.
We will call such an expression an {\it $(n;\,r_1,\dots,r_m)$-formula}, or an {\it $(n,\vec r)$-formula},
where $\vec r=(r_1,\dots,r_m)$; it is of type $\prod_{i=1}^m(\0^{r_i}\too\1)\too(\0^n\too\1)$.
An $(n,\vec r)$-formula is called {\it closed} if no individual variables occur freely in it, and
{\it $\lambda$-closed} if no variables occur freely in it.
Similarly to the above, one can define lifted connectives, predicate variables/constants and quantifiers
for $(n,\vec r)$-formulas and use them to give an inductive construction of $(n,\vec r)$-formulas.
Given a predicate valuation $\pval$ and a variable assignment $\iass$, an arbitrary
$(n;\,r_1,\dots,r_m)$-formula $G$ is interpreted, using this inductive construction, by a function
$|G|^{\iass\pval}_\I\:\prod_{i=1}^m\Hom(\D^{r_i},\O)\to\Hom(\D^n,\O)$, which depends
only on the $\alpha$-equivalence class of $G$.
If $G$ is closed, then $|G|^\pval_\I:=|G|^{\iass\pval}_\I$ does not depend on $\iass$; and if $G$ is
$\lambda$-closed, then $|G|_\I:=|G|^\pval_\I$ does not depend on $\pval$ either.

\begin{remark} Let us note that $\lambda$-closed $(n,\vec r)$-formulas are similar to traditional schemata.
Indeed, term variables and $r_i$-ary formula variables that occur in a schema can be considered as ``blanks''
that may be filled in by terms and by $r_i$-formulas; to be precise, each blank must be highlighted by
a marker so that blanks of the same color are to be filled in by the same term or $r_i$-formula.
Let us note, however, that these ``blanks'' themselves do not correspond to any {\it variables} in the sense of
$\lambda$-calculus, but rather to expressions of the form $v\mapsto v$, where $v$ is an individual variable or
a predicate variable.
\end{remark}

\subsubsection{Meta-formulas}\label{meta-formulas}

We will now define the extension of the language $\L$ by the {\it language of meta-logic} --- not to be
confused with a meta-language of $\L$ (which we do not need to treat formally since what is usually
expressed by meta-variables in schemata will be achieved by means of the language of meta-logic).
The additional ingredients of the extension of $\L$ are no longer specific to one's choice of a logic
(and for this reason we need no separate notation for the extension); instead, they will be used
to express meta-problems (or rather their Platonic ``ideas''/Fregean ``senses'').
Just like propositions and problems without parameters are interpreted by elements of the set $\O$,
meta-problems without parameters will be interpreted by elements of another set $\Qm$.
We will further assume that there is a ``reflection'' function $\ocf:\O\to\Qm$ converting a
proposition into a meta-problem asking one to prove it (or a problem into a meta-problem asking one
to solve this problem); and a ``truth value'' (or ``meta-reflection'') function
$\wnf\:\Qm\to\{\Top,\Bot\}$ converting a meta-problem into a judgement (=meta-meta-proposition)
asserting the solubility of this meta-problem.
The choice of $\Qm$, $\ocf$ and $\wnf$ is included, along with a choice of an $\L$-structure $\I$
and some further data to be discussed in a moment, in a {\it meta-$\L$-structure} $\J$, also called
a {\it meta-interpretation} of $\L$.

We will see in \S\ref{models} that the usual, Tarski-style model theory corresponds to the straightforward
{\it two-valued meta-interpretation} $\J=\I_+$, which sets $\Qm=\{\Top,\Bot\}$ and $\wnf=\id$ no matter what
$\O$ is (i.e., meta-problems are postulated to have no observable ``sense'' beyond their ``reference'').
However, we will also see in \S\ref{meta-semantics} that, in fact, a more interesting option is contained
in Kolmogorov's implicit problem interpretation of the {\it meta-logic} of intuitionistic logic \cite{Kol}.
Thus we leave $\Qm$ unspecified at this point.

Now this all is semantics (or rather meta-semantics); to represent it syntactically, we need yet another
base type, to be denoted $\m$, along with a constant $\Rt:\1\to\m$.
Thus $\ocf$ is supposed to be an interpretation of $\Rt$.
Now any formula can be regarded, via $\Rt$, as an expression of type $\m$.
In fact, this will be essentially our only source of ``elementary'' expressions of type $\m$, since we are not
going to use any variables of type $\m$.
Thus, we can already define an {\it atomic meta-formula} as an expression of the form $\Rt(F)$,
where $F$ is an arbitrary formula.

Arbitrary {\it meta-formulas} are built out of atomic meta-formulas using only the following constants:
two {\it meta-connectives} and two varieties of a {\it meta-quantifier} (the universal one).
The meta-connectives are the {\it meta-conjunction} $\&:\m\x\m\too\m$, and the {\it meta-implication}
$\impord:\m\x\m\too\m$.%
\footnote{The meta-absurdity, the meta-disjunction and the existential meta-quantifier are not really
needed, but will eventually be discussed as a side remark (see \S\ref{meta-disjunction}).
All of these, as well as the meta-conjunction, would be definable (see \S\ref{meta-disjunction}) in terms
of quantification over variables of type $\m$ (which we do not include in our meta-logic, in contrast to Paulson).}
Their interpretations $|\&|_\J\:\Qm\x\Qm\to\Qm$ and $|\impord|_\J\:\Qm\x\Qm\to\Qm$ are
supposed to express conjunction and implication between judgements.
Under the two-valued meta-interpretation these are given by the usual truth tables, i.e.,
$|\&|_{\I_+}(q,r)=\Top$ if and only if $q=r=\Top$, and $|\impord|_{\I_+}(q,r)=\Bot$
if and only if $q=\Top$ and $r=\Bot$.
Then there is the {\it first-order meta-quantifier} $\q:(\0\too\m)\too\m$ and
the {\it $n$-ary second-order meta-quantifier} $\q_n:((\0^n\too\1)\too\m)\too\m$ for each
$n=0,1,2,\dots$.
The first-order meta-quantifier $\q$ is interpreted by a function $|\q|_\J:\Hom(\D,\Qm)\to\Qm$;
and the $n$-ary second-order meta-quantifier $\q_n$ by a function
$|\q_n|_\J\:\Hom(\Hom(\D^n,\O),\Qm\big)\to\Qm$.
Under the two-valued meta-interpretation these behave like the classical universal quantifier, i.e.,
$|\q|_{\I_+}(f)=\Top$ if and only $f(x)=\Top$ for all $x\in\D$, and $|\q_n|_{\I_+}(f)=\Top$
if and only $f(\phi)=\Top$ for all $\phi\in\Hom(\D^n,\O)$.

In practice, meta-quantifiers will be written like the old-style (early 20th century) universal quantifiers,%
\footnote{Peano's notation for universal quantification, superseded by Gentzen's analogue
$\forall$ of Peano's $\exists$.} but with fancy parentheses $\mq{\cdot}$ so as to avoid visual confusion
with the ordinary parentheses $(\cdot)$.
Namely, if $\frak q:(\Delta\too\m)\too\m$ is a meta-quantifier (either of them), $\F$ is an expression of
type $\m$, and $x$ is a variable of type $\Delta$, then $\mq{x}\,\F$ abbreviates the expression
$\frak q(x\mapsto \F)$.
Here $x$ is said to be {\it bound} by the meta-quantifier, and $\F$ is called the {\it scope} of
(the occurrence of) the meta-quantifier.
Moreover, $\mq{x_0,\dots,x_n}\,\F$ abbreviates $\mq{x_0}(\mq{x_1,\dots,x_n}\,\F)$, where the $x_i$ may be
of different types.

This completes our definition of the meta-logical extension of $\L$: it consists of the constants
$\Rt$, $\mand$, $\imp$, $\q$ and $\q_0,\q_1,\dots$.
Since these are fixed forever (in the present paper), the set of meta-formulas can be reconstructed
from $\L$ alone.

The definition of a meta-$\L$-structure $\J$ is also complete: it consists of an $\L$-structure $\I$,
the set $\Qm$, the maps $\ocf$ and $\wnf$, and two elements $|\!\mand\!|_\J$, $|\!\!\imp\!\!|_\J$
of $\Hom(\Qm\x\Qm,\Qm)$, an element $|\q|_\J$ of $\Hom(\Hom(\D,\Qm),\Qm)$ and an element
$|\q_n|_\J$ of $\Hom(\Hom(\Hom(\D^n,\O),\Qm),\Qm)$ for each $n\in\N$.

When no confusion arises, we will omit the symbol $\Rt$ in writing out expressions of type $\m$.
Accordingly, formulas will be regarded as special (namely, atomic) meta-formulas.
However, one should keep in mind that meta-connectives and meta-quantifiers cannot be used {\it inside}
of a formula.

As usual, $\F\iff\G$ abbreviates $(\F\imp\G)\mand(\G\imp\F)$; ``$\iff$''
is called {\it meta-equivalence}.
We will stick to the following order of precedence of logical and meta-logical symbols (in groups of equal
priority, starting with higher precedence/stronger binding):
\begin{enumerate}
\item $\neg$, $\exists$ and $\forall$;
\item $\land$ and $\lor$;
\item $\to$ and $\tofrom$;
\item $\mq{\cdot}$;
\item $\&$;
\item $\imp$ and $\iff$.
\end{enumerate}

By an {\it $n$-meta-formula} we mean an expression of the form $x_1,\dots,x_n\mapsto\F$, where
the $x_i$ are pairwise distinct individual variables and $\F$ is a meta-formula.
By an {\it $(n;\,r_1,\dots,r_m)$-meta-formula}, or an {\it $(n,\vec r)$-meta-formula}, where
$\vec r=(r_1,\dots,r_m)$, we mean an expression of the form $\phi_1,\dots,\phi_m\mapsto\G$, where
the $\phi_i$ are pairwise distinct predicate variables of (not necessarily distinct) arities $n_1,\dots,n_m$,
and $\G$ is an $n$-meta-formula.
Thus an $(n,\vec 0)$-meta-formula (where $\vec 0$ is the unique $0$-tuple) is an $n$-meta-formula, and
a $0$-meta-formula is a meta-formula.
Let us call an $(n,\vec r)$-meta-formula {\it closed} if individual variables do not
occur freely in it, and {\it $\lambda$-closed} in no variables occur freely in it.

Like before, one can lift the meta-connectives and the varieties of the meta-quantifier to the level
of $(n,\vec r)$-meta-formulas, and use the lifts to give an inductive construction of
$(n,\vec r)$-meta-formulas.
(Let us note that for a meta-formula $\F$, such a recursive description will involve
$(n;\,r_1,\dots,r_m)$-meta-formulas with $n>0$ if $\F$ contains either a meta-level quantifier or
a first-order meta-quantifier, and with $m>0$ if $\F$ contains a second-order meta-quantifier.)

Using this inductive construction, every $(n;\,r_1,\dots,r_m)$-meta-formula $\F$ along with
a variable assignment $\iass$ and a predicate valuation $\pval$ are easily seen to yield a function
$|\F|^{\iass\pval}_\J\:\prod_{i=1}^m\Hom(\D^{r_i},\Qm)\to\Hom(\D^n,\Qm)$.
If $\F$ is closed, then this function, written $|\F|^\pval_\J$, does not depend on $\iass$;
and if $\F$ is $\lambda$-closed, then the function, written $|\F|_\J$, also does not depend on $\pval$.

In particular, a $\lambda$-closed meta-formula $\F$ is interpreted by a specific element $|\F|_\J\in\Qm$.
A closed meta-formula $\F$ along with a predicate valuation $\pval$ yield an element $|\F|^\pval_\J\in\Qm$.
Given a predicate valuation $\pval$ and a variable assignment $\iass$, an arbitrary meta-formula $\F$ is
interpreted by an element $|\F|^{\iass\pval}_\J\in\Qm$.

\formulas

\begin{example} Let us parse the following $\mm\lambda$-closed meta-formula:
\[\mq{\gamma}\ {\gamma\lor\neg\gamma}.\]
We have omitted the reflection function, and here it is shown:
\[\mq{\gamma}\ \Rt({\gamma\lor\neg\gamma}).\]
Let us also expand the meta-quantifier, so as to make the $\mm\lambda$-abstraction explicit:
\[\q_0\big(\gamma\mapsto\Rt({\gamma\lor\neg\gamma})\big).\]

Now let us assume that some meta-interpretation $\J$ is given.
Then the $\mm\lambda$-closed $(0;0)$-formula $\gamma\mapsto{\gamma\lor\neg\gamma}:\1\to\1$
is interpreted by a function $\O\to\O$.
Therefore $\gamma\mapsto\Rt({\gamma\lor\neg\gamma}):\1\to\m$ is interpreted
by a function $\O\to\Qm$.
Hence $\q_0\big(\gamma\mapsto\Rt({\gamma\lor\neg\gamma})\big):\m$ is interpreted by
a specific element of $\Qm$.
\end{example}

\begin{example} Let us parse the following $\mm\lambda$-closed meta-formula:
\[\mq{\gamma}\big(\mq{\tr x}\ {\gamma(\tr x)\imp\forall \tr x\,\gamma(\tr x)}\big).\]
We have omitted the reflection functions, and here they are shown:
\[\mq{\gamma}\Big(\mq{\tr x}\ \Rt\big(\gamma(\tr x)\big)\imp
\Rt\big(\forall\tr x\,\gamma(\tr x)\big)\Big).\]
Let us also expand the meta-quantifiers, so as to make the $\mm\lambda$-abstractions explicit:
\[\q_1\Bigg(\gamma\mapsto\bigg(\q\Big(\tr x\mapsto\Rt\big(\gamma(\tr x)\big)\Big)\imp
\Rt\big(\forall\tr x\,\gamma(\tr x)\big)\bigg)\Bigg).\]

Now let us assume that some meta-interpretation $\J$ is given.
Then the $\mm\lambda$-closed $(0;\,1)$-formula
$\gamma\mapsto\forall\tr x\,\gamma(\tr x):(\0\too\1)\too\1$ is interpreted by
a function $\Hom(\D,\O)\to\O$ (namely, the interpretation of $\forall$), and the $\mm\lambda$-closed
$(1;\,1)$-formula
$\gamma\mapsto\big(\tr x\mapsto\gamma(\tr x)\big):(\0\too\1)\too(\0\too\1)$
is interpreted by a function $\Hom(\D,\O)\to\Hom(\D,\O)$ (namely, the identity function).
Hence
$\gamma\mapsto\Rt\big(\forall\tr x\,\gamma(\tr x)\big):(\0\too\1)\too\m$
is interpreted by a function $\Hom(\D,\O)\to\Qm$, and
$\gamma\mapsto\Big(\tr x\mapsto\Rt\big(\gamma(\tr x)\big)\Big):(\0\too\1)\too
(\0\too\m)$
is interpreted by a function $\Hom(\D,\O)\to\Hom(\D,\Qm)$.
Next,
$\gamma\mapsto\q\Big(\tr x\mapsto\Rt\big(\gamma(\tr x)\big)\Big):(\0\too\1)\too\m$
is interpreted by a function $\Hom(\D,\O)\to\Qm$, and consequently
\[\gamma\mapsto\bigg(\q\Big(\tr x\mapsto\Rt\big(\gamma(\tr x)\big)\Big)\imp
\Rt\big(\forall\tr x\,\gamma(\tr x)\big)\bigg):(\0\too\1)\too\m\]
is also interpreted by a function $\Hom(\D,\O)\to\Qm$.
Finally, the meta-formula in question is interpreted by the interpretation of $\q_1$ applied
to the latter function; this is a specific element of $\Qm$.
\end{example}

\metameta

By a {\it principle} we mean a $\lambda$-closed meta-formula involving no meta-connectives.
In other words, a principle is a formula meta-quantified over all its free individual variables and
over all its predicate variables.
A meta-formula will be called {\it purely logical} if it involves no function symbols and no predicate constants.

\begin{example}
Fermat's theorem is a formula, like most other mathematical theorems that are expressible in
a first-order language, whereas the principle of excluded middle is a purely logical principle, like most other
logical principles that are expressible in a first-order language.
There are important first-order theories (such as Peano arithmetic, Tarski's elementary geometry and
Zermelo--Fraenkel set theory) where most axioms and theorems are formulas, but some are principles,
though none are purely logical (e.g.\ the principle of mathematical induction, Tarski's continuity principle
and the set-theoretic principles of separation and replacement).
\end{example}

A meta-formula $\F$ is said to be {\it valid} in a meta-$\L$-structure $\J$ if
$\wnf(|\F|_\J^{\iass\pval})=\Top$ for arbitrary variable assignment $\iass$ and predicate valuation $\pval$.
When $\F$ is closed, this is equivalent to saying that
$\wnf(|\F|_\J^\pval)=\Top$ for an arbitrary predicate valuation $\pval$; and when $\F$ is
$\lambda$-closed, this is equivalent to saying that $\wnf(|\F|_\J)=\Top$.

A meta-formula $\F$ is said to be {\it valid} in an $\L$-structure $\I$ if it is valid in
the two-valued meta-interpretation $\I_+$.
We also call $\F$ {\it valid} in the pair $(\I,\pval)$ if $\wnf(|\F|_{\I_+}^{\iass\pval})=\Top$
for arbitrary variable assignment $\iass$.
Dually, $\F$ {\it satisfiable} in the pair $(\I,\pval)$ if $\wnf(|\F|_{\I_+}^{\iass\pval})=\Top$
for some variable assignment $\iass$.

We denote by $\Turnstile\F$, or in more detail $\Turnstile_\J\F$, the judgement that
the meta-formula $\F$ is valid in $\J$.
The judgement $\Turnstile_\J(\F_1\mand\dots\mand\F_m\imp\G)$ is also abbreviated by
\[\F_1,\dots,\F_m\Turnstile_\J\G.\]
When this judgement is true, we also say that $\G$ is a {\it semantic consequence} of
the $\F_i$ in the meta-$\L$-structure $\J$.
When $\J=\I_+$, we abbreviate $\Turnstile_{\I_+}$ by $\Turnstile_\I$, and speak of
{\it semantic consequence} in the $\L$-structure $\I$.%
\footnote{When the $\F_i$ and $\G$ are formulas, this is Tarski's standard definition of semantic
consequence \cite{Ta1}.}

\subsection{Meta-logic}

Our meta-logic is in essence a second-order fragment of {\tt Isabelle}'s meta-logic \cite{Pau1}, which is in turn
essentially a fragment of higher-order intuitionistic logic.%
\footnote{Our meta-logic only deals with meta-formulas, whereas {\tt Isabelle}'s with arbitrary expressions
of types generated from the base types $\0$, $\1$, $\m$ using the type constructor $\too$.
{\tt Isabelle}'s meta-logic additionally includes quantification over variables of arbitrary types
(of which only two are needed for dealing with first-order logics, so we treat them separately),
equality, $\lambda$-abstraction and function application (we prefer to leave these at the informal
meta-meta-level, since their explicit treatment does not seem to add much conceptual clarity in our case).
Due to our omission of {\tt Isabelle}'s meta-rules for equality, our meta-logic is not capable of dealing
with first-order logics with equality --- although some theories with equality can be handled.
(The author is unconvinced that this is a deficiency since modern type theories provide an
arguably more adequate framework for dealing with equality than first-order logic.)
Due to our omission of unrestricted quantification, Russell's definitions of connectives
(see \S\ref{meta-disjunction}) no longer apply, and we find it convenient to add the meta-conjunction to
our meta-logic (though it could have been eliminated in principle by using the exponential law).}

One need not panic about the fact that the meta-logic is intuitionistic and not classical
(just like with $\lambda$-calculus, which is closely related to a fragment of intuitionistic logic).
Firstly, when arguing about meta-logic, we have to do it in some meta-meta-logic (along with some
meta-meta-theory over it), and this one will be classical (and informal) in our approach.
Secondly, the semantics of our meta-logic according to the standard (Tarski's) setup of model theory
turns out to be classical.
On the other hand, it is not unexpected that our meta-logic is intuitionistic: this simply reflects
the fact that derivation of formulas is a constructive activity; thus the primitive notion is that of
the {\it problem} (or task) of deriving a given formula, and the more familiar notion of derivability
of the formula is secondary: it merely asserts the existence of a solution of the said problem.
One can discuss semantics of the metalogic without degenerating into classical logic (we will do it in 
\S\ref{meta-semantics}), but one does not have to do it before
discussing semantics of an object logic, be it intuitionistic, classical or modal.
Thus by defining a logic in terms of meta-logic we will not complicate our task of understanding
intuitionistic logic; on the contrary, this will simplify matters, as it is always easier to understand
something that has a formal definition.

\subsubsection{Meta-rules of inference} \label{meta-rules}
There are just two meta-rules of inference for each meta-connective: the introduction rule and
the elimination rule.%
\footnote{These meta-rules themselves belong to meta-meta-logic, which we acknowledge by drawing bold
horizontal lines.}
In the meta-rules below, $\F$ and $\G$ are meta-formulas, $x$ is an individual variable, $T$ is a term,
$\gamma$ is an $n$-ary predicate variable, and $\Sigma$ is an $n$-formula, for some $n=0,1,2,\dots$.%
\footnote{It should be noted that this sentence sneaks in meta-meta-quantification over $\F$, $\G$,
$x$, $T$, $\gamma$ and $\Sigma$.
One does not need to worry about it too much, because all substantial issues appear to be already addressed
on the previous two levels.
In particular, the meta-rules for $\mq{\cdot}$ take care of the entire business of capture-avoiding
substitutions.
This provides an opportunity to focus on more ``substantial'' issues at the object level (in a traditional
sense that is ignorant of the interpretation of substitution by pullback).}
The meta-rules are stated in the natural deduction style, which means that they operate on deductions
from assumptions.
In particular, the $\imp$-introduction meta-rule below is understood thus: starting from a deduction
of $\G$ from a list of assumptions that includes $\F$, we pass to a deduction of $\F\imp\G$ from
the same list of assumptions but with $\F$ removed.
(If some other entries of the list are identical with $\F$, they are kept.)
See \cite{Gi}, \cite{Pr1} for further details on natural deduction (but it should be clear how it works
from the examples below).

\begin{enumerate}[label=]
\item ($\alpha$-conversion) \quad
$\Dfrac{\begin{matrix}\vdots\\ \F\end{matrix}}{\G}$, if $\F$ is $\alpha$-equivalent to $\G$
\bigskip

\item ($\&$-introduction) \quad
$\Dfrac{\begin{matrix}\vdots\\ \F\end{matrix}\qquad
\begin{matrix}\vdots\\ \G\end{matrix}}{\F\mand\G}$
\bigskip

\item ($\&$-elimination) \quad
$\Dfrac{\begin{matrix}\vdots\\ \F\mand\G\end{matrix}}{\F}\qquad\text{and}\qquad
\Dfrac{\begin{matrix}\vdots\\ \F\mand\G\end{matrix}}{\G}$
\bigskip

\item ($\imp$-elimination) \quad
$\Dfrac{\begin{matrix}\vdots\\ \F\end{matrix}\qquad
\begin{matrix}\vdots\\ \F\imp\G\end{matrix}}{\G}$
\bigskip

\item ($\imp$-introduction) \quad
$\Dfrac{\begin{matrix}[\F]\\ \vdots\\ \G\end{matrix}}{\F\imp\G}$
\bigskip

\item (meta-generalization)

$\Dfrac{\begin{matrix}\vdots\\ \F\end{matrix}}{\mq{x}\F}$,
provided that $x$ does not occur freely in any of the assumptions;
\bigskip

$\Dfrac{\begin{matrix}\vdots\\ \F\end{matrix}}{\mq{\gamma}\F}$,
provided that $\gamma$ does not occur freely in any of the assumptions.
\bigskip

\item (meta-specialization) \quad
$\Dfrac{\begin{matrix}\vdots\\ \mq{x}\F\end{matrix}}{\F[x/T]}$,
provided that $T$ is free for $x$ in $\F$;
\bigskip

$\Dfrac{\begin{matrix}\vdots\\ \mq{\gamma}\F\end{matrix}}{\F[\gamma/\Sigma]}$,
provided that $\Sigma$ is free for $\gamma$ in $\F$.
\end{enumerate}
\bigskip

Let us note that $\F[x/T]$ has the same meaning as $\F|_{x:=T}$ in the first-order meta-specialization
rule, and if $n=0$, then also $\F[\gamma/\Sigma]$ has the the same meaning as $\F|_{\gamma:=\Sigma}$ in
the second-order meta-specialization rule; but not so for $n>0$, since in that case the $\lambda$-expression
$\F|_{\gamma:=\Sigma}$ is not itself a meta-formula, and only $\beta$-reduces to the meta-formula
$\F[\gamma/\Sigma]$.

One could in fact drop the ``free for'' hypotheses and understand the meta-specialization rules
as enabling one to pass from a given deduction of $\mq{x}\F$ or $\mq{\gamma}\F$ to the deduction
of any meta-formula in the $\alpha$-equivalence class to which $(x\mapsto\F)(T)$ or
$(\gamma\mapsto\F)(\Sigma)$ $\beta$-reduces.
This seems reasonable, since we anyway have the $\alpha$-conversion meta-rule.
Nevertheless, the ``free for'' conditions are mentioned above in order to facilitate comparison with
traditional treatments of first-order logics.

The ``free for'' conditions are defined in \S\ref{free substitution}.
Let us write them out explicitly in our situation to see that they are similar to the usual ones used in
first-order logic (cf.\ e.g.\ Kleene \cite{Kl}*{\S18, \S34}, Church \cite{Ch}*{\S35}): only the clause (3)
below is unusual, but it is analogous to (0).

Namely, that $T$ is {\it free for $x$ in $\F$} means that
\begin{enumerate}[start=0]
\item no free occurrence of $x$ in $\F$ is in the scope of a quantifier or a meta-quantifier
over any individual variable $y$ that occurs in $T$.
\end{enumerate}
And that the $n$-formula $\Sigma=x_1,\dots,x_n\mapsto\Theta$ %
\footnote{Let us note that we cannot just write $\Sigma(x_1,\dots,x_n)$ instead of $\Theta$, because
that does not exclude the unwanted situation that some $x_i$ occurs freely in $\Sigma$.} 
is {\it free for $\gamma$ in $\F$} means that
\begin{enumerate}[resume]
\item every free occurrence of $\gamma$ in $\F$ is in an atomic formula of the form
$\gamma(T_1,\dots,T_n)$ such that each $T_i$ is free for $x_i$ in $\Theta$;
\item no free occurrence of $\gamma$ in $\F$ is in the scope of a quantifier or a meta-quantifier over
any individual variable $y$ that occurs freely in $\Sigma$;
\item no free occurrence of $\gamma$ in $\F$ is in the scope of a meta-quantifier over
any predicate variable $\delta$ that occurs in $\Sigma$.
\end{enumerate}

From the viewpoint of $\lambda$-calculus, conditions (2) and (3) are checked upon
the substitution of $\Sigma$ for $\gamma$; whereas (1) is checked upon the evaluation of
the resulting expression, which involves an auxiliary substitution.

The point of the $\alpha$-conversion rule is that object-level inference rules and laws
such as $\fm{\mq{\alpha,\tr t}\ \alpha(\tr t)\to\exists \tr x\,\alpha(\tr x)}$ are now meta-formulas
(and not just some schemata involving meta-variables) and so do not automatically yield their copies
with different names of bound variables.
Variables bound by meta-quantifiers can be easily renamed using the specialization and generalization
meta-rules; but no such procedure exists for variables bound by usual quantifiers.

It is not hard to see that it suffices to assume the $\alpha$-conversion rule only for atomic
meta-formulas (i.e.\ the $\Rt$-images of formulas).
Indeed, an arbitrary meta-formula is built out of atomic ones using meta-connectives and meta-quantifiers;
using the meta-rules of inference one can get rid of the meta-connectives and the meta-quantifiers, at
the cost of multiplying the number of deductions, adding new assumptions, and remembering that some
variables do not occur freely in the assumptions.

\subsubsection{Examples of deductions} \label{deductions}

The only deductions that are provided from start are the {\it trivial} deductions, in which a
meta-formula is deduced from itself.
A meta-formula $\F$ is called {\it deducible} if using the meta-rules one can obtain (from
the trivial deductions) a deduction of $\F$ from no assumptions.
We will also abbreviate the phrase ``to obtain (from the trivial deductions, using the meta-rules)
a deduction of $\F$ from the assumption $\G$'' by ``to deduce $\F$ from $\G$''.

\begin{example} \label{trivial} Starting from the trivial deduction of $\F$ from itself, we can apply the
$\imp$-introduction meta-rule to obtain a deduction of $\F\imp\F$ from no assumptions.
Thus the meta-formula $\F\imp\F$ is deducible.
\end{example}

\begin{example} \label{weakening}
Starting from the trivial deduction of $\G$ from itself, we may view it as a deduction
of $\G$ from the assumptions $\F$ and $\G$.
Then the $\imp$-introduction yields a deduction of $\F\imp\G$ from the assumption $\G$.
Applying the $\imp$-elimination once again, we get that $\G\imp(\F\imp\G)$ is deducible.
Also, in the case where $\G$ is of the form $\H\imp\H$, we can combine the deduction of $\F\imp\G$ from
the assumption $\G$ with Example \ref{trivial} to get that $\F\imp(\H\imp\H)$ is deducible.
\end{example}

\begin{example} \label{mand-intro}
Let us show that $(\H\imp\F)\mand(\H\imp\G)\Imp(\H\imp\F\mand\G)$ is deducible.
For this, it suffices to deduce $\H\imp\F\mand\G$ from $(\H\imp\F)\mand(\H\imp\G)$.
For this, it in turn suffices to deduce $\F\mand\G$ from the hypotheses $\H$ and
$(\H\imp\F)\mand(\H\imp\G)$.
This is done as follows: we first deduce $\H\imp\F$ and (separately) $\H\imp\G$ from
$(\H\imp\F)\mand(\H\imp\G)$; and use the $\imp$-elimination meta-rule to deduce $\F$ from
$\H$ and $\H\imp\F$, and $\G$ from $\H$ and $\H\imp\G$.
This gives us deductions of $\F$ and $\G$ from the hypotheses $\H$ and
$(\H\imp\F)\mand(\H\imp\G)$; thus we can use the $\&$-introduction meta-rule to obtain the
desired deduction.
\end{example}

\begin{example} \label{1-2-meta} Similarly to Example \ref{mand-intro}, 
 $\mq{x}\big(\H\imp\F\big)\Imp\big(\H\imp\mq{x}\F\big)$ is deducible, where
$x$ is an individual or predicate variable that does not occur freely in $\H$.
\end{example}

\begin{example}\label{exp-example}
Let us deduce $(\F\mand\G\imp\H)\Iff\big(\F\imp(\G\imp\H)\big)$,
the exponential meta-law.
By the $\&$-elimination, it suffices to deduce both implications: ``$\imp$'' and ``$\when$''.
For this it suffices to deduce $\F\imp(\G\imp\H)$ from $\F\mand\G\imp\H$ and vice versa.

In one direction, it suffices, by the $\imp$-introduction, to deduce $\G\imp\H$ from the hypotheses
$\F$ and $\F\mand\G\imp\H$.
For this it suffices, by the $\imp$-introduction, to deduce $\H$ from $\G$, $\G$ and
$\F\mand\G\imp\H$.
Now from the hypotheses $\F$ and $\G$ we get, by the $\&$-introduction, $\F\mand\G$.
From the latter and $\F\mand\G\imp\H$ we get $\H$ by the $\imp$-elimination.

Conversely, it suffices, by the $\imp$-introduction, to deduce $\H$ from the assumptions $\F\mand\G$ and
$\F\imp(\G\imp\H)$.
Now $\F\mand\G$ yields, by the $\&$-elimination, $\F$ and $\G$.
From $\F$ and $\F\imp(\G\imp\H)$ we get $\G\imp\H$ by the $\imp$-elimination.
Similarly, from $\G$ and $\G\imp\H$ we get $\H$, as desired.
\end{example}

\begin{example}\label{mq-commutes}
It is also easy to check that $\mq{x}\F\mand\mq{x}\G\Iff\mq{x}(\F\mand\G)$
is deducible and $\mq{x}(\F\imp\G)\Imp(\mq{x}\F\imp\mq{x}\G)$ is deducible, where
$x$ is either an individual variable or a predicate variable.
\end{example}

\begin{example} \label{composition metaformula}
Given a deduction of $\H$ from $\G$ and a deduction of $\G$ from $\F$, by combining them we obtain
a deduction of $\H$ from $\F$.
Using this observation, it is easy to show that $(\F\imp\G)\mand(\G\imp\H)\imp(\F\imp\H)$
is deducible.
\end{example}

\begin{remark}\label{meta deduction theorem}
If one can deduce $\G$ from $\F$, then by the $\imp$-introduction,
$\F\imp\G$ is deducible, and hence by the $\imp$-elimination, the deducibility of $\F$ implies
that of $\G$.

However, the converse is not true: if the deducibility of $\F$ implies that of $\G$,
it does not follow that $\G$ can be deduced from $\F$.
Indeed, it can be shown that the meta-absurdity $\fm{\mq{\gamma}\gamma}$ and the Peirce meta-principle
$\fm{\mq{\alpha,\beta}\Big(\big((\alpha\imp\beta)\imp\alpha\big)\imp\alpha\Big)}$ are both not deducible;
but still the former cannot be deduced from the latter.
\end{remark}

\subsubsection{Hilbert-style formulation}\label{hilbert-style meta-logic}

For purposes of models (see \S\ref{models}) it is convenient to record a ``Hilbert-style'' 
formulation of the meta-logic, consisting mostly of {\it meta-laws}, i.e.\ inference meta-rules with
no hypotheses.
To simplify notation, we omit the horizontal bar in meta-laws and the vertical dots illustrating 
the deduction.
Like before, $\F$, $\G$ and $\H$ stand for arbitrary meta-formulas, $x$ is an individual variable, 
$T$ is a term, $\gamma$ is an $n$-ary predicate variable, and $\Phi$ is an $n$-formula, for some 
$n=0,1,2,\dots$.

\begin{enumerate}
\item \label{alpha-eq} $\F\imp\G$, if $\F$ is $\alpha$-equivalent to $\G$;
\smallskip

\item\label{meta-law:identity} $\F\imp\F$;
\medskip

\item\label{meta-rule:composition} $(\F\imp\G)\mand(\G\imp\H)\Imp(\F\imp\H)$;
\medskip

\item\label{meta-rule:exponential} 
$\big((\F\mand\G)\imp\H\big)\Iff\big(\F\imp(\G\imp\H)\big)$;
\medskip

\item\label{meta-law:conjunction} $\F\mand\G\imp\F$ \ and \ $\F\mand\G\imp\G$;
\smallskip

\item\label{meta-rule:conjunction} $(\H\imp\F)\mand(\H\imp\G)\Imp(\H\imp\F\mand\G)$;
\medskip

\item\label{meta-law:forall1} $\mq{x}\F\imp\F[x/T]$,
provided that $T$ is free for $x$ in $\F$;
\medskip

\item\label{meta-law:forall2} $\mq{\gamma}\F\imp\F[\gamma/\Phi]$,
provided that $\Phi$ is free for $\gamma$ in $\F$;
\medskip

\item\label{meta-law:forall3} $\mq{x}(\H\imp\F)\Imp(\H\imp\mq{x}\F)$,
provided that $x$ does not occur freely in $\H$;
\medskip

\item\label{meta-law:forall4} $\mq{\gamma}(\H\imp\F)\Imp(\H\imp\mq{\gamma}\F)$,
provided that $\gamma$ does not occur freely in $\H$;
\bigskip

\item\label{meta-rule:forall1} $\Dfrac{\F}{\mq{x}\F}$,
provided that $x$ does not occur freely in the assumptions;
\bigskip

\item\label{meta-rule:forall2} $\Dfrac{\F}{\mq{\gamma}\F}$,
provided that $\gamma$ does not occur freely in the assumptions;
\bigskip

\item\label{meta-rule:modus ponens} $\Dfrac{\F,\,\F\imp\G}{\G}$.
\end{enumerate}
\bigskip

The three meta-rules with hypotheses are same as before, whereas the meta-laws have already been shown
to be deducible above, except for (\ref{alpha-eq}), (\ref{meta-law:conjunction}), 
(\ref{meta-law:forall1}) and 
(\ref{meta-law:forall2}), which follow immediately from the corresponding meta-rules using 
the $\imp$-introduction.

Conversely, we immediately get the $\alpha$-conversion, $\mand$-elimination and specialization meta-rules
from the corresponding meta-laws, using the $\imp$-elimination meta-rule.
Also, from the meta-laws (\ref{meta-law:identity}) and (\ref{meta-rule:exponential}) 
we get $\F\imp(\G\imp\F\mand\G)$, which yields the $\mand$-introduction meta-rule by a double application 
of $\imp$-elimination.

It remains to recover the $\imp$-introduction meta-rule.
(This is similar to the usual proof of the deduction theorem.)
We induct on the length of the given deduction of $\G$ from $\F$. 
If this deduction is trivial, i.e.\ $\G=\F$, then the desired conclusion $\G\imp\F$ is yielded by
the meta-law (\ref{meta-law:identity}).
Now given a deduction of $\G$ from $\F$, there are two possibilities for the last meta-rule used in 
this deduction.

Case 1. The last meta-rule is the $\imp$-elimination with hypotheses of the form $\H$ and $\H\imp\F$.
We may assume that $\G\imp\H$ and $\G\imp(\H\imp\F)$ are deducible.
Then by a double application of the meta-law (\ref{meta-rule:exponential}), $\H\imp(\G\imp\F)$ is 
also deducible.
Also, the meta-laws (\ref{meta-rule:composition}) and (\ref{meta-rule:exponential}) yield that 
$(\G\imp\H)\imp\Big(\big(\H\imp(\G\imp\F)\big)\imp(\G\imp\F)\Big)$ is deducible.
Applying $\imp$-elimination twice, we conclude that $\G\imp\F$ is also deducible.

Case 2. The last meta-rule is the generalization with hypothesis of the form $\F'$, where 
$\F=\mq{x}\F'$ and $x$ is either an individual or a predicate variable that does not occur freely in $\G$.
We may assume that $\G\imp\F'$ is deducible.
Then by the generalization meta-rule, $\mq{x}(\G\imp\F')$ is deducible.
Hence by (\ref{meta-law:forall3}) or (\ref{meta-law:forall4}), $\G\imp\F$ is also deducible.

\subsubsection{Enderton-style formulation} \label{enderton}

For purposes of models (see \S\ref{models}) we also need an ``Enderton-style'' formulation of the meta-logic,
whose only meta-rule is $\imp$-elimination.%
\footnote{Hilbert-style deductive systems for first-order logics usually include some form of 
the generalization rule.
In particular, so do the original systems of Hilbert--Bernays \cite{HB} and Hilbert--Ackermann \cite{HA}.
Enderton's deductive system \cite{End} includes only one inference rule, the {\it modus ponens}.
The generalization rule has the status of a meta-theorem in his system, just like the Deduction Theorem.}
The notation is the same as in the Hilbert-style formulation, except that we now additionally need 
a tuple $\vec z$ of individual variables and a tuple $\vec\beta$ of predicate variables.

\begin{enumerate}[label=(\arabic*$'$)]
\item \label{alpha-eq+} $\mq{\vec z}\mq{\vec\beta}\big(\F\imp\G\big)$, if $\F$ is $\alpha$-equivalent to $\G$;
\smallskip

\item\label{meta-law:identity+} $\mq{\vec z}\mq{\vec\beta}\big(\F\imp\F\big)$;
\medskip

\item\label{meta-rule:composition+} 
$\mq{\vec z}\mq{\vec\beta}\big((\F\imp\G)\mand(\G\imp\H)\Imp(\F\imp\H)\big)$;
\medskip

\item\label{meta-rule:exponential+} 
$\mq{\vec z}\mq{\vec\beta}\Big(\big((\F\mand\G)\imp\H\big)\Iff\big(\F\imp(\G\imp\H)\big)\Big)$;
\medskip

\item\label{meta-law:conjunction+} 
$\mq{\vec z}\mq{\vec\beta}\big(\F\mand\G\imp\F\big)$ \ and \ 
$\mq{\vec z}\mq{\vec\beta}\big(\F\mand\G\imp\G\big)$;
\smallskip

\item\label{meta-rule:conjunction+} 
$\mq{\vec z}\mq{\vec\beta}\big((\H\imp\F)\mand(\H\imp\G)\Imp(\H\imp\F\mand\G)\big)$;
\medskip

\item\label{meta-law:forall1+} $\mq{\vec z}\mq{\vec\beta}\big(\mq{x}\F\imp\F[x/T]\big)$,
provided that $T$ is free for $x$ in $\F$;
\medskip

\item\label{meta-law:forall2+} $\mq{\vec z}\mq{\vec\beta}\big(\mq{\gamma}\F\imp\F[\gamma/\Phi]\big)$,
provided that $\Phi$ is free for $\gamma$ in $\F$;
\medskip

\item\label{meta-law:forall3+} $\mq{\vec z}\mq{\vec\beta}\big(\mq{x}(\F\imp\G)\Imp(\mq{x}\F\imp\mq{x}\G)\big)$;
\medskip

\item\label{meta-law:forall4+} 
$\mq{\vec z}\mq{\vec\beta}\big(\mq{\gamma}(\F\imp\G)\Imp(\mq{\gamma}\F\imp\mq{\gamma}\G)\big)$;
\medskip

\item\label{meta-law:forall5} $\mq{\vec z}\mq{\vec\beta}\big(\F\imp\mq{x}\F\big)$,
provided that $x$ does not occur freely in $\F$;
\medskip

\item\label{meta-law:forall6} $\mq{\vec z}\mq{\vec\beta}\big(\F\imp\mq{\gamma}\F\big)$,
provided that $\gamma$ does not occur freely in $\F$;
\medskip

\item $\Dfrac{\F,\,\F\imp\G}{\G}$.
\end{enumerate}
\bigskip

All the meta-laws of the Enderton-style system are easily deducible in the Hilbert-style system.
Also, all the meta-laws of the Hilbert-style system are easily deducible in the Enderton-style system
(in fact, all of them except for (\ref{meta-law:forall3}) and (\ref{meta-law:forall4}) are already present in 
the Enderton-style system, by considering empty $\vec z$ and $\vec\beta$).

It remains to recover the generalization meta-rules in the Enderton-style system.
Let us first observe that if $\F$ and $\G$ are meta-formulas and $v$ is either individual or predicate 
variable, $\mq{v}\G$ is deducible from $\mq{v}\F$ and $\mq{v}(\F\imp\G)$ in the Enderton-style system.
Indeed, $\mq{v}(\F\imp\G)\imp(\mq{v}\F\imp\mq{v}\G)$ is an instance of \ref{meta-law:forall3+} or
\ref{meta-law:forall4+}, and the assertion follows by a double application of the $\imp$-elimination.
Also, if $\F$ is a meta-formula and $v$ is either individual or predicate variable that does not occur
freely in $\F$, then $\mq{v}\F$ is easily deducible from $\F$ in the Enderton-style system.

Now suppose that we are given a deduction from assumptions in the system combining the Enderton-style system 
and the two generalization meta-rules.
Let us consider any application of these meta-rules, i.e.\ the meta-generalization of a meta-formula $\F$ 
with respect to an individual or predicate variable, call it $v$, that is not free in any of the assumptions.
The only way that $v$ can first arise as a free variable in the deduction of $\F$ from the assumptions 
is by occurring as a free variable in some instance $\G$ of some meta-law used in this deduction.
Arguing by induction, we may assume that generalization meta-rules are not used within the deduction of
$\F$ from the assumptions (regarded as a subtree of the deduction tree under consideration).
But then we amend this deduction as follows: (i) $\G$ is replaced by another instance $\mq{v}\G$ of 
the same meta-law; (ii) every application of the $\imp$-elimination, inferring $\G'$ from $\F'$ and 
$\F'\imp\G'$, is replaced as follows: if $v$ occurs freely neither in $\F'$ nor in $\G'$, nothing is changed; 
if $v$ occurs freely in $\F'$, the $\imp$-implication is replaced by the deduction (just mentioned above) of 
$\mq{v}\G'$ from $\mq{v}\F'$ and $\mq{v}\F'\imp\G'$; and if $v$ occurs freely in $\G'$ but not in $\F'$, then
we first deduce $\mq{v}\F'$ from $\F'$ and then $\mq{v}\G'$ from $\mq{v}\F'$ and $\mq{v}\F'\imp\G'$;
(iii) the application of meta-generalization to $\F$ is dropped.
In the end of the inductive procedure, all applications of the generalization meta-rules will have been 
eliminated.

\subsubsection{Deducing from assumptions} \label{deductions2}
The following examples deal with specific formulas (regarded as atomic meta-formulas).

\formulas

\begin{example}\label{mp-example} Let us deduce $\beta$ from the assumptions $\alpha$,
$\alpha\to\beta$ and the following {\it modus ponens} (object-level) rule:
\begin{enumerate}[label=(\roman*)]
\item $\mq{\phi,\psi}\ \big(\phi\to\psi\Imp(\phi\imp\psi)\big)$.
\end{enumerate}
Using the second-order meta-specialization, from (i) we infer
$\alpha\to\beta\Imp(\alpha\imp\beta)$.
From this and our assumption $\alpha\to\beta$ we infer, using the $\imp$-elimination,
$\alpha\imp\beta$.
Applying it again, we infer $\beta$.
\end{example}

\begin{example} \label{ibp} Assuming (i) and the following principle (which holds in
intuitionistic logic, cf.\ \S\ref{tautologies}, (\ref{not-not-LEM})):
\begin{enumerate}[start=2,label=(\roman*)]
\item $\mq{\gamma}\ \neg\neg(\gamma\lor\neg\gamma)$,
\end{enumerate}
let us deduce the following implication between principles:
\[\mq{\gamma}\ \neg\neg\gamma\to\gamma\Imp\mq{\gamma}\ \gamma\lor\neg\gamma.\]

Taking into account the $\imp$-introduction, it suffices to deduce
$\mq{\gamma}\ \gamma\lor\neg\gamma$ from (i), (ii) and $\mq{\gamma}\ \neg\neg\gamma\to\gamma$.
Now from (ii) we get $\neg\neg(\beta\lor\neg\beta)$ by the second-order meta-specialization;
and from $\mq{\gamma}\ \neg\neg\gamma\to\gamma$ we similarly get
$\neg\neg(\beta\lor\neg\beta)\to(\beta\lor\neg\beta)$.
Finally, similarly to Example \ref{mp-example}, these yield $\beta\lor\neg\beta$; and then the second-order
meta-generalization yields $\mq{\beta}\ \beta\lor\neg\beta$.
This is the same as $\mq{\gamma}\ \gamma\lor\neg\gamma$ up to $\gamma$-equivalence.
\end{example}

\begin{remark} \label{ibp2} By contrast, the rule
\[\mq{\gamma}\ (\neg\neg\gamma\to\gamma\imp\gamma\lor\neg\gamma)\]
cannot be deduced from (i) and (ii), because it meta-specializes (by setting $\gamma=\neg\beta$) to
$\neg\neg\neg\beta\to\neg\beta\imp\neg\beta\lor\neg\neg\beta$, where the premise
holds in intuitionistic logic (see \S\ref{tautologies}, (\ref{triple negation})), but the
conclusion does not (see \S\ref{Jankov} below).
(In other words, the said rule is not even admissible for intuitionistic logic.)
\end{remark}

\begin{example}\label{subst-example}
Let us deduce $\exists \tr x\,\big(\gamma(\tr x)\to\gamma(\tr y)\big)$ from (i) and the following
assumptions (which will be among the laws of intuitionistic logic):
\begin{enumerate}[start=3,label=(\roman*)]
\item $\mq{\gamma}\ \gamma\to\gamma$;
\item $\mq{\beta,\tr t}\ \beta(\tr t)\to\exists \tr x\,\beta(\tr x)$.
\end{enumerate}
Using the first-order meta-specialization, from (iii) we infer $\gamma(\tr y)\to\gamma(\tr y)$,
and from (iv) we infer $\mq{\beta}\ \beta(\tr y)\to\exists \tr x\,\beta(\tr x)$.
From the latter we infer, using the second-order meta-specialization,
$\big(\gamma(\tr y)\to\gamma(\tr y)\big)\to\exists \tr x\,\big(\gamma(\tr x)\to\gamma(\tr y)\big)$.
(Here we use that $\tr x\mapsto\gamma(\tr x)\to\gamma(\tr y)$ is free for $\beta$ in
$\beta(\tr y)\to\exists \tr x\,\beta(\tr x)$.
For this, it is essential that $\tr y$ is free in $\beta(\tr y)\to\exists \tr x\,\beta(\tr x)$.
Thus we would not be able to make this step directly from (iv).)
The remainder of the argument is similar to that of Example \ref{mp-example}.
\end{example}

\begin{example} \label{wrong} Let us deduce $\forall \tr x\, \big(\gamma(\tr x)\to\gamma(\tr y)\big)$
from (i), (iii) and the following rule
\begin{enumerate}[start=5,label=(\roman*)]
\item $\mq{\beta,\tr t}\ \big(\beta(\tr t)\imp\forall \tr x\,\beta(\tr x)\big)$.
\end{enumerate}
Using the first-order meta-specialization, from (v) we can still infer
$\mq{\beta}\ \big(\beta(\tr y)\imp\forall \tr x\, \beta(\tr x)\big)$.
From the latter we can still infer, using the second-order meta-specialization, the meta-formula
$\big(\gamma(\tr y)\to\gamma(\tr y)\big)\imp\forall \tr x\, \big(\gamma(\tr x)\to\gamma(\tr y)\big)$.
Since we still have $\gamma(\tr y)\to\gamma(\tr y)$ from (iii), by the $\imp$-elimination
we conclude $\forall \tr x\, \big(\gamma(\tr x)\to\gamma(\tr y)\big)$.
\end{example}

\begin{remark} \label{wrong2}
Since the formula deduced in Example \ref{wrong} is contradictory (either classically or intuitionistically),
we conclude that (v) is an incorrect interpretation of the generalization (object-level) rule.
Moreover, the trouble is clearly in that $\tr x\mapsto\gamma(\tr x)\to\gamma(\tr y)$ is free for $\beta$ in
$\beta(\tr y)\imp\forall \tr x\,\beta(\tr x)$ --- but it shouldn't be, if this were to be
the generalization rule.
There is no such trouble with the true generalization (object-level) rule:
\begin{enumerate}[start=5,label=(\roman*$'$)]
\item $\mq{\beta}\ \big(\mq{\tr t}\ \beta(\tr t)\imp\forall \tr x\,\beta(\tr x)\big)$.
\end{enumerate}
This does not meta-specialize by substituting $\tr x\mapsto\gamma(\tr x)\to\gamma(\tr y)$ for $\beta$,
because it is
not free for $\beta$ in $\mq{\tr t}\ \beta(\tr t)\imp\forall \tr x\,\beta(\tr x)$.
The first-order meta-specialization does not apply either, since the scope of the meta-quantifier is not
the entire meta-formula in (v$'$).
\end{remark}

\begin{example} \label{gen} Assuming (i), let us deduce
$\mq{\tr x}\ \phi(\tr x)\to\psi(\tr x)\imp\big(\mq{\tr x}\ \phi(\tr x)\imp\mq{\tr x}\ \psi(\tr x)\big)$.

By the $\imp$-introduction, it suffices to deduce $\mq{\tr x}\ \phi(\tr x)\imp \mq{\tr x}\ \psi(\tr x)$,
assuming (i) and $\mq{\tr x}\ \phi(\tr x)\to\psi(\tr x)$.
For this it suffices, by the same token, to deduce $\mq{\tr x}\ \psi(\tr x)$, assuming (i),
$\mq{\tr x}\ \phi(\tr x)\to\psi(\tr x)$ and $\mq{\tr x}\ \phi(\tr x)$.
The latter two assumptions yield $\phi(\tr x)\to\psi(\tr x)$ and $\phi(\tr x)$ by the meta-specialization.
From these we get $\psi(\tr x)$ similarly to Example \ref{mp-example}.
Finally, the meta-generalization yields $\mq{\tr x}\ \psi(\tr x)$.

Let us note that the last step is only possible since $\tr x$ is not free in our assumptions.
\end{example}

\begin{example} \label{subst-add} The meta-formula in the conclusion of Example \ref{subst-example} can be
restated as
$\forall \tr y\, \exists \tr x\, \big(\gamma(\tr x)\to\gamma(\tr y)\big)$, if in addition to
the original hypotheses (i), (iii), (iv) we assume the generalization rule (v$'$).

Indeed, since $\tr y$ is not free in any of the assumptions, using meta-generalization we get
$\mq{\tr y}\ \exists \tr x\,\big(\gamma(\tr x)\to\gamma(\tr y)\big)$ from the conclusion of \ref{subst-example}.
Now the generalization rule (v$'$) can be applied, using the second-order meta-specialization
and the $\imp$-elimination, to obtain the desired closed formula
$\forall \tr y\,\exists \tr x\, \big(\gamma(\tr x)\to\gamma(\tr y)\big)$.

Similarly, the conclusion of Example \ref{wrong} follows in the form
$\forall \tr y\,\forall \tr x\, \big(\gamma(\tr x)\to\gamma(\tr y)\big)$ from
the assumptions (i), (iii), (v$'$) and the faulty assumption (v).
\end{example}

\metameta

\subsubsection{Specialization} \label{specialization}

A {\it specialization} of a meta-formula $\F$ is any meta-formula that is a result of $\beta$-reduction
(possibly involving $\alpha$-conversions) of a $\lambda$-expression of the form
$\big(\vec\phi\mapsto (\vec x\mapsto\F)\big)(\vec\Sigma)(\vec T)$, where $\vec x$ is an $n$-tuple of
pairwise distinct individual variables, $\vec\phi$ is an $m$-tuple of pairwise distinct predicate variables
of arities $r_i$, $\vec T$ is an $n$-tuple of terms, and $\vec\Sigma$ is a tuple of $r_i$-formulas.

In other words, a meta-formula $\F'$ is a specialization of a meta-formula $\F$ if and only if
$\F'=\F[\vec x/\vec T,\,\vec\phi/\vec\Sigma]$, where $\vec T$ is free for $\vec x$ in $\F$ and
$\vec\Sigma$ is free for $\vec\phi$ in $\F$.

Clearly, specialization can always be obtained by first using the generalization meta-rule, then
the specialization meta-rule, and finally the $\alpha$-conversion meta-rule.

Thus, if a meta-formula $\F$ is deducible from assumptions that contain no free variables, then any its
specialization is deducible from the same assumptions.

\begin{example}
Let us consider the formula $\F=\fm\phi(\tr x)$ (regarded as an atomic meta-formula),
the term $T=\tr z$ and the $1$-formula $F=\tr y\mapsto\fm{\exists\tr z\ \psi(\tr y)}$.
Then
\begin{multline*}
\big(\fm\phi\mapsto(\tr x\mapsto\F)\big)(F)(T)
=\Big(\fm\phi\mapsto\big(\tr x\mapsto\fm\phi(\tr x)\big)\Big)
\big(\tr y\mapsto\fm{\exists\tr z\ \psi(\tr y)}\big)(\tr z)\\
\overset\beta=\Big(\tr x\mapsto\big(\tr y\mapsto\fm{\exists\tr z\ \psi(\tr y)}\big)(\tr x)\Big)(\tr z)
\overset\beta=\big(\tr y\mapsto\fm{\exists\tr z\ \psi(\tr y)}\big)(\tr z)
\overset\beta=\tr y\mapsto\fm{\exists\tr z\,\psi(\tr y)}|_{\tr y:=\tr z}
\overset\alpha=\fm{\exists\tr x\ \psi(\tr z)}.
\end{multline*}
Of course, the final $\alpha$-conversion can be avoided by the cost of replacing $F$ with
the $\alpha$-equivalent $1$-formula $F'=\tr y\mapsto\fm{\exists\tr x\ \psi(\tr y)}$.
The result depends on the choice of $F'$, but any other choice will yield a formula that differs
only in the name of the bound variable.
\end{example}

\formulas

\begin{example} \label{metaq-example1} The formula $\alpha\land\alpha\to\alpha$ is a specialization
of $\alpha\land\beta\to\alpha$.

Similarly, the meta-formula
$\big(\alpha(\tr x)\to\alpha(\tr x)\big)\Imp
\big(\alpha(\tr x)\to\forall\tr x\,\alpha(\tr x)\big)$ is
a specialization of
$\big(\beta\to\alpha(\tr x)\big)\Imp\big(\beta\to\forall\tr x\,\alpha(\tr x)\big)$.

On the other hand, the meta-formula
$\mq{\tr x}\big(\alpha(\tr x)\to\alpha(\tr x)\big)\Imp
\big(\alpha(\tr x)\to\forall\tr x\,\alpha(\tr x)\big)$ is
not a specialization of
$\mq{\tr x}\big(\beta\to\alpha(\tr x)\big)\Imp\big(\beta\to\forall\tr x\,\alpha(\tr x)\big)$.
\end{example}

\begin{example} \label{metaq-example2} The formula
$\big(\beta(\tr y)\to\beta(\tr y)\big)\To\exists\tr x\,\big(\beta(\tr x)\to\beta(\tr y)\big)$ is
a specialization of
$\alpha(\tr y)\to\exists\tr x\,\alpha(\tr x)$ (see Example \ref{subst-example}).

Similarly, the meta-formula
$\big(\beta(\tr y)\to\beta(\tr y)\big)\Imp\forall\tr x\,\big(\beta(\tr x)\to\beta(\tr y)\big)$ is
a specialization of
$\alpha(\tr y)\imp\forall\tr x\,\alpha(\tr x)$ (see Example \ref{wrong}).

On the other hand, the meta-formula
$\mq{\tr y}\big(\beta(\tr y)\to\beta(\tr y)\big)\Imp\forall\tr x\,\big(\beta(\tr x)\to\beta(\tr y)\big)$ is
not a specialization of
$\mq{\tr y}\ \alpha(\tr y)\imp\forall\tr x\,\alpha(\tr x)$ (see Remark \ref{wrong2}).
\end{example}

\metameta

\subsection{Logics}

The purpose of this section is to introduce and discuss some ``syntactic meta-sugar''.

\subsubsection{Rules}

The {\it first-order meta-closure} $\mc{1}\F$ of the meta-formula $\F$ is $\mq{\vec x}\F$,
where $\vec x$ is the tuple of all individual variables occurring freely in $\F$.
The {\it second-order meta-closure} $\mc{2}\F$ is $\mq{\vec\gamma}\F$,
where $\vec\gamma$ is the tuple of all predicate variables occurring freely in $\F$.
Their combination $\mc{2}\mc{1}\F$ will be called the {\it full meta-closure} of $\F$.
If $\Delta$ is a formula (and only in this case) we abbreviate $\mc{2}\mc{1}\Delta$ by $\prin\Delta$.

Thus every principle can be written in the form $\prin\Delta$, and every meta-formula of the form
$\prin\Delta$ is a principle, where by our convention $\Delta$ must be a mere formula.

A (structural) {\it rule}, written $\Gamma_1,\dots,\Gamma_m/\Delta$, or, in more detail,
\[\frac{\Gamma_1,\dots,\Gamma_m}{\Delta},\]
where $\Gamma_1,\dots,\Gamma_m$ and $\Delta$ are formulas, is an abbreviation for the meta-formula
\[\mc{2}\Big(\mc{1}\Gamma_1\,\mand\dots\mand\,\mc{1}\Gamma_m\Imp \mc{1}\Delta\Big).\]
The formulas $\Gamma_1,\dots,\Gamma_m$ are called the {\it premisses} of the rule, and $\Delta$ its
{\it conclusion}.

What happens when $m=0$, i.e.\ when the list of premisses is empty?
If the empty meta-conjunction is defined as an abbreviation of some specific deducible meta-formula $\Th$
(for example, $\mq\gamma\gamma\imp\mq\gamma\gamma$), then
a rule with no premisses, $/\,\Delta$, denotes the meta-formula $\mc{2}\big(\Th\imp\mc{1}\Delta\big)$.
But this is the same as the meta-formula $\mc{2}\mc{1}\Delta$, since it is easy to deduce
$\big(\Th\imp\F\big)\Iff\F$.
Thus rules with no premisses can be identified with principles.

Let us spell out some informal justification for the definition of a rule.
The point of the second-order meta-quantification (``$\mc{2}\!$'') is clear at once from Example \ref{ibp} and
Remark \ref{ibp2}.
The point of the first-order meta-quantification (``$\mc{1}\!$'') is that we want to use principles and rules
as basic assumptions in our deductions; but having free variables in the assumptions would ruin many correct
deductions, such as the ones in Examples \ref{1-2-meta}, \ref{gen} and \ref{subst-add}.
Finally, the reason why the first-order meta-closure is applied separately to each premiss of the rule
(in contrast to the second meta-closure) is clear from Example \ref{wrong} and Remark \ref{wrong2}.
Let us note that
$\mc{1}\Gamma_1\mand\dots\mand\mc{1}\Gamma_m\Iff\mc{1}(\Gamma_1\mand\dots\mand\Gamma_m)$
is deducible (see Example \ref{mq-commutes}).

A byproduct of the fact that principles and rules are $\lambda$-closed is that they have no nontrivial
specializations.
Let us recall that a specialization of a meta-formula is obtained by first applying
meta-generalization, then meta-specialization, and finally $\alpha$-equivalence.
By a {\it special case} of a rule we mean any rule that can be obtained from it by {\it first} applying
meta-specialization, {\it then} meta-generalization, and finally $\alpha$-equivalence.

This turns out to work out differently for principles and for rules with premisses.
Clearly, a principle $\prin\Delta'$ is a special case of a principle $\prin\Delta$ if and only if
the formula $\Delta'$ is a specialization of the formula $\Delta$.
(That is, if and only if $\Delta'=\Delta[\vec x/\vec T,\,\vec\phi/\vec\Sigma]$, where $\vec T$ is free for
$\vec x$ in $\Delta$ and $\vec\Sigma$ is free for $\vec\phi$ in $\Delta$.)
This is not equivalent to saying that $\mc{1}\Delta'$ is a specialization of $\mc{1}\Delta$.
For instance, $\fm{\mc{1}\alpha(\tr x)}$ is not a specialization of $\fm{\mc{1}\beta}$, but
$\fm{\alpha(\tr x)}$ is a specialization of $\fm\beta$.

But for rules with premisses, we cannot use first-order meta-specialization.
Thus a rule $\Gamma_1',\dots,\Gamma_m'\,/\,\Delta'$ is a special case of a rule
$\Gamma_1,\dots,\Gamma_m\,/\,\Delta$ if and only if the meta-formula
$\mc{1}\Gamma_1'\mand\dots\mand\mc{1}\Gamma_m'\imp\mc{1}\Delta'$ is a specialization of
$\mc{1}\Gamma_1\mand\dots\mand\mc{1}\Gamma_m\imp\mc{1}\Delta$.
(That is, if and only if for some tuple $\vec\Sigma$ of $r_i$-formulas that is free for $\vec\phi$ in
each $\Gamma_j$ and in $\Delta$, each $\Gamma_j'=\Gamma_j(\vec\phi/\vec\Sigma)$ and
$\Delta'=\Delta(\vec\phi/\vec\Sigma)$.)

\formulas

\begin{example} \label{metaq-example1'} (Cf.\ Example \ref{metaq-example1}.)
The principle $\prin\alpha\land\alpha\to\alpha$ is
a special case of the principle $\prin\alpha\land\beta\to\alpha$.

Similarly, the principle
$\prin\big(\alpha(\tr x)\to\alpha(\tr x)\big)\To
\big(\alpha(\tr x)\to\forall\tr x\,\alpha(\tr x)\big)$ is
a special case of
$\prin\big(\beta\to\alpha(\tr x)\big)\To\big(\beta\to\forall\tr x\,\alpha(\tr x)\big)$.

\smallskip
However, the rule
$\dfrac{\alpha(\tr x)\to\alpha(\tr x)}{\alpha(\tr x)\to\forall\tr x\,\alpha(\tr x)}$ is
not a special case of
$\dfrac{\beta\to\alpha(\tr x)}{\beta\to\forall\tr x\,\alpha(\tr x)}$.
\end{example}

\begin{example} \label{metaq-example2'} (Cf.\ \ref{metaq-example2}.) The principle
$\prin\big(\beta(\tr y)\to\beta(\tr y)\big)\To\exists\tr x\,\big(\beta(\tr x)\to\beta(\tr y)\big)$ is
a special case of
$\prin\alpha(\tr y)\to\exists\tr x\,\alpha(\tr x)$.

Similarly, the principle
$\prin\big(\beta(\tr y)\to\beta(\tr y)\big)\To\forall\tr x\,\big(\beta(\tr x)\to\beta(\tr y)\big)$ is
a special case of
$\prin\alpha(\tr y)\to\forall\tr x\,\alpha(\tr x)$.

\smallskip
However, the rule
$\dfrac{\beta(\tr y)\to\beta(\tr y)}{\forall\tr x\,\big(\beta(\tr x)\to\beta(\tr y)\big)}$ is
not a special case of
$\dfrac{\alpha(\tr y)}{\forall\tr x\,\alpha(\tr x)}$.
\end{example}

\metameta

\subsubsection{Logics and theories}

A {\it derivation system}%
\footnote{Also called a ``deductive system'' in the literature.
For our purposes it is convenient to distinguish {\it deductions} in meta-logic from {\it derivations}
within a specific object logic.}
$\Ds$ in a first-order language $\L$ is a meta-formula of the form \[\H_1\mand\dots\mand\H_k,\]
where each $\H_i$ is a purely logical rule (possibly with no premisses) in $\L$.
In other words, $\Ds$ is the meta-conjunction of finitely many purely logical principles and other
purely logical rules (with at least one premise), which are called, respectively, the {\it laws}%
\footnote{These correspond to what is usually called  ``axiom schemata''.
Apart from the conflation of schemata and principles, this standard terminology is objectionable in
the case of intuitionistic logic, as it fits well with the Brouwer--Heyting approach, but not with
Kolmogorov's approach.
On Kolmogorov's approach, the laws of intuitionistic logic are not propositions accepted
without proof, but rather a toolkit of elementary constructions that can be used as black boxes.}
and the {\it inference rules} of $\Ds$.

A {\it logic} is a meta-equivalence class of derivation systems in some first-order language $\L$.
In other words, derivation systems $\Ds$ and $\Ds'$ are said to {\it determine the same logic} if
the meta-formula $\Ds\iff\Ds'$ is deducible.
The purely logical part of $\L$ is called the {\it language of the logic} $L$ and is denoted $\L(L)$.
A logic $L^+$ is called an {\it extension} of a logic $L$ if there exist derivation systems $\Ds$ and
$\Ds'$ such that $\Ds$ determines $L$ and $\Ds\mand\Ds'$ determines $L^+$.

A meta-formula $\F$ is called {\it derivable} in the logic $L$ determined by a derivation system $\Ds$
if the meta-formula $\Ds\imp\G$ is deducible (in the meta-logic) --- or equivalently, if $\F$ is
deducible from the hypothesis $\Ds$ (in the meta-logic).
Clearly, adding a derivable principle or rule to a derivation system $\Ds$
does not affect derivability of principles and rules in the logic determined by $\Ds$.

A principle $\prin\Delta$ is derivable in the logic if and only if either $\prin\Delta$ is a special case
of a law, or there exist derivable principles $\prin\Gamma_1,\dots,\prin\Gamma_m$ such that
the rule $\Gamma_1,\dots,\Gamma_m\,/\,\Delta$ is a special case of an inference rule.
More generally, a rule $\Gamma_1,\dots,\Gamma_m\,/\,\Delta$ is derivable in the logic if and only if
either the principle $\prin\Delta$ is one of the $\prin\Gamma_i$,%
\footnote{But not just a special case of one of the $\prin\Gamma_i$.
Thus the formulas $\Gamma_i$ are {\it not} treated just as if the principles $\prin\Gamma_i$ were
laws.}
or it is a special case of a law, or there exist finitely many formulas
$\Delta_1,\dots,\Delta_k$ such that each rule $\Gamma_1,\dots,\Gamma_m\,/\,\Delta_i$ is
derivable, $i=1,\dots,k$, and $\Delta_1,\dots,\Delta_k\,/\,\Delta$ is a special case of
an inference rule.
(The ``if'' assertion is trivial, and the converse is not hard.)

It is also not hard to see that every special case of a derivable rule is derivable, and
that every special case of a derivable principle is a derivable principle.
In fact, every specialization of a derivable meta-formula is a derivable meta-formula; in particular,
every specialization of a derivable formula is a derivable formula.

A first-order {\it theory} $\Th$ over a logic $L$ is a finite meta-conjunction of principles
in a first-order language $\L$ whose purely logical part is the language of $L$.%
\footnote{This corresponds to the traditional notion of a theory given by finitely many axioms and
finitely many axiom schemata.}
The language $\L$ is called {\it the language of the theory} $\Th$ and is denoted $\L(\Th)$, 
and these principles are called the {\it axioms}%
\footnote{
In the context of intuitionistic logic (and its non-classical extensions) these may be called {\it postulates}
following Euclid (and his ancient commentators Geminus and Proclus, who have actually made the distinction
between axioms and postulate quite clear, see \cite{M3}).
It should be noted, however, that the modern meaning of the word {\it postulate} is not very
different from that of {\it axiom}, as opposed to the ancient Greek word which meant {\it request},
{\it demand} (see \cite{M3}).}
of $\Th$.
In contrast to the laws of a derivation system, the axioms of a theory may involve function symbols
and predicate constants.
Because of this, a typical theory (such as the theory of groups or the theory of an order relation) would
involve no predicate variables, and hence also no meta-quantifiers (as long as all its axioms are
closed formulas).
Meta-quantifiers would be needed, however, to express principles that are not purely logical (such as
the principle of mathematical induction).
An {\it empty theory} over $L$ has no axioms (but its language can be different from the language of $L$).

A formula $\G$ of $\L$ is a {\it theorem}%
\footnote{Or a {\it soluble problem} in intuitionistic logic (and its non-classical extensions).}
of a theory $\Th$ over a logic $L$ if the meta-formula $\Th\imp\G$ is derivable in $L$,
in other words, if $\G$ is deducible from $\Ds\mand\Th$ (in the meta-logic), where $\Ds$ is
a derivation system for $L$.
This is equivalent to saying that either the principle $\prin\G$ is a special case of
a law of $\Ds$ or of an axiom of $\Th$, or there exist theorems $\F_1,\dots,\F_m$
such that the rule $\F_1,\dots,\F_m\,/\,\G$ is a special case of an inference rule of $\Ds$.
(Let us note that an axiom that involves no meta-quantifiers has no special cases other than itself.)
A theory $\Th$ is called {\it consistent} if its theorems do not include all formulas of $\L$.
It is easy to see that $\Th$ is inconsistent (i.e.\ not consistent) if and only if $\fm{\Th\imp\mq\gamma\gamma}$
is derivable in $L$.

Of course, assertions called ``theorems'', ``lemmas'' or ``propositions'' in the present text are not supposed
to be theorems of any first-order theory.
They are meta-meta-theorems.
But some assertions of the present text (such as those starting with ``the following formula is
derivable in such and such logic'') do signify theorems of the empty theories over classical,
intuitionistic or QS4 logics.
Also, if logics are regarded as meta-theories over the meta-logic, then for instance \S\ref{deductions}
is devoted to some meta-theorems of the empty meta-theory over the meta-logic, and \S\ref{deductions2}
to those of some non-empty meta-theories.

\subsubsection{Syntactic consequence}\label{syntactic consequence}

If $L$ is the logic determined by a derivation system $\Ds$, we denote by $\turnstile\F$,
or in more detail $\turnstile _L\F$, the judgement that the meta-formula $\F$ is derivable in the logic.
The judgement $\turnstile (\F_1\mand\dots\mand\F_m\imp\G)$ is also abbreviated by
\[\F_1,\dots,\F_m\turnstile\G.\]
When this judgement is true, we also say that $\G$ is a {\it (syntactic) consequence} of
the $\F_i$.%
\footnote{Thus consequence belongs to meta-meta-logic, and not meta-logic.
It is clear from Example \ref{fake-admissible} below (see also Remark \ref{meta deduction theorem})
that it would not be a good idea to
explain the judgement $\turnstile\F$ by identifying it with the meta-formula
$\Ds\imp\F$ itself.}
Here the $\F_i$ and $\G$ can be arbitrary meta-formulas; but a few cases are of particular interest.
Namely, $\F_1,\dots,\F_m\turnstile\G$ might be of the following form:
\begin{enumerate}
\item[(1)] $\turnstile\mc{2}\mc{1}\Delta$, where $\Delta$ is a formula.
This is saying that $\prin\Delta$ is a derivable principle.

\item[(1$'$)] $\mc{2}\mc{1}\Gamma_1,\dots,\mc{2}\mc{1}\Gamma_m\turnstile\mc{2}\mc{1}\Delta$, where
the $\Gamma_i$ and $\Delta$ are formulas.
This is saying that the principle $\prin\Delta$ is a consequence of the principles $\prin\Gamma_i$.

\item[(1$''$)] $\mc{2}\Big(\mc{1}\Theta_{11}\mand\dots\mand\mc{1}\Theta_{1n_1}\imp\mc{1}\Gamma_1\Big),\dots,
\mc{2}\Big(\mc{1}\Theta_{m1}\mand\dots\mand\mc{1}\Theta_{mn_m}\imp\mc{1}\Gamma_m\Big)\turnstile
\mc{2}\Big(\mc{1}\Xi_1\mand\dots\mand\mc{1}\Xi_n\imp\mc{1}\Delta\Big)$, where
the $\Theta_{ij}$, $\Xi_j$, $\Gamma_i$ and $\Delta$ are formulas.
This is saying that the rule $\Xi_1,\dots,\Xi_n\,/\,\Delta$ is a consequence of the rules
$\Theta_{i1},\dots,\Theta_{in_i}/\,\Gamma_i$.

\item[(2)] $\turnstile \mc{1}\Delta$, where $\Delta$ is a formula.
This is equivalent to $\turnstile\mc{2}\mc{1}\Delta$
(since there are no free predicate variables in $\Ds$).

\item[(2$'$)] $\mc{1}\Gamma_1,\dots,\mc{1}\Gamma_m\turnstile \mc{1}\Delta$, where the $\Gamma_i$
and $\Delta$ are formulas.
Similarly, this is the case if and only if the rule $\Gamma_1,\dots,\Gamma_m\,/\,\Delta$ is derivable.%
\footnote{Thus $\mc{1}\Gamma_1,\dots,\mc{1}\Gamma_m\turnstile \mc{1}\Delta$ is the traditional
``varied variables'' version of syntactic consequence for formulas, as in the textbooks by Schoenfield
and Mendelson.}

\item[(3)] $\turnstile\Delta$, where $\Delta$ is a formula.
This is saying that the formula $\Delta$ is derivable, and is the case if and only if
$\turnstile \mc{1}\Delta$ (since there are no free individual variables in $\Ds$).

\item[(3$'$)] $\Gamma_1,\dots,\Gamma_m\turnstile\Delta$, where the $\Gamma_i$ and $\Delta$ are formulas.
This is saying that the formula $\Delta$ is a syntactic consequence of the formulas $\Gamma_i$.%
\footnote{Thus $\Gamma_1,\dots,\Gamma_m\turnstile\Delta$ is the more usual ``fixed variables'' version of
syntactic consequence for formulas, as in the textbooks by Church, Enderton, Troelstra and van Dalen.}
As long as the $\Gamma_i$ contain no free individual variables, we similarly have that
$\Gamma_1,\dots,\Gamma_m\turnstile\Delta$ if and only if
$\mc{1}\Gamma_1,\dots,\mc{1}\Gamma_m\turnstile \mc{1}\Delta$.
In general, for logics in which the deduction theorem and the exponential law hold,
$\Gamma_1,\dots,\Gamma_m\turnstile\Delta$ if and only if
$\turnstile\Gamma_1\land\dots\land\Gamma_m\to\Delta$.
\end{enumerate}

What is this {\it deduction theorem}?
In logics where implication is present, a usual rule of inference is the {\it modus ponens} rule:
$$\fm{\frac{\phi,\,\phi\to\psi}{\psi}}.$$
Using this rule, we get that $\turnstile\Gamma\to\Delta$ implies $\Gamma\turnstile\Delta$
(see Examples \ref{mp-example} and \ref{exp-example});
and more generally, that $\Gamma_1,\dots,\Gamma_n\turnstile\Gamma\to\Delta$ implies
$\Gamma_1,\dots,\Gamma_n,\Gamma\turnstile\Delta$ (where $\Gamma$ and $\Delta$ are formulas).
In some logics the converse also holds, in which case it is called the {\it deduction theorem}.
In particular, the deduction theorem holds for classical and intuitionistic logics, where it is proved 
similarly to the argument in \S\ref{hilbert-style meta-logic} recovering the $\imp$-introduction meta-rule 
from the Hilbert-style formulation of the meta-logic.
By combining the deduction theorem with the exponential law
$\fm{\prin\big((\alpha\land\beta)\to\gamma\big)\Tofrom\big(\alpha\to(\beta\to\gamma)\big)}$,
we easily get the last assertion in (3$'$) above.

The deduction theorem is not as banal as it might seem.
Indeed, the {\it modus ponens} rule yields, in fact, slightly more: the rule $\Gamma\,/\,\Delta$ is
a consequence of the principle $\prin\Gamma\to\Delta$.
(If the principle is purely logical, this is equivalent to saying that if we add
$\prin\Gamma\to\Delta$ as a new law to the derivation system, then $\Gamma\,/\,\Delta$
becomes derivable.)
However, the deduction theorem, whenever it holds, does {\it not} imply that the principle
$\prin\Gamma\to\Delta$ is a consequence of the rule $\Gamma\,/\,\Delta$.
(Which in the purely logical case is equivalent to saying that $\prin\Gamma\to\Delta$ becomes derivable
once $\Gamma\,/\,\Delta$ is set as a new inference rule.)
This is simply not true in general --- already in classical logic (see
Examples \ref{meta-implication1} and \ref{deduction example}).
It follows, incidentally, that for closed formulas $\Gamma$ and $\Delta$, the meta-formula
$(\Gamma\imp\Delta)\Imp(\Gamma\to\Delta)$ need not be derivable, already in classical logic
(see also Examples \ref{meta-implication0} and \ref{meta-implication3}).

Consequence between principles and rules will often be referred to using the words such as ``implies'' and
``follows'', when this cannot lead to confusion; thus ``$\prin\Gamma$ implies $\prin\Delta$'' signifies a
meta-formula different form that signified by ``the principle that $\Gamma$ implies $\Delta$''
(see Example \ref{ibp} and Remark \ref{ibp2}).
A set of rules $\Ru_1,\dots,\Ru_k$ is said to be {\it interderivable} (or simply {\it equivalent},
when this cannot lead to confusion) with a set of rules $\Su_1,\dots,\Su_l$ if each $\Ru_i$ is
a consequence of $\Su_1,\dots,\Su_l$, and each $\Su_i$ is a consequence of $\Ru_1,\dots,\Ru_k$.
The notion of consequence between principles and rules is going to be indispensable for our purposes;
its explicit form seems to occur only rarely in the literature (see \cite{Um}, \cite{Ci}, \cite{Ry}*{1.4.8},
\cite{Je}), but indirectly it
is used more often (in particular, through the equivalent notions of the order relation on extensions
of a given logic and of the basis for a set of rules, such as the set of all admissible rules).

Of the judgements listed above, (3$'$) implies (2$'$), which in turn implies (1$'$)
(see Example \ref{mq-commutes}), and these implications are not reversible (see Example \ref{wrong}
and Remark \ref{wrong2} and, respectively, Example \ref{ibp} and Remark \ref{ibp2}).
To these judgements we may add also:
\begin{enumerate}
\item[(0$'$)] $\turnstile\Gamma_1,\dots,\turnstile\Gamma_m$ imply $\turnstile\Delta$;
\item[(0$''$)] $\turnstile\Theta_{11},\dots,\Theta_{1n_1}/\,\Gamma_1,\dots,\turnstile
\Theta_{m1},\dots,\Theta_{mn_m}/\,\Gamma_m$ imply $\turnstile\H_1,\dots,\H_n\,/\,\Delta$.
\end{enumerate}
Then (1$''$) implies (0$''$) by the $\imp$-elimination meta-rule, and in particular (1$'$) implies (0$'$).
But not conversely, since non-derivable principles need not be consequences of each other
(even in classical logic, see Example \ref{principle implication} below; also, in zero-order intuitionistic
logic, $\prin\neg\neg\gamma\lor\neg\gamma$ does not imply $\prin\gamma\lor\neg\gamma$, which in turn
does not imply $\prin\ab$, see \S\ref{Jankov} below).
In fact, it is easy to see that principles $\prin\Gamma_1,\dots,\prin\Gamma_m$ imply
$\prin\Delta$ if and only if the judgements $\G\turnstile\Gamma_1,\dots,\G\turnstile\Gamma_m$
imply $\G\turnstile\Delta$ for each $\lambda$-closed meta-formula $\G$ (by considering
$\G=\prin\Gamma_1\mand\dots\mand\prin\Gamma_m$).
The same argument also establishes the same assertion with a different choice of $\G$'s, namely:
$\prin\Gamma_1,\dots,\prin\Gamma_m$ imply $\prin\Delta$ over a logic $L$ if and only if
$\turnstile _{L'}\Gamma_1,\dots,\turnstile _{L'}\Gamma_m$ imply $\turnstile _{L'}\Delta$ for any extension
$L'$ of $L$ by additional laws (or by additional laws and inference rules).
Similarly, (1$''$) over a logic $L$ is equivalent to (0$''$) over all extensions $L'$ of $L$ by additional
laws and inference rules.

\begin{remark}
Kleene's sequent ``$\Gamma_1,\dots,\Gamma_m\turnstile ^{x_1,\dots,x_k}\Delta$'', where
$\vec x=(x_1,\dots,x_k)$ is the tuple of individual variables that ``may be varied'' in the deduction,
can be identified with our judgement
$\mq{\vec x}\Gamma_1,\dots,\mq{\vec x}\Gamma_m\turnstile\mq{\vec x}\Delta$.
In fact, according to Kleene \cite{Kl}*{\S24}, his sequent notation ``is not fully explicit, as it does
not show for which of the assumption formulas a given superscript variable may be varied. [...]
When there is occasion to be more explicit, the facts may be stated verbally, e.g.\ as in Lemma 8a below.''
Our notation includes a formalization of such verbal comments, by simply meta-quantifying the assumption
formulas over distinct collections of individual variables.
(The meta-quantifiers in the conclusion are redundant, as long as they repeat those that are common to
all of the assumptions.)
\end{remark}

\subsubsection{Admissible rules}\label{logic}

A rule $\Ru$ is called {\it admissible} for the logic determined by a derivation system $\Ds$ if
every derivable principle of the logic determined by the extended derivation system $\Ds\mand\Ru$ is also 
derivable in the original logic.
(That is, if $\turnstile\Ru\imp\Phi$ implies $\turnstile\Phi$ for each formula $\Phi$.)
Thus adding an admissible rule to a derivation system does not change the set of derivable principles, but 
unless this rule is derivable, it changes the logic (i.e.\ the syntactic consequence relation).
This can happen even with classical logic, as we will see in \S\ref{models}.
Clearly, a purely logical rule $\Gamma_1,\dots,\Gamma_m\,/\,\Delta$ is admissible
if and only if for every its special case
$\Gamma_1[\vec\phi/\vec\Phi],\dots,\Gamma_m[\vec\phi/\vec\Phi]\,/\,\Delta[\vec\phi/\vec\Phi]$
(where $\vec\Phi$ is free for $\vec\phi$ in each $\Gamma_i$ and in $\Delta$)
such that each $\Gamma_i[\vec\phi/\vec\Phi]$ is a derivable formula, $\Delta[\vec\phi/\vec\Phi]$ is
also a derivable formula.

\begin{remark}
For this reason, admissible rules of the form $\Gamma_1,\dots,\Gamma_m\,/\,\Delta$ are sometimes stated in
the literature in the form
$$\Dfrac{\turnstile\Gamma_1,\dots,\turnstile\Gamma_m}{\turnstile\Delta}$$
or verbally, ``if $\turnstile\Gamma_1,\dots,\turnstile\Gamma_m$,
then $\turnstile\Delta$'', where the predicate variables $\phi_1,\dots,\phi_k$ occurring in
the $\Gamma_i$ and in $\Delta$ are tacitly understood to range over all
$r_j$-formulas $\Phi_1,\dots,\Phi_k$ that are free for $\phi_1,\dots,\phi_k$ in each $\Gamma_i$
and in $\Delta$ (that is, with an implicit meta-meta-quantifier $\forall\vec\phi$ meant at the beginning
of the sentence).
The verbal version is particularly confusing, since its two alternative readings: (i) without assuming
any implicit meta-meta-quantification as in the judgement (0$'$) above, (ii) assuming another implicit
meta-meta-quntification: ``if $\forall\vec\phi\turnstile\Gamma_1,\dots,\forall\vec\phi\turnstile\Gamma_m$,
then $\forall\vec\phi\turnstile\Delta$'' --- are entirely different judgements, which are equivalent to
each other, but not to the admissibility of the rule.
In older literature, however, one often finds {\it inference rules} stated in this ambiguous language
(which is even more confusing, and certainly indicates the author's acceptance of the ``formulaic''
conception of logic).
\end{remark}

Clearly, every derivable rule is admissible.
Moreover, a purely logical rule $\Gamma_1,\dots,\Gamma_m\,/\,\Delta$ is derivable if and only if for every its 
special case $\Gamma_1[\vec\phi/\vec\Phi],\dots,\Gamma_m[\vec\phi/\vec\Phi]\,/\,\Delta[\vec\phi/\vec\Phi]$
(where $\vec\Phi$ is free for $\vec\phi$ in each $\Gamma_i$ and in $\Delta$) and for every formulas
$\Psi_1,\dots,\Psi_l$ such that $\mc{1}\Psi_1,\dots,\mc{1}\Psi_l\turnstile\Gamma_i[\vec\phi/\vec\Phi]$
for each $i$, we also have $\mc{1}\Psi_1,\dots,\mc{1}\Psi_l\turnstile\Delta[\vec\phi/\vec\Phi]$.
(This follows by considering the case $l=k$, $\Psi_i=\Gamma_i[\vec\phi/\vec\Phi]$.)

Let us note that if a new law or inference rule is added to a derivation system, previously
derivable rules do not cease to be derivable, but previously admissible rules may cease to be admissible.

\begin{example} \label{fake-admissible}
One might be suspecting that a rule $\Gamma\,/\,\Delta$ is admissible if and only if
the meta-formula
$\H:=\mc{2}\Big(\big(\Ds\imp \mc{1}\Gamma\big)\Imp\big(\Ds\imp \mc{1}\Delta\big)\Big)$
is derivable.
This is not so, since the derivability of $\H$ is in fact equivalent to
the derivability, and not admissibility of the rule $\Gamma\,/\,\Delta$.

Indeed, the derivability of $\H$ amounts to the derivability of $\Ds\imp \mc{1}\Delta$ from
the assumption $\Ds\imp \mc{1}\Gamma$.
This in turn amounts to the derivability of $\mc{1}\Delta$ from the assumptions
$\Ds$ and $\Ds\imp \mc{1}\Gamma$.
But derivability from these assumptions is equivalent to derivability from the assumptions
$\Ds$ and $\mc{1}\Gamma$.
In turn, the derivability of $\mc{1}\Delta$ from the assumptions $\Ds$ and $\mc{1}\Gamma$ amounts
to the derivability of $\mc{1}\Gamma\imp \mc{1}\Delta$ from the assumption $\Ds$.
This is the same as the derivability of the rule $\Gamma\,/\,\Delta$.
\end{example}

Clearly, derivability and admissibility of a rule are preserved when the rule is replaced by one that is
interderivable with it.
In particular, an admissible rule that is interderivable with a principle must be derivable.
Je\v r\'abek \cite{Je} proved that in a number of logics, including zero-order intuitionistic
logic and S4, every rule is interderivable with a principle plus a finite set of admissible rules.

Lorentzen's book, where the term ``admissible rule'' was first used, also dealt with the following notion:
a rule $\Ru$ is {\it admissible relative to} a set of rules $\Su_1,\dots,\Su_k$ if $\Ru$ becomes
an admissible rule upon adjoining $\Su_1,\dots,\Su_k$ as new inference rules to the logic
\cite{Lor}*{pp.\ 24, 40} (see also \cite{ScH}).
When $\Ru$ is a principle, this coincides with the notion of consequence; and when $k=0$, the two notions
(i.e., consequence and relative admissibility) reduce, respectively, to derivability of a rule and
admissibility of a rule.

We will call a rule $\Ru$ {\it stably admissible} over a given logic if it is admissible over any extension
of this logic, obtained by adjoining any number of new laws and inference rules.
Thus derivable rules are stably admissible; and in contrast to admissible rules, if a new law
or inference rule is added to the derivation system, previously stably admissible rules do not cease to be
stably admissible.
(However, a stably admissible rule can cease to be stably admissible upon a mere extension of the language of
the logic, see Examples \ref{st-adm} and \ref{equality-st-adm}.)

Clearly, a purely logical rule $\Gamma_1,\dots,\Gamma_m\,/\,\Delta$ is stably admissible
if and only if for every its special case
$\Gamma_1[\vec\phi/\vec\Phi],\dots,\Gamma_m[\vec\phi/\vec\Phi]\,/\,\Delta[\vec\phi/\vec\Phi]$
(where $\vec\Phi$ is free for $\vec\phi$ in each $\Gamma_i$ and in $\Delta$)
the principle $\prin\Delta[\vec\phi/\vec\Phi]$ is a consequence of the principles
$\prin\Gamma_1[\vec\phi/\vec\Phi],\dots,\prin\Gamma_m[\vec\phi/\vec\Phi]$.
This uses the characterization of consequence between principles given in \S\ref{syntactic consequence}.
Since that characterization involved some freedom, we obtain as a byproduct that a purely logical rule $\Ru$
is stably admissible over a logic if and only if it is admissible over any extension of this logic,
obtained by adjoining any number of new laws (but not inference rules).
By using the same freedom, we also get that $\Ru$ is stably admissible if and only if $\G\turnstile\Ru\imp\Phi$ 
implies $\G\turnstile\Phi$ for every formula $\Phi$ and every $\lambda$-closed meta-formula $\G$.
This is equivalent to saying that $\turnstile (\Ru\imp\Phi)\imp\Phi$ for every formula $\Phi$
(by considering $\G=\Ru\imp\Phi$).

On the other hand, it is easy to see that a rule $\Ru$ is derivable if and only if 
$\turnstile\Ru\imp\F$ implies $\turnstile\F$ for every rule $\F$ (by considering $\F=\Ru$).
Similarly, a rule $\Ru$ is derivable if and only if $\G\turnstile\Ru\imp\F$ implies $\G\turnstile\F$ for every 
meta-formula $\F$ and every $\lambda$-closed meta-formula $\G$.
Thus, similarly to the above, a rule $\Ru$ is derivable if and only if $\turnstile (\Ru\imp\F)\imp\F$ for every 
rule $\F$ (or for every $\lambda$-closed meta-formula $\F$).

Admissibility, stable admissibility and derivability are all distinct notions, already in
classical logic (see Examples \ref{omitting}, \ref{st-adm} and \ref{equality-st-adm}).

\subsection{Intuitionistic logic (syntax)}\label{logics}

\subsubsection{Derivation system}\label{intlogic}

\formulas

A formalization of intuitionistic logic QH in terms of a derivation system was found by 
Heyting \cite{He4} (1928 and 1931), extending partial formalizations by Kolmogorov \cite{Kol0} (1925;
the $\to,\neg,\exists,\forall$ fragment%
\footnote{Kolmogorov's point of departure was Hilbert's derivation system for classical logic, which did
not use other connectives.
Presumably Kolmogorov was also aware of Russell's definition of $\land$, $\lor$ and $\exists$ in terms of
$\to$ and the second-order $\forall$, which is valid intuitionistically, see \cite{Ac} and \cite{Pr1}*{p.\ 67}.}
with the omission of the explosion principle $\prin\ab\to\alpha$ and with inadvertent omission of
the universal specialization principle $\prin\forall \tr x\,\alpha(\tr x)\to\alpha(\tr y)$; see \cite{Pl})
and Glivenko (1928 and 1929; the zero-order fragment).%
\footnote{It should be mentioned that an algebraic formulation of the zero-order fragment
is found already in Skolem's 1919 paper (see \cite{vP}).}
Here is an equivalent formulation due to Spector (as in \cite{Tr3}; see also \cite{Tr2}), where the usual
side conditions are eliminated as explained below.
\smallskip

\begin{enumerate}[label=(\Roman*)]
\item\label{rule:modus ponens} (modus ponens rule) $\dfrac{\alpha,\,\alpha\to\beta}{\beta}$
\smallskip

\item\label{axiom:identity} $\prin\alpha\to\alpha$
\medskip

\item\label{rule:composition} $\dfrac{\alpha\to\beta,\,\beta\to\gamma}{\alpha\to\gamma}$
\medskip

\item\label{rule:exponential} (exponential law)
$\dfrac{(\alpha\land\beta)\to\gamma}{\alpha\to(\beta\to\gamma)}
\quad\text{and}\quad\dfrac{\alpha\to(\beta\to\gamma)}{(\alpha\land\beta)\to\gamma}$
\medskip

\item\label{axiom:explosion} (explosion principle) $\prin\ab\to\alpha$
\medskip

\item\label{axiom:conjunction} $\prin\alpha\land\beta\to\alpha$ \ and \ $\prin\alpha\land\beta\to\beta$
\smallskip

\item\label{axiom:disjunction} $\prin\alpha\to\alpha\lor\beta$ \ and \ $\prin\beta\to\alpha\lor\beta$
\medskip

\item\label{rule:conjunction} $\dfrac{\gamma\to\alpha,\,\gamma\to\beta}{\gamma\to\alpha\land\beta}$
\medskip

\item\label{rule:disjunction} $\dfrac{\alpha\to\gamma,\,\beta\to\gamma}{\alpha\lor\beta\to\gamma}$
\medskip

\item\label{axiom:forall} $\prin\forall \tr x\,\alpha(\tr x)\to\alpha(\tr t)$
\medskip

\item\label{axiom:exists} $\prin\alpha(\tr t)\to\exists \tr x\,\alpha(\tr x)$
\medskip

\item\label{rule:forall-open} $\dfrac{\gamma\to\alpha(\tr x)}{\gamma\to\forall \tr x\,\alpha(\tr x)}$
\medskip

\item\label{rule:exists-open} $\dfrac{\alpha(\tr x)\to\gamma}{\exists \tr x\,\alpha(\tr x)\to\gamma}$
\end{enumerate}
\bigskip

\noindent
Classical predicate logic QC is obtained by adding just one more law:
\smallskip

\begin{enumerate}[resume,label=(\Roman*)]
\item\label{axiom:excluded middle} (law of excluded middle) $\prin p\lor\neg p$
\end{enumerate}
\bigskip

\metameta

Let us note that the traditional schematic versions of \ref{axiom:forall} and \ref{axiom:exists} with side
conditions follow immediately from \ref{axiom:forall} and \ref{axiom:exists} respectively by applying
the two specialization meta-rules (see \S\ref{meta-rules}), in either order.
We have a notational choice of specializing the unary predicate variable $\fm\alpha$ either to a $1$-formula
of the form $x\mapsto\Theta$, in which case we get

\[\turnstile\forall \tr x\,\Theta\to \Theta[\tr x/T]
\text{, provided that $T$ is free for $\tr x$ in $\Theta$;}\]
\[\turnstile\Theta[\tr x/T]\to\exists \tr x\,\Theta
\text{, provided that $T$ is free for $\tr x$ in $\Theta$,}\]
\smallskip

\noindent
or to an arbitrary $1$-formula $\Sigma$, in which case we get

\[\turnstile\forall \tr x\,\Sigma(\tr x)\to\Sigma(T)
\text{, provided that $T$ is free for $\tr x$ in $\Sigma(\tr x)$
and $\tr x$ does not occur freely in $\Sigma$;}\]
\[\turnstile\Sigma(T)\to\exists \tr x\,\Sigma(\tr x)
\text{, provided that $T$ is free for $\tr x$ in $\Sigma(\tr x)$
and $\tr x$ does not occur freely in $\Sigma$,}\]
\smallskip

\noindent
where the $\lambda$-expressions $\Sigma(\tr x)$, $\Sigma(T)$ stand for their normal forms (which are
formulas).

Let us note that the second side condition, ``provided that $\tr x$ does not occur freely in $\Sigma$'',
cannot be dropped.
For instance, if $\Sigma$ is the arithmetical $1$-formula $\tr y\mapsto\tr x=\tr y$, and $T$ is the arithmetical
term $\tr x+1$, then $\forall \tr x\,\Sigma(\tr x)\to\Sigma(T)$ is not a valid formula.

Similarly, the traditional schematic versions of \ref{rule:forall-open} and \ref{rule:exists-open} with side
conditions follow immediately from \ref{rule:forall-open} and \ref{rule:exists-open} respectively by applying
the second-order meta-specialization and (for one of the traditions) the first-order meta-generalization
(see \S\ref{meta-rules}).
Namely, by specializing $\fm\gamma$ to a formula $\Gamma$, and $\fm\alpha$ to a $1$-formula of the form
$\tr x\mapsto\Theta$, we get

\[\mq{\tr x}\Gamma\to\Theta\ \turnstile\ \Gamma\to\forall \tr x\,\Theta\text{, provided that $\tr x$ does not
occur freely in $\Gamma$};\]
\[\mq{\tr x}\Theta\to\Gamma\ \turnstile\ \exists \tr x\,\Theta\to\Gamma\text{, provided that $\tr x$ does not
occur freely in $\Gamma$},\]
\smallskip

\noindent
which corresponds to the tradition of ``varied variables'' (in Kleene's terminology), found in the textbooks
by Schoenfield and Mendelson.
Now by combining these with the meta-generalization with respect to $\tr x$, we get

\[\Gamma\to\Theta\ \turnstile\ \Gamma\to\forall \tr x\,\Theta\text{, provided that $\tr x$ does not
occur freely in $\Gamma$ and in the assumptions};\]
\[\Theta\to\Gamma\ \turnstile\ \exists \tr x\,\Theta\to\Gamma\text{, provided that $\tr x$ does not
occur freely in $\Gamma$ and in the assumptions},\]
\smallskip

\noindent
which corresponds to the tradition of ``fixed variables'', found in the textbooks by Church, Troelstra and
van Dalen.
Like before, we can replace $\Theta$ with $\Sigma(\tr x)$, at the cost of the additional side condition
``provided that $\tr x$ does not occur freely in $\Sigma$''.

In particular, the special case of \ref{rule:forall-open} with $\fm\gamma=\triv$, when rewritten as
\smallskip

\begin{enumerate}[start=12,label=(\Roman*$_\triv$)]
\item\label{rule:generalization} (generalization rule)
$\fm{\dfrac{\alpha(\tr x)}{\forall \tr x\,\alpha(\tr x)}}$
\end{enumerate}
\smallskip

\noindent (see Example \ref{top-example} below), corresponds to the following traditional schematic forms
of the generalization rule:

\[\mq{\tr x}\Sigma(\tr x)\ \turnstile\ \forall \tr x\,\Sigma(\tr x)\text{, provided that $\tr x$ does not
occur freely in $\Sigma$};\]
\[\Sigma(\tr x)\ \turnstile\ \forall \tr x\,\Sigma(\tr x)\text{, provided that $\tr x$ does not
occur freely in $\Sigma$ and in the assumptions},\]
\smallskip

\noindent
where the $\lambda$-expression $\Sigma(\tr x)$ again stands for its normal form (which is a formula).
The side conditions cannot be dropped here as well (see Example \ref{wrong} and Remark \ref{wrong2}).

\begin{example}\label{top-example} Let us show that
formulas $\Phi$ and $\triv\to\Phi$ are interderivable in intuitionistic logic.
Indeed, by \ref{axiom:conjunction} $\turnstile\Phi\land\triv\to\Phi$, hence by the exponential law,
$\turnstile\Phi\to(\triv\to\Phi)$, and thus by {\it modus ponens}, $\Phi\turnstile\triv\to\Phi$.
Conversely, since $\triv=\ab\to\ab$, by \ref{axiom:identity} $\turnstile\triv$, and therefore by
{\it modus ponens}, $\triv\to\Phi\turnstile\Phi$.
\end{example}

\subsubsection{Modifications and dualizations}\label{intlogic-modified}

Since intuitionistic logic satisfies the deduction theorem (see \cite{Tr3}, where it is proved in
Spector's derivation system) and the exponential law, $\Phi,\Psi\turnstile\Xi$
is equivalent to $\turnstile\Phi\land\Psi\to\Xi$ (see \S\ref{syntactic consequence}).

Thus, should we wish to have fewer inference rules, we may well replace the inference rules in
\ref{rule:composition}, \ref{rule:exponential}, \ref{rule:conjunction} and \ref{rule:disjunction}
by the corresponding laws (in which the horizontal bar is replaced by $\to$, and
the comma by $\land$).

As for the inference rules \ref{rule:forall-open} and \ref{rule:exists-open}, for many purposes it is
convenient to replace them by the generalization rule \ref{rule:generalization} along with the following ones:
\smallskip

\formulas
\begin{enumerate}[start=12,label=(\Roman*$'$)]
\item\label{rule:forall} $\dfrac{\forall \tr x\,(\gamma\to\alpha(\tr x))}
{\gamma\to\forall \tr x\,\alpha(\tr x)}$
\smallskip

\item\label{rule:exists} $\dfrac{\forall \tr x\,(\alpha(\tr x)\to\gamma)}
{\exists \tr x\,\alpha(\tr x)\to\gamma}$
\end{enumerate}
\smallskip

\noindent
Indeed, using the generalization rule \ref{rule:generalization} and its converse, which follows from
\ref{axiom:forall}, we can clearly exchange \ref{rule:forall} for \ref{rule:forall-open} and
\ref{rule:exists} for \ref{rule:forall-open}.

Now by the deduction theorem, the inference rules \ref{rule:forall} and \ref{rule:exists} can be replaced
by the corresponding laws (in which the horizontal bar is replaced by $\to$).

The generalization rule is, however, strictly weaker than the corresponding principle,
$\prin\fm{\alpha(\tr x)\to\forall \tr x\,\alpha(\tr x)}$ (see \S\ref{models}).
Informally, this can be explained as follows.
If, for instance, $\neg\,\tr x=\tr x+1$ has been proved, we should be able to conclude
$\forall \tr x\ \neg\,\tr x=\tr x+1$; yet from $\exists \tr x\ \,\tr x=1$ we cannot conclude
$\exists \tr x\,\forall \tr x\ \,\tr x=1$ (which is equivalent to $\exists \tr y\,\forall \tr x\ \,\tr x=1$);
also, from $\neg\forall \tr x\ \,\tr x=1$ we cannot conclude $\neg\,\tr x=1$
(which, if proved, would imply $\forall \tr x\ \neg\,\tr x=1$).

Let us note dual forms of the rules \ref{rule:forall} and \ref{rule:exists}:
\smallskip

\begin{enumerate}[start=12,label=(\Roman*$'^*$)]
\item\label{rule:forall-dual}
$\dfrac{\exists \tr x\,(\gamma\to\alpha(\tr x))}{\gamma\to\exists \tr x\,\alpha(\tr x)}$
\smallskip

\item\label{rule:exists-dual}
$\dfrac{\exists \tr x\,(\alpha(\tr x)\to\gamma)}{\forall \tr x\,\alpha(\tr x)\to\gamma}$
\end{enumerate}
\smallskip

\noindent
These, or rather their corresponding principles, follow from the laws \ref{axiom:exists} and
\ref{axiom:forall} respectively, using the rule \ref{rule:exists-open} and the generalization rule.
Similarly, there are also dual forms of the rules \ref{rule:conjunction} and \ref{rule:disjunction}.

The laws \ref{axiom:forall} and \ref{axiom:exists} also imply the converses of the rules
\ref{rule:forall-open} and \ref{rule:exists-open}, and hence (using the generalization rule) also
the converses of \ref{rule:forall} and \ref{rule:exists}.
However, the converses of the rules \ref{rule:forall-dual} and \ref{rule:exists-dual} are not
derivable in intuitionistic logic (see \ref{Harrop1}), and therefore not even admissible in it, since
they are easily seen to be interderivable with the principles
$\prin\exists \tr t\,\big(\exists \tr x\,\alpha(\tr x)\to\alpha(\tr t)\big)$ and
$\prin\exists \tr t\,\big(\alpha(\tr t)\to\forall \tr x\,\alpha(\tr x)\big)$, respectively.
In classical logic, these principles are easily seen to be derivable.

Dual to the generalization rule is the following instantiation rule
\smallskip

\begin{enumerate}[start=12,label=(\Roman*$_\triv^*$)]
\item\label{rule:generalization-dual} $\dfrac{\exists \tr x\,\gamma}{\gamma}$,
\end{enumerate}
\smallskip

\noindent
whose corresponding principle follows from \ref{rule:exists-open} and \ref{axiom:identity}.
The more general rule $\exists \tr x\,\alpha(\tr x)\,/\,\alpha(\tr x)$, which is the ``contrapositive'' of
the generalization rule,
is not even admissible in classical logic (by considering any of the two classically derivable formulas
just mentioned above, $\exists \tr x\,\big(\exists \tr y\,\beta(\tr y)\to\beta(\tr x)\big)$ or
$\exists \tr x\,\big(\beta(\tr x)\to\forall \tr y\,\beta(\tr y)\big)$).

The rule $\exists \tr x\,\alpha(\tr x)\,/\,\alpha(\tr x)$ is also not admissible in intuitionistic logic,
by considering the intuitionistically derivable formula
$\exists \tr x\,\big(\beta(\tr x)\to\beta(\tr y)\big)$ (see Example \ref{subst-example}).
But the second variable is essential here: as long as the language contains no function symbols, and
$\Phi$ is a closed $1$-formula, $\turnstile\exists \tr x\,\Phi(\tr x)$ implies
$\turnstile\Phi(\tr x)$ by the ``existence property'' of intuitionistic logic (Gentzen; see \cite{Kl2}).
The related ``disjunction property'' will be discussed in \S\ref{Tarski completeness} below.

\metameta

\subsubsection{Additional meta-connectives} \label{meta-disjunction}

In studying models, it seems worthwhile to be armed with what was not needed previously:
the {\it meta-absurdity} $\metabot$, the {\it meta-disjunction} $\F\parallel\G$ and
the {\it existential first-order meta-quantifier} $\eq{x}\F$.
These are assumed to satisfy the following meta-rules, where $\F$, $\G$ and $\H$ are meta-formulas,
$x$ is an individual variable and $T$ is a term.

\begin{enumerate}[label=]
\item ($\metabot$-elimination) \quad
$\Dfrac{\begin{matrix}\vdots\\ \metabot\end{matrix}}{\F}$
\bigskip

\item ($\parallel$-introduction) \quad
$\Dfrac{\begin{matrix}\vdots\\ \F\end{matrix}}{\F\parallel\G}\qquad\text{and}\qquad
\Dfrac{\begin{matrix}\vdots\\ \G\end{matrix}}{\F\parallel\G}$
\bigskip

\item ($\parallel$-elimination) \quad
$\Dfrac{\begin{matrix}\\ \vdots\\ \F\parallel\G\end{matrix}\qquad
\begin{matrix}[\F]\\ \vdots\\ \H\end{matrix}\qquad
\begin{matrix}[\G]\\ \vdots\\ \H\end{matrix}}{\H}$
\bigskip

\item ($\eq{\cdot}$-introduction) \quad
$\Dfrac{\begin{matrix}\vdots\\ \F[x/T]\end{matrix}}{\eq{x}\F}$,
provided that $T$ is free for $x$ in $\F$
\bigskip

\item ($\eq{\cdot}$-elimination) \quad
$\Dfrac{\begin{matrix}\\ \vdots\\ \eq{x}\F\end{matrix}\qquad
\begin{matrix}[\F]\\ \vdots\\ \H\end{matrix}}{\H}$,
provided that $x$ does not occur freely in $\H$ and in the assumptions.
\end{enumerate}

On the other hand, Russell's well-known second-order definitions of $\ab$, $\lor$, $\exists$ and $\land$
(see \cite{Ac} and \cite{Gi}*{\S11.3.2}) suggest the following versions of $\metabot$, $\parallel$,
$\eq{\cdot}$ and $\mand$:
\begin{itemize}
\item $\varmetabot:=\fm{\mq{\alpha}\alpha}$;
\item $\F\varmor\G:=\mq{\gamma}\big((\F\imp\gamma)\mand(\G\imp\gamma)\imp\gamma\big)$;%
\footnote{Let us note that $\mand$ can be eliminated here by a mere application of the exponential meta-law.}
\item $\vareq{x}\F:=\mq{\gamma}\big(\mq{x}(\F\imp\gamma)\imp\gamma\big)$;
\item $\F\varmand\G:=\mq{\gamma}\Big(\big(\F\imp(\G\imp\gamma)\big)\imp\gamma\Big)$,
\end{itemize}
where $\gamma$ is a nullary predicate variable that does not occur freely in the meta-formulas
$\F$, $\G$.

When $\F$ and $\G$ are formulas (regarded as atomic meta-formulas), we can identify $\mand$ with
$\varmand$; $\parallel$ with $\varmor$; and $\eq{\cdot}$ with $\vareq{\cdot}$.
We only include full details for the case of meta-conjunction:

\begin{proposition} \label{redundancy}
(a) $\Phi\varmand\Psi\Iff\Phi\mand\Psi$ is deducible for all formulas $\Phi$, $\Psi$;

(b) each of $\Phi$, $\Psi$ can be deduced from $\Phi\varmand\Psi$, and $\Phi\varmand\Psi$ can be deduced
from $\Phi$ and $\Psi$ without using the $\&$-introduction and the $\&$-elimination meta-rules;

(c) Let $\H$ be a meta-formula where all occurrences of $\&$ are applied to formulas.
Then $\H$ is deducible if and only if the meta-formula obtained from $\H$ by replacing
every occurrence of $\&$ with a $\varmand$ is deducible without using the $\&$-introduction and
the $\&$-elimination meta-rules.
\end{proposition}

\begin{proof}[Proof. (b)]
To deduce $\Phi$ from $\Phi\varmand\Psi$ we use the meta-specialization with $\gamma=\Phi$
along with the fact that $\Phi\imp(\Psi\imp\Phi)$ is deducible (see Example \ref{weakening}).
To deduce $\Psi$ from $\Phi\varmand\Psi$ we use the meta-specialization with $\gamma=\Psi$
along with the fact that $\Phi\imp(\Psi\imp\Psi)$ is deducible (see Example \ref{weakening}).

To deduce $\Phi\varmand\Psi$ from $\Phi$ and $\Psi$, let us first note that $\gamma$ can be deduced
from $\Phi$, $\Psi$ and $\Phi\imp(\Psi\imp\gamma)$ by a double application of the $\imp$-elimination.
Now by the $\imp$-introduction, we get a deduction of $\big(\Phi\imp(\Psi\imp\gamma)\big)\imp\gamma$
from $\Phi$ and $\Psi$, and it remains to apply the meta-generalization.
\end{proof}

\begin{proof}[(a)] The $\&$-elimination meta-rule along with the $\varmand$-introducing deduction
provided by (b) yield a deduction of $\Phi\varmand\Psi$ from $\Phi\mand\Psi$.
Similarly one gets a deduction of $\Phi\mand\Psi$ from $\Phi\varmand\Psi$.
The assertion now follows by the $\imp$-introduction.
\end{proof}

\begin{proof}[(c)] The ``if'' assertion follows from (a), and the ``only if'' assertion from (b).
\end{proof}

Let us also discuss the special case of the $\eq{\cdot}$-rules where $x$ does not occur freely in $\Phi$.

\begin{example} \label{RP-translation}
Let us show that if $\Phi$ is a formula, $\Phi\Iff\mq{\gamma}\big((\Phi\imp\gamma)\imp\gamma\big)$ is
deducible as long as $\gamma$ does not occur freely in $\Phi$.

The $\when$ implication follows by meta-specializing $\gamma$ to $\Phi$.

To get the converse implication, by the meta-generalization, it suffices to deduce
$(\Phi\imp\gamma)\imp\gamma$ from $\Phi$.
But $\Phi\imp\big((\Phi\imp\gamma)\imp\gamma\big)$ is easily deducible by a double application of
the exponential meta-law.
\end{example}

\formulas

\begin{proposition} \label{RP-translation2} In intuitionistic logic,

(a) $\turnstile\phi\Iff\mq{\gamma}(\phi\to\gamma)\to\gamma$;

(b) $\turnstile\alpha\land\beta\Iff\mq{\gamma}\big(\alpha\to (\beta\to\gamma)\big)\to\gamma$;

(c) $\turnstile\alpha\lor\beta\Iff\mq{\gamma}(\alpha\to\gamma)\to\big((\beta\to\gamma)\to\gamma\big)$;

(d) $\turnstile\exists\tr x\,\alpha(\tr x)\Iff\mq{\gamma}\forall\tr x\,\big(\alpha(\tr x)\to\gamma\big)\to\gamma$;

(e) $\turnstile\ab\Iff\mq{\gamma}\gamma$.
\end{proposition}

Parts (b), (c), (d),(e) yield meta-logical ``definitions'', up to meta-equivalence, of $\ab$;
of $\land$ and $\lor$ in terms of $\to$; and of $\exists$ in terms of $\to$ and $\forall$.
Of course, these ``definitions'' also apply in classical logic.
Let us note that part (e) says that $\turnstile\ab\iff\tilde\metabot$.

\begin{proof}[Proof. (a)] This is similar to Example \ref{RP-translation}.
The $\when$ follows by meta-specializing $\gamma$ to $\phi$.

To get the converse, by the meta-generalization, it suffices to derive
$(\phi\to\gamma)\to\gamma$ from $\phi$.
By the modus ponens rule, it suffices to derive $\phi\to\big((\phi\to\gamma)\to\gamma\big)$.
But it is easily derivable by a double application of the exponential law.
\end{proof}

\begin{proof}[(b),(c),(d),(e)]
These follow from (a) by appropriately meta-specializing $\phi$.
Namely, (b) follows by the exponential law; (c) by (\ref{dm1}) and the exponential law;
(d) by (\ref{gq1}); and (e) is straightforward.
\end{proof}

\metameta

\begin{proposition} \label{connectives vs meta-connectives}
In intuitionistic logic,

(a) $\turnstile\Phi\land\Psi\iff\Phi\mand\Psi$ for any formulas $\Phi$, $\Psi$;

(b) $\turnstile\forall x\,\Phi\iff\mq{x}\Phi$ for any formula $\Phi$;

(c) $\turnstile\metabot\imp\ab$;

(d) $\turnstile\Phi\parallel\Psi\imp\Phi\lor\Psi$ for any formulas $\Phi$, $\Psi$;

(e) $\turnstile \eq{x}\Phi\imp\exists x\,\Phi$ for any formula $\Phi$;

(f) $\turnstile\Phi\to\Psi\imp(\Phi\imp\Psi)$ for any formulas $\Phi$, $\Psi$.
\end{proposition}

This clearly transfers to any extension of intuitionistic logic, including classical logic.

The converses to parts (d), (e), (f) do not hold already in zero-order classical logic
(and hence also in zero-order intuitionistic logic), see Examples \ref{meta-implication0} and
\ref{meta-disjunction0}.

\begin{proof}[Proof. (a)] From \ref{rule:conjunction} with $\fm\gamma=\triv$ and the $\mand$-elimination
meta-rule we get $\turnstile\Phi\mand\Psi\imp\Phi\land\Psi$.
Conversely, from \ref{axiom:conjunction} and the {\it modus ponens} rule we have
$\turnstile\Phi\land\Psi\imp\Phi$ and $\turnstile\Phi\land\Psi\imp\Psi$.
Hence by the $\&$-introduction, $\turnstile\Phi\land\Psi\imp\Phi\mand\Psi$.
\end{proof}

\begin{proof}[(b)] Here $\when$ follows from the generalization rule, and $\imp$ from
\ref{axiom:forall} and the {\it modus ponens} rule, along with the generalization meta-rule.
\end{proof}

\begin{proof}[(c)] This follows from the $\metabot$-elimination rule.
\end{proof}

\begin{proof}[(d)] From \ref{axiom:disjunction} and the {\it modus ponens} rule we have
$\turnstile\Phi\imp\Phi\lor\Psi$ and $\turnstile\Psi\imp\Phi\lor\Psi$.
Hence by the $\parallel$-elimination, $\turnstile\Phi\parallel\Psi\imp\Phi\lor\Psi$.
\end{proof}

\begin{proof}[(e)] This is similar to (d).
\end{proof}

\begin{proof}[(f)] This follows from the {\it modus ponens} rule.
\end{proof}

\begin{remark} If $\F$ is a meta-formula with no occurrences of $\metabot$, then
$\varmetabot\imp\F$ is deducible.
Indeed, from $\varmetabot$ we can deduce any formula (regarded as an atomic meta-formula)
by the meta-specialization.
The assertion now follows by induction, using $\&$- $\imp$-, $\parallel$-, $\eq\cdot$-introduction
and the meta-generalization.
\end{remark}

\subsection{Models} \label{models0}

\subsubsection{Meta-models}\label{models}

Let us call a meta-rule {\it absolute} if it involves no conditions on the assumptions of the deduction.
In particular, the $\imp$-elimination meta-rule is absolute; and so is every meta-law.
Thus the Enderton-style system of \S\ref{enderton} consists entirely of absolute meta-rules.

Let us fix a language $\L$.
An absolute meta-rule is said to be {\it satisfied} in a meta-$\L$-structure $\J$ if for every
its instance 
\[\Dfrac{\F_1,\dots,\F_n}{\G}\]
and every individual assignment $\iass$ and every predicate valuation $\pval$, if
$\wnf(|\F_i|^{\iass\pval}_\J)=\Top$ for each $i$, then $\wnf(|\G|^{\iass\pval}_\J)=\Top$.

In particular, a meta-law is satisfied in $\J$ if and only $\wnf(|\G|^{\iass\pval}_\J)=\Top$
for every its instance $\G$ and every individual assignment $\iass$ and every predicate valuation $\pval$.
Let us note also that the $\alpha$-conversion rule is trivially satisfied in every meta-$\L$-structure.

A meta-$\L$-structure $\J$ will be called a {\it model of the meta-logic} if it satisfies 
all meta-rules of the Enderton-style formulation of the meta-logic in \S\ref{enderton}.

It is easy to see that if an absolute meta-rule is derivable in the Enderton-style system, then it is
satisfied in any model of the meta-logic.

This applies, in particular, to the $\mand$-introduction and $\mand$-elimination meta-rules.
Thus we obtain

\begin{proposition} \label{mand-redundant}
Let $\J$ be a model of the meta-logic, $\F$ and $\G$ meta-formulas, $\iass$ a variable assignment and
$\pval$ a predicate valuation.

Then $|\F\mand\G|_\J^{\iass\pval}=\Top$ if and only if
$|\F|_\J^{\iass\pval}=\Top$ and $|\G|_\J^{\iass\pval}=\Top$.
\end{proposition}

The same goes for the specialization meta-rules.
However, the generalization meta-rules are not absolute.

Let us call a meta-$\L$-structure $\J$ a {\it semi-standard model of the meta-logic} if it satisfies 
all absolute meta-rules of the Hilbert-style formulation of the meta-logic in \S\ref{enderton} and also
the following two conditions (corresponding to the generalization meta-rules):
\begin{enumerate}[label=(\Alph*)]
\item 
If $\wnf(|\F|_\J^{\iass'\pval})=\Top$ for every variable assignment $\iass'$ that agrees with
$\iass$ on all individual variables other than $x$, then $\wnf(|\mq{x}\F|_\J^{\iass\pval})=\Top$;
\smallskip

\item if $\wnf(|\F|_\J^{\iass\pval'})=\Top$ for every predicate valuation $\pval'$ that agrees with
$\pval$ on all predicate variables other than $\gamma$, then $\wnf(|\mq{\gamma}\F|_\J^{\iass\pval})=\Top$.
\medskip
\end{enumerate}

It is easy to see, following the straightforward derivation of the Enderton-style system from 
the Hilbert-style system, that every semi-standard model of the meta-logic is a model of the meta-logic.

\begin{proposition} \label{meta-redundant}
Let $\J$ be a semi-standard model of the meta-logic, $\F$ and $\G$ meta-formulas, $\iass$ a variable 
assignment and $\pval$ a predicate valuation.

(a) $\wnf(|\mq{x}\F|_\J^{\iass\pval})=\Top$ if and only if $\wnf(|\F|_\J^{\iass'\pval})=\Top$ for every
variable assignment $\iass'$ that agrees with $\iass$ on all individual variables other than $x$.

(b) $\wnf(|\mq{\gamma}\F|_\J^{\iass\pval})=\Top$ if and only if $\wnf(|\F|_\J^{\iass\pval'})=\Top$
for every predicate valuation $\pval'$ that agrees with $\pval$ on all predicate variables other than $\gamma$.
\end{proposition}

\begin{proof} Let us prove (a).
The ``if'' assertion is precisely the condition corresponding to the generalization meta-rule.
So it remains to prove the ``only if'' assertion.
The condition corresponding to the specialization meta-rule is: for any $\F$, $\iass$ and $\pval$,
if $\wnf(|\mq{x}\F|_\J^{\iass\pval})=\Top$, then $\wnf(|\F[x/T]|_\J^{\iass\pval})=\Top$,
as long as $T$ is free for $x$ in $\F$.

Let $y$ be a variable that does not occur freely in $\F$.
Given an assignment $\iass'$ as above, let $\iass''$ be the assignment with $\iass''(y)=\iass'(x)$
and $\iass''(z)=\iass(z)$ for every variable $z$ other than $y$.
Since $y$ does not occur freely in $\F$, we have
$|\mq{x}\F|_\J^{\iass\pval}=|\mq{x}\F|_\J^{\iass''\pval}$.
By the condition corresponding to the meta-specialization rule,
if $\wnf(|\mq{x}\F|_\J^{\iass''\pval})=\Top$, then $\wnf(|\F[x/y]|_\J^{\iass''\pval})=\Top$.
Since $y$ does not occur freely in $\F$, we also have $|\F[x/y]|_\J^{\iass''\pval}=|\F|_\J^{\iass'\pval}$.
Thus if $\wnf(|\mq{x}\F|_\J^{\iass\pval})=\Top$, then $\wnf(|\F|_\J^{\iass'\pval})=\Top$.

The proof of (b) is similar.
\end{proof}

Let us note that while the $\imp$-elimination meta-rule is absolute, the $\imp$-introduction meta-rule is not.

Let us call a meta-$\L$-structure $\J$ a {\it standard model of the meta-logic} if it satisfies all 
absolute meta-rules of the original natural deduction formulation of the meta-logic in \S\ref{meta-rules},
the conditions (A) and (B) above (corresponding to the generalization meta-rules), and also the following
condition (corresponding to the $\imp$-introduction meta-rule):
\begin{enumerate}[label=(\Alph*), start=3]
\item If either $\wnf(|\F|_\J^{\iass\pval})=\Bot$ or
$\wnf(|\G|_\J^{\iass\pval})=\Top$, then $\wnf(|\F\imp\G|_\J^{\iass\pval})=\Top$.
\end{enumerate}

It is easy to see that any standard model of the meta-logic is a semi-standard model of the meta-logic.

Condition (C) is converse to the condition of satisfying the $\imp$-elimination meta-rule.
Thus we obtain

\begin{proposition} \label{classical meta-implication}
Let $\J$ be a standard model of the meta-logic, $\F$ and $\G$ meta-formulas, $\iass$ a variable assignment and
$\pval$ a predicate valuation.

$\wnf(|\F\imp\G|_\J^{\iass\pval})=\Top$ if and only if
$\wnf(|\F|_\J^{\iass\pval})=\Top$ implies $\wnf(|\G|_\J^{\iass\pval})=\Top$.
\end{proposition}

It is easy to see that a two-valued meta-interpretation $\J=\I_+$ is always a standard model of the meta-logic.

In fact, the converse also holds: any standard model of the meta-logic reduces to a two-valued one.

\begin{theorem} Let $\J$ be a standard model of the meta-logic.
Then there exists a two-valued model $\I_+$ of the meta-logic such that any given meta-formula $\F$ is valid
in $\J$ if and only if it is valid in $\I_+$.
\end{theorem}

\begin{proof}
The given model $\J$ includes an $\L$-structure $\I$, which in turn determines a two-valued model 
$\I_+$ of the meta-logic, with $\ocf_{\I_+}\:\O\to\{\Top,\Bot\}$ defined as the composition 
$\O\xr{\ocf_\J}\Qm\xr{\wnf_\J}\{\Top,\Bot\}$.
For any atomic meta-formula $\F$ and any variable assignment $\iass$ and any predicate valuation $\pval$
we have $\wnf_\J(|\F|_\J^{\iass\pval})=\Top$ if and only if $|\F|_{\I_+}^{\iass\pval}=\Top$.
Now from \ref{mand-redundant}, \ref{meta-redundant} and \ref{classical meta-implication} we get that 
the same holds for an arbitrary meta-formula $\F$.
In particular, a meta-formula $\F$ is valid in $\J$ if and only if it is valid in $\I_+$.
\end{proof}

\subsubsection{Two-valued meta-models}

Let us now record the conditions corresponding to all meta-rules (also for the additional
meta-connectives) in two-valued models of the meta-logic.
\medskip

\begin{itemize}
\item $|\F\mand\G|_{\I_+}^{\iass\pval}=\Top$ if and only if
$|\F|_{\I_+}^{\iass\pval}=\Top$ and $|\G|_{\I_+}^{\iass\pval}=\Top$.
\medskip
\item $|\F\parallel\G|_{\I_+}^{\iass\pval}=\Top$ if and only if
$|\F|_{\I_+}^{\iass\pval}=\Top$ or $|\G|_{\I_+}^{\iass\pval}=\Top$.
\medskip
\item $|\F\imp\G|_{\I_+}^{\iass\pval}=\Top$ if and only if
$|\F|_{\I_+}^{\iass\pval}=\Top$ implies $|\G|_{\I_+}^{\iass\pval}=\Top$.
\medskip
\item $|\mq{x}\F|_{\I_+}^{\iass\pval}=\Top$ if and only if $|\F|_{\I_+}^{\iass'\pval}=\Top$
for every variable assignment $\iass'$ that agrees with $\iass$ on all individual variables other than $x$.
\medskip
\item $|\mq{\gamma}\F|_{\I_+}^{\iass\pval}=\Top$ if and only if $|\F|_{\I_+}^{\iass\pval'}=\Top$
for every predicate valuation $\pval'$ that agrees with $\pval$ on all predicate variables other than $\gamma$.
\medskip
\item $|\eq{x}\F|_{\I_+}^{\iass\pval}=\Top$ if and only if $|\F|_{\I_+}^{\iass'\pval}=\Top$
for some variable assignment $\iass'$ that agrees with $\iass$ on all individual variables other than $x$.
\medskip
\item $|\metabot|_{\I_+}=\Top$ if and only if $\ocf(\O)=\{\Top\}$.
\end{itemize}
\medskip

Here are some special cases:

\begin{proposition}
Let $\I$ be an $\L$-structure and $\F$ a meta-formula.

(a) Let $\pval$ be a predicate valuation.
Then $|\mc{1}\F|_{\I_+}^{\pval}=\Top$ if and only if $|\F|_{\I_+}^{\iass\pval}=\Top$ for every
variable assignment $\iass$.

(b) Suppose that $\F$ is closed.
Then $|\mc{2}\F|_{\I_+}=\Top$ if and only if
$|\F|_{\I_+}^\pval=\Top$ for every predicate valuation $\pval$;

(c) Let $\pval$ be a predicate valuation.
Then $|\ec{1}\F|_{\I_+}^{\pval}=\Top$ if and only if $|\F|_{\I_+}^{\iass\pval}=\Top$ for some
variable assignment $\iass$.
\end{proposition}

Let us note that the proposition can be rephrased as follows:
\begin{enumerate}[label=(\alph*)]
\item a meta-formula $\F$ is valid in $(\I,\pval)$ if and only if $\mc{1}\F$ is valid in $(\I,\pval)$;
\item a closed meta-formula $\F$ is valid in $\I$ if and only if $\mc{2}\F$ is valid in $\I$;
\item a meta-formula $\F$ is satisfiable in $(\I,\pval)$ if and only if $\ec{1}\F$ is valid
in $(\I,\pval)$.
\end{enumerate}

\begin{corollary} \label{two-valued meta-interpretation}
Let $\I$ be an $\L$-structure.

(a) A principle $\prin\Phi$ is valid in $\I$ if and only if the formula $\Phi$ is valid in $\I$
(that is, $\ocf(|\Phi|_\I^{\iass\pval})=\Top$ for every variable assignment $\iass$ and every
predicate valuation $\pval$).

(b) A rule $\Phi_1,\dots,\Phi_n\,/\,\Psi$ is valid in $\I$ if and only if for every predicate valuation
$\pval$, if the formulas $\Phi_1,\dots,\Phi_n$ are valid in $(\I,\pval)$, then $\Psi$ is also valid in
$(\I,\pval)$ (in other words, if $\ocf(|\Phi_j|_\I^{\iass\pval})=\Top$ for each $j$ and every variable
assignment $\iass$, then $\ocf(|\Psi|_\I^{\iass\pval})=\Top$ for every variable assignment $\iass$).
\end{corollary}

\subsubsection{Models}\label{models2}
Let us fix a logic $L$ based on a language $\L$ and given by a derivation system $\Ds$.

An $\L$-structure $M$ is called a {\it model} of $L$ if the meta-formula $\Ds$ is valid in $M$
(in other words, if all the laws and inference rules of $\Ds$ are valid in $M_+$).%
\footnote{Let us note that models as defined here include neither a predicate valuation nor 
a variable assignment.}
If an $\L$-structure $M$ is a model of $L$, then clearly all derivable principles of $L$ are valid in $M$, and more
generally, it is easy to see that
\[\F_1,\dots,\F_k\turnstile _L\G\ \text{ implies }\
\F_1,\dots,\F_k\Turnstile_M\G\]
for arbitrary meta-formulas $\F_i$ and $\G$.
In particular,
\[\text{if a rule}\ \Phi_1,\dots,\Phi_k\,/\,\Psi\ \text{ is derivable in $L$, then it is valid in every
model of $L$.}\]
For many logics one also has the converse implication (see e.g.\ \cite{Ry}*{1.4.11} for
the zero-order case).
In fact, a logic $L$ is said to be {\it complete} with respect to a family of its models $M_t$
(possibly consisting of just one model), if every purely logical principle that is valid in each $M_t$
is derivable in $L$, and {\it strongly complete} if every purely logical rule that is valid in
each $M_t$ is derivable in $L$.
Thus the family $M_t$ is not strongly complete if and only if there exists a proper extension $L^+$ of $L$
such that each $M_t$ is a model of $L^+$.

A {\it model} of a theory $\Th$ over $L$ is an $\L(\Th)$-structure $M$ that is a model of $L$ such that 
the meta-formula $\Th$ is valid in $M$ (in other words, such that all axioms of $\Th$ are valid in $M$).
Clearly, every theorem of $\Th$ is valid in all models of $\Th$.
The theory $\Th$ is called {\it complete} with respect to a family $M_t$ of its models if every
formula that is valid in each $M_t$ is a theorem of $\Th$.
Thus $\Th$ is complete with respect to the family $M_t$ if and only if every principle (possibly not purely
logical) that is valid in each $M_t$ is derivable from $\Th$ in $L$.
We may also call $\Th$ {\it strongly complete} with respect to the family $M_t$ if every rule (possibly not
purely logical) that is valid in each $M_t$ is derivable from $\Th$ in $L$.

Given a first-order language $\L$, let $\L_+$ denote the language obtained by extending $\L$ by countably 
many predicate constants of each arity.
Thus if $\L$ is the language of a logic $L$, we have an arity-preserving bijection $\phi_L$ between 
the predicate variables of $\L$ and the predicate constants of $\L_+$.
It extends to a bijection $\phi_L$ between the formulas of $\L$ and those formulas of $\L_+$ that involve
no predicate variables.
If $M$ is an $\L^+$-structure, let $M^-$ denote the $\L$-structure obtained from $M$ by discarding the 
interpretation of the predicate constants that are not in $\L$. 

\begin{proposition} \label{strong completeness}
If every consistent theory $\Th$ over a logic $L$ such that $\L(\Th)=\L(L)_+$ and $\Th$ involves no predicate 
variables is complete with respect to some family $M_{\Th,t}$ of its models, then $L$ is strongly complete with 
respect to the two-parameter family $M_{\Th,t}^-$ of its models.
\end{proposition}

The proof also works if $L$ is replaced by an empty theory over $L$.

\begin{proof}
If a rule $\Phi_1,\dots,\Phi_k\,/\,\Psi$ of the language $\L(L)$ is not derivable in $L$, let 
$\Th=\mc{1}\phi_L(\Phi_1)\mand\dots\mand\mc{1}\phi_L(\Phi_k)$.
Then $\phi_L(\Psi)$ is not a theorem of $\Th$.
In particular, $\Th$ is consistent, and hence $\phi_L(\Psi)$ is not valid in $M_{\Th,t}$ for some $t$.
But each $\mc{1}\phi_L(\Phi_i)$ is valid in $M_{\Th, t}$.
Thus $\mc{1}\phi_L(\Phi_1)\mand\dots\mand\mc{1}\phi_L(\Phi_k)\imp\mc{1}\phi_L(\Psi)$ is not valid in 
$M_{\Th, t}$, and hence the rule $\Phi_1,\dots,\Phi_k\,/\,\Psi$ is not valid in $M_{\Th, t}$ or,
equivalently, in $M_{\Th,t}^-$.
\end{proof}

\begin{remark}\label{Dragalin}
Strong completeness is closely related to Dragalin's composite theories \cite{Dr}.
We call a pair $\Th=(\Th_+,\Th_-)$ a {\it composite theory} if $\Th_+$ is a theory, i.e., 
a finite meta-conjunction of principles, called the {\it axioms} of $\Th$, and $\Th_-$ is 
a finite meta-disjunction (see \ref{meta-disjunction}) of principles, which might be called 
the {\it taboos} of $\Th$.
The empty meta-disjunction will be identified with $\tilde\metabot=\fm{\mq{\gamma}\gamma}$.
A composite theory $\Th$ is called {\it inconsistent} if $\Th_+\turnstile \Th_-$.

A {\it model} of a composite theory $\Th$ over a logic $L$ is a model $M$ of $L$ such that $\Th_+$ is valid in $M$ 
and $\Th_-$ is not valid in $M$.
Since meta-disjunction, like all other meta-connectives, is treated classically in the two-valued 
meta-interpretation, to say that $\Th_-$ is not valid in $M$ is equivalent to saying that
none of the taboos of $\Th$ is valid in $M$ --- as long as there is at least one taboo.
Else we have $\Th_-=\fm{\mq{\gamma}\gamma}$, and to say $\Th_-$ is not valid in $M$ is equivalent to saying that
not all formulas are valid in $M$.
This is equivalent to the usual requirement that $\D$ be nonempty (which is thus ``explained'' by 
composite theories). 

If every consistent composite theory $\Th$ over $L$ such that $\L(\Th)=\L(L)_+$ and $\Th$ involves no
predicate variables has a model $M_{\Th}$, then $L$ is strongly complete with respect to the family $M_{\Th}^-$.
Indeed, if a rule $\Phi_1,\dots,\Phi_n\,/\,\Psi$ is not derivable, then the composite theory $\Th$ 
with axioms $\mc{1}\phi_L(\Phi_1),\dots,\mc{1}\phi_L(\Phi_n)$ and the only taboo $\mc{1}\phi_L(\Psi)$ 
is consistent.
Then $\mc{1}\phi_L(\Phi_1),\dots,\mc{1}\phi_L(\Phi_n)$ are valid in $M_\Th$, whereas $\mc{1}\phi_L(\Psi)$ is not.
Hence $\mc{1}\phi_L(\Phi_1)\mand\dots\mand\mc{1}\phi_L(\Phi_n)\imp\mc{1}\phi_L(\Psi)$ is not valid in $M_\Th$.
Therefore the rule $\Phi_1,\dots,\Phi_n\,/\,\Psi$ is not valid in $M_\Th$ or, equivalently, in $M_\Th^-$.
\end{remark}

\subsubsection{Valuation fields} \label{valuation fields}

Given an $\L$-structure $\I$ given by $\D$, $\O$ and interpretations of function symbols, predicate
constants, connectives and quantifiers, we call a family $\H$ of subsets $\H_n\subset\Hom(\D^n,\O)$, 
$n=0,1,2,\dots$, a {\it valuation field} with respect to $\I$ if it is closed with respect to 
the interpretations of function symbols, predicate constants, connectives and quantifiers provided by $\I$.
In other words, $\H=(\H_n)$ is a valuation field if and only if each valuation $\pval$ such that
$\pval(\phi)\in\H_n$ for every $n$-ary predicate variable $\phi$, $n=0,1,2,\dots$, also satisfies 
$|F|_\I^\pval\in\H_n$ for each $n$-formula $F$ of the language, $n=0,1,2,\dots$.

A valuation field $\H$ {\it contains} a valuation $\pval$ if $\pval(\phi)\in\H_n$ for every $n$-ary predicate 
variable $\phi$, $n=0,1,2,\dots$.
The valuation field {\it generated by} a valuation $\pval$, denoted $\left<\pval\right>$, is the smallest 
valuation field that contains $\pval$.
Clearly, $\left<\pval\right>=(\H_n)$ where each $\H_n$ consists of $|F|_\I^\pval$ where $F$ runs over
all $n$-formulas of the language.

Given a valuation field $\H=(\H_n)$ with respect to an $\L$-structure $\I$, a {\it meta-$L$-structure $\J$ 
restricted over} $\H$ is defined like a usual meta-$\L$-structure, except that 
the second-order meta-quantifiers are now interpreted by functions $|\q_n|_\J\:\Hom(\H_n,\Qm\big)\to\Qm$.
In particular, the {\it two-valued meta-$L$-structure restricted over} $\H$ is the modification $\I_+(\H)$
of the usual two-valued meta-$\L$-structure $\I_+$ obtained by re-interpreting the second-order meta-quantifiers 
by functions $|\q_n|_{\I_+(\H)}\:\Hom(\H_n,\Qm\big)\to\Qm$ defined in the same way as before, i.e.\ by
$|\q_n|_{\I_+(\H)}(f)=\Top$ if and only $f(\phi)=\Top$ for all $\phi\in\H_n$.

It is not hard to see that for any valuation field $\H$, $\I_+(\H)$ is a standard model of the meta-logic 
in the sense of \S\ref{models}, except that all valuations mentioned there are now required to be contained 
in $\H$.
(It is here that we use that $\H$ is a valuation field and not just an arbitrary family of subsets of
the function sets $\Hom(\D^n,\O)$.)

Given a valuation field $\H$, with respect to an $\L$-structure $M$, the latter is called
a {\it model} of $\L$ {\it with respect to $\H$} if the meta-formula $\Ds$ is valid in $M_+(\H)$.

\begin{remark} Valuation fields are parallel to (and, in a sense, generalize) algebras of admissible sets, 
which distinguish the so-called ``general frames'' from usual frames.
\end{remark}

\subsection{Classical logic (semantics)} \label{classical models}

\subsubsection{Two-valued models}\label{two-valued}

In a {\it two-valued} model of classical logic, the set $\O$ is taken to be
$\{\Top,\Bot\}$, with $\ocf=\id$, and the connectives and the quantifiers are interpreted according
to the usual truth tables (see \S\ref{Tarski truth}).
(Thus two-valued models of classical logic differ only in the choice of the domain $\D$.)
It is not hard to check that, indeed, the laws and the inference rules of classical logic
(see \S\ref{logics}) hold under this interpretation.

A model of a theory over classical logic is {\it two-valued} if it extends a two-valued model of classical logic
(by additionally interpreting the predicate constants and the function symbols in the language of the theory).
By a strong form of G\"odel's completeness theorem \cite{RS}*{VIII.5.2}, every consistent theory $\Th$ over 
classical logic is complete with respect to the family of its two-valued models with domain $\N$.%
\footnote{Let us recall that our theories correspond to the traditional notion of a theory given by finitely
many axioms and finitely many axiom schemata (which due to our finiteness restrictions on the language
consist of countably many axioms).
This is included in the countably axiomatized traditional theory of \cite{RS}*{VIII.5.2} 
(in fact, the finitely axiomatized case would already suffice for our present purposes).}
In fact, the proof of this assertion in \cite{RS} actually shows that $\Th$ is complete with respect to
a family $M_{\Th,t}$ of its two-valued models with domain $\N$ that differ only in the interpretation of
the predicate constants, as they all have the same interpretation of the function symbols.
(Namely, if $\frak f$ is an $n$-ary function symbol, $|\frak f|\:\N^n\to\N$ is defined by
$|\frak f|(\overline{t_1},\dots,\overline{t_n})=\overline{\frak f(t_1,\dots,t_n)}$, where
$\overline{\cdot}\:\T(\L)\to\N$ is some bijection between all terms of the language and natural numbers.)
Thus from \ref{strong completeness} we get

\begin{theorem} \label{strong completeness thm}
Classical logic is strongly complete with respect to its unique two-valued model with domain $\N$.

Every empty theory over classical logic is strongly complete with respect to a family of its two-valued models 
with domain $\N$.
This family consists of just one model if there are no predicate constants in the language.
\end{theorem}

The condition that there be no predicate constants in the language cannot be dropped
by considering a propositional (=nullary predicate) constant $\frak c$.
In any two-valued model $M$, either $|\frak c|_M=\Top$ or $|\frak c|_M=\Bot$; hence one of
the non-derivable principles $\prin \frak c$, $\prin\neg \frak c$ is valid in $M$.

We will next use two-valued models to differentiate derivable, admissible and stably admissible rules,
as well as principles.

\begin{example}\label{gen rule}
Although the generalization rule $\fm p(\tr x)\,/\,\forall \tr x\,\fm p(\tr x)$ is derivable in
classical logic (see \ref{rule:generalization}), the corresponding
formula $\fm p(\tr t)\to\forall \tr x\,\fm p(\tr x)$ is not.

Indeed, let us consider a two-valued model $M$ of classical logic with $\D_M=\{0,1\}$ and
a predicate valuation $\pval$ such that $|\fm p|^\pval_M(0)=\Bot$ and $|\fm p|^\pval_M(1)=\Top$.
Then $|\fm p(\tr x)\to\forall \tr x\,\fm p(\tr x)|^{\iass\pval}_M$
is $\Top$ if $\iass(\tr x)=0$ and $\Bot$ if $\iass(\tr x)=1$.
Hence $\fm p(\tr x)\to\forall \tr x\,\fm p(\tr x)$ is not valid in $(M,\pval)$, and therefore
also in $M$.
Thus it is not derivable in classical logic.

Let us note that the meta-formula $\mq{\tr x}\fm p(\tr x)\imp\forall \tr x\,\fm p(\tr x)$ must be
valid in $(M,\pval)$, since it is an immediate consequence of the generalization rule.
This is easy to check directly: it is interpreted by the implication ``if $\fm p(\tr x)$ is valid
in $(M,\pval)$, then $\forall \tr x\,\fm p(\tr x)$ is valid in $(M,\pval)$'', in which both
the hypothesis and the conclusion are false.
\end{example}

\begin{example} \label{cbot}
In zero-order classical logic with a propositional constant $\frak c$ added to the language,
the rule $\frak c\,/\,\clbot$ is admissible but not derivable (and not stably admissible).

Indeed, a two-valued model $M$ such that $|\frak c|_M=\Top$ shows that the rule $\frak c\,/\,\clbot$ is not
derivable.
Moreover it shows that the principle $\prin\frak c$ does not imply the principle $\prin\clbot$, and
hence that the rule $\frak c\,/\,\clbot$ is not stably admissible.
On the other hand, a two-valued model $M$ such that $|\frak c|_M=\Bot$ shows that the formula $\frak c$
is not derivable.
Hence the rule $\frak c\,/\,\clbot$, which has no special cases apart from itself, is admissible.
\end{example}

\begin{example}\label{omitting}
In classical logic, the rule \[\frac{\exists \tr x\,\fm p(\tr x),\,\exists \tr x\,\neg\fm p(\tr x)}{\clbot}\]
is admissible but not derivable.

To see the latter, we can use the same $M$ and $\pval$ as in the previous example (\ref{gen rule}).
Then $\exists\tr x\,\fm p(\tr x)$ and $\exists\tr x\,\neg\fm p(\tr x)$ are both valid in $(M,\pval)$,
yet $\clbot$ is not.
Hence the rule in question is not valid in $M$, and therefore it is not derivable in classical logic.

On the other hand, suppose that for some 1-formula $P$
we have both $\turnstile\exists \tr x\,P(\tr x)$ and $\turnstile\exists \tr x\,\neg P(\tr x)$.
Let $N$ be a two-valued model of classical logic with domain consisting of a single element,
$\D_N=\{*\}$.
Then $\exists \tr x\,P(\tr x)$ and $\exists \tr x\,\neg P(\tr x)$ must both be valid in $N$, that is,
$|\exists \tr x\,P(\tr x)|_N^{\iass\pval}=|\exists \tr x\,\neg P(\tr x)|_N^{\iass\pval}=\Top$
for every individual assignment $\iass$ and every predicate valuation $\pval$.
Since $*$ is the only element in $\D_N$, this means that
$|P|_N^{\iass\pval}(*)=|\neg P|_N^{\iass\pval}(*)=\Top$, which is a contradiction.
Hence for any special case of the rule in question, the derivability of the premisses is impossible,
and in particular implies the derivability of the conclusion.
Thus the rule is admissible.
\end{example}

\begin{remark}\label{Dzik}
In fact, these examples are characteristic of non-derivable admissible rules in classical logic.
They are all such that any specialization of the conjunction of their premisses is non-derivable,
and even not satisfiable in any two-valued model with one-element domain \cite{Dz}.
\end{remark}

\begin{example} \label{st-adm}
The rule of the previous example is, in fact, stably admissible, not only for classical but,
more generally, for intuitionistic logic (and also for empty theories over them, with 
no predicate constants of arity $>0$ in the language).
Let us note that this rule is not derivable in intuitionistic logic, since it is not derivable in classical
logic.

To prove the stable admissibility, we need to show that for any 1-formula $\Sigma$, the principles
$\prin\exists \tr x\,\Sigma(\tr x)$ and $\prin\exists \tr x\,\neg\Sigma(\tr x)$ imply $\prin\ab$.
Indeed, let $\vec\gamma$ be the tuple of all predicate variables that occur in $\Sigma$, and let
$\vec r=(r_1,\dots,r_k)$ be the tuple of their arities.
Next let $\ab_{\vec r}=(\ab_{r_1},\dots,\ab_{r_n})$, where $\ab_r$ denotes the $r$-formula
$\var_1,\dots,\var_r\mapsto\ab$.
(Alternatively, we could have used $\triv$ in place of $\ab$, or both of them.)
Then the two principles in question meta-specialize to the formulas
$\exists \tr x\,\Sigma(\tr x)[\vec\gamma/\ab_{\vec r}]$ and
$\exists \tr x\,\neg\Sigma(\tr x)[\vec\gamma/\ab_{\vec r}]$ respectively.
Since the language contains no predicate constants of arity $>0$, the formula
$\Sigma(\tr x)[\vec\gamma/\ab_{\vec r}]$ contains no occurrences of individual variables,
including $\tr x$.
But then our formulas $\exists \tr x\,\Sigma(\tr x)[\vec\gamma/\ab_{\vec r}]$ and
$\exists \tr x\,\neg\Sigma(\tr x)[\vec\gamma/\ab_{\vec r}]$ are equivalent (over intuitionistic logic) to
$\Sigma(\tr x)[\vec\gamma/\ab_{\vec r}]$ and $\neg\Sigma(\tr x)[\vec\gamma/\ab_{\vec r}]$
respectively.
Thus they indeed have $\ab$ as a joint consequence.
\end{example}

\begin{example}\label{equality-st-adm}
In classical logic with language extended by one binary predicate constant (denoted $=$),
the rule $\exists \tr x\,\fm p(\tr x),\,\exists \tr x\,\neg\fm p(\tr x)\,/\,\clbot$
is not stably admissible.

Indeed, let us show that the principles $\prin\exists \tr x\,\ \tr x=\tr y$ and
$\prin\exists \tr x\,\ \neg \tr x=\tr y$ do not imply $\prin\clbot$; or in other words that
if they are added as laws to classical logic with equality, the resulting logic is still consistent.
Indeed, $\mq{\tr y}\exists \tr x\,\ \tr x=\tr y$ or, equivalently, $\forall \tr y\,\exists \tr x\,\ \tr x=\tr y$
is derivable from $\forall\tr y\,\ \tr{y=y}$, and therefore is valid in any model of classical logic with equality;
whereas $\mq{\tr y}\exists \tr x\,\ \neg \tr x=\tr y$ or, equivalently,
$\forall \tr y\,\exists \tr x\,\ \neg \tr x=\tr y$ is valid in any model of classical logic with equality
with non-singleton domain.
(See e.g.\ \cite{Men} concerning classical logic with equality.)
Thus classical logic with $=$ added to the language (and no additional laws or inference rules)
has two-valued models where both $\exists \tr x\,\ \tr x=\tr y$ and $\exists \tr x\,\ \neg \tr x=\tr y$
are valid; whereas $\clbot$ is not valid in any model.
\end{example}

\begin{proposition}\label{der-adm}
In zero-order classical logic (with no propositional constants added to the language) 
admissible rules are derivable.
\end{proposition}

\begin{proof}
If a rule $P_1,\dots,P_k\,/\,Q$ is not derivable, then by the strong completeness theorem
\ref{strong completeness thm}, this rule is not satisfied in the two-valued model $M$ of zero-order
classical logic.
Then there is a special case $P_1',\dots,P_k'\,/\,Q'$ of the rule in question and a 
valuation $\pval$ such that each $|P_i'|_M^\pval=\Top$ and $|Q'|_M^\pval=\Bot$
(see e.g.\ Proposition \ref{two-valued meta-interpretation}(b)).
Let $\vec p=(p_1,\dots,p_n)$ be the tuple of all propositional variables that occur freely in
any of the formulas $P_1',\dots,P_k',Q'$, and let $\vec{\frak c}=(\frak c_1,\dots,\frak c_n)$,
where $\frak c_i=\cltop$ if $\pval(p_i)=\Top$ and $\frak c_i=\clbot$ if $\pval(p_i)=\Bot$.
Then $(P_1',\dots,P_k'\,/\,Q')[\vec p/\vec{\frak c}]$ is a new rule, which can be written
$P_1'',\dots,P_k''\,/\,Q''$,
and which is another special case of the original rule $P_1,\dots,P_k\,/\,Q$.
By design, each $|P_i''|_M^\pval=\Top$ and $|Q''|_M^\pval=\Bot$.
But since $P_1'',\dots,P_k''$ and $Q''$ contain no propositional variables, their
interpretations do not depend on $\pval$, and we may also write $|P_i''|_M=\Top$ and $|Q''|_M=\Bot$,
and consider $P_i''$ and $Q''$ as principles.
But zero-order classical logic is complete with respect to $M$ (see \ref{strong completeness thm} again),
so each $P_i''$ is derivable, and $Q''$ is not derivable (in fact, $\neg Q''$ is derivable due to
$|\neg Q''|_M=\Top$.)
Thus $P_1,\dots,P_k\,/\,Q$ is not admissible.
\end{proof}

Two-valued models can also to be used to compare implication between principles or rules with that
between the judgements of their derivability.

\begin{proposition}\label{rule implication}
In zero-order classical logic (with no propositional constants added to the language) rules $\Ru_1,\dots,\Ru_k$ 
imply a rule $\Su$ if and only if the judgements $\turnstile \Ru_1,\dots,\turnstile \Ru_k$ imply $\turnstile \Su$.
\end{proposition}

If propositional constants are added to the language, the assertion is no longer true
(see Example \ref{cbot2} below).

\begin{proof}
The ``only if'' assertion holds in any logic.
Conversely, let us assume the implication between the judgements of derivability, and let us
consider two cases.
If $\Ru_1,\dots,\Ru_k$ all happen to be derivable, then so must be $\Su$, and so it is trivially
a consequence of $\Ru_1,\dots,\Ru_k$.
If some $\Ru_i$ is not derivable, then by the proof of \ref{der-adm} it has a special case
$P_1,\dots,P_n\,/\,Q$ (involving only $\cltop$, $\clbot$ and other connectives, but no propositional
variables) such that each $P_j$ is derivable, and $Q\to\clbot$ is also derivable.
Thus by adjoining $\Ru_i$ to the logic, we will be able to derive $\clbot$.
Then everything becomes derivable, including $\Su$.
Thus $\Su$ is a consequence of $\Ru_i$ alone; in particular, it is a joint consequence of $\Ru_1,\dots,\Ru_k$.
\end{proof}

Proposition \ref{rule implication} does not generalize to first-order classical logic:

\begin{example} \label{principle implication}
In classical logic, the principle $\prin\exists \tr x\,\fm p(\tr x)\to\forall \tr x\,\fm p(\tr x)$ does not
imply $\prin\clbot$, even though the judgement
$\turnstile\exists \tr x\,\fm p(\tr x)\to\forall \tr x\,\fm p(\tr x)$ implies $\turnstile\clbot$.

Indeed, $\exists \tr x\,\fm p(\tr x)\to\forall \tr x\,\fm p(\tr x)$ is valid in all models with a singleton
domain, whereas $\clbot$ is not valid in any model; thus there is no implication between the principles.
On the other hand, $\exists \tr x\,\fm p(\tr x)\to\forall \tr x\,\fm p(\tr x)$ is not valid in some models;
thus the judgement $\turnstile\exists \tr x\,\fm p(\tr x)\to\forall \tr x\,\fm p(\tr x)$ is false, and
so implies any other judgement.
\end{example}

\subsubsection{Leibniz--Euler models}\label{Euler}

It was known already to Leibniz and became well-known from the work of Euler that classical propositional 
logic can be interpreted by set-theoretic operations (of course, at a level of knowledge that is informal 
by modern standards; see \cite{Ba}).
The predicate case (but of course not the modern formulations) essentially appeared in the work of Pierce 
and Schr\"oder in the 19th century.

To describe a {\it Leibniz--Euler} model $V=V(X,\D)$ of classical logic,%
\footnote{In the literature, these seem to be best known as a special case of Boolean-valued models.
However, for the purposes of the present treatise there is no benefit in dealing with general 
Boolean algebras (and consequently, also general Heyting algebras) so there is no point in using
that terminology here, as it would only obscure matters.}
one starts with arbitrary sets $X$ and $\D$.
The set $\O$ will be the set of arbitrary subsets of $X$, and the function $\ocf\:\O\to\{\Top,\Bot\}$
sends the entire $X$ to $\Top$ and every proper subset $S\subsetneq X$ to $\Bot$.

Thus, to specify a valuation, each predicate variable of arity $n$ must be interpreted by 
a $\D^n$-indexed family of subsets of $X$.
Upon a choice of a valuation $\pval$, a closed formula will be interpreted by a subset of $X$, 
and more generally a closed $n$-formula $F$ will be interpreted by a $\D^n$-indexed family 
$|F|^\pval_V\:\D^n\to\O$ of subsets of $X$.
Upon a further choice of a variable assignment $\iass$, an arbitrary formula will be interpreted by 
a subset of $X$, and more generally an arbitrary $n$-formula $F$ will be interpreted by a $\D^n$-indexed 
family $|F|^{\iass\pval}_V\:\D^n\to\O$ of subsets of $X$.
Validity of an $n$-formula $F$ in the model, $\Turnstile F$, means that
$|F|^{\iass\pval}_V(t_1,\dots,t_n)=X$ for each tuple $(t_1,\dots,t_n)\in\D^n$ and for an arbitrary valuation 
$\pval$ and assignment $\iass$.

Finally, disjunction of formulas is interpreted by union of sets, negation by complement (within $X$), and
universal quantification by $\D$-indexed union.
In symbols, if $F$ and $G$ are formulas and $H$ is a 1-formula,
\begin{itemize}
\item $|F\lor G|^{\iass\pval}_V=|F|^{\iass\pval}_V\cup|G|^{\iass\pval}_V$;
\item $|\neg F|^{\iass\pval}_V=X\but |F|^{\iass\pval}_V$;
\item $|\exists \tr x\,H(\tr x)|^{\iass\pval}_V=\bigcup_{d\in D}|H|^{\iass\pval}_V(d)$.
\end{itemize}

Consequently, conjunction is interpreted by intersection, and existential quantification by $\D$-indexed
intersection.
The interpretation of implication is also determined by those disjunction and negation.
The interpretation of $\clbot$, which is classically equivalent to $\fm{p\land\neg p}$, is determined by
those of conjunction and negation; and the interpretation of $\cltop$ is determined by those of $\clbot$
and negation.
Thus $\clbot$ is interpreted by the empty union (that is, $\emptyset$) and $\cltop$ by the empty intersection
(that is, $X$).

We call this interpretation a {\it Leibniz--Euler model} of classical logic; indeed, it is not hard to check
that the laws and the inference rules of classical logic (see \S\ref{logics}) hold under
this interpretation (cf.\ \cite{RS}).
For example, let us consider the generalization rule, $\fm{q(\tr x)\,/\,\forall\tr x\, q(\tr x)}$.
To check that it is valid in a Leibniz--Euler model $V=V(X,\D)$, we pick some predicate valuation $\pval$
(cf.\ \ref{two-valued meta-interpretation}).
Then $|\fm q|_V^\pval$ is some function from $\D$ to the set of all subsets of $X$;
$|\forall\tr x\,\fm q(\tr x)|_V^\pval$ is the intersection of all its values; and
$|\fm q(\tr x)|_V^{\iass\pval}$ is its value at $\iass(x)$, where the individual assignment $\iass$
associates to every individual variable, including $\tr x$, an arbitrary element of $\D$.
Thus if $|\fm q(\tr x)|_V^{\iass\pval}=X$ for every individual assignment $\iass$, then
$|\forall\tr x\,\fm q(\tr x)|_V^{\pval}=X$, as desired.

In the case where $X$ is a singleton, $\{*\}$, Leibniz--Euler models reduce to the two-valued models
discussed previously.
As long as no predicate constants are added to the language, two-valued models suffice to discern classical truth
(see Theorem \ref{strong completeness thm}).
Moreover, with an arbitrary language, a formula $F$ that is not valid in some Leibniz--Euler model
$V=V(X,\D)$ is also not valid in some two-valued model $M$ with the same domain $\D$.
Namely, we can get $M$ from $V$ by keeping the same interpretations of the function symbols,
picking an $x\in X$ such that $x\notin |F|_V^{\iass\pval}$ for some individual assignment $\iass$ and
predicate valuation $\pval$, and setting the interpretation $|\frak p|_M$ of a predicate constant $\frak p$
to be the function sending a tuple $\vec d\in\D^n$ to $\{*\}$ if $x\in |\frak p|_V(\vec d)$ and to
$\emptyset$ otherwise.

On the other hand, classical logic (or the empty theory over it, with any language)
is complete with respect to some Leibniz--Euler model with domain $\N$ and with a fixed predicate 
valuation \cite{RS}*{VIII.3.3, II.8.1}, but of course not with respect to any particular two-valued model 
with a fixed valuation --- not even in the propositional case.
Also, a $\lambda$-closed formula $F$ is not valid in a two-valued model $M$ if and only if
$\neg F$ is valid in $M$; of course, this strange property does not hold in Euler models.

Nevertheless, an Euler model such that $X=\D^k$ for some $k$ can be derived from a two-valued model,
by introducing $k$ implicit parameters and changing the arities of predicate constants.
Thus we fix a domain $\D$, and to interpret an $n$-ary predicate constant $\frak p$
in an Euler model $V$ with $X=\D^k$ we look at the interpretation of an $(n+k)$-ary predicate constant
$\frak p'$ in a two-valued model $M$ with the same domain $\D$,
and set $$|\frak p|_V(x_1,\dots,x_n)=\{(x_{n+1},\dots,x_{n+k})\in\D^k\mid
|\frak p'|_M(x_1,\dots,x_{n+k})=\Top\}.$$

\begin{example} \label{meta-disjunction0}
In classical logic, $F\parallel G\turnstile F\lor G$ for any formulas
$F$, $G$ by \ref{connectives vs meta-connectives}; but the converse,
$F\lor G\turnstile F\parallel G$, does not hold already in the zero-order case.

Indeed, let $V$ be a Leibniz--Euler model of
classical logic with $X=\{0,1\}$, and let $\pval$ be a valuation such that
$\pval(\fm p)=\{0\}$ and $\pval(\fm q)=\{1\}$.
Then $|\fm p\lor\fm q|_V^\pval=\{0,1\}$, so $|\fm p\lor\fm q|_{V_+}^\pval=\Top$; whereas
$|\fm p|_V^\pval=\{0\}$ and $|\fm q|_V^\pval=\{1\}$, so $|\fm p|_{V_+}^\pval=|\fm q|_{V_+}^\pval=\Bot$,
and consequently $|\fm p\parallel\fm q|_{V_+}^\pval=\Bot$.
Thus $\fm p\lor\fm q\imp\fm p\parallel\fm q$ is not valid in $V$, and consequently is not derivable
in classical logic.

Let us note that $F\lor G\imp F\parallel G$ is clearly valid for any formulas $F$, $G$
in every two-valued model of classical logic.
\end{example}

\begin{example} \label{meta-disjunction1}
The previous example (\ref{meta-disjunction0}) can be improved:
$\mc{2}F\lor G\turnstile\mc2(F\parallel G)$ does not hold already in zero-order classical logic.

Indeed, let $V$ be a Leibniz--Euler model of
classical logic with $X=\{0,1\}$, and let $\pval$ be a valuation such that $\pval(\fm p)=\{0\}$.
Then $|\neg\fm p|_V^\pval=\{1\}$, so $|\fm p|_{V_+}^\pval=|\neg\fm p|_{V_+}^\pval=\Bot$,
and consequently $|\fm p\parallel\neg\fm p|_{V_+}^\pval=\Bot$.
Thus $\mc2(\fm p\parallel\neg\fm p)$ is not valid in $V$, and consequently is not derivable
in classical logic.
On the other hand, $\mc2\fm p\lor\neg\fm p$ is the principle of excluded middle, which is certainly
derivable in classical logic.
\end{example}

\begin{example} \label{meta-implication0}
In classical logic, $F\to G\turnstile F\imp G$ for any formulas
$F$, $G$ by the {\it modus ponens} rule; but the converse,
$F\imp G\turnstile F\to G$, does not hold already in the zero-order case.

Indeed, let $V$ be a Leibniz--Euler model of
classical logic with $X=\{0,1\}$, and let $\pval$ be a valuation such that
$\pval(\fm p)=\{0\}$.
Then $|\fm p\to\clbot|_V^\pval=\{1\}$, so $|\fm p\to\clbot|_{V_+}^\pval=\Bot$; whereas
$|\fm p|_V^\pval=\{0\}$ and $|\clbot|_V^\pval=\{\emptyset\}$, so
$|\fm p|_{V_+}^\pval=|\clbot|_{V_+}^\pval=\Bot$, and consequently $|\fm p\imp\clbot|_{V_+}^\pval=\Top$.
Thus $(\fm p\imp\clbot)\imp\fm p\to\clbot$ is not valid in $V$, and consequently is not derivable
in classical logic.

Let us note that $(F\imp G)\imp F\to G$ is clearly valid for any formulas $F$, $G$
in every two-valued model of classical logic.
\end{example}

\begin{example} \label{meta-implication1}
A rule $F\,/\,G$ need not imply the principle $\prin F\to G$ in zero-order classical
logic with language extended by a propositional constant $\frak c$.

Indeed, similarly to the previous example (\ref{meta-implication0}), $\Turnstile_V \frak c\imp\bot$
does not imply $\Turnstile_V \frak c\to\bot$
for a Leibniz--Euler model $V$ with $X_V=\{0,1\}$ and $|\frak c|_V=\{0\}$.
\end{example}

\begin{remark} \label{meta-implication2}
A rule $F\,/\,G$ implies the principle $\prin F\to G$ in zero-order classical logic (with no
propositional constants added to the language).
Indeed, $\turnstile F\,/\,G$ implies $\turnstile\prin F\to G$ by the deduction theorem,
and therefore $(F\,/\,G)$ implies $\prin F\to G$ by \ref{rule implication}.
\end{remark}

\begin{example} \label{meta-implication3}
$\turnstile\mc{2}(F\imp G)\imp\mc{2}F\to G$ need not imply
$\turnstile\mc{2}\big((F\imp G)\imp F\to G)$ already in zero-order classical logic (with no propositional 
constants added to the language).

Indeed, let $F=\fm p$ and $G=\clbot$.
Then the former meta-formula is derivable by Remark \ref{meta-implication1}; whereas the latter one is
not derivable by Example \ref{meta-implication0}.
\end{example}

\begin{example} \label{cbot2}
If $\Ru$ and $\Su$ are rules, and the judgement $\turnstile \Ru$ implies $\turnstile \Su$, then
the rule $\Ru$ need not imply $\Su$ in zero-order classical logic with language extended by
a propositional constant $\frak c$.

Indeed, let $\Ru=\frak c\,/\,\clbot$ and $\Su=\prin\neg\frak c$ (or, equivalently,
$\Su=\cltop\,/\,\neg\frak c$).
Then $\Ru$ is not derivable (see Example \ref{cbot}), and consequently $\turnstile\Ru$ implies everything,
including $\turnstile\Su$.
On the other hand, $\Ru$ does not imply $\Su$ (see Example \ref{meta-implication1}).
\end{example}

\begin{example}\label{deduction example}
A rule $F\,/\,G$ need not imply the principle $\prin F\to G$ in classical logic, 
even if the formulas $F$ and $G$ are closed.

Indeed, let $F=\exists \tr x\,\fm p(\tr x)\land\exists \tr x\,\neg\fm p(\tr x)$ and $G=\bot$.
We will construct an Euler model by combining the two-valued models of Examples \ref{gen rule} and
\ref{omitting}.

Thus let $M$ and $N$ be the two-valued models of classical logic with $\D_M=\{0,1\}$ and $\D_N=\{*\}$.
Let $V$ be an Euler model with $X=\{\mu,\nu\}$ and $\D=\{0,1\}$.
Given a predicate valuation $\pval_M$ for $M$ and a predicate valuation $\pval_N$ for $N$, let 
$\pval=\pval(\pval_M,\pval_N)$ be the predicate valuation for $V$ defined on each $r$-ary predicate variable $q$ 
by $\mu\in\pval(q)(x_1,\dots,x_r)$ if and only if $\pval_M(q)(x_1,\dots,x_r)=\Top$ and by
$\nu\in\pval(q)(x_1,\dots,x_r)$ if and only if $\pval_N(q)(*,\dots,*)=\Top$.
It follows that for every closed formula $P$,
$\mu\in |P|_V^\pval$ if and only if $P$ is valid in $(M,\pval_M)$, and
$\nu\in |P|_V^\pval$ if and only if $P$ is valid in $(N,\pval_N)$.

It is easy to see that valuations of the form $\pval(\pval_M,\pval_N)$ form a valuation field $K$.
Let us note that these are precisely those valuations $\pval$ for $V$ that satisfy for each $r=1,2,\dots$ and 
each $r$-ary predicate variable $q$ the following alternative: either $\nu\in\pval(q)(x_1,\dots,x_r)$ for all 
$x_1,\dots,x_r\in\{0,1\}$ or $\nu\notin\pval(q)(x_1,\dots,x_r)$ for all $x_1,\dots,x_r\in\{0,1\}$.
 
Let $\pval_M$ be a predicate valuation for $M$ such that $\pval_M(\fm p)(0)=\Bot$ and
$\pval_M(\fm p)(1)=\Top$, and let $\pval_N$ be some predicate valuation for $N$.
Let $\pval=\pval(\pval_M,\pval_N)$.

Clearly, $F$ is valid in $(M,\pval_M)$, and consequently $|F|_V^\pval$ contains $\mu$.
Hence $|F\to\bot|_V^\pval$ does not contain $\mu$, and therefore $|F\to\bot|_{V_+}^\pval=\Bot$.
Thus $\prin F\to\clbot$ is not valid in $V$ with respect to the valuation field $K$.

On the other hand, $F$ is clearly not valid in $(N,\pval_N)$ (regardless of the choice of $\pval_N$),
and consequently $|F|_V^\pval$ does not contain $\nu$.
Hence $|F\imp\bot|_{V_+}^\pval=\Top$.
This works regardless of the choice of $\pval_M$, so we get that $F\,/\,\clbot$ is valid in $V$ with
respect to the valuation field $K$.
\end{example}

\subsubsection{Platonist interpretation I (naive)}\label{Frege+}

We now turn to contentual interpretations of the meta-logical notions introduced above, in the case of 
classical object logic.
 
As discussed in \S\ref{universalism}, a natural understanding of classical logic, closely related to 
the Platonist philosophy, is that closed formulas signify propositions --- which are either true or false.
This view is compatible with the usual Tarski-style model theory, as presented in \S\ref{meta-formulas} 
and \S\ref{models}, and extends to give an interpretation of the meta-logic of classical logic.

Namely, a domain $\D$ is fixed, and $\O$ is taken to be a class of propositions, consisting of 
a prescribed class of contentful (e.g.\ mathematical) primitive propositions and of composite propositions, 
which are obtained inductively from primitive ones by using contentual classical connectives and 
quantifiers over arbitrary $\D$-indexed families of propositions.
For instance, primitive propositions could be of the form ``$n=m$'', where $n$ and $m$ are numerals, 
e.g.\ ``$0=1$'' and ``$1=1$''.
In this case composite propositions would include ``$0=1$ or not $0=1$'' and, if $\D$ is the set $\N$ 
of numerals, also ``there exists an $n$ such that not $n=0$''.
Formal classical connectives and quantifiers ($\lor:\1^2\too\1$, $\exists:(\0\too\1)\too\1$, etc.) are 
interpreted straightforwardly by the contentual ones (``or''$\:\O^2\to\O$, 
``exists''$\:\Hom(\D,\O)\to\O$, etc.).

The truth function $\ocf\:\O\to\Qm=\{\Top,\Bot\}$ is defined explicitly on primitive propositions, and 
is extended inductively to composite ones by the usual truth tables (see \S\ref{Tarski truth}).
This determines a contentual interpretation of classical logic and its meta-logic, which we will call 
the {\it Platonist interpretation}.
 
Since we are assuming that each proposition is either true or false, the Platonist interpretation
is incompatible with the specific interpretation of \S\ref{contentual}, although it is compatible with the
preceding sketch there (if ``provability'' in that sketch is taken to be in a complete theory, such as Tarski's
elementary geometry).
 
\begin{example}\label{platonist example}
Let us record the Platonist interpretations of the judgements of admissibility, 
stable admissibility and derivability of a purely logical rule $F\,/\,G$.
Let $\vec\phi=(\phi_1,\dots,\phi_m)$ be the tuple of all predicate variables that occur freely in $F$
or in $G$, and let $\vec r=(r_1,\dots,r_m)$ be the tuple of their arities.
Below $\vec\Phi$ is understood to run over the set $\Fm_{\vec r}:=\Fm_{r_1}\x\dots\x\Fm_{r_m}$ of all tuples 
$(\Phi_1,\dots,\Phi_m)$ such that each $\Phi_i$ is an $r_i$-formula.

\begin{itemize}
\item Derivability: $\forall\D,\O\ \forall\pval\,
(\forall\iass\,|F|^\pval_\iass\to\forall\iass\,|G|^\pval_\iass)$;

\item Stable admissibility: $\forall\vec\Phi\,\forall\D,\O\ 
(\forall\pval,\iass\,|F[\vec\phi/\vec\Phi]|^\pval_\iass
\To\forall\pval,\iass\,|F[\vec\phi/\vec\Phi]|^\pval_\iass)$;

\item Admissibility:
$\forall\vec\Phi\,(\forall\D,\O\ \forall\pval,\iass\,|F[\vec\phi/\vec\Phi]|^\pval_\iass
\To\forall\D,\O\ \forall\pval,\iass\,|F[\vec\phi/\vec\Phi]|^\pval_\iass)$.
\end{itemize}
\end{example}
 
If we identify with each other all true propositions and, separately, all false propositions, 
we will get a two-valued model of classical logic (see \S\ref{two-valued}), which clearly determines 
the validity of formulas and meta-formulas in the Platonist interpretation.
Thus the Platonist interpretation can be seen as a mere rewording of two-valued models.
In particular, it may be argued to be an adequate understanding of classical logic on the grounds that 
classical logic is strongly complete with respect to a two-valued model with countable domain 
(see Theorem \ref{strong completeness thm}).

However, every two-valued model $M$ of classical logic obviously has the property that if some rule 
$F\,/\,G$ is valid in $M$, where $F$ and $G$ are closed formulas, then also the principle $\prin F\to G$ 
is valid in $M$.
This property is not satisfied by Leibniz--Euler models of classical logic (see 
Examples \ref{meta-implication1} and \ref{deduction example}).
Thus $F\,/\,G$ generally does not imply $\prin F\to G$ in classical logic, but this is not detected
by the Platonist interpretation!

\subsubsection{Modified Platonist interpretation} \label{mPlatonist}

It is not hard to amend the Platonist interpretation so as to overcome the said deficiency.

In addition to a domain of discourse $\D$, we fix a ``hidden parameter'' domain $\E$.
Now $\O$ is taken to be a class of predicates on $\E$, consisting of a prescribed class of contentful 
(e.g.\ mathematical) primitive predicates and composite predicates, which are obtained inductively 
from the primitive ones by using contentual classical connectives and quantifiers over arbitrary 
$\D$-indexed families of predicates.
For instance, if $\E$ is the rational plane $\Q\x\Q$, primitive predicates could be of the form 
``$nq=mr$'', where $n$ and $m$ are numerals, e.g.\ ``$3q=2r$''.
In this case composite predicates would include ``$q=2r$ or $2q=r$'' and, if $\D$ is the set $\N$ of 
numerals, also ``there exists an $n$ such that $nq=r$''.
Formal classical connectives and quantifiers are interpreted straightforwardly by the contentual ones.

The value of $\ocf\:\O\to\Qm=\{\Top,\Bot\}$ on a predicate $P\in\O$ will be $\Top$ if and only if
$P(e)$ is true for all $e\in\E$, where ``true'' is defined explicitly for primitive $P\in\O$ and all 
$e\in\E$ and is extended inductively to composite $P\in\O$ and all $e\in E$ by the usual truth tables, 
separately for each $e\in\E$.
This determines a contentual interpretation of classical logic and its meta-logic, which we will call 
the {\it modified Platonist interpretation}.

If we again identify with each other all true propositions and, separately, all false propositions, 
we will get a Leibniz--Euler model of classical logic, which determines the validity of formulas and 
meta-formulas in the modified Platonist interpretation.
Thus the modified Platonist interpretation can be seen as a mere rewording of the Leibniz--Euler models.
In particular, it does detect that $F\,/\,G$ generally does not imply $\prin F\to G$ in classical logic.

\begin{example}\label{platonist example2}
In the notation of Example \ref{platonist example}, the modified Platonist interpretations of the same 
judgements are as follows, where $e$ is understood to run over $\E$:

\begin{itemize}
\item Derivability: $\forall\D,\E,\O\ \forall\pval\,
\big(\forall\iass\,\forall e\,|F|^\pval_\iass(e)\to\forall\iass\,\forall e\,|G|^\pval_\iass(e)\big)$;

\item Stable admissibility: 

$\forall\vec\Phi\,\forall\D,\E,\O\ 
\big(\forall\pval,\iass\,\forall e\,|F[\vec\phi/\vec\Phi]|^\pval_\iass(e)
\To\forall\pval,\iass\,\forall e\,|F[\vec\phi/\vec\Phi]|^\pval_\iass(e)\big)$;

\item Admissibility:

$\forall\vec\Phi\,\big(\forall\D,\E,\O\ \forall\pval,\iass\,\forall e\,|F[\vec\phi/\vec\Phi]|^\pval_\iass(e)
\To\forall\D,\E,\O\ \forall\pval,\iass\,\forall e\,|F[\vec\phi/\vec\Phi]|^\pval_\iass(e)\big)$.
\end{itemize}
\end{example}

\section{Topology, translations and principles}\label{topology}

Our next goal is to naturally arrive at Tarski's topological models of intuitionistic logic
by means of a symbolic analysis of the BHK interpretation.

\subsection{Deriving a model from BHK}\label{weak BHK}

Given a problem $\Gamma$, let $\wn\Gamma$ denote, as in \S\ref{about-bhk}, the proposition
{\it A solution of $\Gamma$ exists}.
In the notation of \S\ref{confusion}, $\wn\Gamma$ can be also expressed as $|\Gamma|\ne\emptyset$,
where $|\Gamma|$ denotes the set of solutions of $\Gamma$.
Validity of propositions of the form $\wn\Gamma$ is understood to be determined, at least to
some extent, by the BHK interpretation (in particular, it should be known from context
for primitive $\Gamma$).
Then the BHK entails, in particular, the following judgements, which involve both intuitionistic
and classical connectives:
\begin{enumerate}
\item $\wn(\Gamma\land\Delta)$ if and only if $\wn\Gamma\land\wn\Delta$;
\item $\wn(\Gamma\lor\Delta)$ if and only if $\wn\Gamma\lor\wn\Delta$;
\item $\wn(\Gamma\to\Delta)$ implies $\wn\Gamma\to\wn\Delta$;
\item $\neg\wn\ab$;
\item $\wn\exists x\,\Gamma(x)$ if and only if $\exists x\,\wn\Gamma(x)$;
\item $\wn\forall x\,\Gamma(x)$ implies $\forall x\wn\Gamma(x)$.
\item $\wn\triv$.
\end{enumerate}
Judgements (1)--(6) are perhaps best understood as immediate consequences of the six assertions in
\S\ref{confusion}, which were in turn obtained directly from the BHK interpretation.
For instance, from $|\Gamma\lor\Delta|=|\Gamma|\sqcup|\Delta|$ we immediately obtain (2), since
$S\sqcup T\ne\emptyset$ if and only if $S\ne\emptyset$ or $T\ne\emptyset$.
On the other hand, since $\triv=\ab\to\ab$ by definition, judgement (7) can be considered as
a special case of either $\wn(\Gamma\to\Gamma)$ or $\wn(\ab\to\Gamma)$.
Both are clearly true on the BHK interpretation: (i) given a solution of $\Gamma$, by doing nothing
to it we get a solution of $\Gamma$, which is a general method --- hence a solution of
$\Gamma\to\Gamma$; (ii) $\ab\to\Gamma$ was already discussed in \S\ref{tautologies}, (\ref{explosion0}).

Let us formulate a theory $T_{BHK}$ over classical logic attempting to capture the meaning of the above
informal judgements.
This will be only a countable theory in the traditional sense, i.e.\ a countable (and recursively enumerable)
set of axioms, not a finite meta-conjunction of formulas and principles.

For each $n$ and each $n$-formula $\Phi$ in the language of intuitionistic logic, let us introduce 
an $n$-ary predicate constant, denoted $\wn\Phi$. 
The following are the axioms of $T_{BHK}$:
\begin{enumerate}
\item $\wn(\Phi\land\Psi)\Tofrom \wn\Phi\land\wn\Psi$;
\item $\wn(\Phi\lor\Psi)\Tofrom \wn\Phi\lor\wn\Psi$;
\item $\wn(\Phi\to\Psi)\To (\wn\Phi\to\wn\Psi)$;
\item $\neg\wn\ab$;
\item $\wn\exists \tr x\, \Phi(\tr x)\Tofrom \exists \tr x\,\wn\Phi(\tr x)$;
\item $\wn\forall \tr x\, \Phi(\tr x)\To \forall \tr x\wn\Phi(\tr x)$;
\item $\wn\triv$.
\end{enumerate}

Suppose that we have a Leibniz--Euler model of $T_{BHK}$ in subsets of a set $X$ with some domain $\D$.
Let us call a subset of $X$ {\it intuitionistic} if it of the form $|\wn\Phi|_\iass$ for some
formula $\Phi$ in the language of intuitionistic logic, and some assignment $\iass$ in $\D$.
Thus $\emptyset$ is intuitionistic by (4), and $X$ is intuitionistic by (7).
The intersection of two intuitionistic sets is a set of the form $|\wn\Phi|_\iass\cap|\wn\Psi|_{\iass'}$.
Here there is no loss of generality in considering only such pairs of formulas $\Phi$, $\Psi$ that the free 
variables of $\Phi$ have no overlap with the free variables of $\Psi$; consequently, we may additionally 
assume that $\iass'=\iass$; we are also free to keep the latter assumption but drop the former one.
Now $|\wn\Phi|_\iass\cap|\wn\Psi|_\iass=|\wn\Phi\land\wn\Psi|_\iass$, which by (1) is the same as
$|\wn(\Phi\land\Psi)|_\iass$, an intuitionistic set.
Similarly by (2), the union of two intuitionistic sets in intuitionistic.
Also (5) implies that unions of some $\D$-indexed families of intuitionistic sets are intuitionistic.

Thus intuitionistic subsets of $X$ are very similar to open subsets in some topology on $X$, except that
they are known to be closed only under some infinite unions.

Then let us try again.
Suppose that we have a Leibniz--Euler model of $T_{BHK}$ in subsets of a set $X$ with some domain $\D$.
Let us call a subset of $X$ {\it open} if it a union of subsets of the form $|\wn\Phi|_\iass$, where 
$\Phi$ is a formula of intuitionistic logic, and $\iass$ is an assignment in $\D$.
Thus the union of any family of open sets is open.
Also, similarly to the above, the intersection of two open sets is open, and $\emptyset$ and $X$ are open.
Thus we get a topology on $X$.

Another apparent deficiency here is that the axiom lists (3) and (6), which provide a minimal formal extract 
from the intrinsically vague BHK clauses for $\to$ and $\forall$, seem to be too weak to yield any 
interesting conclusions as they provide only one-sided constraints on the intuitinistic $\to$ and $\forall$.

One way to strengthen (3) and (6) to something bidirectional is as follows:
\begin{itemize}
\item[(3$'$)] $\big(\wn\Xi\to\wn(\Phi\to\Psi)\big)\Iff \big(\wn\Xi\to(\wn\Phi\to\wn\Psi)\big)$;
\item[(6$'$)] $\big(\wn\Xi\to\wn\forall \tr x\, \Phi(\tr x)\big)\Iff 
\big(\wn\Xi\to\forall \tr x\,\wn\Phi(\tr x)\big)$.
\end{itemize}
Indeed, (3$'$) implies (3) by considering $\Xi=\Phi\to\Psi$, and (6$'$) implies (6) by considering
$\Xi=\forall \tr x\,\Phi(\tr x)$.
Conversely, (3) implies the $\imp$ meta-implication in (3$'$), and (6) implies the $\imp$ 
meta-implication in (6$'$), by using the modus ponens rule of classical logic.
As for the $\when$ meta-implications, they specialize by considering $\Xi=\triv$ and using (7) to
the following partial converses to (3) and (6):
\begin{itemize}
\item[(3$''$)] $\wn\Phi\to\wn\Psi\Imp\wn(\Phi\to\Psi)$;
\item[(6$''$)] $\forall\tr x\,\wn\Phi(\tr x)\Imp\wn\forall\tr x\,\Phi(\tr x)$.
\end{itemize}

\begin{remark}
Let us emphasize that there is nothing in the BHK interpretation that would suggest strengthening
(3) and (6) in this particular way --- although {\it some} strengthening of (3) and (6) is certainly
present in the BHK interpretation.
In the QHC calculus of \cite{M1}, which pretends to capture some formal aspects of the BHK interpretation
along the lines of the present section, the $\when$ meta-implications in (3) and (6) correspond to the rules
$\fm{\wn\gamma\to(\wn\alpha\to\wn\beta)\,/\,\wn\gamma\to\wn(\alpha\to\beta)}$ and
$\fm{\wn\gamma\to\forall\tr x\,\wn\alpha(x)\,/\,\wn\gamma\to\wn\forall\tr x\,\alpha(x)}$, which are
shown to be non-derivable, and even non-admissible in QHC \cite{M2}.
\end{remark}

Let us denote the extension of $T_{BHK}$ by the meta-axiom lists (3$'$) and (6$'$) by $T'_{BHK}$.
(Here (3) and (6) could of course be dropped from $T'_{BHK}$.) 
Then (3$''$) and (6$''$) yield:
\begin{itemize}
\item[(3$'''$)] $T'_{BHK}\turnstile\wn\Phi\to\wn\Psi$ implies $T'_{BHK}\turnstile\wn(\Phi\to\Psi)$;
\item[(6$'''$)] $T'_{BHK}\turnstile\forall\tr x\,\wn\Phi(\tr x)$ implies 
$T'_{BHK}\turnstile\wn\forall\tr x\,\Phi(\tr x)$.
\end{itemize}

Now let us assume that our Leibniz--Euler model of $T_{BHK}$ happens to satisfy also the meta-axioms 
of $T'_{BHK}$. (One way to achieve this is to consider the theory $T''_{BHK}$ consisting of all
formulas that are consequences of $T'_{BHK}$, and consider a Leibniz--Euler model of $T''_{BHK}$
along with a valuation with respect to which $T''_{BHK}$ is complete, cf.\ \cite{RS}*{VIII.3.3, II.8.1}.
This model with valuation will necessarily satisfy all admissible, and not only derivable, rules 
of $T'_{BHK}$.)
Then (3$'$) guarantees that $|\wn(\Phi\to\Psi)|_\iass$ is the largest open set contained in 
$|\wn\Phi\to\wn\Psi|_\iass$, and (6$'$) guarantees that $|\wn\forall \tr x\,\Phi(\tr x)|_\iass$ 
is the largest open set contained in $|\forall \tr x\,\wn\Phi(\tr x)|_\iass$.
Symbolically,
$|\wn(\Phi\to\Psi)|_\iass=\Int\big((X\but|\wn\Phi|_\iass)\cup|\wn\Psi|_\iass\big)$ and 
$|\wn\forall \tr x\, \Phi(\tr x)|_\iass=\Int\bigcap_{d\in D} |\wn\Phi(d)|_\iass$, 
where $\Int$ denotes topological interior.

Thus we can also go the other way round: given any topological space $X$ and a domain $\D$, we can 
interpret every predicate constant $\wn\alpha$, where $\alpha$ is a problem variable, by an arbitrary
open subset of $X$ (for any given assignment).
Then we can interpret all our predicate constants $\wn\Phi$ using inductively the above interpretations of 
$\wn(\Phi\land\Psi)$, $\wn(\Phi\lor\Psi)$, $\wn(\Phi\to\Psi)$, $\wn\ab$, $\wn\exists\tr x\,\Xi(\tr x)$, 
$\wn\forall\tr x\,\Xi(\tr x)$, $\wn\triv$ in terms of given interpretations of $\wn\Phi$, $\wn\Psi$, 
$\wn\Xi(\tr x)$.
(In particular, each $|\wn\Phi|_\iass$ will be an open subset of $X$.) 
This yields a model of $T'_{BHK}$.

Now, how does $T'_{BHK}$ fit with the official laws and inference rules of intuitionistic logic?
(3) implies $T'_{BHK}\turnstile\wn\Phi\,\mand\,\wn(\Phi\to\Psi)\imp\wn\Psi$, and from (3$'''$) we immediately 
obtain $T'_{BHK}\turnstile\wn(\Phi\to\Phi)$ and $T'_{BHK}\turnstile\wn(\ab\to\Phi)$.
If, in addition to (3$'''$), we take into account (1) and (2), we also get 
$T'_{BHK}\turnstile\wn(\Phi_1\land\Phi_2\to\Phi_i)$ and
$T'_{BHK}\turnstile\wn(\Phi_i\to\Phi_1\lor\Phi_2)$ for each $i=1,2$.
Using (5) and (6), we similarly obtain 
$T'_{BHK}\turnstile\wn\big(\Phi(\tr t)\to\exists\tr x\,\Phi(\tr x)\big)$ and
$T'_{BHK}\turnstile\wn\big(\forall\tr x\,\Phi(\tr x)\to\Phi(\tr t)\big)$.

To verify that 
$T'_{BHK}\turnstile\wn\big(\Xi\to\Phi(\tr x)\big)\imp
\wn\big(\Xi\to\forall\tr x\,\Phi(\tr x)\big)$,
by (3$''$) and (6$'$) it suffices to check that 
$T'_{BHK}\turnstile\wn\big(\Xi\to\Phi(\tr x)\big)\to\big(\wn\Xi\to\forall\tr x\,\wn\Phi(\tr x)\big)$,
which indeed follows from (3).
This takes care of \ref{rule:forall}, and similarly one can deal with \ref{rule:exists}, 
\ref{rule:conjunction} and \ref{rule:disjunction}.
The same approach also works for \ref{rule:composition} and \ref{rule:exponential}.

Thus if we regard each open set $|\wn\Phi|_\iass$ as an interpretation of the intuitionistic formula $\Phi$, 
and not just of the predicate constant $\wn\Phi$, we get a model of intuitionistic logic.

\subsection{Tarski models}\label{Tarski}

Let us summarize in closed terms the models of intuitionistic logic that we have just obtained.
In the zero-order case, they were discovered independently by Stone \cite{Sto}, Tang \cite{Tang}
(see also \S\ref{provability}) and Tarski \cite{Ta2} in the 1930s, with Tarski having also established
completeness of zero-order intuitionistic logic with respect to this class of models.
The models were extended to the first-order case by Mostowski, with completeness established by Rasiowa
and Sikorski (see \S\ref{Tarski completeness}).

We will consider only models of intuitionistic logic, and not of empty theories over it.
Thus there will be no problem constants and function symbols in the language.

We fix an arbitrary topological space $X$ and a domain (i.e., an arbitrary set) $\D$.
Then the set $\O$ consists of all open subsets of $X$, and the function $\ocf:\O\to\{\Top,\Bot\}$ sends 
the entire space $X$ to $\Top$, and each proper open subset of $X$ to $\Bot$.

Thus, to specify a valuation, each problem variable $\alpha$ of arity $n$ must be interpreted by 
a $\D^n$-indexed family of open subsets of $X$ (i.e., by a function $|\alpha|\:\D^n\to\O$).
Upon a choice of a valuation $\pval$, a closed formula will be interpreted by an open subset of $X$, 
and more generally a closed $n$-formula $\Phi$ will be interpreted by a $\D^n$-indexed family of open subsets 
of $X$, to be denoted $|\Phi|$ or in more detail $|\Phi|^{\pval}$, or in still more detail $|\Phi|^\pval_X$.
(The valuation will be suppressed in the notation when it is clear from context.)
Upon a further choice of a variable assignment $\iass$, an arbitrary formula will be interpreted by 
an open subset of $X$, and more generally an arbitrary $n$-formula $\Phi$ will be interpreted by 
a $\D^n$-indexed family of open subsets of $X$, to be denoted $|\Phi|$ or in more detail 
$|\Phi|_\iass$ or in still more detail $|\Phi|^\pval_\iass$.
(The assignment will be suppressed in the notation when it is clear from context.)
Validity of an $n$-formula $\Phi$ in the model, $\Turnstile\Phi$, means that
$|\Phi|_\iass^\pval(t_1,\dots,t_n)=X$ for each tuple $(t_1,\dots,t_n)\in\D^n$ and for an arbitrary 
valuation $\pval$ and assignment $\iass$.

Finally, intuitionistic connectives and quantifiers are interpreted as follows:
\begin{itemize}
\item $|\Phi\lor\Psi|=|\Phi|\cup|\Psi|$;
\item $|\Phi\land\Psi|=|\Phi|\cap|\Psi|$;
\item $|\Phi\to\Psi|=\Int\big((X\but |\Phi|)\cup |\Psi|\big)$;
\item $|\ab|=\emptyset$;
\item $|\exists \tr x\,\Xi(\tr x)|=\bigcup_{d\in D}|\Xi|(d)$;
\item $|\forall \tr x\,\Xi(\tr x)|=\Int\bigcap_{d\in D}|\Xi|(d)$,
\end{itemize}
where $\Phi$ and $\Psi$ are formulas and $\Xi$ is a 1-formula, and $\iass$ and $\pval$ are fixed.
Let us note that $|\Phi\to\Psi|$ is the union of all open sets $U$ such that $U\cap |\Phi|\subset U\cap |\Psi|$.

In particular, we have $|\neg\Phi|=\Int(X\but |\Phi|)=X\but\Cl |\Phi|$ and $|\neg\neg\Phi|=\Int(\Cl|\Phi|)$.
Thus {\it decidable} formulas (i.e.\ $\Phi$ such that $\turnstile\Phi\lor\neg\Phi$) are represented by clopen
(=closed and open) sets; and {\it stable} formulas (i.e.\ $\Phi$ such that $\turnstile\neg\neg\Phi\to\Phi$) 
are represented by regular open sets (i.e.\ sets equal to the interior of their closure).
Stable formulas coincide (by (\ref{triple negation}) and (\ref{double negation}) in \S\ref{tautologies}) 
with those that are equivalent to a negated formula.
Decidable formulas are stable by (\ref{decidable-stable}).
On the other hand, if $\Phi\lor\neg\Phi$ is stable, then $\Phi$ is decidable by (\ref{not-not-LEM}).

Tarski models yield simple and intuitive proofs that many classically valid principles are not derivable 
in intuitionistic logic.
Here are a few examples.

\subsubsection{Constant Domain Principle} \label{constant domain}

The converse to (\ref{q5}), known as the {\it Constant Domain Principle}
(for reasons related to Kripke models):
\FM{\prin\forall \tr x\,\big(\alpha\lor\pi(\tr x)\big)\To\alpha\lor\forall \tr x\,\pi(\tr x)}
is not derivable in intuitionistic logic, as witnessed by the Tarski model with $X=\R$ and $\D=\N$ and 
a valuation with $|\fm\pi|(n)=(-\frac1n,\frac1n)$ and $|\fm\alpha|=\R\but\{0\}$.
A simple modification of this model and valuation, which we purposely leave to
the reader, shows non-derivability of the following Negative Constant Domain Principle:
\FM{\prin\forall \tr x\,\big(\neg\alpha\lor\pi(\tr x)\big)\To\neg\alpha\lor\forall \tr x\,\pi(\tr x).}

In fact, the contrapositive of the Constant Domain Principle is derivable in intuitionistic logic:
\FM{\turnstile\neg\forall \tr x\,\big(\alpha\lor\pi(\tr x)\big)\Tofrom
\neg\big(\alpha\lor\forall \tr x\,\pi(\tr x)\big).}
Indeed, the right hand side is equivalent, by (\ref{implication0''}) and (\ref{q3}),
to $\fm{\neg\forall \tr x\,\big(\neg\alpha\to\pi(\tr x)\big)}$, and the ``$\from$'' implication follows from
(\ref{implication0'}).
(The ``$\to$'' implication is the contrapositive of (\ref{q5}).)

However, the following two-variable generalization of the Constant Domain Principle:
\FM{\prin\forall \tr y\,\forall \tr x\,\big(\rho(\tr y)\lor\pi(\tr x,\tr y)\big)\To
\forall \tr y\,\big(\rho(\tr y)\lor\forall \tr x\,\pi(\tr x,\tr y)\big)}
does not turn into an intuitionistically derivable principle upon taking the contrapositive:
\FM{\prin\neg\forall \tr y\,\big(\rho(\tr y)\lor\forall \tr x\,\pi(\tr x,\tr y)\big)\To
\neg\forall \tr y\,\forall \tr x\,\big(\rho(\tr y)\lor\pi(\tr x,\tr y)\big).}
Indeed, the latter principle is not valid in the Tarski model with $X=\D=\R$, by considering a valuation with
$|\fm\pi|(x,y)=(y-e^x,y+e^x)$ and $|\fm\rho|(y)=\R\but\{y\}$.
We will call the latter principle (i.e., the contrapositive) the {\it Parametric Distributivity Principle}.
It may be regarded as ``dual'' to the Constant Domain Principle in that they have a common generalization
but no common special cases.

Let us note that the Parametric Distributivity Principle can be equivalently reformulated as
\FM{\prin\neg\neg\Big(\forall \tr y\,\forall \tr x\,\big(\rho(\tr y)\lor\pi(\tr x,\tr y)\big)\To
\forall \tr y\,\big(\rho(\tr y)\lor\forall \tr x\,\pi(\tr x,\tr y)\big)\Big)}
using the intuitionistically derivable principles (\ref{double contrapositive}) and (\ref{neg-neg-imp}).
Hence it is a special case of a variant of Kleene's Principle:
\FM{\prin\neg\neg\forall \tr y\,\Big(\forall \tr x\,\big(\rho(\tr y)\lor\pi(\tr x,\tr y)\big)\To
\big(\rho(\tr y)\lor\forall \tr x\,\pi(\tr x,\tr y)\big)\Big)}
Kleene's principle \cite{Kl}*{\S80, Theorem 58(b)(iii)} has $\fm\pi(\tr x)$ instead of $\fm\pi(\tr x,\tr y)$.

\subsubsection{Independence of Premise} \label{Harrop1}

The converse to (\ref{q2}) is known as the principle of {\it Independence of Premise}:
\FM{\prin\big(\delta\to\exists \tr x\,\pi(\tr x)\big)\To\exists \tr x\,\big(\delta\to\pi(\tr x)\big).}
Its contrapositive
\FM{\prin\neg\big(\delta\to\exists \tr x\,\pi(\tr x)\big)\Tofrom
\neg\exists \tr x\,\big(\delta\to\pi(\tr x)\big)}
is derivable in intuitionistic logic, which follows from (\ref{implication1}) and (\ref{quantifiers1}).
By contrast, even if $\fm\delta$ is specialized to $\neg\fm\alpha$, and $\exists$ is ``specialized''
to $\lor$ in the principle of Independence of Premise, the resulting formula, known as
the {\it Kreisel--Putnam Principle}:
\FM{\prin(\neg\alpha\to\beta\lor\gamma)\To (\neg\alpha\to\beta)\lor(\neg\alpha\to\gamma)}
is not derivable in intuitionistic logic.
Indeed, let $X=\R^2$,
\begin{gather*}
|\fm\beta|=\{(x,y)\mid x>0\},\\
|\fm\gamma|=\{(x,y)\mid y>0\}
\end{gather*}
\smallskip
and $|\fm\alpha|=|\neg(\fm\beta\lor\fm\gamma)|$,
so that $|\neg\fm\alpha|=\{(x,y)\mid x>0\lor y>0\}$.
Then $|\neg\fm\alpha\to\fm\beta\lor\fm\gamma|=\R^2$, whereas
\begin{gather*}
|\neg\fm\alpha\to\fm\beta|=\{(x,y)\mid x>0\lor y<0\};\\
|\neg\fm\alpha\to\fm\gamma|=\{(x,y)\mid x<0\lor y>0\}.
\end{gather*}
\smallskip
Hence $|\neg\fm\alpha\to\fm\beta|\cup
|\neg\fm\alpha\to\fm\gamma|=\R^2\but\{(0,0)\}$.

This Tarski model shows also that {\it Harrop's Rule}:
\FM{\frac{\neg\alpha\to\beta\lor\gamma}
{(\neg\alpha\to\beta)\lor(\neg\alpha\to\gamma)}}
is not derivable in intuitionistic logic.
(Alternatively, if it were derivable, then the Kreisel--Putnam principle would also be derivable
by the deduction theorem.)
However, it is admissible, see Theorem \ref{Harrop-admissible} below (hence, in particular, this rule
does not imply the Kreisel--Putnam principle).
The stronger rule:
\FM{\frac{\delta\to\beta\lor\gamma}
{(\delta\to\beta)\lor(\delta\to\gamma)}}
is not even admissible, by considering its special case with $\fm\delta$ substituted by $\fm{\beta\lor\gamma}$.
In this case the meta-quantified premiss is derivable in intuitionistic logic, but the
meta-quantified conclusion, $\prin\fm{(\beta\lor\gamma\to\beta)\lor(\beta\lor\gamma\to\gamma)}$,
is clearly equivalent to the following {\it G\"odel--Dummett principle}:
\FM{\prin(\beta\to\gamma)\lor(\gamma\to\beta),}
which is certainly not derivable in intuitionistic logic (for instance, the previous model in $\R^2$ works to 
show this).
Let us note that, in fact, our specialization of $\fm\delta$ to $\fm{\beta\lor\gamma}$ led to no loss of generality,
since $\fm{\prin(\beta\lor\gamma\to\beta)\lor(\beta\lor\gamma\to\gamma)}$ clearly implies (using the exponential 
law) the quantifier-free independence of premise,
\FM{\prin(\delta\to\beta\lor\gamma)\to(\delta\to\beta)\lor(\delta\to\gamma).}
Thus the latter principle as well as its rule version considered above are each equivalent to the
G\"odel--Dummett principle.
These equivalent principles are strictly stronger than the Kreisel--Putnam principle (see Example
\ref{KP vs GD} below).

A similar argument shows that the following stable Harrop's rule:
\FM{\frac{\neg\neg(\beta\lor\gamma)\to(\beta\lor\gamma),\ \neg\alpha\to\beta\lor\gamma}
{(\neg\alpha\to\beta)\lor(\neg\alpha\to\gamma)}}
is equivalent to the following stable G\"odel--Dummett principle:
\FM{\frac{\neg\neg(\beta\lor\gamma)\to(\beta\lor\gamma)}{(\beta\to\gamma)\lor(\gamma\to\beta)}}
and to the following stable Kreisel--Putnam principle:
\FM{\frac{\neg\neg(\beta\lor\gamma)\to(\beta\lor\gamma)}
{(\neg\alpha\to\beta\lor\gamma)\To (\neg\alpha\to\beta)\lor(\neg\alpha\to\gamma)};}
and, moreover, that $\neg\fm\alpha$ can be replaced by $\fm\delta$ in these stable versions.
In fact, these all are equivalent forms of the original Harrop rule.
Indeed, assuming $\fm{\neg\alpha\to\beta\lor\gamma}$ and writing
$\Phi=\fm{\neg\alpha\land\beta}$ and $\Psi=\fm{\neg\alpha\land\gamma}$, we have
$\fm{\neg\alpha\tofrom\Phi\lor\Psi}$.
Then also $\fm{\neg\neg(\Phi\lor\Psi)\to(\beta\lor\gamma)}$, and hence the stable Harrop rule implies
$\fm{(\neg\alpha\to\Phi)\lor(\neg\alpha\to\Psi)}$, which in turn implies
$\fm{(\neg\alpha\to\beta)\lor(\neg\alpha\to\gamma)}$.

By another argument of the same type, the rule
\FM{\frac{\beta\land\gamma\to\delta}
{(\beta\to\delta)\lor(\gamma\to\delta)}}
implies the G\"odel--Dummett principle $\fm{\prin(\beta\to\gamma)\lor(\gamma\to\beta)}$, which in turn
implies Skolem's principle (see \cite{vP}), the converse of (\ref{dm2}):
\FM{\prin(\beta\land\gamma\to\delta)\to(\beta\to\delta)\lor(\gamma\to\delta);}
thus the three are equivalent.
Let us also note that by ``generalizing'' $\land$ and $\lor$ to $\forall$ and $\exists$,
Skolem's form of the G\"odel--Dummett principle transforms into the converse of (\ref{gq2}):
\FM{\prin(\forall \tr x\, \gamma(\tr x)\to\delta)\To\exists \tr x\,(\gamma(\tr x)\to\delta),}
which is therefore also an independent principle (by considering a two-element domain).

\begin{remark}
We will see in \S\ref{Jankov} that the principle
\FM{\prin(\beta\land\gamma\to\neg\alpha)\to(\beta\to\neg\alpha)\lor(\gamma\to\neg\alpha)}
is strictly stronger than the Kreisel--Putnam principle.
This might seem surprising, since the two principles are clearly equivalent to the rules
\FM{\frac{\neg\neg\delta\to\delta}{(\beta\land\gamma\to\delta)\to(\beta\to\delta)\lor(\gamma\to\delta)}
\qquad\text{and}\qquad
\frac{\neg\neg\delta\to\delta}{(\delta\to\beta\lor\gamma)\to(\delta\to\beta)\lor(\delta\to\gamma)}}
respectively --- even though these rules have the same premiss and equivalent meta-closures of the conclusions.
As we have seen, it is enough to meta-quantify over $\fm\delta$:
\FM{\turnstile\mq{\delta}\,(\beta\land\gamma\to\delta)\to(\beta\to\delta)\lor(\gamma\to\delta)
\Iff\mq{\delta}\,(\delta\to\beta\lor\gamma)\to(\delta\to\beta)\lor(\delta\to\gamma).}
\end{remark}

\subsection{Topological completeness}\label{Tarski completeness}

Rasiowa and Sikorski proved that every consistent first-order theory $\Th$ over intuitionistic logic
is complete with respect to the Tarski model in a certain subspace $X=X_\Th$ of
the Baire space%
\footnote{That is, $\N^\N$ with the product topology. This space is well-known to be homeomorphic
to the set of irrational reals.} 
with $\D=\N$ \cite{RS}*{X.3.2}.%
\footnote{Thus if $\Psi$ is not a theorem of $\Th$, then $\Psi$ is not valid in $X_\Th$.
It is additionally shown in \cite{RS}*{X.3.4 and remark after X.3.5} that if $\neg\Phi$ is not a theorem of 
$\Th$, then $\Phi$ is satisfiable in the Tarski model in a certain subspace $X=X'_\Th$ of 
the Baire space with $\D=\N$.}
In particular, intuitionistic logic is complete with respect to one such Tarski model (in fact, even with 
respect to a fixed valuation), and strongly complete with respect to the class of all such Tarski models 
(see \ref{strong completeness}).
The latter result was improved by Dragalin, who showed that intuitionisitic logic is strongly complete with
respect to the Tarski model with $\D=\N$ and with $X$ the Baire space itself \cite{Dr}%
\footnote{Dragalin's theorem \cite{Dr}*{3.5.2} says that every composite theory (see Remark \ref{Dragalin})
has a Tarski model with $\D=\N$ and with $X$ the Baire space; this trivially implies the assertion 
(see Remark \ref{Dragalin}).} 
(and, in fact, complete even with respect to a fixed valuation in such a model).
It is easy to see that intuitionistic logic is not strongly complete with respect to any Tarski model 
with connected $X$, including Euclidean spaces (see \S\ref{omniscience}).

Originally, Tarski \cite{Ta2} (see also \cite{McT3}; for a sketch, see \cite{SU}*{2.4.1}, and for
simplified proofs see \cite{Sl}, \cite{Krem1}, \cite{Krem2} and references there) proved that if $X$
is a separable metrizable space with no isolated points (for instance, $\R^n$ for any $n>0$), then
zero-order intuitionistic logic is complete with respect to the Tarski model in $X$.
For instance, the Kreisel--Putnam Principle (and, in fact, even Harrop's Rule)
is also disproved by the Tarski model with $X=\R$, by considering a valuation with $|\neg\fm\alpha|=(0,1)$,
$|\fm\beta|=\bigcup_{i=1}^\infty (\frac1{2i+2},\frac1{2i})$ and
$|\fm\gamma|=\bigcup_{i=1}^\infty (\frac1{2i+1},\frac1{2i-1})$.
In general, it should be noted that open subsets of
$\R$ can be quite complicated (see \cite{N}*{\S III.6.9}).

Tarski's result was only recently extended to first-order intuitionistic logic by Kremer \cite{Krem4},
who showed it to be complete with respect to the Tarski model with $\D=\N$ in any {\it zero-dimensional}
separable metrizable space $X$ with no isolated points; and strongly complete with respect to Tarski models 
with $\D=\N$ in open subspaces of $X$.%
\footnote{Kremer's theorem says, in particular, that if $\Gamma$ and $\Delta$ are closed formulas such that
$\Gamma\to\Delta$ is not a theorem of intuitionistic logic, then
$|\Gamma|^\pval\but|\Delta|^\pval\ne\emptyset$ for some valuation $\pval$ in the Tarski model in $X$ with $\D=\N$.
If $\Psi$ is not a theorem of intuitionistic logic, then its universal closure $\Delta$ is also not a theorem, 
and hence $|\triv|^\pval\but|\Delta|^\pval\ne\emptyset$ for some $\pval$, i.e., $|\Delta|^\pval\ne X$, and 
so $\Psi$ is not valid in $X$.
If a rule $\Phi_1,\dots,\Phi_n\,/\Psi$ is not derivable in intuitionistic logic, then $\Gamma\to\Delta$
is not a theorem of intuitionistic logic, where $\Gamma$ is the universal closure of 
$\Phi_1\land\dots\land\Phi_n$ and $\Delta$ is the universal closure of $\Psi$.
Hence $|\Gamma|^\pval\but|\Delta|^\pval\ne\emptyset$ for some valuation $\pval$ in the Tarski model in $X$
with $\D=\N$.
If we let $Y=|\Gamma|^\pval$ and $|\Xi|_\iass^{\pval_Y}=|\Xi|_\iass^{\pval}\cap Y$ for every formula $\Xi$ 
and every assignment $\iass$ in $\D=\N$, then it is easily checked by induction that $\pval_Y$ is 
a valuation in the Tarski model in $Y$ with $\D=\N$; clearly, $|\Gamma|^{\pval_Y}=Y$ and 
$|\Delta|^{\pval_Y}\ne Y$.
Hence $\Phi_1,\dots,\Phi_n$ are valid in $(Y,\pval_Y)$, and $\Psi$ is not valid in $(Y,\pval_Y)$.
Thus $\Phi_1,\dots,\Phi_n\,/\Psi$ is not valid in $Y$.}
These include the Cantor set, the Baire space, and the set of rational numbers.
On the other hand, intuitionistic logic is not complete with respect to the Tarski model with any $\D$ in any 
locally connected space, including open subsets of Euclidean spaces \cite{Krem4}.

\subsubsection{Disjunction property}

Interestingly, if $\Phi$ is a closed formula such that $|\Phi\lor\neg\Phi|^\pval=X$ for some valuation $\pval$
for the Tarski model in a connected space $X$, then either $|\Phi|^\pval=X$ or $|\neg\Phi|^\pval=X$
(since the only clopen subsets of $X$ are $\emptyset$ and $X$).
In other words, for closed formulas $\Phi$ the judgement $\Phi\lor\neg\Phi\Turnstile\Phi\parallel\neg\Phi$ 
is valid in Tarski models with connected $X$ (for instance, for $X=\R$).
It is clearly not valid for any disconnected $X$; in particular, the judgement 
$\Phi\lor\neg\Phi\turnstile\Phi\parallel\neg\Phi$ does not hold in zero-order intuitionistic logic.

However, the corresponding external judgement holds: in zero-order intuitionistic logic,
if $\turnstile\Phi\lor\neg\Phi$, then either $\turnstile\Phi$ or $\turnstile\neg\Phi$.
This is a special case of the following remarkable fact, originally announced by G\"odel
and proved by Gentzen (see also \cite{Kl}*{\S80, Theorem 57(a)},
\cite{RS}*{X.8.1}, \cite{McT3}*{Theorem 4.4}):

\begin{theorem}[Disjunction Property]\label{disjunction property}
In intuitionistic logic, if $\turnstile\Phi\lor\Psi$ then either
$\turnstile\Phi$ or $\turnstile\Psi$, for any closed formulas $\Phi$ and $\Psi$.
\end{theorem}

It is easy to deduce this from the completeness of Tarski models.

The proof employs the following construction, which will also be used in sequel: the Alexandroff
dual cone $C^*W$ over the space $W$ is the union $W\cup\{\hat 0\}$ (as a set) endowed with
the topology whose open sets are as follows: (i) open subsets of $W$; (ii) $W\cup\{\hat 0\}$.

\begin{proof} Suppose that $\turnstile\Phi\lor\Psi$, but $\not\turnstile\Phi$ and $\not\turnstile\Psi$.
Then there are valuations $\pval_X$, $\pval_Y$ for the Tarski models in topological spaces $X$, $Y$ with
domains $\D_X$ and $\D_Y$ such that $|\Phi|^{\pval_X}\ne X$ and $|\Psi|^{\pval_Y}\ne Y$.
Let $\D$ be a set that admits surjections $d_X$ onto $\D_X$ and $d_Y$ onto $\D_Y$. 
Let us define valuations $\pval_X'$, $\pval_Y'$ for the Tarski models in $X$ and $Y$ with domain $\D$ by 
$|\alpha|_\iass^{\pval_X'}=|\alpha|_{d_X\circ\iass}^{\pval_X}$ and
$|\alpha|_\iass^{\pval_Y'}=|\alpha|_{d_Y\circ\iass}^{\pval_Y}$ for every problem variable $\alpha$
and every assignment $\iass$ in $\D$.
Then it is easy to see that $|\Xi|^{\pval_X'}=|\Xi|^{\pval_X}$ and $|\Xi|^{\pval_Y'}=|\Xi|^{\pval_Y}$
for every closed formula $\Xi$. 
In particular, $|\Phi|^{\pval_X'}\ne X$ and $|\Psi|^{\pval_Y'}\ne Y$.

Let $Z:=C^*(X\sqcup Y)$ be the Alexandroff dual cone over the disjoint union $X\sqcup Y$.
Let us define a valuation $\pval_Z$ for the Tarski model in $Z$ with domain $\D$ by 
$|\alpha|_\iass^{\pval_Z}=|\alpha|_\iass^{\pval_X'}\cup|\alpha|_\iass^{\pval_Y'}$ for every problem variable 
$\alpha$ and every assignment $\iass$ in $\D$.
Then it is easy to check by induction (using that $X$ and $Y$ are open in $Z$) that
$|\Xi|_\iass^{\pval_X}=|\Xi|_\iass^{\pval_Z}\cap X$ and $|\Xi|_\iass^{\pval_Y}=|\Xi|_\iass^{\pval_Z}\cap Y$ 
for any formula $\Xi$ and any assignment $\iass$ in $\D$.
In particular, $|\Phi|^{\pval_Z}\not\supset X$ and $|\Psi|^{\pval_Z}\not\supset Y$.

By our hypothesis, $|\Phi\lor\Psi|^{\pval_Z}=Z$.
In particular, either $\hat 0\in|\Phi|^{\pval_Z}$ or $\hat 0\in|\Psi|^{\pval_Z}$.
However, the minimal open set containing $\hat 0$ is the entire space $Z$.
Hence either $|\Phi|^{\pval_Z}=Z$ or $|\Psi|^{\pval_Z}=Z$, which is a contradiction.
\end{proof}

\subsubsection{Harrop's rule}
The Disjunction Property can be regarded as a multiple-conclusion admissible rule in zero-order
intuitionistic logic.
A rather similar argument establishes the admissibility of one genuine (single-conclusion) rule
in first-order intuitionisitc logic:

\begin{theorem} \label{Harrop-admissible} Harrop's rule 
\FM{\frac{\neg\gamma\to\phi\lor\psi}
{(\neg\gamma\to\phi)\lor(\neg\gamma\to\psi)}}
is an admissible rule for intuitionistic logic.
\end{theorem}

See \cite{Ie2} for a direct syntactic proof not appealing to the completeness theorem.

It should be mentioned that Harrop's rule is, in fact, stably admissible \cite{Pru}.

The admissibility of Harrop's rule implies the following generalization of the Disjunction Property:

\begin{corollary} In intuitionistic logic, if $\neg\Gamma\turnstile\Phi\lor\Psi$ then either 
$\neg\Gamma\turnstile\Phi$ or $\neg\Gamma\turnstile\Psi$, for any closed formulas $\Gamma$, $\Phi$ and $\Psi$.
\end{corollary}

\begin{proof} 
The hypothesis along with the Deduction Theorem imply $\turnstile\neg\Gamma\to\Phi\lor\Psi$.
Since Harrop's rule is admissible, we have $\turnstile(\neg\Gamma\to\Phi)\lor(\neg\Gamma\to\Psi)$.
By the Disjunction Property, it follows that either $\turnstile\neg\Gamma\to\Phi$ or
$\turnstile\neg\Gamma\to\Psi$.
This along with the {\it modus ponens} implies the assertion.
\end{proof}

\begin{proof}[Proof of Theorem \ref{Harrop-admissible}] 
We will prove the admissibility of the following rule, which as shown in \S\ref{Harrop1}
is equivalent to Harrop's rule:
\FM{\frac{\neg\neg(\gamma\lor\delta)\to(\gamma\lor\delta)}{(\gamma\to\delta)\lor(\delta\to\gamma)}.}
Suppose that $\turnstile\neg\neg(\Gamma\lor\Delta)\to(\Gamma\lor\Delta)$ and
$\not\turnstile (\Gamma\to\Delta)\lor(\Delta\to\Gamma)$ for some formulas $\Gamma$ and $\Delta$.
Then by Tarski's completeness theorem, there exist a space $X$ and a domain $\D$ such that
$|(\Gamma\to\Delta)\lor(\Delta\to\Gamma)|_{\iass_0}^{\pval_X}\ne X$ for some assignment $\iass_0$ in $\D$
and some valuation $\pval_X$ in the Tarski model in $X$ with domain $\D$.
Then, in particular, $|\Gamma|_{\iass_0}^{\pval_X}\not\subset|\Delta|_{\iass_0}^{\pval_X}$ and 
$|\Delta|_{\iass_0}^{\pval_X}\not\subset|\Gamma|_{\iass_0}^{\pval_X}$.

Let us define a valuation $\pval_Y$ in the Tarski model in 
$Y:=|\Gamma\lor\Delta|_{\iass_0}^{\pval_X}$ with domain $\D$ by 
$|\alpha|_{\iass}^{\pval_Y}=|\alpha|_{\iass}^{\pval_X}\cap Y$ for every problem variable $\alpha$ and 
every assignment $\iass$ in $\D$.
Then $|\Xi|_{\iass}^{\pval_Y}=|\Xi|_{\iass}^{\pval_X}\cap Y$ for every formula $\Xi$ and every assignment $\iass$
(using that $Y$ is open in $X$).
In particular, $|\Gamma\lor\Delta|_{\iass_0}^{\pval_Y}=Y$, and neither of 
$|\Gamma|_{\iass_0}^{\pval_Y}$, $|\Delta|_{\iass_0}^{\pval_Y}$ lies in the other one.

Let $Z$ be the Alexandroff dual cone $C^*Y=Y\cup\{\hat 0\}$.
Let us define a valuation $\pval_Z$ in the Tarski model in $Z$ with domain $\D$ by
$|\alpha|_{\iass}^{\pval_Z}=|\alpha|_{\iass}^{\pval_Y}$ for every problem variable $\alpha$ and every assignment 
$\iass$ in $\D$.
Then  $|\Xi|_{\iass}^{\pval_Y}=|\Xi|_{\iass}^{\pval_Z}\cap Z$ for every formula $\Xi$ and every assignment $\iass$ 
(using that $Y$ is open in $Z$).
In particular, neither of $|\Gamma|_{\iass_0}^{\pval_Z}$, $|\Delta|_{\iass_0}^{\pval_Z}$ lies in the other one.
Also, $|\Gamma\lor\Delta|_{\iass_0}^{\pval_Z}$ contains $Y$, and consequently
$|\neg\neg(\Gamma\lor\Delta)|_{\iass_0}^{\pval_Z}=Z$ (since $\Cl Y=Z$).
But then our hypothesis $\turnstile\neg\neg(\Gamma\lor\Delta)\to(\Gamma\lor\Delta)$ implies
that $|\Gamma\lor\Delta|_{\iass_0}^{\pval_Z}=Z$.

In particular, either $\hat 0\in|\Gamma|_{\iass_0}^{\pval_Z}$ or $\hat 0\in|\Delta|_{\iass_0}^{\pval_Z}$.
However, the minimal open set containing $\hat 0$ is the entire space $Z$.
Thus either $|\Gamma|_{\iass_0}^{\pval_Z}=Z$ or $|\Delta|_{\iass_0}^{\pval_Z}=Z$, which is a contradiction.
\end{proof}

\subsubsection{Alexandroff spaces}

McKinsey and Tarski also proved that zero-order intuitionistic logic is complete with respect to
the class of Tarski models with finite $X$.
Finite topological spaces are included in the class of {\it Alexandroff spaces}, where
the intersection of any (possibly infinite) family of open subsets of $X$ is open.
The relation $x\in\Cl\,\{y\}$ on the points $x$, $y$ of an Alexandroff
space is a preorder, i.e.\ it is reflexive and transitive.
Conversely, any preordered set $P$ can be endowed with its Alexandroff topology, where a subset of $P$
is defined to be closed if it is an order ideal.
This gives a one-to-one correspondence between Alexandroff spaces and preordered sets, under which $T_0$
spaces%
\footnote{That is, spaces where at least one out of any two distinct points is contained in an open set
not containing the other point.}
correspond precisely to posets.

In the zero-order case, Tarski models in Alexandroff spaces with fixed valuations are also known 
(in somewhat different terms) as Kripke models.
In general, Tarski models in Alexandroff spaces with fixed valuations are the same as 
``Kripke models with constant domain''.
(See \ref{Kripke} concerning the case of ``non-constant domain'').
First-order intuitionistic logic is not complete with respect to the class of Tarski models in 
Alexandroff spaces: clearly, the Constant Domain Principle holds in all such models.
In fact, this is the only obstacle to completeness: by a result of S. G\"orneman, the extension of
intuitionistic logic by the Constant Domain Principle is complete with respect to Tarski models
in Alexandroff spaces (see \cite{Gab}*{\S3.3, Corollary 8}).

Let us note that if the Alexandroff topology on a poset is $T_1$,%
\footnote{A topological space is called $T_1$ if for every pair of distinct points, each is contained in
an open set not containing the other one; or equivalently if all singleton subsets are closed.}
then it is discrete.
Let us note also that every Alexandroff space $X$ admits a continuous open surjection onto the $T_0$ Alexandroff
space obtained as the quotient of $X$ by the equivalence relation whose equivalence classes are
the minimal nonempty open sets of $X$.

If $K$ is a simplicial complex (or more generally a cell complex or a cone complex,
see \cite{M4}), viewed as a topological space (with the metric topology, see \cite{M5}), let $X(K)$ be
the face poset of $K$, viewed as a $T_0$ Alexandroff space.
By sending a point of $K$ to the minimal simplex (or cell or cone) that contains it, we obtain
a continuous open surjection $K\to X(K)$.
Conversely, given a $T_0$ Alexandroff space $X$, viewed as a poset, we have its order complex,
whose simplices can be arranged into cones of a cone complex $K(X)$ such that $X(K(X))=X$
(see \cite{M4}).
Thus we have a continuous open surjection $K(X)\to X$.
Thus a Tarski model in an Alexandroff space $X$ gives rise to a Tarski model in the polyhedron $K(X)$,
and all completeness results can be transferred accordingly.


\begin{example}\label{KP vs GD}
The G\"odel--Dummett principle $\fm{\prin(\beta\to\gamma)\lor(\gamma\to\beta)}$ is obviously valid in 
all Alexandroff spaces associated to totally ordered posets.
It is not valid in every Alexandroff space $X$ corresponding to a poset that
contains a pair of non-comparable elements $q,r$ with a lower bound $p$.
Indeed, the smallest open set $U_p$ containing $p$ also contains $q$ and $r$, and consequently $p$ lies
neither in $\Int\big(U_q\cup(X\but U_r)\big)$ (which does not contain $r$) nor in 
$\Int\big(U_r\cup(X\but U_q)\big)$ (which does not contain $q$).

Meanwhile, the Kreisel--Putnam principle
$\fm{\prin(\neg\alpha\to\beta\lor\gamma)\to (\neg\alpha\to\beta)\lor(\neg\alpha\to\gamma)}$
is valid, for instance, in the three-element poset $P$ consisting of a pair of incomparable elements
$q,r$ and their lower bound $p$.
Indeed, $\neg\fm\alpha$ can only be interpreted by a regular open set, and the only regular open sets in $P$ 
are $\emptyset$, $P$, $\{q\}$ and $\{r\}$.
It is easy to see that the Kreisel--Putnam principle is satisfied under any valuation such that
$|\neg\fm\alpha|=|\ab|$ or $|\fm\neg\alpha|=|\triv|$ or
$|\fm\gamma|\subset|\fm\beta|$ or $|\fm\beta|\subset|\fm\gamma|$.
This leaves only the following options: $|\fm\beta|=\{q\}$ and $|\fm\gamma|=\{r\}$ (or vice versa),
with $\neg\fm\alpha$ interpreted by either $\{q\}$ or $\{r\}$.
In each of these cases either $|\neg\fm\alpha|=|\fm\beta|$ or $|\neg\fm\alpha|=|\fm\gamma|$,
and the Kreisel--Putnam principle is again satisfied.

Thus the Kreisel--Putnam principle does not imply the G\"odel--Dummett principle.
\end{example}

\begin{remark} In fact, the Kreisel--Putnam principle holds in every poset $P$ satisfying the
following {\it van Benthem condition} (cf.\ \cite{Hum}*{p.\ 886}):
$$\forall x\in P\ \forall y\,,z\ge x\ \exists u\ge x\
\big(u\le y,z\text{ and }\forall v\ge u\ \exists w\ge v\ (w\ge y\text{ or } w\ge z)\big).$$
Indeed, suppose that $P$ satisfies this condition, but the Tarski model in $P$ does not
satisfy the Kreisel--Putnam principle.
Then $$\Int(\Cl(A)\cup B\cup C)\not\subset\Int(\Cl(A)\cup B)\cup\Int(\Cl(A)\cup C)$$
for some open subsets $A$, $B$, $C$ of $P$.
So there exists an $x\in\Int(\Cl(A)\cup B\cup C)$ such that
$x\notin\Int(\Cl(A)\cup B)$ and $x\notin\Int(\Cl(A)\cup C)$.
Then the smallest open set $U_x$ containing $x$ contains some $y\notin\Cl(A)\cup B$
and some $z\notin\Cl(A)\cup C$.
Thus $y,z\ge x$ and $y,z\notin\Cl A$; also, $y\notin B$ and $z\notin C$.
Let $u$ be given by the van Benthem condition.
Then $u\notin B$, $u\notin C$, so $u\notin B\cup C$.
Also, $u\in \Cl(A)\cup B\cup C$ since $u\ge x$.
Hence $u\in\Cl A$.
Then there exists a $v\ge u$ such that $v\in A$.
If $w$ is given by the van Benthem condition, then $w\in A$ and either $w\ge y$ or $w\ge z$.
Then either $y\in\Cl A$ or $z\in\Cl A$, which is a contradiction.
\end{remark}

\begin{remark}
Iemhoff \cite{Ie}*{3.14} characterized admissible rules of zero-order intuitionistic logic as
those rules that are valid in the Tarski model in every $T_0$ Alexandroff space whose
corresponding poset has a smallest element and a ``tight'' lower bound for every finite subset $Y$,
that is, a lower bound $z$ such that every $x>z$ satisfies $x\ge y$ for some $y\in Y$.
(It should be noted here that the admissibility of a rule $\Phi/\Psi$ means not just that 
$\turnstile\Phi$ implies $\turnstile\Psi$, but that $\turnstile\Phi'$ implies $\turnstile\Psi'$ for every 
special case $\Phi'/\Psi'$ of $\Psi/\Psi$; whereas the validity of $\Phi/\Psi$ in a class of models means 
just that $\Turnstile_\pval\Phi$ implies $\Turnstile_\pval\Psi$ for all valuations $\pval$ in all those models.
But, of course, the validity of $\Phi/\Psi$ in a class of models implies the validity of every special case 
$\Phi'/\Psi'$ of $\Psi/\Psi$ in the same class of models, due to the quantification over all valuations.
It is essential here that valuations are not included in our definition of a ``model''.)

Iemhoff also gave another characterization of admissible rules of zero-order intuitionistic logic.
Let us define several operations on Tarski models-with-valuations $(X,\pval)$: ($O_1$) ``restrict'' a given 
$(X,\pval)$ by reducing $X$ to the smallest open neighborhood of some point;
($O_2$) extend a given $(X,\pval)$ in any way over the Alexandroff dual cone $C^*X$;
($O_3$) combine a given collection of $(X_i,\pval_i)$ by taking the disjoint union of $X_i$.
(All these operations are used in the proofs of \ref{disjunction property}, \ref{Harrop-admissible} above.)
Then a rule is admissible for zero-order intuitionistic logic if and only if it is valid in every Tarski 
model-with-valuation obtained by applying $O_3$ to any family $F$ of Tarski models-with-valuation 
in finite $T_0$ spaces such that $F$ is in the image of $O_2$, and $F$ is closed under $O_1$ and $O_2\circ O_3$ 
\cite{Ie}*{3.15}.
\end{remark}

\subsection{Jankov's principle}\label{Jankov}

A striking omission in our list of intuitionistic validities in \S\ref{tautologies} is
the converse to one the de Morgan laws (\ref{deMorgan2}):
\FM{\prin\neg(\gamma\land\delta)\To\neg\gamma\lor\neg\delta\text.
\tag{JP}}

To see that this {\it Jankov's principle} is not derivable in intuitionistic logic, it suffices to substitute
a stable but not decidable formula $\Phi$ (e.g.\ $\Phi=\neg\fm\alpha$) for $\fm\gamma$ and 
$\neg\Phi$ for $\fm\delta$.
Jankov's principle (also known as ``de Morgan's principle'' or the ``weak principle of excluded middle'')
may have various appearances:

\formulas
\begin{proposition}\label{Jankov's logic} The following are equivalent to (JP):
\begin{enumerate}
\item $\prin\neg\neg(\alpha\lor\beta)\Tofrom\neg\neg\alpha\lor\neg\neg\beta$;
\item $\prin\neg\alpha\lor\neg\neg\alpha$;
\smallskip
\item $\dfrac{\neg\neg\beta\to\beta}{\beta\lor\neg\beta}$.
\end{enumerate}
\end{proposition}

Let us note that the admissibility of the rule (3) amounts to the assertion that stable formulas are decidable
(which of course does not hold in intuitionisitc logic).

\begin{proof} To obtain (1) from (JP), we substitute $\neg\alpha$ for $\gamma$ and $\neg\beta$ for $\delta$, 
and use the contrapositive of the other de Morgan tautology (\ref{deMorgan1}),
$\turnstile\neg(\alpha\lor\beta)\tofrom\neg\alpha\land\neg\beta$.

To obtain (2) from (1), we apply the latter with $\neg\alpha$ susbstituted for $\beta$, and use 
(\ref{triple negation}) and (\ref{not-not-LEM}).

Next, from (\ref{implication-exp}) and (\ref{neg-neg-drop}) we get
$\turnstile\neg(\gamma\land\delta)\To(\neg\neg\gamma\to\neg\delta)$, and it follows that
$\turnstile\neg(\gamma\land\delta)\To(\neg\gamma\lor\neg\neg\gamma\to\neg\gamma\lor\neg\delta)$.
Upon a double application of the exponential law (\ref{exponential}), we get
$\turnstile\neg\gamma\lor\neg\neg\gamma\to\big(\neg(\gamma\land\delta)\to\neg\gamma\lor\neg\delta)\big)$,
so that (2) implies (JP).

Finally, applying the rule (3) with $\neg\alpha$ substituted for $\beta$, we get (2). Conversely,
we trivially have $\turnstile (\neg\alpha\tofrom\beta)\To(\neg\alpha\lor\neg\neg\alpha\to\beta\lor\neg\beta)$.
The exponential law transforms this validity into
$\turnstile\neg\alpha\lor\neg\neg\alpha\To\big((\neg\alpha\tofrom\beta)\to\beta\lor\neg\beta\big)$.
Applying the latter with $\neg\beta$ substituted for $\alpha$, we get that (2) implies (3).
\end{proof}
\metameta

Jankov's principle is satisfied in all {\it extremally disconnected} spaces, that is, spaces where
the closure of every open set is open (or equivalently, every regular open set is clopen).
All extremally disconnected metrizable spaces are discrete (see
\cite{En}*{Exer.\ 6.2.G(a)}), and in a certain model of ZF, all extremally disconnected compact
Hausdorff spaces are finite \cite{Mor}.
However, in ZFC, a well-known example of an extremally disconnected compact Hausdorff space is
the Stone--\v Cech compactification of any infinite discrete space \cite{En}.

On the other hand, the Alexandroff space corresponding to a poset with a greatest element
is always extremally disconnected.
In general, the Alexandroff space corresponding to a poset $P$ is extremally disconnected if and only if
``principal filters of $P$ are directed sets'', that is, every pair of elements of $P$ that has 
a lower bound also has an upper bound.
Indeed, if $U$ is an open subset of $P$ whose closure is not open, then some $z\in\Cl U$ is
bounded above by some $y\notin\Cl U$ (since $\Cl U$ is not open).
Also $z\notin U$ (since $U$ is open), but it is bounded above also by some $x\in U$ (since $z\in \Cl U$).
However, if $u$ is an upper bound of $x$ and $y$, then $u\in U$ (since $x\in U$) and
therefore $y\in\Cl U$, which is a contradiction.
Conversely, if $x,y\in P$ have a lower bound $z$ but no upper bound, then the minimal open set $U_x$
containing $x$ has a non-open closure (since $\Cl U_x$ contains $z$ but not $y$).

\formulas

The G\"odel--Dummett principle implies Jankov's principle; more precisely,
the latter is a special case of the principle
\[
\prin(\gamma\land\delta\to\neg\alpha)\To(\gamma\to\neg\alpha)\lor(\delta\to\neg\alpha)\text,
\tag{JP$'$}
\]
which is in turn a special case of Skolem's form of the G\"odel--Dummett principle (see \S\ref{Harrop1}).
Jankov's principle does not imply the G\"odel--Dummett principle by considering, for instance,
the four-element poset consisting of a pair of incomparable elements along with their upper and
lower bounds.

On the other hand, Jankov's principle implies the Kreisel--Putnam principle,
\[\prin(\neg\alpha\to\gamma\lor\delta)\To (\neg\alpha\to\gamma)\lor(\neg\alpha\to\delta)\text.\]
Indeed, we clearly have
$\turnstile (\neg\alpha\to\gamma\lor\delta)\To(\neg\alpha\lor\neg\neg\alpha\to\gamma\lor\delta\lor\neg\neg\alpha)$.
Hence by the exponential law, $\prin\neg\alpha\lor\neg\neg\alpha$ implies
$\prin(\neg\alpha\to\gamma\lor\delta)\To\gamma\lor\delta\lor\neg\neg\alpha$, and it remains to
observe that
$\turnstile\gamma\lor\delta\lor\neg\neg\alpha\To
(\gamma\lor\neg\neg\alpha)\lor(\delta\lor\neg\neg\alpha)$ and by (\ref{implication0}),
$\turnstile (\gamma\lor\neg\neg\alpha)\lor(\delta\lor\neg\neg\alpha)\To
(\neg\alpha\to\gamma)\lor(\neg\alpha\to\delta)$.
The Kreisel--Putnam principle does not imply Jankov's principle by considering, for instance,
the three-element poset of Example \ref{KP vs GD}.

By contrast, (JP$'$) is in fact equivalent to Jankov's principle.
Indeed, by the exponential law (JP$'$) is equivalent to
\[\prin\big(\alpha\to\neg(\gamma\land\delta)\big)\To(\alpha\to\neg\gamma)\lor(\alpha\to\neg\delta)\text,\]
where substituting $\neg\beta$ for $\alpha$ leads to no loss of generality due to (\ref{neg-neg-drop}).
Now Jankov's principle implies
$\prin\big(\neg\beta\to\neg(\gamma\land\delta)\big)\To(\neg\beta\to\neg\gamma\lor\neg\delta)$, whereas
the Kreisel--Putnam principle implies
$\prin(\neg\beta\to\neg\gamma\lor\neg\delta)\To(\neg\beta\to\neg\gamma)\lor(\neg\beta\to\neg\delta)$.

\metameta

To summarize, the following strict entailments of principles have been discussed:
\begin{center}
G\"odel--Dummett pr.\ $\turnstile $ \ Jankov's principle \ $\turnstile $ \ Kreisel--Putnam pr.\
$\turnstile $ \ Harrop's rule
\end{center}

\subsubsection{Mints-Citkin rule and Rose's formula} \label{Mints-Citkin}

\formulas
Harrop's rule is a consequence of the following Mints--Citkin rule \cite{Ci}
\[\frac{(\lambda\to\mu)\to\phi\lor\psi}{\big((\lambda\to\mu)\to\phi\big)\lor
\big((\lambda\to\mu)\to\psi\big)\lor\big((\lambda\to\mu)\to\lambda\big)}\text.
\]
Indeed, the Mints--Citkin rule with $\mu=\ab$ implies Harrop's rule,
since $\turnstile (\neg\lambda\to\lambda)\to(\neg\lambda\to\phi)$ (due to
$\turnstile\neg\lambda\land(\neg\lambda\to\lambda)\land\neg\lambda\to\ab$ and $\turnstile\ab\to\phi$).

Since Harrop's rule is not derivable (see \S\ref{Harrop1}), the Mints--Citkin rule is also not derivable.
But it is known to be admissible (see \cite{Ie}*{proof of Theorem 3.13}).
In fact, a slight generalization of this rule is the first in an infinite list of admissible rules
(the so-called Visser rules) that imply every admissible rule of zero-order intuitionistic logic,
by a celebrated result of P. Rozi\`ere \cite{Roz} and R. Iemhoff \cite{Ie}.

But in contrast to Harrop's rule, the Mints--Citkin rule is not stably admissible
\cite{Ie0}*{Theorem 24}.
Indeed, if it were stably admissible, then the Kreisel--Putnam principle
\[\prin(\neg\alpha\to\beta\lor\gamma)\To (\neg\alpha\to\beta)\lor(\neg\alpha\to\gamma)\text,\]
whose formula happens to be a special case of the premise of the Mints--Citkin rule,
would imply the principle
\[\prin
\big((\neg\alpha\to\beta\lor\gamma)\to(\neg\alpha\to\beta)\big)\lor
\big((\neg\alpha\to\beta\lor\gamma)\to(\neg\alpha\to\gamma)\big)\lor
\big((\neg\alpha\to\beta\lor\gamma)\to\neg\alpha\big)\text.\]
By a well-known theorem of Kreisel and Putnam, the extension of zero-order intuitionistic logic
by the Kreisel--Putnam principle has the Disjunction Property (see \cite{dJ}*{p.\ 3} for a short proof).
Thus the Kreisel--Putnam principle would have to imply either
$\prin(\neg\alpha\to\beta\lor\gamma)\to(\neg\alpha\to\beta)$
or $\prin(\neg\alpha\to\beta\lor\gamma)\to(\neg\alpha\to\gamma)$ or
$\prin(\neg\alpha\to\beta\lor\gamma)\to\neg\alpha$.
However, none of these three principles are derivable even in classical logic, in which 
the Kreisel--Putnam principle is certainly derivable.
\metameta

\formulas
It follows from Proposition \ref{Jankov's logic}, (3)$\imp$(2), that the principle
$\prin(\neg\neg\alpha\to\alpha)\to\alpha\lor\neg\alpha$ implies 
$\prin\neg\alpha\lor\neg\neg\alpha$.
In contrast, the Lemmon--Scott rule
\[
\frac{(\neg\neg\alpha\to\alpha)\to\alpha\lor\neg\alpha}{\neg\alpha\lor\neg\neg\alpha}
\]
\metameta
is not derivable.
Indeed, let us consider the Tarski model with $X=\R\x(\R_+\cup\R_-)\cup\{(0,0)\}$, where 
$\R_+=\{x\in\R\mid x>0\}$ and $\R_-=\{x\in\R\mid x<0\}$ and a valuation with
$|\fm\alpha|=(\R_+\cup\R_-)\x\R_+$.
Then $|\neg\fm\alpha|=\R\x\R_-$ and $|\neg\neg\fm\alpha|=\R\x\R_+$, so that
$|\neg\fm\alpha\lor\neg\neg\fm\alpha|=\R\x(\R_+\cup\R_-)=X\but\{(0,0)\}$.
On the other hand, $|\fm\alpha\lor\neg\fm\alpha|=\R\x\R_-\cup(\R_+\x\R_-)\x\R_+=|\neg\neg\fm\alpha\to\fm\alpha|$,
so that $|(\neg\neg\fm\alpha\to\fm\alpha)\to\fm\alpha\lor\neg\fm\alpha|=X$.
(See \cite{Ry}*{3.5.21} for a model in an Alexandroff space.)

\begin{remark} Our example above works to show non-derivability of the special case of the Lemmon--Scott rule 
with $\fm\alpha$ substituted by $\neg\fm\beta\lor\neg\fm\gamma$, by considering a valuation with
$|\fm\beta|=\R_+\x\R_+\cup\R\x\R_-$ and $|\fm\gamma|=\R_-\x\R_+\cup\R\x\R_-$.
On the other hand, the Lemmon--Scott rule is clearly equivalent to the rule
\FM{\frac{(\neg\neg\alpha\to\alpha)\to\neg\alpha\lor\neg\neg\alpha}{\neg\alpha\lor\neg\neg\alpha}\text.}
Thus our model shows non-derivability of Rose's formula
$\big((\neg\neg\Phi\to\Phi)\to\neg\Phi\lor\neg\neg\Phi\big)\to\neg\Phi\lor\neg\neg\Phi$, where
$\Phi=\neg\fm\beta\lor\neg\fm\gamma$.
(See \cite{N}*{\S III.6.9} for a model in the real line.)
\end{remark}

\begin{remark} It is not essential in the previous examples that $X$ is non-locally compact.
The same effect can be achieved in a Tarski model in $Y=\{(x,y)\in\R^2\mid |y|\ge|x|\}$, by considering 
a valuation with $|\fm\alpha|$, $|\fm\beta|$, $|\fm\gamma|$ as above but intersected with $Y$. 
\end{remark}

\formulas
However, the Lemmon-Scott rule is admissible; in fact, as observed in \cite{Ie3}, it follows easily from a special case of 
the Mints--Citkin rule:
\[\frac{(\neg\neg\alpha\to\alpha)\to\alpha\lor\neg\alpha}{\big((\neg\neg\alpha\to\alpha)\to\alpha\big)\lor
\big((\neg\neg\alpha\to\alpha)\to\neg\alpha\big)\lor
\big((\neg\neg\alpha\to\alpha)\to\neg\neg\alpha\big)}\text.
\]
Indeed, it is not hard to check that the second disjunct of the conclusion implies $\neg\alpha$, whereas 
the first and the third disjuncts each imply $\neg\neg\alpha$.

Let us now sketch an alternative argument that the Mints--Citkin rule is not stably admissible (a variation of
\cite{Ie1}*{proof of Proposition 23}).
If the Mints--Citkin rule were stably admissible, the Scott principle
\[\prin\big((\neg\neg\alpha\to\alpha)\to\alpha\lor\neg\alpha\big)\To\neg\alpha\lor\neg\neg\alpha\text,\]
whose formula happens to be a special case of the premise of the Mints--Citkin rule,
would imply the principle
\[\prin\big(\hat\alpha\to\neg\alpha\big)\lor\big(\hat\alpha\to\neg\neg\alpha\big)\lor
\big(\hat\alpha\to(\neg\neg\alpha\to\alpha)\big)\text,\]
where $\hat\alpha=(\neg\neg\alpha\to\alpha)\to\alpha\lor\neg\alpha$.
By a well-known theorem of Scott, the extension of zero-order intuitionistic logic
by the Scott principle has the Disjunction Property (see \cite{dJ}*{pp.\ 4--5}).
Thus the Scott principle would have to imply either
$\prin\big((\neg\neg\alpha\to\alpha)\to\alpha\lor\neg\alpha\big)\to\neg\alpha$
or $\prin\big((\neg\neg\alpha\to\alpha)\to\alpha\lor\neg\alpha\big)\to\neg\neg\alpha$
or $\prin\big((\neg\neg\alpha\to\alpha)\to\alpha\lor\neg\alpha\big)\to(\neg\neg\alpha\to\alpha)$.
However, the first two of the latter tree principles are not derivable even in classical logic, in which
the Scott principle is certainly derivable; and the third implies $\prin\beta\lor\neg\beta$ by substituting 
$\beta\lor\neg\beta$ for $\alpha$.
However, the Scott principle does not imply $\prin\beta\lor\neg\beta$, since Jankov's principle
$\prin\neg\alpha\lor\neg\neg\alpha$ implies the Scott principle but does not imply $\prin\beta\lor\neg\beta$
(see \S\ref{Jankov}).

\begin{remark} The rule
\[\frac{(\neg\neg\alpha\to\alpha)\to\alpha\lor\neg\alpha}{\alpha\lor\neg\alpha}\]
\vskip3pt
\noindent
and its corresponding principle 
$\prin\big((\neg\neg\alpha\to\alpha)\to\alpha\lor\neg\alpha\big)\to\alpha\lor\neg\alpha$
are equivalent to $\prin\alpha\lor\neg\alpha$.
Indeed, the latter principle clearly implies the former principle, whereas the rule is clearly equivalent to
\[\frac{(\neg\neg\alpha\to\alpha)\to\alpha\lor\neg\alpha}{\neg\neg\alpha\to\alpha}\text,\]
\vskip6pt
\noindent
which implies $\prin\alpha\lor\neg\alpha$ by substituting $\beta\lor\neg\beta$ for $\alpha$.

Moreover, the formula $\big((\neg\neg\alpha\to\alpha)\to\alpha\lor\neg\alpha\big)\to\alpha\lor\neg\alpha$
is equivalent to $\neg\neg\alpha\to\alpha$.
Indeed, it is clearly equivalent to
$\big((\neg\neg\alpha\to\alpha)\to\alpha\lor\neg\alpha\big)\to(\neg\neg\alpha\to\alpha)$,
and it suffices to show that
\[\turnstile\Big(\big((\neg\neg\alpha\to\alpha)\to\alpha\lor\neg\alpha\big)\to(\neg\neg\alpha\to\alpha)\Big)\To
(\neg\neg\alpha\to\alpha)\text.\]
Indeed, we can rewrite this judgement, by the exponential law, as
\[\turnstile\neg\neg\alpha\land\Big(\big((\neg\neg\alpha\to\alpha)\to\alpha\lor\neg\alpha\big)\to
(\neg\neg\alpha\to\alpha)\Big)\To\alpha\text.\tag{$*$}\]
The left hand side of ($*$) implies, using (\ref{co-exponential}),
\[\Big(\neg\neg\alpha\to\big((\neg\neg\alpha\to\alpha)\to\alpha\lor\neg\alpha\big)\Big)\to
(\neg\neg\alpha\to\alpha)\text,\]
which by the exponential law rewrites as
\[\big((\neg\neg\alpha\to\alpha)\to(\neg\neg\alpha\to\alpha\lor\neg\alpha)\big)\to
(\neg\neg\alpha\to\alpha)\text,\]
which in turn is equivalent to
$\neg\neg\alpha\to\alpha$.
On the other hand, the left hand side of ($*$) also implies $\neg\neg\alpha$; thus it implies
$\neg\neg\alpha\land(\neg\neg\alpha\to\alpha)$, which in turn implies the right hand side of ($*$).
\end{remark}

\metameta

\begin{remark} See \cite{Bek} concerning stable admissibility of rules in arithmetic.
\end{remark}

\subsection{Markov's principle}\label{Markov}

\formulas
Similarly missing among intuitionistic laws is the following converse to (\ref{quantifiers2}), which
is sometimes referred to as the {\it Generalized Markov Principle}:
\[
\prin\neg\forall \tr x\,\pi(\tr x)\To\exists \tr x\,\neg \pi(\tr x)\text.
\tag{GMP}
\]
This is of course the ``generalization'' of Jankov's principle above with $\exists$ and $\forall$
in place of $\lor$ and $\land$.
On the other hand, (GMP) can also be ``specialized'' by replacing $\exists$ and $\forall$ with 
$\land$ and $\to$, and the resulting principle $\prin\turnstile\neg (\chi\to\pi)\To(\chi\land\neg\pi)$ 
is also non-derivable.
Indeed, it specializes to $\prin\neg\neg\chi\to\chi$ by substituting $\pi$ with $\ab$.

The stable case of (GMP),
\[\frac{\neg\neg\pi(\tr x)\to\pi(\tr x)}{\neg\forall \tr x\,\pi(\tr x)\To\exists \tr x\,\neg \pi(\tr x)}\text,
\]
is clearly equivalent to the principle 
$\prin\neg\forall \tr x\,\neg\rho(\tr x)\To\exists \tr x\,\neg\neg\rho(\tr x)$, which is in turn equivalent,
via (\ref{quantifiers1}), to the following {\it Strong Markov Principle}
\[
\prin\neg\neg\exists \tr x\,\rho(\tr x)\To\exists \tr x\,\neg\neg\rho(\tr x)
\tag{SMP}
\]
whose converse is the intuitionistic law (\ref{SMP-converse}).
Using (\ref{quantifiers1}) and (\ref{triple negation}), (SMP) can be equivalently
rewritten as
\[
\prin\neg\neg\exists \tr x\,\neg\neg\rho(\tr x)\To\exists \tr x\,\neg\neg\rho(\tr x)\text,
\]
which is in turn clearly equivalent to the rule
\[\frac{\neg\neg\rho(\tr x)\to\rho(\tr x)}{\neg\neg\exists\tr x\,\rho(\tr x)\to\exists\tr x\,\rho(\tr x)}\text,
\]
whose admissibility amounts to the assertion that if $\rho(\tr x)$ is stable, then so is 
$\exists \tr x\,\rho(\tr x)$.
Let us note that $\forall \tr x\,\rho(\tr x)$ is always stable if $\rho(\tr x)$ is, due to
the intuitionistic law (\ref{DNS-converse}).

\metameta
It is easy to see that (SMP) holds in all Alexandroff spaces whose corresponding preorder is
a directed set (i.e., every pair of elements has an upper bound).
Nevertheless, Jankov's principle does not imply (SMP) by considering the Stone--\v Cech compactification
of $\N$ (where the union of the regular open sets $\{n\}$, $n\in\N$, is not regular open).
It is not hard to construct a simpler model of this kind that does not depend on the uncountable axiom
of choice (see Proposition \ref{Stone-Cech} below).

\formulas
If $\exists$ is ``specialized'' to $\land$ in (SMP), the resulting principle
$\prin\neg\neg(\chi\land\rho)\To(\chi\land\neg\neg\rho)$ is still non-derivable (indeed, 
it specializes to $\prin\neg\neg\chi\to\chi$ by substituting $\rho$ with $\triv$), but its stable case
$\prin\neg\neg(\neg\psi\land\neg\pi)\To(\neg\psi\land\neg\pi)$ is derivable by (\ref{deMorgan3}) and
(\ref{deMorgan1}).
If $\exists$ is ``specialized'' to $\lor$ in (SMP), the resulting principle
$\prin\neg\neg(\rho_1\lor\rho_2)\To\neg\neg\rho_1\lor\neg\neg\rho_2$ is still non-derivable, and so
is its stable case (since the union of two regular open sets may well fail to be regular open).
However, the decidable case
\[
\frac{\rho_1\lor\neg\rho_1\text,\ \ \rho_2\lor\neg\rho_2}
{\neg\neg(\rho_1\lor\rho_2)\To\neg\neg\rho_1\lor\neg\neg\rho_2}
\]
is derivable, since the premise implies $(\rho_1\lor\rho_2)\lor\neg(\rho_1\lor\rho_2)$ using (\ref{deMorgan1}), 
which in turn implies $\neg\neg(\rho_1\lor\rho_2)\To\rho_1\lor\rho_2$ by (\ref{decidable-stable}).

The following {\it Strong Markov Rule} asserts that (SMP)
holds for all decidable $\rho(\tr x)$:
\[
\frac{\rho(\tr x)\lor\neg\rho(\tr x)}
{\neg\neg\exists \tr x\,\rho(\tr x)\To\exists \tr x\,\rho(\tr x)}\text;
\tag{SMR}
\]
or equivalently that (GMP) holds for all decidable $\pi(\tr x)$:
\[
\frac{\pi(\tr x)\lor\neg\pi(\tr x)}
{\neg\forall \tr x\,\pi(\tr x)\To\exists \tr x\,\neg\pi(\tr x)}\text.
\tag{SMR$'$}
\]
\metameta
To see that (SMR) is not derivable, let us consider the Tarski model with $\D=\N$ and $X=\N_+$, the one-point
compactification of the countable discrete space $\N$, and a valuation with $|\fm\rho|(n)=\{n\}$.
Then each $|\fm\rho|(n)$ is clopen, but their union is not regular open.
Another useful Tarski model where (SMR) fails has $\D=\N$ and $X=2^\N$, the Cantor set of all functions
$\N\to\{0,1\}$; each subset $|\fm\rho|(n)$ of $2^\N$ consisting of all functions $f$ such that $f(n)=1$
is clopen, but the union of these subsets, $2^\N\but\{(0,0,\dots)\}$, is not regular open.

These examples actually show that even the following {\it Markov Rule} is not derivable:
\FM{\frac{\rho(\tr x)\lor\neg\rho(\tr x),\ \
\neg\neg\exists \tr x\,\rho(\tr x)}{\exists \tr x\,\rho(\tr x)}\text.}
But the latter is strictly weaker than the Strong Markov Rule.
Indeed, let us consider a Tarski model with $\D=\N$ and $X=\N_+\x[0,1]\cup\{\infty\}\x(-1,0]$, and
a valuation with $|\fm\rho|(n)=\{n\}\x[0,1]$.
Again each $|\fm\rho|(n)$ is clopen, and their union is not regular open.
On the other hand, if $\Cl(\bigcup_{i=1}^\infty U_i)=X$, where each $U_i$ is clopen, then at least
one $U_i$ contains a neighborhood of $\{\infty\}\x[-1,1]$, and it follows that in fact $\bigcup U_i=X$.

\begin{theorem} The Markov rule is admissible in intuitionistic logic.
\end{theorem}

\begin{proof}
Suppose that for some 1-formula $\Phi$ we have $\turnstile\Phi(\tr x)\lor\neg\Phi(\tr x)$
and $\turnstile\neg\neg\exists \tr x\,\Phi(\tr x)$, but $\not\turnstile\exists \tr x\,\Phi(\tr x)$.
Then by the completeness theorem, there exists a Tarski model in some space $X$ with some domain $\D$
such that some point $p_X\in X$ is not contained in $|\exists \tr x\,\Phi(\tr x)|^{\pval_X}_{\iass_0}$ for some 
valuation $\pval_X$ and some assignment $\iass_0$.
(Of course, the latter is needed only if $\exists\tr x\,\Phi(\tr x)$ is not a closed formula.) 

Let $Y$ be the Alexandroff space corresponding to the three element poset consisting of a pair of
incomparable elements $p_Y$, $q$ and of their lower bound $r$.
Let $Z$ be the gluing $X\cup_{p_X=p_Y}Y$, that is, the quotient space of the disjoint union $X\sqcup Y$
by the equivalence relation whose only non-singleton class, to be denoted by $p$, consists of $p_X$ and $p_Y$.
We will identify $X$ and $Y$ with their images in $Z$, and in particular from now on we will write $p$
in place of both $p_X$ and $p_Y$.
Let $\pval_Z$ be the valuation for the Tarski model in $Z$ with domain $\D$ defined by 
$|\alpha|^{\pval_Z}_\iass=|\alpha|^{\pval_X}_\iass$ for every problem variable $\alpha$ 
and every assignment $\iass$.
Since $X$ is open in $Z$, it is easy to check by induction that 
$|\Xi|^{\pval_X}_\iass=|\Xi|^{\pval_Z}_\iass\cap X$ for every
formula $\Xi$ and every assignment $\iass$.
In particular, $|\Phi|^{\pval_X}_{\iass_0}(d)=|\Phi|^{\pval_Z}_{\iass_0}(d)\cap X$ for each $d\in\D$.

Then $p\notin|\Phi|^{\pval_Z}_{\iass_0}(d)$ for each $d\in\D$.
Since $Y$ is connected and each $|\Phi|^{\pval_Z}_{\iass_0}(d)$ is clopen (due to our hypothesis that
$\Phi(\tr x)$ is decidable), we also have $q\notin|\Phi|^{\pval_Z}_{\iass_0}(d)$ for each $d\in\D$.
Hence $q\notin|\exists \tr x\,\Phi(\tr x)|^{\pval_Z}_{\iass_0}$.
Since $Z\but\{q\}$ is closed, $q$ also does not belong to
$|\neg\neg\exists\tr x\,\Phi(\tr x)|^{\pval_Z}_{\iass_0}=\Int\Cl|\exists \tr x\,\Phi(\tr x)|^{\pval_Z}_{\iass_0}$.
But this contradicts our hypothesis that $\turnstile\neg\neg\exists \tr x\,\Phi(\tr x)$.
\end{proof}

The following is sometimes called the {\it Markov Principle}:
\FM{\prin\forall \tr x\,\big(\rho(\tr x)\lor\neg\rho(\tr x)\big)\To
\big(\neg\neg\exists \tr x\,\rho(\tr x)\to\exists \tr x\,\rho(\tr x)\big)\text.}
It is strictly stronger than the Markov Rule due to the admissibility of the latter; moreover, it is
strictly stronger than the Strong Markov Rule.
Indeed, let us consider a Tarski model with $\D=\N$ and $X=\N_+\x[0,1]\cup[0,\infty]\x\{0\}$ and
a valuation with $|\fm\rho|(n)=U_n:=\{n\}\x(0,1]$.
Then $\bigcup_{i=1}^\infty U_i=\N\x(0,1]$ and $\Int\Cl(\bigcup_{i=1}^\infty U_i)=\N_+\x(0,1]$, and
consequently
$|\neg\neg\exists\tr x\,\fm\rho(\tr x)\to\exists\tr x\,\fm\rho(\tr x)|=X\but(\{\infty\}\x[0,1])$.
On the other hand, $U_i\cup\Int(X\but U_i)=X\but\{(n,0)\}$, and therefore
$|\forall\tr x\,\big(\fm\rho(\tr x)\lor\neg\fm\rho(\tr x)\big)|=X\but(\N_+\x\{0\})$, which is not a subset of
$X\but(\{\infty\}\x[0,1])$.
Nevertheless, the model satisfies the Strong Markov Rule, since the only clopen subsets of $X$ are $\emptyset$
and $X$.

\begin{remark}
Markov's original formulation of his principle was concerned with the case where $x$ ranges over $\N$.%
\footnote{In some form, Markov's Principle was discussed long before A. A. Markov, Jr.\ (who himself called it
the ``Leningrad Principle''), in particular by P. S. Novikov \cite{N2} and, according to Hilbert
\cite{Hi2}*{p.\ 268}, already by Kronecker.}
A decidable problem $\rho(n)$ is equivalent to the problem of verifying that $\rho_n=1$,
where each $\rho_n\in\{0,1\}$ is defined by $\rho_n=1$ if $\rho(n)$ is soluble,
and $\rho_n=0$ if $\neg\rho(n)$ is soluble.
Under the BHK interpretation (assuming that $\rho$ is decidable), the problem $\exists n\,\rho(n)$ asks
essentially to find a one in the sequence $(\rho_i)$, and the problem $\neg\neg\exists n\,\rho(n)$,
or equivalently $\neg\forall n\,\neg\rho(n)$, asks essentially to prove that there exists no proof
that each entry in the sequence is zero.
Markov's Principle is then asserting that there exists a general method that, given a sequence of zeroes
and ones, and a proof that this sequence cannot be proved to be all $0$'s, produces a $1$ somewhere in
this sequence.
Of course, we are not really told how to find this $1$ by constructive means, so this principle is
not justified by the clarified BHK interpretation.

However, on a very computational reading of the BHK interpretation, Markov's principle is validated.
Indeed, suppose we have a Turing machine that reads a sequence of zeroes and ones until it finds a one.
If this machine continues forever on some sequence, this can (perhaps) be equated with a proof that
the sequence is all zeroes.
Now, given a proof that the machine does not continue forever on some sequence, by actually running it on
this sequence we get a general method of finding a one in the sequence.
In this connection Markov's constructivist school considered Markov's Principle to be constructively
acceptable (cf.\ \cite{TV}).
\end{remark}

\subsubsection{Principle of Omniscience}\label{omniscience}
\formulas
The following {\it Rule of Omniscience}:%
\footnote{This is a form of Bishop's ``Principle of Omniscience'' \cite{Bi}.}
\[
\frac{\pi(\tr x)\lor\neg\pi(\tr x)}{\exists \tr x\,\neg\pi(\tr x)\,\lor\,\forall \tr x\,\pi(\tr x)}
\tag{RO}
\]
implies (SMR$'$) by (\ref{implication0'}).
Using (\ref{decidable-stable}) and (\ref{quantifiers1}), one can equivalently reformulate (RO) 
as
\[\frac{\rho(\tr x)\lor\neg\rho(\tr x)}
{\exists \tr x\,\rho(\tr x)\,\lor\,\neg\exists \tr x\,\rho(\tr x)}\text.\]
(To obtain this from (RO), substitute $\neg\rho(\tr x)$ for $\pi(\tr x)$; and for the converse, substitute
$\neg\pi(\tr x)$ for $\rho(\tr x)$.)

\metameta
From the latter form of (RO) it is clear that its admissibility amounts to the assertion that if 
$\fm\rho(\tr x)$ is decidable, then so is $\exists \tr x\,\fm\rho(\tr x)$.
Also, (RO) is valid in those spaces where arbitrary union of clopen sets is clopen.
These include all Alexandroff spaces and all connected spaces.
In particular, Tarski models in Euclidean spaces satisfy (RO), and so intuitionistic
logic is not strongly complete with respect to this class of models.
On the other hand, (RO) is not valid in a Tarski model in the subset $[-1,0]\cup\{\frac1n\mid n\in\N\}$
of the real line, where (SMR) is valid; indeed, in this space, a union of clopen sets is regular open, but
generally not clopen.

\formulas
The symmetric assertion, that $\forall\tr x\,\rho(\tr x)$ is decidable whenever
$\rho(\tr x)$ is, amounts to the admissibility of what can be called the {\it Weak Rule of Omniscience}:
\[
\frac{\pi(\tr x)\lor\neg\pi(\tr x)}
{\forall \tr x\,\pi(\tr x)\,\lor\,\neg\forall \tr x\,\pi(\tr x)}\text,
\tag{WRO}
\]
because it follows from (RO) by (\ref{quantifiers2}).
\metameta
The converse implication fails by considering the Tarski model in $\N_+$, where
(RO) is not valid (indeed, already (SMR) is not valid), but (WRO) is valid, since the only open set of $\N_+$ 
that is not clopen is $\N$, but it is neither the interior of any larger set nor the intersection of
any family of clopen sets.
On the other hand, (WRO) together with (SMR) obviously entail (RO).
Thus in fact (RO) is equivalent to the meta-conjunction of (WRO) and (SMR):

\[\turnstile\quad\text{(RO) }\Iff\text{ (WRO) }\mand\text{ (SMR)}.\]

\begin{remark}
When $x$ ranges over $\N$, the three principles can be recast under the BHK interpretation in terms
of the Cantor set $2^\N$ of all functions $\N\to\{0,1\}$.
If $\infty$ denotes the constant function $n\mapsto 1$, and
$p$ is understood to range over $2^\N$, then
the three principles assert the existence of  general methods for solving the following problems:

\begin{itemize}
\item Weak Rule of Omniscience:
$p=\infty\lor\neg (p=\infty)$;
\item Strong Markov Rule:
$\neg (p=\infty)\to p\ne\infty$;
\item Rule of Omniscience:
$p=\infty\lor p\ne\infty$,
\end{itemize}
where $p=q$ denotes the problem, {\it Prove that $p(n)=q(n)$ for
all $n\in\N$}, and $p\ne q$ denotes the problem, {\it Find an $n\in\N$
such that $p(n)\ne q(n)$}.
As observed by Escardo \cite{Es2}, there is no loss
of generality in permitting $p$ to range just over the subset
$\N_+\subset 2^\N$ consisting of all non-increasing functions.
(Note that $\N_+=N\cup\{\infty\}$, where $N$ is a copy of $\N$.)
This follows from the existence of a constructive retraction $r$
of $2^\N$ onto $\N_+$, defined by $r(p)(n)=\min\{p(k)\mid k\le n\}$.
\end{remark}


\subsection{$\neg\neg$-Translation}\label{negneg}

A translation of classical predicate logic QC
into intuitionistic predicate logic QH was discovered essentially by
Kolmogorov (even though he only had his formalization of a fragment of QH
available at the moment) \cite{Kol0} (see also \cite{Coq}, \cite{UP}, \cite{Pl}).

\begin{theorem}\label{negneg-theorem}
Given a formula $F$ in the language of classical logic, insert the double negation $\neg\neg$
in front of every subformula, and replace $\top$ and $\bot$ by $\triv$ and $\ab$, respectively.
Then the resulting formula $F_{\neg\neg}$ is derivable in intuitionistic logic if and only if 
$F$ is derivable in classical logic.

Moreover, a rule $F_1,\dots,F_k\,/\,G$ is derivable in classical logic if and only if the rule 
$(F_1)_{\neg\neg},\dots,(F_k)_{\neg\neg}\,/\,G_{\neg\neg}$ is derivable in intuitionistic logic.
\end{theorem}

Kolmogorov's translation was later rediscovered by G\"odel and
Gentzen, who gave proofs in the setup of QH and also observed that
prefixing $\land$, $\forall$ and (in the case of Gentzen) $\to$ by
$\neg\neg$ is superfluous:

\formulas
\begin{proposition} \label{Goedel-Gentzen}
The following holds in intuitionistic logic:

(a) $\turnstile\neg\neg\alpha\land\neg\neg\beta\Tofrom
\neg\neg(\neg\neg\alpha\land\neg\neg\beta)$;

(b) $\turnstile (\neg\neg\alpha\to\neg\neg\beta)\Tofrom
\neg\neg(\neg\neg\alpha\to\neg\neg\beta)$;

(c) $\turnstile\forall \tr x\,\neg\neg\alpha(\tr x)\Tofrom
\neg\neg\forall \tr x\,\neg\neg\alpha(\tr x)$.
\end{proposition}

\begin{proof} By (\ref{deMorgan1}) and (\ref{implication-exp}),
if $\gamma$ and $\delta$ are stable, then so are $\gamma\land\delta$
and $\gamma\to\delta$.
Also, by (\ref{quantifiers1}), if $\gamma(\tr x)$ is
stable, then so is $\forall \tr x\,\gamma(\tr x)$.
But by (\ref{triple negation}), $\turnstile\neg\neg\epsilon\tofrom\epsilon$ holds
for stable $\epsilon$.
\end{proof}

\metameta
\begin{proof}[Proof of Theorem \ref{negneg-theorem}]
The ``only if'' assertion is trivial: any derivation of $F_{\neg\neg}$ in intuitionistic logic is
also valid classically, but clearly $F_{\neg\neg}$ is classically equivalent to $F$.
Similarly for rules.

For the converse, we may consider instead of Kolmogorov's interpretation $F\mapsto F_{\neg\neg}$
its (Gentzen's) simplified version described above, to be denoted $F\mapsto F'$.
Then it suffices to check that if $\prin F$ is a law of classical logic, then 
$\prin F'$ is derivable in intuitionistic logic; and if $F_1,\dots,F_k\,/\,G$ is 
an inference rule of classical logic, then $F_1',\dots,F_k'\,/\,G'$
is derivable in intuitionistic logic.

\formulas
For principles and rules that do not involve $\lor$ or $\exists$, this is trivial, by \ref{Goedel-Gentzen}.
From the law $\prin\alpha(\tr x)\to\exists\tr x\,\alpha(\tr x)$ we get that
$\prin\neg\neg\alpha(\tr x)\to\neg\neg\exists\tr x\,\neg\neg\alpha(\tr x)$ is derivable, using that 
$\turnstile\beta\to\neg\neg\beta$.
From the inference rule $\alpha(\tr x)\to\gamma/\,\exists\tr x\,\alpha(x)\to\gamma$ we get that
$\neg\neg\alpha(\tr x)\to\neg\neg\gamma\,/\,\neg\neg\exists\tr x\,\neg\neg\alpha(\tr x)\to\neg\neg\gamma$ 
is derivable, using that $\turnstile\neg\neg(\neg\neg\gamma)\to\neg\neg\gamma$.
The analogous law and inference rule involving $\lor$ are treated similarly.
Finally, $\prin\neg\neg(\neg\neg\alpha\lor\neg(\neg\neg\alpha))$ is derivable, using 
$\turnstile\neg\neg(\gamma\lor\neg\gamma)$.
\end{proof}

As for $\lor$ and $\exists$, let us note that one can eliminate them altogether using the intuitionistic 
de Morgan laws (\ref{deMorgan1}), (\ref{quantifiers1}):
\formulas
\[\turnstile\neg\neg(\alpha\lor\beta)\Tofrom\neg(\neg\alpha\land\neg\beta)\text,\]
\[\turnstile\neg\neg\exists \tr x\,\alpha(\tr x)\Tofrom\neg\forall \tr x\,\neg\alpha(\tr x)\text.\]
\metameta
To summarize, the $\neg\neg$-translation of classical logic into
intuitionistic can be presented as follows:
\begin{itemize}
\item predicate variables are interpreted as problem variables prefixed by $\neg\neg$;
\item $\top$ and $\bot$ are replaced by $\triv$ and $\ab$, respectively;
\item classical $\land$, $\forall$ and $\to$ (and hence $\neg$)
are interpreted as the intuitionistic ones;
\item classical $\lor$ and $\exists$ are interpreted as the
$\neg$-conjugates of intuitionistic $\land$ and $\forall$.
\end{itemize}

Topologically, we have interpreted classical logic as the logic of
regular open sets, where $\land$, $\forall$ and $\to$ (and hence $\neg$)
are interpreted as in Tarski models, whereas $\lor$ and $\exists$
are interpreted via the Tarski interpretations of $\land$ and $\forall$.

In embedding classical logic into intuitionistic in a type theoretic
context it may be essential to avoid the prefixing
of atomic subformulas (see \cite{BR}).
This is achieved by the Kuroda embedding, which places $\neg\neg$ before
the entire formula and also {\it after} each universal quantifier.
It is still equivalent to Kolmogorov's embedding:%
\footnote{Other variants of the $\neg\neg$-translation achieve still more
economy in the use of $\neg$ (see \cite{FO}).}

\formulas
\begin{proposition} \label{Kuroda}
The following holds in intuitionistic logic:

(a) $\turnstile\neg\neg(\alpha\land\beta)\Tofrom
\neg\neg(\neg\neg\alpha\land\neg\neg\beta)$;

(b) $\turnstile\neg\neg(\alpha\lor\beta)\Tofrom
\neg\neg(\neg\neg\alpha\lor\neg\neg\beta)$;

(c) $\turnstile\neg\neg(\alpha\to\beta)\Tofrom
\neg\neg(\neg\neg\alpha\to\neg\neg\beta)$;

(d) $\turnstile\neg\neg\exists \tr x\,\alpha(\tr x)\Tofrom\neg\neg\exists \tr x\,\neg\neg\alpha(\tr x)$.
\end{proposition}
\metameta

\begin{proof} Using respectively (\ref{deMorgan3}), (\ref{deMorgan1}),
(\ref{implication1}) and (\ref{quantifiers1}), the left hand sides
can be equivalently rewritten in a form where each atomic subformula
is prefixed by a $\neg$.
In this form, each atomic subformula can be replaced by its double
negation, due to (\ref{triple negation}).
\end{proof}

Of course, instead of placing $\neg\neg$ in front of the entire formula it suffices
(in view of \ref{Goedel-Gentzen}) to put it before every $\exists$ and $\lor$.
This yields an ``essentially local'' version of the translation (similar to Dowek's version \cite{Dow}):
\begin{itemize}
\item prefix every classical $\exists$ and $\lor$ by a $\neg\neg$;
\item postfix every classical $\forall$ and $\land$ by a $\neg\neg$;
\item if the entire formula is a single atom, prefix it by a $\neg\neg$;
\item replace $\top$, $\bot$ by $\triv$, $\ab$.
\end{itemize}
(Here a classical $\land$ really needs to be postfixed by a $\neg\neg$ only when there are
no other types of connectives ``around''.)

\subsubsection{$\neg\neg$-Shift principle}\label{shift}
Prefixing only the entire formula with $\neg\neg$ (and replacing $\top$, $\bot$ by $\triv$, $\ab$), which 
is known as Glivenko's transformation, does not work in the presence of universal quantifiers.
Indeed, let $F$ denote the formula
$\fm{\forall \tr x\,\pi(\tr x)\,\lor\,\exists \tr x\,\neg\pi(\tr x)}$, which we
encountered in the Rule of Omniscience and which is derivable in classical logic.
If Glivenko's transformation were taking classically derivable formulas
to intuitionistically derivable formulas, then $\prin\neg\neg F$, that is,
\formulas
\[\prin\neg\neg\,\big(\forall \tr x\,\pi(\tr x)\,\lor\,\exists \tr x\,\neg\pi(\tr x)\big)\text,\]
would be an intuitionistically derivable principle.
The latter principle is equivalent, via (\ref{deMorgan1}) and (\ref{implication-exp}), to
\[\prin\neg\exists \tr x\,\neg\pi(\tr x)\To\neg\neg\forall \tr x\,\pi(\tr x)\text,\]
which is nothing but the contrapositive of (GMP), Generalized Markov's Principle.
By (\ref{quantifiers1}), the latter principle is in turn equivalent to
the following {\it $\neg\neg$-Shift Principle}:
\[\prin\forall \tr x\,\neg\neg\pi(\tr x)\To\neg\neg\forall \tr x\,\pi(\tr x)\text. \tag{DNS}\]
Note that the converse to the $\neg\neg$-Shift Principle is nothing but
the intuitionistically derivable principle (\ref{DNS-converse}).
If $\forall$ is ``specialized'' to $\land$, the $\neg\neg$-Shift
Principle itself becomes intuitionistically derivable (\ref{neg-neg-and});
and if $\forall$ is ``specialized'' to $\to$, the resulting principle
$\prin(\chi\to\neg\neg\pi)\To\neg\neg (\chi\to\pi)$ is again
intuitionistically derivable (by (\ref{neg-neg-imp}) and (\ref{neg-neg-drop})).
\metameta

To see that the $\neg\neg$-Shift Principle is not derivable, let us consider
the Tarski model with $X=\R$ and $\D=\Q$, and a valuation with $|\fm\pi|(q)=\R\but\{q\}$.
Then $|\forall \tr x\,\neg\neg\fm\pi(\tr x)|=\R$, whereas
$|\neg\neg\forall \tr x\,\fm\pi(\tr x)|=\emptyset$ (indeed, $|\forall\tr x\,\fm\pi(\tr x)|=\emptyset$
since $\R\but\Q$ has empty interior in $\R$).
Another Tarski model where the $\neg\neg$-Shift Principle is not valid can be obtained by observing that 
the poset of integers (ordered in the usual way) with its Alexandroff topology
contains only countably many open sets, whose intersection is empty,
and no regular open sets other than $\emptyset$ and the whole space.

\formulas
By \ref{Goedel-Gentzen}, the $\neg\neg$-Shift Principle
can be equivalently rewritten as
\[\prin\neg\neg\forall \tr x\,\neg\neg\pi(\tr x)\To\neg\neg\forall \tr x\,\pi(\tr x)\text. \tag{DNS$'$}\]
It follows that the $\neg\neg$-Shift Principle holds for stable $\pi(\tr x)$.
In other words, the contrapositive of (SMP), the Strong Markov Principle,
is derivable in intuitionistic logic.

Yet another equivalent formulation of the $\neg\neg$-Shift Principle was discussed in \S\ref{PEM}:
\[\prin\neg\neg\forall \tr x\,(\pi(\tr x)\lor\neg\pi(\tr x))\text. \tag{DNS$''$}
\]
To see that (DNS) implies (DNS$''$) it suffices to note that by (\ref{not-not-LEM}),
$\turnstile\forall \tr x\,\neg\neg(\pi(\tr x)\lor\neg \pi(\tr x))$.
Conversely, (DNS$''$) implies (DNS$'$) using that
$\turnstile\neg\neg\alpha\land(\alpha\lor\neg\alpha)\To\alpha$.

Let us note that the $\neg\neg$-Shift Principle is valid, for instance, in the Tarski model in
the one-point compactification $\N_+$ of the countable discrete space $\N$ with any $\D$.
Indeed, $|\alpha\lor\neg\alpha|$ must be dense in $\N_+$ due to
$\turnstile\neg\neg(\alpha\lor\neg \alpha)$ (see (\ref{not-not-LEM})).
Hence $|\alpha\lor\neg\alpha|$ contains $\N$.
Therefore $|\forall \tr x\,(\pi(\tr x)\lor\neg\pi(\tr x))|$ also contains $\N$, and thus
$|\neg\neg\forall \tr x\,(\pi(\tr x)\lor\neg\pi(\tr x))|=\N_+$.
\metameta


\subsection{Provability translation}

\subsubsection{Modal logic QS4}\label{S4}
Suppose that classical connectives and quantifiers are interpreted by
set-theoretic operations on subsets of a set $X$ (as in \S\ref{Euler}),
and that $X$ also happens to be a topological space.
As long as the Leibniz--Euler model is closed under the topological interior operator $\Int$,
this operator can be regarded as interpreting some additional unary connective $\Box$.
In a Leibniz--Euler model, equivalent formulas are always interpreted by the same subset;
thus such a $\Box$ must certainly be well-defined on equivalence classes of formulas, in the sense that if 
$F\tofrom G$ is derivable in classical logic, then $\Box F\tofrom\Box G$ is to be postulated.
This condition along with the usual axioms of a topological space $X$ in terms of interior,
\begin{itemize}
\item $\Int P\subset P$;
\item $\Int P\subset\Int(\Int P)$;
\item $X\subset\Int X$;
\item $\Int P\cap\Int Q=\Int(P\cap Q)$,
\end{itemize}
correspond to the following properties of $\Box$, which suggest the
intended reading of $\Box F$ as ``there exists a proof of $F$''.%
\footnote{It is remarkable that Orlov, who first introduced these axioms
in 1928, did so in order to give a provability explanation of propositional
intuitionistic logic \cite{Or}*{\S\S6,7} (see \cite{M1}*{\S\ref{g1:letters1}}).
His work remained virtually unknown, while the same axioms were rediscovered
in a few years by Lewis, Becker and G\"odel (see \cite{Goe2}, \cite{McT3}*{Theorem 2.1}).}
\formulas
\begin{itemize}
\item (reflection) $\prin\Box p\to p$;
\item (proof checking) $\prin\Box p\to\Box\Box p$;
\item (modus ponens) $\prin\Box p\land\Box (p\to q)\To\Box q$.
\medskip
\item (necessitation) $\dfrac{p}{\Box p}$;
\end{itemize}
\medskip
Indeed, the modus ponens principle follows from:
\begin{enumerate}[label=(\roman*)]
\item $\prin\Box p\land\Box q\Tofrom\Box(p\land q)$;
\medskip
\item $\dfrac{p\tofrom q}{\Box p\tofrom\Box q}$.
\end{enumerate}
\medskip
using that in classical logic, $\turnstile p\land(p\to q)\Tofrom p\land q$.
Conversely, the modus ponens principle implies
\begin{enumerate}[resume*]
\item $\prin\Box(p\to q)\To(\Box p\to\Box q)$
\end{enumerate}
by the exponential law (regarded as a law of classical logic).
Now (iii) and necessitation imply (ii) and (i).
In more detail, the $\leftarrow$ implication in (i) is proved using
that in classical logic, $\turnstile p\land q\to p$ and $\turnstile p\land q\to q$; and
the $\to$ implication in (i), transformed by the exponential law,
is proved using that in classical logic, $\turnstile p\to (q\to (p\land q))$.
\metameta

The four properties of $\Box$ (reflection, proof checking, necessitation and modus ponens), regarded as three 
laws and one inference rule, are supplemented by the usual laws and inference
rules of classical logic (see \S\ref{logics}), with understanding that they now apply to all formulas, 
possibly containing the new connective $\Box$.
The resulting logic is known as QS4, and its zero-order version as S4.
The connective $\Box$ is called a modality and is assigned the same level of precedence as $\neg$.

By design, a Leibniz--Euler model of classical logic in subsets of a topological
space extends to a {\it topological model} of QS4 by interpreting $\Box$ by
the interior operator.
It is well-known that QS4 is complete with respect to its topological models \cite{RS}
with countable domain (see also \cite{Krem3}).

\formulas
Let us note that the laws of QS4 can be equivalently restated in terms of
$\Diamond:=\neg\Box\neg$:
\begin{enumerate}
\item $\prin p\to\Diamond p$;
\item $\prin \Diamond\Diamond p\to\Diamond p$;
\item $\prin (\Diamond p\to\Diamond q)\to\Diamond(p\to q)$.
\end{enumerate}
However, the rule $\Diamond p\,/\,p$, which is the ``contrapositive'' of the necessitation rule
$p\,/\,\Box p$, is not even an admissible rule of QS4.
Indeed, $q\to\Box q$ is not derivable in QS4, but $\Diamond(q\to\Box q)$ is 
(see \S\ref{formal provability} below).
Alternatively, since $\alpha\lor\neg\alpha$ is not derivable in intuitionistic logic, whereas
$\prin\neg\neg(\alpha\lor\neg\alpha)$ is, it follows (see \S\ref{provability} below) that
$\Box q\lor\Box\neg\Box q$ is not derivable in QS4, whereas $\neg\Box\neg(\Box q\lor\Box\neg\Box q)$ is.
\metameta

\begin{example}\label{QS4-admissible}
An example of an admissible non-derivable rule of QS4 is \FM{\frac{\neg\Box p\land\neg\Box\neg p}{\bot}\text.}
Indeed, in a topological model, $|\fm{\neg\Box p\land\neg\Box\neg p}|=\Cl |\fm p|\but\Int |\fm p|$, that is,
the topological frontier $\Fr |\fm p|$.
Since $\Fr |\fm p|$ may well be the entire space (for instance, the frontier of $\Q$ as a subset of $\R$ is $\R$),
the rule is not derivable.

On the other hand, if $F$ is a formula, we claim that $\not\turnstile\neg\Box F\land\neg\Box\neg F$.
Indeed, in a topological model with a valuation and an assignment such that all predicate variables 
are interpreted by the entire space, $|F|$ will be either empty or the entire space; in either case, 
$|\neg\Box F\land\neg\Box\neg F|=\Fr(|F|)=\emptyset$.
\end{example}

\subsubsection{QS4 and formal provability}\label{formal provability}

In what sense does the modality $\Box$ of S4 and QS4 correpond to provability?
Clearly, provability in the sense of the ``classical BHK'' (see \S\ref{confusion}) satisfies the laws
and inference rules of QS4 --- albeit in a trivial way, with $\Box F\tofrom F$ being valid.

However, $\Box$ fails to directly represent formal provability in any theory containing Peano Arithmetic.
Indeed, reflection and necessitation imply $\turnstile\Box(\Box\bot\to\bot)$; if $\Box$ is to be interpreted
as formal provability, then this would be saying that the consistency of the theory can be proved within 
the theory, contradicting G\"odel's second incompleteness theorem (see \cite{Sm}).
Thus G\"odel spoke of derivability ``understood not in a particular system, but in the absolute sense
(that is, one can make it evident)'' and suggested an intended reading of $\Box$ as ``is provable in
the absolute sense'' \cite{Goe1}.
This, of course, raises the question of what, if anything, $\Box$ ever has to do with actual proof theory.

In fact, it is more natural to compare $\Box$ not with G\"odel's predicate of formal provability 
$\Bew(\# F):=\exists m\Prov(m,\# F)$, where $\Prov(m,n)$ is a formula expressing the predicate that 
the $m$th proof is a proof of the $n$th formula in the G\"odel numbering, but rather with its strengthened 
variants such as Mostowski's
\[M(\# F):=\exists m\big(\Prov(m,\# F)\land\neg\Prov(m,\#\ 0\!=\!1)\big),\] which is
equivalent to $\Bew(\# F)$ but the equivalence is not a theorem of Peano Arithmetic.
Indeed, the consistency formula for Mostowski's predicate, $\neg M(\#\ 0\!=\!1)$, is trivially a theorem of
Peano Arithmetic --- just like $\neg\Box\bot$ is trivially a theorem of QS4.

However, neither $M(\# F)$ nor any other arithmetic provability predicate $B(\# F)$
expressible in Peano Arithmetic can interpret the modality $\Box$ of QS4.
Indeed, by L\"ob's theorem (see \cite{Sm}*{4.1.1}, \cite{JdeJ}), any arithmetic provability predicate of
the form $\boxdot F=B(\# F)$ satisfying the (schematic) principles $\boxdot(F\to G)\to(\boxdot F\to\boxdot G)$ 
and $\boxdot F\to\boxdot\boxdot F$ and the (schematic) rule $F\,/\,\boxdot F$ also satisfies L\"ob's principle:
$\boxdot(\boxdot F\to\boxdot F)\to\boxdot F$.
On the other hand, the contrapositive of L\"ob's principle for $F=\bot$:
$\neg\boxdot\bot\to\neg\boxdot(\neg\boxdot\bot)$ is a modal version of G\"odel's second incompleteness theorem,
which is inconsistent with the theorem $\neg\Box\bot$ of QS4.
(It should be noted that L\"ob's theorem implies also the original second incompleteness theorem, since
$\boxdot F=\Bew(\# F)$ does satisfy the three principles in the hypothesis of L\"ob's theorem.)

The next natural step in the quest for a proof-theoretic interpretation of the modality $\Box$ of QS4
would be to get rid of the constraint imposed by working within a fixed first-order theory (such as
the Peano Arithmetic).
If we extend a theory $T$ that admits a G\"odel numbering by the reflection principle
$R_T:=\Bew_T(\# F)\to F$, where $\Bew_T$ denotes G\"odel's provability predicate for $T$, then 
the formula expressing the consistency of $T$, $\Consis(T):=\Bew_T(\#\ 0\!=\!1)\to 0\!=\!1$ follows, and so does
$\Consis\big(T+\Consis(T)\big)=\Bew_T\big(\#\Bew_T(\#\ 0\!=\!1)\big)\to 0\!=\!1$; and so on.
If we now write $\boxdot F=\Bew_T(\# F)$, then the principles $\boxdot F\to F$, $\boxdot F\to\boxdot\boxdot F$ 
and $\boxdot(F\to G)\to(\boxdot F\to\boxdot G)$ are derivable in $T+R_T$.
However, the rule $F\,/\,\boxdot F$ is not even admissible for $T+R_T$: while $\neg\boxdot\bot$ is provable 
(due to $R_T$), $\boxdot(\neg\boxdot\bot)$ is not (since $\neg\boxdot(\neg\boxdot\bot)$ is provable, due 
to L\"ob's theorem for $T$).

An analysis of the latter failure suggests that maybe instead of finding a single proof-theoretic 
interpretation of $\Box$, different occurrences of $\Box$ could be interpreted differently, according to their 
``depth'' in a formula.
This effectively means that we are lifting a formula of QS4 to a formula of a {\it polymodal} logic,
which has infinitely many modalities: $[0],\,[1],\,[2],\,\dots$, by labelling every occurrence of $\Box$
with some number; and then seeking an arithmetical interpretation of this polymodal logic, with each $[n]$
interpreted by formal provability in a theory $T_n$, where each $T_{i+1}$ contains at least $T_i$ and
its reflection principle $R_{T_i}$.
Because of L\"ob's theorem, we do not want to label an occurrence of $\Box$ as $[n]$ if its scope
contains an occurrence of $\Box$ already labelled as $[k]$ for some $k\ge n$.

\formulas
But it turns out that such a lift is impossible.
Indeed, let us consider the formula $\Diamond(p\to\Box p)$, which can be
equivalently presented as $\neg\Box(p\land\neg\Box p)$.
To see that it is derivable in QS4, observe that $\Box(p\land\neg\Box p)$ implies, firstly, $\Box p$
and, secondly, $\Box\neg\Box p$, which in turns implies $\neg\Box p$.
Thus $\Box(p\land\neg\Box p)$ implies $\bot$, as desired.
As observed in \cite{Ku10}, this argument has ``a flavor of self-referentiality''; indeed, if we try to
repeat it with labels, we get that $[1](p\land\neg [0]p)$ implies, firstly, $[1]p$ and, secondly,
$[1]\neg[0] p$, which in turn implies $\neg[0]p$.
Yet there is nothing contradictory about this, as $[0]p$ implies $[1]p$ but not conversely.
\metameta

A different approach was suggested informally by G\"odel, but remained unpublished until 1995; a closely related, 
but somewhat different approach was suggested and formalized independently by Art\"emov (see \cite{JdeJ}*{\S10}).
Their idea is to lift a formula of S4 to a polymodal logic whose modalities are indexed by ``proof terms'',
which are to be interpreted by actual proofs in the arithmetical model.
Art\"emov's approach takes advantage of the fact that Peano Arithmetic proves $\Prov(n,\# F)$
for some numeral $n$ if and only if it proves $F$ (see \cite{Sm}*{3.2.4}).
Thus it can be said that on Art\"emov's interpretation, the modality $\Box$ of S4 represents ``existence
of proofs'' in Peano Arithmetic (and not just in its extensions by reflection principles), where ``proofs''
have the usual meaning of formal proofs (except that Art\"emov needs one ``proof'' to be able to prove
several formulas), but ``existence'' is understood in an explicit sense, not expressible internally
in Peano Arithmetic.

Another well-known approach is that $\Box F$ can be represented in the form $F\land\boxdot F$,
where it is $\boxdot$ that should be interpreted as an arithmetic provability predicate.
(Topologically, $\neg\boxdot\neg$ corresponds to Cantor's derivative operation; see \cite{Esa}
for further details.)
In particular, by well-known results of (i) Solovay and (ii) Kuznetsov--Muravitsky, Goldblatt and Boolos 
(see \cite{Tr1}, \cite{Esa}, \cite{JdeJ}), $\Box F:=F\land\Bew(\# F)$ satisfies the laws and 
inference rules of QS4, and besides those only one additional quantifier-free principle, namely the following 
principle Grz: $\Box(\Box(F\to\Box F)\to F)\to F$.
This does not contradict L\"ob's theorem because $\Bew(\# F)\land F$ is not expressible in the form $B(\# F)$.

\subsubsection{Provability translation}\label{provability}
Tarski models motivate the following embedding of intuitionistic logic into QS4: interpret every problem
variable as a predicate variable prefixed by a box (because each problem variable corresponds to an 
open set in a Tarski model with chosen valuation and assignment), replace $\ab$ and $\triv$ by
$\bot$ and $\top$ respectively, and insert a box in front of every implication and in front of every 
universal quantifier (because of their special treatment in Tarski models).
This is a version of G\"odel's translation (see below) due to McKinsey--Tarski,
as extended to the first-order case by Rasiowa--Sikorski, Maehara and
Prawitz--Malmn\"as (see \cite{Tr1}, \cite{N}, \cite{FF}).

Thus every formula $\Phi$ in the language of intuitionistic logic corresponds to a formula $\Phi'$
in the language of QS4 (namely, the formula obtained from $\Phi$ by the modification just described) so that 
the interpretation of $\Phi$ in a Tarski model of intuitionistic logic is the same as the interpretation 
of $\Phi'$ in the corresponding topological model of QS4.
  
It is straightforward to check directly that if $\prin\Phi$ is a law of intuitionistic logic, 
then $\prin\Phi'$ is derivable in QS4; and if $\Phi_1,\dots,\Phi_k\,/\,\Psi$ is 
an inference rule of intuitionistic logic, then $\Phi_1',\dots,\Phi_k'\,/\,\Psi'$
is derivable in QS4. Alternatively, this follows from the strong completeness of QS4 with respect to
topological models.

The converse implications follow from the strong completeness of intuitionistic logic with respect 
to Tarski models. 
(A syntactic verification is also possible, but it relies on Gentzen's cut elimination theorem; 
see e.g.\ \cite{Min2}*{p.~489}.)

Of course, one could as well insert a box in front of {\it every} subformula of $\Phi$ (in addition to
replacing $\ab$ and $\triv$ by $\bot$ and $\top$); the resulting formula $\Phi_\Box$ will be equivalent 
to $\Phi'$ in QS4:

\formulas
\begin{proposition} \label{QS4-McKinsey}
The following holds in QS4:

(a) $\turnstile(\Box p\land\Box q)\Tofrom\Box(\Box p\land\Box q)$;

(b) $\turnstile(\Box p\lor\Box q)\Tofrom\Box(\Box p\lor\Box q)$;

(c) $\turnstile\exists \tr x\,\Box p(\tr x)\Tofrom\Box\exists \tr x\,\Box p(\tr x)$.
\end{proposition}
\metameta

The topological counterpart of this proposition is saying essentially
that finite intersections and arbitrary unions of open sets are open.
Any proof of this fact from the axioms of a topological space in terms
of $\Int$ should translate into a proof of \ref{QS4-McKinsey}.

%

By combining the above, we obtain

\begin{theorem}\label{box-translation}
Given a formula $\Phi$ in the language of intuitionistic logic, insert a box in front of every subformula, 
and replace $\ab$, $\triv$ by $\bot$, $\top$ respectively.
Then the resulting formula $\Phi_\Box$ is derivable in QS4 if and only if
$\Phi$ is derivable in intuitionistic logic.

Moreover, a rule $\Phi_1,\dots,\Phi_k\,/\,\Psi$ is derivable in intuitionistic logic if and only if the rule 
$(F_1)_{\Box},\dots,(F_k)_{\Box}\,/\,G_{\Box}$ is derivable in QS4.
\end{theorem}

G\"odel's original translation, as extended to the first-order case (cf.\ \cite{Tr1}) is as follows:
given a formula $\Phi$, insert a box {\it after} each instance of $\lor$, $\to$ and $\exists$.
Then the resulting formula $\Phi''$ is derivable in QS4 if and only if
$\Phi$ is derivable in intuitionistic logic.
Indeed, $\Phi''$ is derivable in QS4 if and only if $\Box\Phi''$ is.
On the other hand, $\Phi_\Box\tofrom\Box\Phi''$ is derivable in QS4:

\formulas
\begin{proposition} \label{QS4-Goedel}
The following holds in QS4:

(a) $\turnstile\Box (p\land q)\Tofrom\Box(\Box p\land\Box q)$;

(b) $\turnstile\Box\forall \tr x\, p(\tr x)\Tofrom\Box\forall \tr x\,\Box p(\tr x)$.
\end{proposition}
\metameta

Here each of the two judgements is proved similarly to their
topological counterpart, $\Int\bigcap_i S_i=\Int\bigcap_i\Int S_i$.
Indeed, here $\supset$ follows from $S_i\supset\Int S_i$.
Conversely, we have $\bigcap_i S_i\subset S_n$ for each $n$, whence
$\Int\bigcap_i S_i\subset\Int S_n$.
Since this holds for each $n$, we get
$\Int\bigcap_i S_i\subset\bigcap_i\Int S_i$.
Applying $\Int$ to both sides completes the proof.

Here is a sample application of the $\Box$-translation:

\begin{proposition}\label{Stone-Cech} Jankov's principle does not imply the Strong Markov Principle.
\end{proposition}

\begin{proof}
Let us consider the Tarski model in the one-point compactification $\N_+$ of the countable discrete space $\N$
with $\D=\N$.
Let us consider the valuation $\pval$ such that $|\fm q|^\pval(n)=\{n\}$ and $|\fm r|^\pval=\{\infty\}$, 
where $\fm q$ is unary and $\fm r$ is nullary, and any other problem variable is also interpreted by 
$\{\infty\}$ on every input.
Since the union of the clopen sets $\{n\}$ is not regular open, (SMR) and in particular (SMP) are not valid
with respect to the valuation field $\left<\pval\right>$ (since their formulas are not valid with respect 
to $\pval$).
To show that, nevertheless, Jankov's principle is valid with respect to $\left<\pval\right>$, it is convenient 
to use the $\Box$-translation.
Let us note that every formula of QS4 is $\pval$-equivalent to (i.e., has the same interpretation under $\pval$
as) a formula whose only problem variables are $\fm q$, $\fm r$.

If $F$ is a quantifier-free closed $m$-formula of QS4 containing no $\Box$'es, then either $\infty$
is contained in $|F|^\pval(\vec n)$ for all $\vec n$, or $\infty$ is contained in $|\neg F|^\pval(\vec n)$ 
for all $\vec n$; also, either $|F|^\pval(\vec n)$ contains a punctured neighborhood of $\infty$ for all $\vec n$, 
or $|\neg F|^\pval(\vec n)$ contains a punctured neighborhood of $\infty$ for all $\vec n$.
Indeed, both assertions are true if $F$ is either one of $\fm q$, $\neg\fm q$, $\fm r$, $\neg\fm r$, and
remain true under conjunction and disjunction; thus they follow by considering the disjunctive normal
form of $F$.

If $|F|^\pval(\vec n)$ contains $\infty$ and $|\neg F|^\pval(\vec n)$ contains a punctured neighborhood 
of $\infty$ for all $\vec n$, then $\Box F$ is $\pval$-equivalent to $F\land\neg\fm r$; 
and in the remaining cases either $|F|^\pval(\vec n)$ contains a punctured neighborhood of $\infty$ 
for all $\vec n$, or $|\neg F|^\pval(\vec n)$ contains a neighborhood of $\infty$ for all $\vec n$, whence 
$\Box F$ is $\pval$-equivalent to $F$.
It follows by induction that every quantifier-free closed $m$-formula is $\pval$-equivalent to
a quantifier-free closed $m$-formula without $\Box$'es.
In other words, every quantifier-free formula is $\pval$-equivalent to a quantifier-free formula 
without $\Box$'es.

If $F$ is a quantifier-free $1$-formula without $\Box$'es, by considering its conjunctive normal form 
we get that $F(\tr x)$ is $\pval$-equivalent to a formula of the form 
$\big(\fm q(\tr x)\lor G\big)\land\big(\neg\fm q(\tr x_1)\lor H\big)$, where $G$ and $H$ contain no 
occurrences of $\tr x$.
Consequently $\forall \tr x\, F(\tr x)$ is $\pval$-equivalent to
$\big(\forall \tr x\,\fm q(\tr x)\lor G\big)\land\big(\forall \tr x_1\,\neg\fm q(\tr x)\lor H\big)$,
which is $\pval$-equivalent to $G\land\big(\fm r\lor H\big)$.
By similarly using the disjunctive normal form, we get that $\exists \tr x\,F(\tr x)$
is $\pval$-equivalent to a formula that contains no quantification over $\tr x$.
It follows by induction that every formula is $\pval$-equivalent to a quantifier-free formula 
without $\Box$'es.

Finally, by combining the above results, we get that for each $m$-formula $F$,
every set $|F|^\pval(n_1,\dots,n_m)$ is either clopen or differs from a clopen set only in the point $\infty$.
(So, for example, it cannot be the set of all even integers.)
Then for each $\pval'\in\left<\pval\right>$ and for each $m$-ary predicate variable $p$,
every set $|p|^{\pval'}(n_1,\dots,n_m)$ is either clopen or differs from a clopen set only in the point $\infty$.
In either case, it is easy to see that $|\Box\neg\Box p|^{\pval'}(n_1,\dots,n_m)$ is clopen.
Thus the $\Box$-translation of $\neg\fm\alpha$, where $\fm\alpha$ is nullary, is interpreted by 
a clopen set under $\pval'$.
Hence the $\Box$-translation of $\prin\neg\fm\alpha\lor\neg\neg\fm\alpha$,
a form of Jankov's principle (see \ref{Jankov's logic}), is valid with respect to $\left<\pval\right>$.
\end{proof}

 \subsection{Independence of connectives and quantifiers}\label{independence}

\begin{theorem} Neither of the intuitionistic connectives and quantifiers $\ab$, $\land$, $\lor$, $\to$, 
$\exists$, $\forall$ can be expressed in terms of the other ones.
\end{theorem}

Our argument attempts to focus on properties that distinguish the connectives, but occasionally
resorts to specific models.
For a proof-theoretic approach see \cite{Pr1}*{p.\ 59}.

\begin{proof}
It is easy to see that $\ab$ is not expressible in terms of the other connectives and quantifiers by considering
a Tarski model in a nonempty space $X$ with the valuation that associates $X$ to every problem variable and 
every tuple of elements of the domain.
Then every formula not involving $\ab$ is interpreted by $X$, but in every Tarski model, $\ab$ is interpreted
by $\emptyset$.

In intuitionistic logic, decidability is preserved under $\lor$ (using (\ref{deMorgan1})), $\land$
(using (\ref{deMorgan2})) and $\to$ (using (\ref{implication0}) and (\ref{implication-coexp})).
In the model showing that the Weak Rule of Omniscience does not imply the Rule of Omniscience
(see \S\ref{omniscience}), decidability is also preserved under $\forall$, but not under $\exists$.
Thus $\exists$ is not expressible in terms of the other connectives and quantifiers.

Due to the intuitionistically derivable principles (\ref{neg-neg-and}), (\ref{neg-neg-imp}), (\ref{neg-neg-or}),
(\ref{DNS-converse}) and (\ref{SMP-converse}), and to the independence of the $\neg\neg$-Shift Principle
(see \S\ref{shift}), $\forall$ is the only connective or quantifier through which $\neg\neg$ cannot be 
``pushed inside''.
Thus $\forall$ is not expressible in terms of the other connectives and quantifiers.

Due to the same intuitionistically derivable principles and to the independence of Jankov's principle (see \S\ref{Markov}),
$\lor$ and $\exists$ are the only connectives or quantifiers through which $\neg\neg$ cannot be 
``pushed outside''.
In models with singleton domain, however, $\neg\neg$ can be pushed outside through $\exists$
(but generally not through $\lor$).
Thus $\lor$ cannot be expressed in terms of other connectives or quantifiers.

(Alternatively, one can combine the use of the singleton domain with the observation that $\lor$ is 
the only connective that does not preserve stability; or with the observation that $\lor$ is the only 
connective that is not preserved under the $\neg\neg$-translation, but if it were expressible in terms
of other connectives and quantifiers in intuitionistic logic, then the same expression would also work
in classical logic.)

If $\to$ is expressible in terms of $\land$, $\lor$, $\forall$, $\exists$ and $\ab$, then so is $\neg$, 
and if so, such an expression would also work in classical logic (with $\ab$ replaced by $\bot$), and 
consequently also in QS4.
Now $\land$, $\lor$ and $\exists$ are preserved under the $\Box$-translation (in the McKinsey--Tarski 
form, see \ref{QS4-McKinsey}), $\ab$ turns into $\bot$, and $\forall$ is preserved from the viewpoint 
of models with singleton domain (Tarksi models of intuitionisitc logic and topological models of QS4).
But $\neg$ is not preserved: $\neg\Box\fm p\to\Box\neg\Box\fm p$ is not derivable in QS4, since in 
topological models, the complement of an open set does not have to be open.

Finally, given a collection $\Gamma_1,\dots,\Gamma_n$ of closed formulas such that $\turnstile\neg\neg\Gamma_i$
for each $i$, any closed formula $\Xi$ obtained from $\Gamma_1,\dots,\Gamma_n$ using 
$\lor$, $\to$ and $\ab$ satisfies either $\turnstile\Gamma_i\to\Xi$ for some $i$ or $\turnstile\Xi\to\ab$.
Indeed, arguing by induction, we may assume that $\Xi$ is either $\Phi\lor\Psi$ or $\Phi\to\Psi$ or $\ab$, 
where either $\turnstile\Gamma_i\to\Phi$ or $\turnstile\Phi\to\ab$, and either $\turnstile\Gamma_j\to\Psi$ 
or $\turnstile\Psi\to\ab$.
If $\Xi=\ab$, then $\turnstile\Xi\to\ab$.
In the case $\Xi=\Phi\lor\Psi$, if $\turnstile\Gamma_i\to\Phi$, then $\turnstile\Gamma_i\to\Xi$; 
if $\turnstile\Gamma_j\to\Psi$, then $\turnstile\Gamma_j\to\Xi$; and 
if $\turnstile\Phi\to\ab$ and $\turnstile\Psi\to\ab$, then $\turnstile(\Phi\lor\Psi)\to\ab$.
It remains to consider the case $\Xi=\Phi\to\Psi$.
If either $\turnstile\Gamma_j\to\Psi$ or $\turnstile\Phi\to\ab$, then $\turnstile\Gamma_j\to\Xi$.
So we may assume that $\turnstile\Gamma_i\to\Phi$ and $\turnstile\Psi\to\ab$. 
Then $\turnstile\Xi\to\neg\Phi$ and $\turnstile\neg\neg\Gamma_i\to\neg\neg\Phi$.
By the hypothesis $\turnstile\neg\neg\Gamma_i$, so $\turnstile\neg\neg\Phi$.
Hence $\turnstile\Xi\to\ab$.

Now as far as models with singleton domain are concerned, the use of $\exists$ and $\forall$ does not yield 
anything new --- but the use of $\land$ does, by considering the Tarski model in $\R$ (with singleton domain) 
and a valuation with $|\fm\alpha|=\R\but\{0\}$ and $|\fm\beta|=\R\but\{1\}$, which suffices to detect that 
$\Xi=\fm\alpha\land\fm\beta$ satisfies neither $\turnstile\fm\alpha\to\Xi$ nor $\turnstile\fm\beta\to\Xi$ 
nor $\turnstile\Xi\to\ab$.
\end{proof}

\section{Sheaf-valued models} \label{sheaves}

The first two sections below contain, in particular, a review of some basic sheaf theory.
This goes only slightly beyond an exposition of a standard material (for sheaves of sets), which can be found 
in Bredon's \cite{Bre} and/or Godement's \cite{God} books (mostly for sheaves of modules).

\subsection{Sheaves}

\subsubsection{Sheaves and presheaves}
A {\it sheaf} (of sets) on a (topological) space $B$ is a map $\F\:E\to B$ that is 
a local homeomorphism (i.e.\ each $e\in E$ has a neighborhood $O$ in $E$ such that 
$\F(O)$ is a neighborhood of $f(e)$ and $\F|_O\:O\to \F(O)$ is a homeomorphism).
Clearly, such an $\F$ is an open map; in particular, $\Supp\F:=\F(E)$ is always open in $B$.
If $\F\:E\to B$ is a sheaf and $b\in B$, the point-inverse $\F_b:=\F^{-1}(b)$ is called 
the {\it stalk} of $\F$ at $b$.

A {\it section} of a sheaf or more generally of a continuous map $f\:E\to B$ over 
an open subset $U\subset B$ is a continuous map $s\:U\to E$ such that $fs=\id_U$.
Clearly, every section is a homeomorphism onto its image.
A section $s\:U\to E$ is said to {\it extend} a section $t\:V\to E$ if
$V\subset U$ and $t=s|_V$.
If $\F\:E\to B$ is a sheaf, it is easy to see that a base of topology on $E$ is given by 
the images of all sections of $\F$ (over all open subsets of $B$).

A {\it morphism} $\F\to\F'$ between sheaves $\F\:E\to B$ and
$\F'\:E'\to B$ is a continuous map $f\:E\to E'$ such that $\F'f=\F$.
Clearly, such a map is itself a local homeomorphism.
A {\it constant} sheaf is the projection $B\x\Xi\to B$, where $\Xi$ is
a discrete space.
A sheaf is {\it locally constant} if it is a covering map, that is,
its restriction over a sufficiently small neighborhood of every point
of $B$ is isomorphic to a constant sheaf.

There are two basic reasons why a sheaf can fail to be
locally constant:
\begin{itemize}
\item Even locally, sections need not extend, as is the case for
the {\it characteristic sheaf} $\chi_U\:U\to B$, where $U$ is
an open subset of $B$ and $\chi_U$ is the inclusion map.
\item Even locally, sections may extend non-uniquely, as is the case
for the {\it amalgamated union} $\chi_U\cup_{\chi_W}\chi_V\:
U\cup_W V\to B$, where $U$ and $V$ are open subsets of $B$ and $W$
is an open subset of $U\cap V$.
Here $U\cup_W V$ is the quotient space of the disjoint union
$U\sqcup V=U\x\{0\}\cup V\x\{1\}\subset B\x\{0,1\}$
by the equivalence relation $(b,0)\sim (b,1)$ if $b\in W$, and
$\chi_U\cup_{\chi_W}\chi_V$ sends the class of $(b,i)$ to $b$.
\end{itemize}
Note that $U\cup_W V$ is non-Hausdorff, as long as some
$b\in (U\cap V)\but W$ lies in the closure of $W$ --- in this case
$(b,0)$ and $(b,1)$ have no disjoint neighborhoods in $U\cup_W V$.
In general, if $s\:U\to E$ and $t\:V\to E$ are sections of a sheaf
$\F\:E\to B$ over open sets, then $\{b\in U\cap V\mid s(b)=t(b)\}$
is open.
If $E$ happens to be Hausdorff, then this set must also
be closed in $U\cap V$.

A {\it presheaf} (of sets) on a (topological) space $B$ is a contravariant 
functor from the category of open subsets of $B$ and their inclusions into
the category of sets.
A {\it morphism} $\phi\:F\to G$ of presheaves on $B$ is a natural
transformation of functors, that is, a collection of maps
$\phi_U\:F(U)\to G(U)$ for all open $U\subset B$ that commute with the restriction maps
$F(j)\:F(U)\to F(V)$ and $G(j)\:G(U)\to G(V)$ for all inclusions $j\:V\emb U$ between open subsets of $B$.

Given a sheaf, or more generally a continuous map $f\:X\to B$, its {\it presheaf of sections}
$\sigma f$ assigns to an open subset $U\subset B$ the set $(\sigma f)(U)$
of all continuous sections of $f$ over $U$, and to an inclusion $j\:V\emb U$ of
open subsets of $B$ the map $(\sigma f)(U)\to(\sigma f)(V)$ given by
restriction of sections.
In general, for an arbitrary presheaf $F$, elements of each $F(U)$
are called {\it sections} over $U$, and the image of a section
$s\in F(U)$ under $F(j)\:F(U)\to F(V)$ is called the {\it restriction}
of $s$ and is denoted $s|_V$.
If $U$ is an open subset of $B$, the presheaf $\Char U:=\sigma\chi_U$
can be described by $(\Char U)(V)=\{\id_V\}$ if $V\subset U$ and
$(\Char U)(V)=\emptyset$ otherwise.

If $F$ is a presheaf on a space $B$ and $b\in B$, a {\it germ} of $F$ at $b$
is an element of the direct limit (=colimit, inductive limit)
$F_b:=\underset{\longrightarrow}{\lim}F(U)$ over all open neighborhoods
$U$ of $b$ ordered by inclusion.
In other words, a germ of $F$ at $b$ is an equivalence class of sections 
$s_U\in F(U)$ over open neighborhoods $U$ of $b$, where $s_U\sim t_V$ 
if $U\cap V$ contains an open neighborhood $W$ of $b$ such that $s_U|_W=t_V|_W$.
In particular, $F_b=\emptyset$ if and only if $F(U)=\emptyset$
for all open neighborhoods $U$ of $b$.

A presheaf $F$ on a space $B$ gives rise to the {\it sheaf of germs} (or sheafafication) of $F$, 
denoted $\gamma F$, whose stalk $(\gamma F)_b$ at $b$ is the set of germs $F_b$;
a base of topology on $E=\bigcup_{b\in B} F_b$ consists of sets of the form 
$O_{s,U}=\{\phi_{U,b}(s)\mid b\in U\}$, where $U\subset B$ is open and 
$\phi_{U,b}\:F(U)\to F_b$ is the natural map.
The open set $\Supp F:=\Supp \gamma F$ can be described as the union of
all open subsets $U$ of $B$ such that $F(U)\ne\emptyset$.

\subsubsection{Sheaf conditions}
Given a sheaf $\F$, it is easy to see that $\gamma\sigma\F$ is isomorphic to $\F$.
Given a presheaf $F$, we have an obvious morphism $F\to\sigma\gamma F$, whose component 
$F(U)\to\sigma\gamma F(U)$ is a bijection if and only if the following two {\it sheaf conditions} hold:
\begin{itemize}
\item If $\{U_\alpha\}$ is an open cover of $U$, and $s,t\in F(U)$ are such that $s|_{U_\alpha}=t|_{U_\alpha}$
for all $\alpha$, then $s=t$.
\item If $\{U_\alpha\}$ is an open cover of $U$, and $s_\alpha\in F(U_\alpha)$ are such that
$s_\alpha|_{U_\alpha\cap U_\beta}=s_\beta|_{U_\alpha\cap U_\beta}$
for all $\alpha$, $\beta$, then there exists an $s\in F(U)$ such that
each $s_\alpha=s|_{U_\alpha}$.
\end{itemize}
Thus sheaves on $B$ are in one-to-one correspondence with presheaves
on $B$ satisfying the sheaf conditions for every open $U\subset B$.
For example, a presheaf of the form $\sigma f$ always satisfies
the sheaf conditions, so it is isomorphic to $\sigma\F$, where
$\F=\gamma\sigma f$ is called the {\it sheaf of germs of sections}
of $f$.
Moreover, every sheaf morphism $\F\to\G$ clearly determines a presheaf
morphism $\sigma\F\to\sigma\G$; and every presheaf morphism $F\to G$
clearly determines a sheaf morphism $\gamma F\to\gamma G$.
It follows that the category of sheaves on $B$ can be identified with
a full subcategory of the category of presheaves on $B$.

\begin{example} If $\pi\:B\x X\to B$ is the projection, then
$\gamma\sigma\pi$ is the {\it sheaf of germs of continuous functions}
$B\to X$.
Note that the total space $E$ of this sheaf is non-Hausdorff already
for $B=X=\R$, since the germs at $0$ of the constant function $f(b)=0$
and the function $g(b)=\max(b,0)$ are distinct, but have no disjoint neighborhoods.
\end{example}

If $C$ is a category, a {\it $C$-valued presheaf} on a space $B$ is a contravariant functor
from the category of open subsets of $B$ and their inclusions to $C$.
When $C$ is a concrete category (for example, the category $Ab$ of abelian groups), each
$C$-valued presheaf on $B$ is also a presheaf of sets on $B$.

Conversely, if $F$ is a presheaf of sets on a space $B$, it determines an $Ab$-valued presheaf $\Z F$ on $B$
by setting $\Z F(U)$ to be the free abelian group $\Z[F(U)]$ with generator set $F(U)$ for every open 
$U\subset B$, and $\Z F(j)\:\Z[F(U)]\to\Z[F(V)]$ to be the linear extension of $F(j)\:F(U)\to F(V)$ 
for every inclusion $j\:V\emb U$ between open subsets of $B$.
It is clear that $F$ satisfies the sheaf conditions if and only if $\Z F$ does.

On the other hand, the sheaf conditions for an $Ab$-valued presheaf $F$ are easily seen to be equivalent to 
exactness of the sequence
\[0\to F(U)\xr{\phi}\prod_\alpha F(U_\alpha)\xr{\psi}\prod_{(\alpha,\beta)} F(U_\alpha\cap U_\beta)\tag{$*$}\]
for any open subset $U\subset B$ and any open cover $\{U_\alpha\}$ of $U$,
where $\phi(s)=(\alpha\mapsto s|_{U_\alpha})$ and
$\psi(\alpha\mapsto s_\alpha)=\left((\alpha,\beta)\mapsto
s_\alpha|_{U_\alpha\cap U_\beta}-s_\beta|_{U_\alpha\cap U_\beta}\right)$.

\begin{example} Let $f\:E\to B$ be a monotone map between posets.
If we endow $E$ and $B$ with the Alexandroff topology, then $f$ is continuous.
Every poset endowed with the Alexandroff topology is weakly homotopy equivalent to its order complex \cite{Mc},
so singular (co)homology of a poset, absolute or relative, may be thought of as (co)homology of its order complex.
For each $i=0,1,\dots$ let us define presheaves $L^i$ and $L_i$ on $B$ by $L^i(U)=H^i(f^{-1}(U);\,\Z)$ and 
$L_i(U)=H_i(E,\,E\but f^{-1}(U);\,\Z)$; the restriction maps $L^i(U)\to L^i(V)$ and $L_i(U)\to L_i(V)$ are
the usual forgetful homomorphisms.
By the Mayer--Vietoris spectral sequence, the sequence ($*$) is exact for $F=L^0$ and, if the order complex of $E$ 
is of dimension $n<\infty$, then also for $F=L_n$.
Thus $L^0$ and $L_n$ satisfy the sheaf conditions.

In general, the sheaf of germs $\H^i(f)=\gamma L^i$ is known as the $i$th Leray sheaf of $f$, and when 
$f=\id_X\:X\to X$, the sheaf of germs $\H_i(X)=\gamma L_i$ is known as the $i$th local homology sheaf of $X$.
If $X$ is the face poset of a graph, then $\H_1(X)$ has stalk $\Z$ at every edge and stalk $\Z^{d-1}$ 
at every vertex of degree $d$.
If $X$ is the face poset of a triangulation of a closed $n$-manifold, then $\H_n(X)$ is locally constant, 
with stalks $\Z$; it will be constant precisely when the manifold is orientable.
If $B$ and $E$ are the face posets of simplicial complexes, and $f$ comes from a simplicial map that triangulates 
a bundle with fiber a closed orientable $k$-manifold, then $\H^k(f)$ is locally constant, with stalks $\Z$; 
it will be constant precisely when the bundle is orientable.
\end{example}

\subsubsection{Sheaves and problems}\label{equations}
Let $f\:X\to B$ be a continuous map (for instance, this could be a real polynomial $f\:\R\to\R$ or a complex
polynomial $f\:\C\to\C$, $f(x)=a_nx^n+\dots+a_0$), and consider the following parametric problem
$\Gamma_f(b)$, $b\in B$:
\begin{center}
{\it Find a solution of the equation $f(x)=b$.}
\end{center}
Thus each point-inverse $f^{-1}(b)$ is nothing but the set of solutions of $\Gamma_f(b)$.

If the parameter $b$ represents experimental data (which inevitably contains noise), then any talk
about the exact value of $b$ (e.g.\ whether $b$ is a rational or irrational number when $B=\R$)
is pointless.
It would be not so unreasonable, however, to assume that the value of $b$ can in principle be
determined up to arbitrary precision.
Indeed, the point of continuous maps is that if a point in the domain is known only up to a sufficient
accuracy, it still makes sense to speak of its image --- it will be known up to a desired accuracy.
This leads us to consider only {\it stable solutions} of the equation $f(x)=b$, that is, such $x_0\in X$
that $f(x_0)=b$ and there exists a neighborhood $U$ of $b$ in the space $B$ over which $f$
has a section.
Indeed, if the parameter $b$ is only known to us up to a certain degree of precision, we should only stick
with a solution $x_0$ of $f(x)=b$ as long as we can be certain that it would not disappear when our
knowledge of $b$ improves.%
\footnote{This may remind the reader of Brouwer's idea that ``all functions are continuous'',
as well as some of Poincar\'e's writings; a more rigorous exposition of these ideas can be found in
M. Escardo's notes \cite{Es} and in P. S. Novikov's physical interpretation of intuitionisic logic
in terms of weighting of masses \cite{N}*{\S III.4}.}
Thus we will consider the following parametric problem $\Delta_b$, $b\in B$:
\begin{center}
{\it Find a stable solution of the equation $f(x)=b$.}
\end{center}
If is easy to see that the set $\F_b\subset f^{-1}(b)$ of all stable solutions of $f(x)=b$
is nothing but the stalk of the sheaf $\F=\gamma\sigma f$ of germs of sections of the map $f$.

For example, if $f\:\R\to\R$ is a polynomial with real coefficients, then stable solutions
of $f(x)=b$ are all the non-repeated real roots of the polynomial $f(x)-b$, along with
its other roots of odd multiplicity.
Indeed, $f$ is a local homeomorphism at each of those roots, but has a local extremum at every root
$x_0$ of even multiplicity.
Hence the perturbed equation $f(x)=b\pm\eps$ (with $+$ chosen in the case of a local maximum,
and $-$ in the case of a local minimum) has no real solutions near $x_0$ as long as $\eps>0$ is
sufficiently small.

If $f\:\C\to\C$ is a polynomial with complex coefficients, then stable solutions of the equation
$f(z)=b$, where $b\in\C$, are precisely all the non-repeated complex roots of the polynomial $f(z)-b$.
Indeed, at a root $z_0$ of multiplicity $d>1$ we can write $f(z)=b+k(z-z_0)^d+o\left((z-z_0)^d\right)$.
So $f$ is locally a $d$-fold branched covering, and therefore admits no section over any neighborhood
of $b=f(z_0)$.
In other words, even though the perturbed equation $f(z)=b+\eps$ does have solutions near $z_0$
whenever $|\eps|$ is sufficiently small, we cannot choose such a solution continuously
depending on $\eps$; but with discontinuous choice, we can never be sure that our chosen explicit
solution will persist when our knowledge of $b$ improves.

Similar behaviour occurs if we consider the problem of finding a root of a complex polynomial
$f_b=a_nz^n+\dots+a_0$ as a function of all its coefficients $b=(a_0,\dots,a_n)$.
Set $\Sigma_n=\{(a_0,\dots,a_n,x)\in\C^{n+2}\mid a_nx^n+\dots+a_0=0\}$,
and let $\phi\:\Sigma_n\to\C^{n+1}$ be the projection along the last coordinate.
Then $\phi^{-1}(b)$ is the set of roots of $f_b$, and the stalk
$(\sigma\phi)_b$ is the set of its stable roots.
By varying just $a_0$ like above we see that no multiple root is stable;
clearly, all non-repeated roots are stable.

\subsection{Operations on sheaves}

\subsubsection{Product and coproduct}
Given a collection of presheaves $F^i$, $i\in I$, their {\it product} $\prod F^i$ and {\it coproduct} 
$\bigsqcup F^i$ are defined respectively by $U\mapsto\prod_i F^i(U)$ and by $U\mapsto\bigsqcup_i F^i(U)$,
with the obvious restriction maps.

If the presheaves $F^i$ satisfy the sheaf conditions, then clearly so does $\prod F^i$.
On the other hand, $\bigsqcup F^i$ generally does not satisfy the second sheaf condition, since for 
a disconnected open set $U\sqcup V$ the sheaf conditions for each $F^i$ imply 
$F^i(U\sqcup V)\simeq F^i(U)\x F^i(V)$, so that 
$\big(\bigsqcup_i F^i\big)(U\sqcup V)\simeq\bigsqcup_i F^i(U)\x F^i(V)$, which is generally not the same as
$\big(\bigsqcup_i F^i\big)(U)\x\big(\bigsqcup_i F^i\big)(V)=
\big(\bigsqcup_i F^i(U)\big)\x\big(\bigsqcup_i F^i(V)\big)$.

Given a collection of sheaves $\F^i$, $i\in I$, their {\it product} $\prod\F^i:=\gamma\prod\sigma\F^i$ 
satisfies $\sigma\prod\F^i\simeq\prod\sigma\F^i$ by the above.
Their {\it coproduct} is defined by $\bigsqcup\F^i=\gamma\bigsqcup\sigma\F^i$.

It is not hard to see that the operations just introduced are precisely the product and the coproduct in 
the categories of sheaves and presheaves.

The coproduct of sheaves $\F^i\:E_i\to B$ can be alternatively
described as the sheaf $\bigsqcup\F^i\:\bigsqcup E_i\to B$,
where $\bigsqcup E_i$ is the disjoint union of spaces, and
$\bigsqcup\F^i$ is defined by $\F^i$ on the $i$th summand.
Thus $\big(\bigsqcup\F^i\big)_b\simeq\bigsqcup(\F^i)_b$
(isomorphism in the category of sets) for each $b\in B$.
More generally, $\big(\bigsqcup F^i\big)_b\simeq\bigsqcup (F^i)_b$
for any presheaves $F^i$ since direct limit commutes with coproducts.

The product of {\it two} sheaves $\F\:E\to B$ and $\F'\:E'\to B$ can be
alternatively described as the fiberwise product
$\F\x\F'\:E\x_B E'\to B$, where $E\x_B E'$ is the subspace of
the Cartesian product $E\x E'$ consisting of all pairs $(e,e')$
such that $\F(e)=\F'(e')$.
Thus $(\F\x\F')_b\simeq\F_b\x\F_b'$.
More generally, $(F\x F')_b\simeq F_b\x F_b'$ for any presheaves $F$
and $F'$ since direct limit commutes with finite products.

The fiberwise product $\big(\prod_B E_i\big)\to B$ of
the sheaves $\F^i\:E_i\to B$ need not be a sheaf when $I$ is infinite.
For, although every $x_i\in(\F^i)_b$ has a neighborhood
in $E_i$ that maps homeomorphically onto a neighborhood $U_i$ of
$b$ in $B$, the fiberwise product may fail to be a local homemorphism
at the point $(i\mapsto x_i)$ if $\bigcap_i U_i$ is not
a neighborhood of $b$.

Given presheaves $F^i$, $i\in I$, for each $n\in I$, the projection $\prod_i F^i\to F^n$ induces
a map $\big(\prod_i F^i\big)_b\to (F^n)_b$ for each $b\in B$.
These in turn yield a map
$$\script G_b\:\left(\prod F^i\right)_b\to\prod\big(F^i\big)_b.$$
An element of the right hand side is a collection of germs of local sections $s_i\:U_i\to E_i$, whereas 
an element of the left hand side is a germ of a collection of local sections $U\to E_i$.
The germs of $s_i$'s may fail to define such a single germ if $\bigcap U_i$ is not a neighborhood of $b$; 
thus $\script G_b$ is not surjective in general.
On the other hand, given two collections of local sections, $\sigma_i\:U\to E_i$ and $\tau_i\:V\to E_i$ 
such that each $\sigma_i$ equals $\tau_i$ on some open set $W_i\subset U\cap V$, there might
be no single open set $W\subset U\cap V$ such that each $\sigma_i|_W=\tau_i|_W$; thus $\script G_b$ is 
not injective in general.%
\footnote{Of course, it is easy to provide a specific example, thus refuting an assertion in 
\cite{God}*{\S II.1.10}.}

It is clear, however, that $\script G_b$ is a bijection when $B$ is an Alexandroff space.

Let us note that the existence of $\script G_b$ implies that if $\big(\prod F^i\big)_b\ne\emptyset$, then each
$(F^i)_b\ne\emptyset$.
The converse fails:

\begin{proposition} \label{Tarski forall fails}
There exist sheaves $\F^1,\F^2,\dots$ on $\R$ such that
each $(\F^i)_b\ne\emptyset$ for each $b\in\R$ but
$\prod\F^i=\chi_\emptyset$.
\end{proposition}

\begin{proof} It suffices to consider, for each $i$, the cover
$C_i=\{(r-\frac1i,r+\frac1i)\mid r\in\R\}$ of $\R$ by all open intervals
of length $2/i$, and the sheaf $\F^i=\bigsqcup_r\chi_{(r-1/i,\,r+1/i)}$.
\end{proof}

\subsubsection{The support of a product}
A topological space $X$ is of {\it covering dimension zero} if every open cover of $X$ has a refinement
consisting of disjoint open sets; and of {\it inductive dimension zero} if it has a base of topology
consisting of clopen sets.
For $T_4$ spaces (i.e.\ $T_1$ spaces where disjoint closed subsets are separated by neighborhoods),
covering dimension zero implies inductive dimension zero (see \cite{En}*{6.2.5, 6.2.6}).
The converse holds for spaces whose topology has a countable base (see \cite{En}*{6.2.7}).
In particular, the two notions are equivalent for separable metrizable spaces, so we may speak of
{\it zero-dimensional separable metrizable spaces}.
Examples of such spaces include the Cantor set and the Baire space, as well as their arbitrary subspaces.
Indeed, it is easy to see that if a space is of inductive dimension zero, then so is every its subspace.

\begin{lemma}\label{0-dim}
If $\F$ is a sheaf on a zero-dimensional separable metrizable space $B$, and $U$ is
an open subset of $B$, then $\F$ has a section over $U$ if and only if $U\subset\Supp\F$.
\end{lemma}

\begin{proof} The ``only if'' implication follows from the definitions.
Given an open $U\subset\Supp\F$, each $x\in U$ is contained in an open set $V_x$ over which $\F$
has a section.
Without loss of generality $V_x\subset U$.
Since $U$ is zero-dimensional, the cover $\{V_x\mid x\in U\}$ of $U$ has a refinement $\{U_\alpha\}$
consisting of disjoint open sets.
Since $\F$ has a section over each $U_\alpha$ and is a sheaf (not just a presheaf), by combining
these sections we get a section over $U$.
\end{proof}

\begin{proposition} \label{product sheaf}
If $F^i$ are presheaves on a space $B$, then $\Supp\prod F^i\subset\Int\big(\bigcap\Supp F^i\big)$.
The reverse inclusion holds in each of the following cases:

(a) $B$ is an Alexandroff space;

(b) $B$ is a zero-dimensional separable metrizable space, and $F^i=\sigma\F^i$ for some $\F^i$.
\end{proposition}

\begin{proof}
The inclusion follows from the existence of the map $\script G_b$ and from the fact that the left hand side
is an open set.
The reverse inclusion for holds for Alexandroff spaces since their $\script G_b$ is bijective.
If $B$ is a zero-dimensional separable metrizable space and $x$ lies in an open set $U$ contained in
$\Supp\F^i$ for each $i$, then by \ref{0-dim} each $\F^i$ has a section over $U$.
Hence so does their product, whose support therefore contains $x$.
\end{proof}

\subsubsection{Hom-sheaf}
If $F$ is a presheaf on a space $B$, and $A\subset B$ is open, then open subsets of $A$ are open in $B$, 
so $F$ determines the restriction presheaf $F|_A$ on $A$.
In the case of a sheaf $\F\:E\to B$, the corresponding construction
generalizes to the case of an arbitrary continuous map $f\:A\to B$,
which induces the inverse image $f^*\F\:f^*E\to A$, where
$f^*E=\{(a,e)\in A\x E\mid f(a)=\F(e)\}$ and $(f^*\F)(a,e)=a$.

If $F$ and $G$ are presheaves on $B$, let us consider the set
$\Hom(F,G)(U)$ of morphisms $F|_U\to G|_U$.
If $V$ is an open subset of $U$, every such morphism $\phi$ restricts to
a morphism $\phi|_V\:F|_V\to G|_V$, and thus we get the
{\it exponential} presheaf $\Hom(F,G)$.
If $G$ satisfies the sheaf conditions, then clearly so does $\Hom(F,G)$.
This enables one to define the sheaf
$\Hom(\F,\G):=\gamma\Hom(\sigma\F,\sigma\G)$, called the {\it sheaf
of germs of morphisms} between the sheaves $\F$ and $\G$.
In general, for any presheaves $F$ and $G$, the stalk
$\Hom(F,G)_b$ consists of germs of morphisms $F\to G$, that is,
of equivalence classes of local morphisms
$\phi_U\:F|_U\to G|_U$ over open neighborhoods $U$ of $b$, where
$\phi_U\sim\phi_V$ if $U\cap V$ contains an open neighborhood
$W$ of $b$ such that $\phi_U|_W=\phi_V|_W$.
A germ of morphisms determines a map between germs, so we get a map
$$\script F_b\:\Hom(F,G)_b\to\Hom(F_b,G_b),$$ where $\Hom(S,T)$ denotes
the set of all maps between the sets $S$ and $T$.
This map is neither injective nor surjective in general:

\begin{proposition} \label{Tarski implication fails}
For any set $S$ there exist sheaves $\F$, $\G$ on $\R$ such that
$\F_0=\G_0=\emptyset$ and $\Hom(\F,\G)_0\simeq S$.
\end{proposition}

\begin{proof}
Let $U=\R\but\{0\}$ and let $\F=\chi_U$.
Let $\G$ be the disjoint union of copies of $\F$ indexed by $S$.
Then for every open neighborhood $V$ of $0$ the set of sheaf
morphisms $\F|_V\to\G|_V$ is in a bijection with $S$, and its elements
persist under restriction.
\end{proof}

Let us mention two more peculiar examples:

\begin{example} \label{Tarski implication fails2} There exist sheaves $\F$, $\G$ on $\R^2$ such that
$\Supp\F=\Supp\G=\R^2\but\{0\}$, yet $\Hom(\F,\G)_{0}=\emptyset$.
Indeed, let $U=\R^2\but\{0\}$ and let $\F=\chi_U$.
Let $\G\:\tilde U\to U\subset\R^2$ be the nontrivial double covering over $U$.
Then for every open neighborhood $V$ of $0$ there exist no sheaf
morphisms $\F|_V\to\G|_V$.
\end{example}

\begin{example} There exist sheaves $\F^n$ and $\G$ on $\R$, $n\in\N$, such that each
$\Hom(\F^n,\G)\ne\chi_\emptyset$ but
$\Hom(\bigsqcup_n\F^n,\G)=\chi_\emptyset$.
Indeed, let $\F^n=\chi_{(-\infty,-1/n)\cup(1/n,\infty)}$ and let $\G=\chi_\emptyset$.
Then $\Hom(\F^n,\G)\simeq\chi_{(-1/n,1/n)}$.
It follows that $\Hom(\bigsqcup_n\F^n,\G)\simeq
\prod_{n\in\N}\chi_{(-1/n,1/n)}\simeq\chi_\emptyset$.
\end{example}

\subsubsection{The support of a Hom-sheaf}

\begin{proposition} \label{Hom-sheaf}
$\Supp\Hom(F,G)\subset\Int\big(\Supp G\cup (B\but\Supp F)\big)$ for presheaves $F$, $G$
on a space $B$, and the reverse inclusion holds in each of the following cases:

(a) $G$ is the characteristic presheaf of an open set, in which case so is $\Hom(F,G)$;

(b) $B$ is a zero-dimensional separable metrizable space, and $G=\sigma\G$ for some $\G$.
\end{proposition}

\begin{proof} The inclusion follows from the existence of the map $\script F_b$ and from the fact that
the left hand side is an open set.
Suppose that $x$ lies in an open set $U$ contained in $\Supp G\cup (B\but\Supp F)$.
Thus if $y\in U$ and $y\in\Supp F$, then $y\in\Supp G$; that is, $U\cap\Supp F\subset U\cap\Supp G$.
Then for each open $V\subset U$ such that $V\not\subset\Supp G$ we have $V\not\subset\Supp F$, hence
$F(V)=\emptyset$ and there is a unique map $F(V)\to G(V)$.
Let us now consider an open $V\subset U$ such that $V\subset\Supp G$.

If $G=\Char(\Supp G)$, then $G(V)=\{\id_V\}$ and there is a unique map $F(V)\to G(V)$.

If $B$ is a zero-dimensional separable metrizable space, and $G=\sigma\G$, then by \ref{0-dim}
$\G$ has a section $s$ over $U\cap\Supp G$, and we have the constant map $F(V)\to G(V)$ onto $\{s|_V\}$.

In either case, the constructed maps $F(V)\to G(V)$ commute with the restriction maps $G(V)\to G(W)$
and $F(V)\to F(W)$ for each open $W\subset V\subset U$, and so determine a morphism of presheaves
$F|_U\to G|_U$.
The germ at $x$ of this morphism is an element of the stalk $\Hom(F,G)_x$, so $x$ is contained in
$\Supp\Hom(F,G)$.

Moreover, in the case (a) we have shown that for every open $U$ contained in $\Supp G\cup (B\but\Supp F)$,
we have a unique morphism $F|_U\to G|_U$; thus $\Hom(F,G)(U)$ contains precisely one element.
If $U$ is not contained in $\Supp G\cup (B\cup\Supp F)$, then $U$ contains a point $y$ such that
$y\notin\Supp G$ and $y\in\Supp F$.
The latter implies that $y$ lies in an open set $V$ such that $F(V)\ne\emptyset$.
Without loss of generality $V\subset U$.
On the other hand, since $y\notin\Supp G$, we have $G(V)=\emptyset$.
Hence there exists no map $F(V)\to G(V)$, and consequently no morphism $F|_U\to G|_U$.
Thus $\Hom(F,G)(U)=\emptyset$.
It then follows that $\Hom(F,G)$ is isomorphic to $\Char\big(\Supp G\cup (B\cup\Supp F)\big)$.
\end{proof}

\begin{corollary} \label{sheaf negation}
$\Hom(F,\Char\emptyset)\simeq\Char\big(\Int(B\but\Supp F)\big)$.
\end{corollary}

The inclusion in \ref{Hom-sheaf} cannot be reversed for sheaves over an Alexandroff space:

\begin{proposition} \label{Alexandroff-implication}
There exist sheaves $\F$, $\G$ over a certain Alexandroff space $B$ such that
$\Supp\F=\Supp\G$, yet $\Supp\Hom(\F,\G)\ne B$.
\end{proposition}

For future purposes we fix some terminology regarding sheaves on Alexandroff spaces.
If $B$ in an Alexandroff space and $p\in B$, let $\left<p\right>=\{q\in B\mid q\ge p\}$, the minimal
open set containing $p$.
Given a sheaf $\F\:E\to B$, for each $p,q\in B$ with $p\le q$ let $\F_{pq}\:\F_p\to\F_q$
denote the restriction map $(\sigma\F)(j)\:(\sigma\F)(\left<p\right>)\to(\sigma\F)(\left<q\right>)$
corresponding to the inclusion $j\:\left<q\right>\emb\left<p\right>$.
Let us note that the topology on $E$ is an Alexandroff topology; its corresponding preorder is given by
$x\le y$ if $x\in\F_p$, $y\in\F_q$, $p\le q$ and $y=\F_{pq}(x)$.

\begin{proof} Let $B$ be the product of two copies of the poset $0<1$.
Let $\F_{(0,0)}=\emptyset$ and $\F_{(1,0)}=\F_{(0,1)}=\F_{(1,1)}=\{z\}$.
Let $\G_{(0,0)}=\emptyset$, $\G_{(1,0)}=\{x\}$, $\G_{(0,1)}=\{y\}$ and $\G_{(1,1)}=\{x,y\}$,
where the maps $\G_{pq}$ are treated as inclusions.
Clearly, $\Hom(\F,\G)_{(0,0)}=\emptyset$.
\end{proof}

In the general case, the inclusion in \ref{Hom-sheaf} can be partially reversed:

\begin{proposition} \label{sheaf implication}
$\Supp G\cup\Int(B\but\Supp F)\subset\Supp\Hom(F,G)$.
\end{proposition}

\begin{proof} If $b\in\Supp G$, then $G(U)$ is nonempty for some open neighborhood $U$ of $b$.
Then we can pick some $s_U\in G(U)$ and define a morphism $F|_U\to G|_U$
by sending every $F(V)$, $V\subset U$, onto $\{s_U|_V\}\subset G(V)$.
If $B\but\Supp F$ contains an open neighborhood $U$ of $b$,
then $F(V)=\emptyset$ for all open $V\subset U$.
Hence a morphism $F|_U\to G|_U$ is defined by observing that all objects
in the domain are empty.
\end{proof}

\subsection{Sheaf-valued models} Let us summarize the behavior of stalks under operations on presheaves:

\begin{proposition}\label{strong BHK-verification}
If $F$, $G$ and $F^i$, $i\in I$, are presheaves on a space $B$, $b\in B$,
then
\begin{itemize}
\item $(F\x G)_b\simeq F_b\x G_b$;
\item $(F\sqcup G)_b\simeq F_b\sqcup G_b$;
\item there is a map $\script F_b\:\Hom(F,G)_b\to\Hom(F_b,G_b)$;
\item $(\Char\emptyset)_b=\emptyset$;
\item $\big(\bigsqcup F^i\big)_b\simeq\bigsqcup (F^i)_b$;
\item there is a map
$\script G_b\:\big(\prod F^i\big)_b\to\prod (F^i)_b$.
\end{itemize}
\end{proposition}

This looks suspiciously similar to the BHK interpretation (see, in particular, \S\ref{confusion}),
which suggests that sheaves might have something to do with intuitionistic logic.

\subsubsection{Sheaf-valued structures}
Let $B$ be a topological space and $\D$ a set.
The set $\O$ will consist of all sheaves over $B$, and the function $\ocf:\O\to\{\Top,\Bot\}$ will send 
a sheaf $\F$ to $\Top$ if and only if it has a global section, i.e.\ $(\sigma\F)(B)\ne\emptyset$.

Thus, to specify a valuation, each problem variable $\alpha$ of arity $n$ must be interpreted by 
a $\D^n$-indexed family of sheaves on $B$ (i.e., by a function $|\alpha|\:\D^n\to\O$).
Upon a choice of a valuation $\pval$, a closed formula will be interpreted by a sheaf on $B$, and 
more generally a closed $n$-formula $\Phi$ will be interpreted by a $\D^n$-indexed family of sheaves on $B$, 
to be denoted $|\Phi|$ or in more detail $|\Phi|^{\pval}$, or in still more detail $|\Phi|^\pval_B$.
(The valuation will be suppressed in the notation when it is clear from context.)
Upon a further choice of a variable assignment $\iass$, an arbitrary formula will be interpreted by 
a sheaf on $B$, and more generally an arbitrary $n$-formula $\Phi$ will be interpreted by 
a $\D^n$-indexed family of sheaves on $B$, to be denoted $|\Phi|$ or in more detail 
$|\Phi|_\iass$ or in still more detail $|\Phi|^\pval_\iass$.
(The assignment will be suppressed in the notation when it is clear from context.)
Validity of an $n$-formula $\Phi$ in the model, $\Turnstile\Phi$, means that the sheaf 
$|\Phi|^\pval_\iass(t_1,\dots,t_n)$ has a global section for each tuple $(t_1,\dots,t_n)\in\D^n$ for an
arbitrary valuation $\pval$ and assignment $\iass$.

Intuitionistic connectives and quantifiers are interpreted by the basic operations on sheaves:
\begin{itemize}
\item $|\Phi\lor\Psi|=|\Phi|\sqcup|\Psi|$;
\item $|\Phi\land\Psi|=|\Phi|\x|\Psi|$;
\item $|\Phi\to\Psi|=\Hom(|\Phi|,|\Psi|)$;
\item $|\ab|=\Char\emptyset$;
\item $|\exists \tr x\,\Xi(\tr x)|=\bigsqcup_{d\in D}|\Xi|(d)$;
\item $|\forall \tr x\,\Xi(\tr x)|=\prod_{d\in D}|\Xi|(d)$,
\end{itemize}
where $\Phi$ and $\Psi$ are formulas and $\Xi$ is a 1-formula, and $\iass$ and $\pval$ are fixed.

In exactly the same way we define {\it pre\/}sheaf-valued structures.
The only essential difference between sheaf- and presheaf-valued structures is in the interpretations 
of $\lor$ and $\exists$, due to the fact that $\sigma\F\sqcup\sigma\G\ne\sigma(\F\sqcup\G)$ in general.

\begin{theorem} \label{global1} Sheaf- and presheaf-valued structures are models of intuitionisitic logic.
\end{theorem}

Since (pre)sheaves form a topos, this is a special case of Palmgren's models of intuitionistic logic
in locally cartesian closed categories with finite sums \cite{Pa1}.

\begin{proof} We need to show that all laws and inference rules of the modified
Spector's derivation system in \S\ref{logics} are valid in the two kinds of structures.
It suffices to show that these laws and inference rules are valid with respect to 
an arbitrary predicate valuation $\pval$.
Thus we fix $\pval$, and $|\cdot|$ will mean $|\cdot|^\pval$.
For those laws and inference rules that involve neither $\lor$ nor $\exists$, the case of
sheaves reduces to the case of presheaves, and so need not be considered separately.

\ref{rule:modus ponens} (modus ponens).
We are given a global section $s$ of $|\fm\alpha|$ and a global section of $|\fm{\alpha\to\beta}|$, 
that is, a presheaf morphism $f\:|\fm\alpha|\to|\fm\beta|$.
Then $f(B)\:|\fm\alpha|(B)\to|\fm\beta|(B)$ sends $s$ to a global section of $|\fm\beta|$.

\ref{axiom:identity}, \ref{rule:composition}
The identity is a natural transformation, and composition of natural transformations is 
a natural transformation.

\ref{axiom:conjunction}, \ref{axiom:disjunction}
The projections of the product of presheaves onto its factors as well as the inclusions of summands into 
the coproduct of presheaves are presheaf morphisms.
Also, the inclusions of summands into the coproduct of sheaves are sheaf morphisms.

\ref{rule:conjunction}, \ref{rule:disjunction}
(Co)product of presheaves is their (co)product in the category of presheaves.
Also, coproduct of sheaves is their coproduct in the category of sheaves.

\ref{rule:exponential} (exponential law)
Given a presheaf morphism $f\:|\fm\alpha|\x|\fm\beta|\to|\fm\gamma|$, we need to construct a presheaf morphism
$g\:|\fm\alpha|\to\Hom(|\fm\beta|,|\fm\gamma|)$.
Let $g(U)$ send an $s\in|\fm\alpha|(U)$ to the following presheaf morphism $h\:|\fm\beta||_U\to|\fm\gamma||_U$.
If $V\subset U$, define $h(V)\:|\fm\beta|(V)\to|\fm\gamma|(V)$ by $t\mapsto f(V)(s|_V,t)$.
If $V'\subset V$, then $f(V')(s|_{V'},t|_{V'})=f(V)(s|_V,t)|_{V'}$ since $s|_{V'}=(s|_V)|_{V'}$ and 
$f$ is a natural transformation.
Hence $h=g(U)(s)$ is a natural transformation.
If $U'\subset U$, then $g(U')(s|_{U'})=g(U)(s)|_{U'}$ as presheaf morphisms 
$|\fm\beta||_{U'}\to|\fm\gamma||_{U'}$ due to $(s|_{U'})|_V=s|_V$ for each $V\subset U'$.
Thus $g$ is a natural transformation.

Conversely, given a $g$, we need to construct an $f$.
Let $f(U)\:|\fm\alpha|(U)\x|\fm\beta|(U)\to|\fm\gamma|(U)$ send a pair $(s,t)$ to $h(U)(t)$, where $h=g(U)(s)$.
If $U'\subset U$, then $f(U')(s|_{U'},t|_{U'})=h'(U')(t|_{U'})$, where $h'=g(U')(s|_{U'})$.
Since $g$ is a natural transformation, $h'=h|_{U'}$ as presheaf morphisms 
$|\fm\beta||_{U'}\to|\fm\gamma||_{U'}$, where again $h=g(U)(s)$.
In particular, $h'(U')=h(U')$ as maps $|\fm\beta|(U')\to|\fm\gamma|(U')$.
Since $h$ is a natural transformation, $h(U')(t|_{U'})=h(U)(t)|_{U'}$.
Thus $f(U')(s|_{U'},t|_{U'})=f(U)(s,t)|_{U'}$, whence $f$ is a natural transformation.

\ref{axiom:explosion} (explosion principle)
$\chi_\emptyset$ is the initial object in the category of sheaves.

\ref{axiom:forall}, \ref{axiom:exists}
For each $t\in D$, the projection onto the $t$th factor is a presheaf morphism
$\prod_{d\in D}|\fm\alpha|(d)\to|\fm\alpha|(t)$, and the inclusion onto the $t$th summand is 
a (pre)sheaf morphism $|\fm\alpha|(t)\to\bigsqcup_{d\in D}|\fm\alpha|(t)$.

\ref{rule:generalization} (generalization rule)
We need to show that if $\fm\alpha(\tr x)$ is valid in a presheaf-valued structure for every variable assignment,
then $\forall \tr x\,\fm\alpha(\tr x)$ is valid in the structure.

Indeed, the hypothesis says that for each $d\in D$ we have a global section of $|\fm\alpha|(d)$.
These yield a global section of the presheaf $\prod_{d\in D}|\fm\alpha|(d)$.

\ref{rule:forall}, \ref{rule:exists}
We need to prove that if $\forall \tr x\,\big(\fm\beta\to\fm\alpha(\tr x)\big)$ is valid in 
a (pre)sheaf structure, then so is $\fm\beta\to\forall \tr x\,\fm\alpha(\tr x)$; and if 
$\forall \tr x\,\big(\fm\alpha(\tr x)\to\fm\beta\big)$ is valid, then so is 
$\exists \tr x\,\fm\alpha(\tr x)\to\fm\beta$.

The hypothesis implies that for each $d\in D$ we have a presheaf morphism
$|\fm\beta|\to|\fm\alpha|(d)$ (resp.\ $|\fm\alpha|(d)\to|\fm\beta|$).
Since (co)product of presheaves is their (co)product in the category of presheaves, we get a presheaf morphism
$|\fm\beta|\to\prod_{d\in D}|\fm\alpha|(d)$ (resp.\ $\bigsqcup_{d\in D}|\fm\alpha|(d)\to|\fm\beta|$).
This works similarly for coproduct of sheaves.
\end{proof}

\subsubsection{Completeness}
We note the following consequence of \ref{strong BHK-verification}.

\begin{proposition} \label{BHK-verification}
If $F$, $G$ and $F^i$, $i\in I$, are presheaves on a space $B$, then
\begin{enumerate}
\item $\Supp (F\x G)=\Supp F\cap\Supp G$;
\item $\Supp(F\sqcup G)=\Supp F\cup\Supp G$;
\item $\Supp\Hom(F,G)\subset\Int\big(\Supp G\cup (B\but\Supp F)\big)$;
\item $\Supp(\Char\emptyset)=\emptyset$;
\item $\Supp\big(\bigsqcup F^i\big)=\bigcup\Supp F^i$;
\item $\Supp\big(\prod F^i\big)\subset\Int\big(\bigcap\Supp F^i\big)$.
\end{enumerate}
\end{proposition}

This is reminiscent of Tarski models, except that we have seen in the previous sections that 
the inclusions in (3) and (6) are strict in general, with (3) being strict even for some Alexandroff spaces.

However, by \ref{product sheaf} and \ref{Hom-sheaf}, the inclusions in (3) and (6) can be reversed
if the presheaves satisfy the sheaf conditions, and $B$ is a zero-dimensional separable metrizable space.
Also, by \ref{0-dim} a sheaf over such a space $B$ has a global section if and only if its support is
the entire $B$.
Thus from Dragalin's and Kramer's completeness results for Tarski models (see \S\ref{Tarski completeness}) 
we obtain

\begin{theorem} \label{completeness}
Intuitionistic logic is strongly complete with respect to its sheaf-valued model with countable domain
over the Baire space $\N^\N$, and complete with respect to its sheaf-valued model with countable domain
over any given zero-dimensional separable metrizable space with no isolated points.
\end{theorem}

Although over zero-dimensional separable metrizable spaces sheaf-valued models can be seen as
merely a proof-relevant conservative extension of Tarski models, in general they tend to be very
different from Tarski models.

In particular, from \ref{Tarski forall fails} it is easy to see that the Tarski model of
intuitionistic logic in $\R$ with $\D=\N$ is not obtainable by taking supports from 
the sheaf-valued model in $\R$ with $\D=\N$.
This example can be improved:

\begin{proposition} The Tarski model of zero-order intuitionistic logic in some Alexandroff space $B$ 
is not obtainable by taking supports from the sheaf-valued model in $B$.
\end{proposition}

\begin{proof}
Let $B=\{x,y,xy,yx,\hat 1\}$ with partial order generated by $x,y<xy$; $x,y<yx$ and $xy,yx<\hat 1$.
Then $\left<x\right>\cap\left<y\right>=\left<xy\right>\cup\left<yx\right>$, and therefore if
$\fm\alpha$, $\fm\beta$, $\fm\gamma$, $\fm\delta$ are interpreted respectively by $\left<x\right>$,
$\left<y\right>$, $\left<xy\right>$, $\left<yx\right>$, then
$(\fm\alpha\land\fm\beta)\to(\fm\gamma\lor\fm\delta)$ is interpreted by $B$.

On the other hand, if $\F$, $\F'$, $\G$, $\G'$ are sheaves with supports
$\left<x\right>$, $\left<y\right>$, $\left<xy\right>$ and $\left<yx\right>$ respectively,
then we have sheaf sections $\Char\left<x\right>\to\F$ and $\Char\left<y\right>\to\F'$
and sheaf morphisms $\G\to\Char\left<xy\right>$ and $\G'\to\Char\left<yx\right>$.
Consequently we have a sheaf morphism $\Hom(\F\x\F',\G\sqcup\G')\to
\Hom(\Char\left<x\right>\x\Char\left<y\right>,\Char\left<xy\right>\sqcup\Char\left<yx\right>)$.
Therefore from \ref{Alexandroff-implication},
$\Hom(\F\x\F',\G\sqcup\G')_x=\emptyset=\Hom(\F\x\F',\G\sqcup\G')_y$.
Thus if $\fm\alpha$, $\fm\beta$, $\fm\gamma$, $\fm\delta$ are interpreted respectively by
$\F$, $\F'$, $\G$, $\G'$, then $(\fm\alpha\land\fm\beta)\to(\fm\gamma\lor\fm\delta)$ is interpreted by
a sheaf with support $\{xy,yx,\hat 1\}$ and not $B$.
\end{proof}

\begin{example} \label{Harrop2}
Let us discuss the Kreisel--Putnam principle and Harrop's rule in the light of sheaf-valued models.
Let us fix a valuation in a sheaf-valued model of intuitionistic logic in a space $B$, and set 
$|\fm\beta|=\chi_{U_1}$, $|\fm\gamma|=\chi_{U_2}$ and $|\neg\fm\alpha|=\chi_U$.
Then by \ref{Hom-sheaf}(a), each $\Hom(\chi_U,\chi_{U_i})\simeq\chi_{V_i}$, where
$V_i=\Int\big(U_i\cup (X\but U)\big)$.
On the other hand, $|\fm\beta\lor\fm\gamma|\simeq\chi_{U_1}\sqcup\chi_{U_2}$.

In our previous discussion of the Kreisel--Putnam principle and Harrop's rule, which was based on
Tarski models (see \S\ref{Harrop1}), we had $B=\R^2$,
$U_1=\{(x,y)\mid x>0\}$, $U_2=\{(x,y)\mid y>0\}$ and $U=U_1\cup U_2$.
It is not hard to see that in this case,
$|\neg\fm\alpha\to\fm\beta\lor\fm\gamma|=\Hom(\chi_U,\chi_{U_1}\sqcup\chi_{U_2})$ is isomorphic
to the amalgamated union $\chi_{V_1}\cup_{\chi_V}\chi_{V_2}$, where $V=\{(x,y)\mid x<0\land y<0\}$.
The support of this amalgamated union is $\R^2\but\{(0,0)\}$, so it does not have a global section
(in fact, it does not even have a section over its support).
So this example no longer works to refute Harrop's rule.
However, it still applies to refute the Kreisel--Putnam principle
(even though in this case $\Supp|\neg\fm\alpha\to\fm\beta\lor\fm\gamma|$ coincides with
$\Supp|(\neg\fm\alpha\to\fm\beta)\lor(\neg\fm\alpha\to\fm\beta)|$), since there are no sheaf morphisms
$(\chi_{V_1}\cup_{\chi_V}\chi_{V_2})|_O\to
(\chi_{V_1}\sqcup\chi_{V_2})|_O$ for any neighborhood $O$ of $(0,0)$.

To refute Harrop's rule with a sheaf-valued model we can, however, take $B=\R^2$,
$U_1=\{(x,y)\mid x>0\land y>0\}$, $U_2=\{(x,y)\mid x<0\land y<0\}$ and $U=U_1\cup U_2$.
Then $V_1=\{(x,y)\mid x>0\lor y>0\}$ and $V_2=\{(x,y)\mid x<0\lor y<0\}$, so that
the support of $|(\neg\fm\alpha\to\fm\beta)\lor(\neg\fm\alpha\to\fm\gamma)|$ is $V_1\cup V_2=\R^2\but\{(0,0)\}$.
On the other hand, $|\neg\fm\alpha\to\fm\beta\lor\fm\gamma|\simeq\chi_{\R^2}$.

Let us now analyze the difference between the two models just considered.
If we write $U_0=\Int((U_1\cup U_2)\but (U_1\cap U_2))$, then
the sheaf morphism $\chi_{U_0}\to\chi_{U_1}\sqcup\chi_{U_2}$
induces a sheaf morphism $\Hom(\chi_U,\chi_{U_0})\to\Hom(\chi_U,\chi_{U_1}\sqcup\chi_{U_2})$.
By \ref{Hom-sheaf}(a), $\Hom(\chi_U,\chi_{U_0})\simeq\chi_{V_0}$, where $V_0=\Int(U_0\cup (X\but U))$.
Thus $|\neg\fm\alpha\to\fm\beta\lor\fm\gamma|$ will have a global section as long as $V_0$ is the entire space,
and in general, $\Supp|\neg\fm\alpha\to\fm\beta\lor\fm\gamma|$ contains $V_0$.
However, $V_0$ need not be contained in
$\Supp|(\neg\fm\alpha\to\fm\beta)\lor(\neg\fm\alpha\to\fm\gamma)|=V_1\cup V_2$ as shown by the second example.
On the other hand, if $V_0\subset V_1\cup V_2$, then $\Hom(\chi_U,\chi_{U_1}\sqcup\chi_{U_2})$
is isomorphic to the amalgamated union
$\chi_{V_1}\cup_{\chi_V}\chi_{V_2}$, where $V=\Int(X\but U)$.
In general, one can show that $\Hom(\chi_U,\chi_{U_1}\sqcup\chi_{U_2})\simeq
(\chi_{V_1}\cup_{\chi_V}\chi_{V_2})\cup_{\chi_W}\chi_{V_0}$, where $W=V_0\cap(V_1\cup V_2)$.
\end{example}

\begin{example}
Let us show that the Negative Constant Domain principle is not valid in some sheaf-valued model.
Let $\D=\N$ and let $B=(-\infty,0]\cup N\subset\R$, where $N=\{\frac1n\mid n\in\N\}$.
Then $N$ is a regular open set in $B$, so $\F=\chi_N$ could interpret a negated formula.
Let $\F^i=\chi_{B\but\{1/i\}}$.
Then $\prod_i\F^i\simeq\chi_{(-\infty,0)}$, so $\F\sqcup\prod_i\F^i\simeq\chi_{B\but\{0\}}$, which has 
no global section.
On the other hand, each $\F\sqcup\F^i$ has a global section, hence so does their product.
Thus there exist no sheaf morphisms $\prod_i(\F\sqcup\F^i)\to\F\sqcup\prod_i\F^i$.
\end{example}

Presheaf-valued models are not as good as sheaf-valued models:

\begin{proposition}\label{global2}
Presheaf-valued models satisfy the Negative Constant Domain Principle.
\end{proposition}

\begin{proof}
We need to construct a morphism of presheaves $\prod_i(F\sqcup F^i)\to F\sqcup\prod_i F^i$, where $F$ is of
the form $|\neg\Phi|$.
Given an open set $U$ and a collection $s=(s_i)$ of sections $s_i\in F(U)\sqcup F^i(U)$, we consider two cases.
If none of the $s_i$ actually belongs to $F(U)$, then we send $s$ via the inclusion
 $\prod_i F^i(U)\subset F(U)\sqcup\prod_i F^i(U)$.
If some $s_i$ belongs to $F(U)$, we send $s$ to the image of this $s_i$ under the inclusion 
$F(U)\subset F(U)\sqcup\prod_i F^i(U)$.
Since $F=|\neg\Phi|$ is the characteristic presheaf of some open set (in fact, of 
a regular open set), $F(U)$ contains at most one element, so the choice of $i$ is irrelevant.
This construction is clearly natural in $U$.
\end{proof}

\begin{proposition} \label{global3}
Presheaf-valued models satisfy the Rule of Omniscience and the Disjunction Property.
\end{proposition}

In sheaf-valued models, every formula interpreted by the characteristic
sheaf of a clopen set is decidable.
Hence the Disjunction Property fails if $B$ is disconnected; and the Rule of Omniscience can be
refuted similarly to \S\ref{Markov}.

\begin{proof}
To prove that the Disjunction Property is valid in presheaf-valued models it suffices to note that $F\sqcup G$ 
has a global section if and only if at least one of the presheaves $F$, $G$ has a global section.

In particular, if $\Phi$ is a formula such that $\Phi\lor\neg\Phi$ is valid with respect to some valuation 
$\pval$ and some assignment $\iass$ in a presheaf-valued model, then either $|\Phi|$ has a global 
section or $|\neg\Phi|$ has a global section; the latter amounts to saying that $|\Phi|=\Char\emptyset$.
If each $|\fm\alpha|(n)$, $n\in\N$, coincides with $\Char\emptyset$, then so does 
$|\exists\tr x\,\fm\alpha(\tr x)|=\bigsqcup_n|\fm\alpha|(n)$; and if some $|\fm\alpha|(n)$ has a global section, 
then so does $\bigsqcup_n|\fm\alpha|(n)$.
Thus $\exists\tr x\,\fm\alpha(\tr x)\lor\neg\exists\tr x\,\fm\alpha(\tr x)$ is valid with respect to $\pval$ 
and $\iass$ as long as $\fm\alpha(\tr x)\lor\neg\fm\alpha(\tr x)$ is.
\end{proof}

It is clear from the proofs of \ref{global2} and \ref{global3} that when the space $B$ is connected,
sheaf-valued models also satisfy the Negative Constant Domain Principle and the Disjunction Property
--- in contrast to Tarski models.

\subsubsection{Relation to Medvedev--Skvortsov models}\label{Medvedev}
Let $B$ be the two-element poset $0<1$ endowed with its
Alexandroff topology.
A presheaf (of sets) $F$ over $B$ such that $F(\emptyset)$ is a singleton is always a sheaf.
(In general, sheafafication of presheaves over $B$ has the effect of identifying the points
of $F(\emptyset)$ with each other.)

A sheaf $\F$ over $B$ consists of two stalks $\F_0=\sigma\F(\{0,1\})$ and $\F_1=\sigma\F(\{0\})$ and
a map $\F_{01}\:\F_0\to\F_1$ corresponding to the inclusion $\{0\}\subset\{0,1\}$.
We have $\Hom(\F,\G)_1=\Hom(\F_1,\G_1)$, whereas $\Hom(\F,\G)_0$ consists of commutative diagrams
$$\begin{CD}
\F_0@>\F_{01}>>\F_1\\
@VVV@VVV\\
\G_0@>\G_{01}>>\G_1.
\end{CD}$$
If $\F_{01}$ and $\G_{01}$ are inclusions, such a diagram amounts
to a map $\phi\:\F_0\to\G_0$ such that $\phi(\F_1)\subset\G_1$.

In the case where the restriction maps of the sheaves are inclusions, by
\ref{strong BHK-verification} we further have
$\left(\left(\bigsqcup_i\F^i\right)_1,\left(\bigsqcup_i\F^i\right)_0\right)=
\left(\bigsqcup_i (\F^i)_1,\bigsqcup_i (\F^i)_0\right)$ and, using that $B$ is
an Alexandroff space, also
$\left(\left(\prod_i\F^i\right)_1,\left(\prod_i\F^i\right)_0\right)=
\left(\prod_i (\F^i)_1,\prod_i (\F^i)_0\right)$.

This is reminiscent of the interpretation of connectives and quantifiers in Medvedev--Skvortsov models 
(see \S\ref{Medvedev-Skvortsov}).
To recover an actual version of Medvedev--Skvortsov models, let us fix a domain $\D$, and one set $X_{rk}$ 
corresponding to each problem variable $\var_k^{\0^r\too\1}$.
Let $X=\prod_{r\in\omega}\prod_{k\in\N}\Hom(\D^r,2^{X_{rk}})$, where $\omega=\{0,1,\dots\}$ and
$\N=\{1,2,\dots\}$. 
Here $2^S$ denotes the set of all subsets of $S$ and $\Hom(Y,Z)$ denotes the set of all maps $Y\to Z$.
Let us consider the Alexandroff space $X^+=X\cup\{\hat 1\}$, where $X$ is regarded as a discrete space, 
or alternatively as a poset with no comparable pair of elements, and $\hat 1$ is an additional element 
that is greater than all elements of $X_\Phi$.
Each $\vec Y=(Y_{rk})\in X$ determines the two-element subspace $\{\vec Y,\hat 1\}$ of $X^+$, which 
is homeomorphic to $B$.

Let us interpret the $k$th $r$-ary problem variable, $\var_k^{\0^r\too\1}$,
by the following $r$-ary family $\F^{rk}(t_1,\dots,t_r)$ of sheaves over $X^+$.
Let $\F^{rk}(\vec t)_{\hat 1}=X_{rk}$ and $\F^{rk}(\vec t)_{\vec Y}=Y_{rk}(\vec t)$, and let
$\F^{rk}(\vec t)_{\vec Y\hat 1}\:Y_{rk}(\vec t)\to X_{rk}$ be the inclusion map.

Let connectives and quantifiers be interpreted as in a sheaf-valued model, except for $\ab$,
which is interpreted by the sheaf $\F$ with $\F_{\vec Y}=\emptyset$ for each $\vec Y\in X$ and
$\F_{\hat 1}=\{\emptyset\}$.
As long as all the sets $X_{rk}$ are non-empty, it is easy to see that this yields a model of
intuitionistic logic with respect to the valuation field generated by the above valuation.

Given a closed formula $\Phi$, let $\var_{k_i}^{\0^{r_i}\too\1}$, $i=1,\dots,m$, be all the problem variables
that occur in $\Phi$.
Let $X_\Phi=\prod_{i=1}^m\Hom(\D^{r_i},2^{X_{r_ik_i}})$ and let $\pi_\Phi\:X\to X_\Phi$ be the projection
onto the subproduct, $\vec Y\mapsto (Y_{r_1k_1},\dots,Y_{r_mk_m})$.

Then $\Phi$ is interpreted by the sheaf $\G$, where, in the notation of \S\ref{Medvedev-Skvortsov}, 
the stalk $\G_{\hat 1}=S$ and the stalk $\G_{\vec Y'}=T(\vec Y)$ for each $\vec Y\in X_\Phi$ and each 
$\vec Y'\in\pi_\Phi^{-1}(\vec Y)$.
Thus the set of global sections of $\G$ can be identified with 
$\bigcap_{\vec Y\in X_\Phi}T(\vec Y)\subset S$, and in particular $\G$ has a global section if and only if
this intersection is nonempty.

\begin{remark}
By adding the least element $\hat 0$ to the poset $X^+$ and considering sheaves over
the resulting Alexandroff space $X^\pm=X\cup\{\hat 0,\hat 1\}$ one can also model a modification
of Medvedev--Skvortsov models (a non-equivariant version of L\"auchli models \cite{La}) where 
all $X_{ij}$ are taken to be a fixed set $U$, and $\ab$ is interpreted not by $(\{\emptyset\},\emptyset)$, 
but by $(U,Z)$ for some fixed $Z\subset U$ (possibly nonempty!), and only $Y_{rk}$ containing $Z$
are considered.
\end{remark}

\subsubsection{Modified BHK interpretation}\label{modified+}

Kolmogorov's interpretation of $\prin\fm\gamma\lor\neg\fm\gamma$ as ``the problem of finding 
a general method of solving the problem $\Gamma\lor\neg\Gamma$ for every contentful problem $\Gamma$''
contains {\it constructive} quantification over all problems; it does not seem to be compatible 
with usual, Tarski-style model theory, which only deals with {\it classical} quantification over 
all valuations.
(It is compatible with a certain alternative model theory constructed below; see \S\ref{extended+}.)
The following modification of Kolmogorov's interpretation overcomes this difficulty by replacing 
the informal (i.e., contentual) constructive quantifier with a combination of an informal classical quantifier 
(enabling compatibility with standard model theory) and a potentially formal constructive quantifier (preventing 
intuitionistic logic from degeneration into classical).
Related ideas are also found in \cite{LOD}*{Definition 4.2} (see also \cite{Ra}*{Ch.\ 7}).

In addition to a domain of discourse $\D$, we fix a ``hidden parameter'' domain $\E$ 
(just like in the modified Platonist interpretation of classical logic).
Now $\O$ is taken to be a class of problems with parameter in $\E$, consisting of a prescribed class 
of contentful (e.g.\ mathematical) primitive parametric problems and of composite parametric problems, which 
are obtained inductively from the primitive ones by using contentual intuitionistic connectives and 
quantifiers over arbitrary $\D$-indexed families of parametric problems.
For instance, if $\E=P\x P$, where $P$ is the set of all finite presentations of groups, primitive
parametric problems could be of the form ``Find an isomorphism between the groups given by 
presentations $p$ and $q$''.
In this case composite parametric problems would include ``Find an isomorphism between the groups given
by $p$ and $q$ or find a proof that it is impossible''.
Formal intuitionistic connectives and quantifiers are interpreted straightforwardly by the contentual ones.

The value of $\ocf\:\O\to\Qm=\{\Top,\Bot\}$ on a parametric problem $\Theta$ will be $\Top$ 
if and only if there exists a general method of solving the problem $\Theta(e)$ for all $e\in\E$.
What is meant by ``solving'' $\Theta(e)$ is defined explicitly for primitive $\Theta\in\O$ and 
all $e\in\E$
and is extended inductively to composite $\Theta\in\O$ and all $e\in\E$ by the usual BHK interpretation 
(see \S\ref{BHK}), separately for each $e\in\E$.
This determines a contentual interpretation of intuitionistic logic and its meta-logic, which 
we will call the {\it modified BHK} interpretation.

Let us note that the arguments of \S\ref{PEM} and \S\ref{Medvedev-Skvortsov} still work to show that
$\prin\fm\gamma\lor\neg\fm\gamma$ is not valid under the modified BHK interpretation.

\begin{example} Let us record the modified BHK interpretations of the judgements of admissibility, 
stable admissibility and derivability of a purely logical rule $\Gamma\,/\,\Delta$.
The tuples $\vec\phi$ and $\vec\Phi$ have the same meaning as in Example \ref{platonist example},
and $e$ is again understood to run over $\E$ as in Example \ref{platonist example2}. 
Like before, $\wn\Theta$ denotes the proposition ``The problem $\Theta$ has a solution''.

\begin{itemize}
\item Derivability: $\forall\D,\E,\O\ \forall\pval\,
\big(\forall\iass\,\wn\forall e\,|\Gamma|^\pval_\iass(e)\to\forall\iass\,\wn\forall e\,|\Delta|^\pval_\iass(e)\big)$;

\item Stable admissibility: 

$\forall\vec\Phi\,\forall\D,\E,\O\ 
\big(\forall\pval,\iass\,\wn\forall e\,|\Gamma[\vec\phi/\vec\Phi]|^\pval_\iass(e)
\To\forall\pval,\iass\,\wn\forall e\,|\Delta[\vec\phi/\vec\Phi]|^\pval_\iass(e)\big)$;

\item Admissibility:

$\forall\vec\Phi\,\big(\forall\D,\E,\O\ \forall\pval,\iass\,\wn\forall e\,|\Gamma[\vec\phi/\vec\Phi]|^\pval_\iass(e)
\To\forall\D,\E,\O\ \forall\pval,\iass\,\wn\forall e\,|\Delta[\vec\phi/\vec\Phi]|^\pval_\iass(e)\big)$.
\end{itemize}
\end{example}

\begin{remark} Let us observe that sheaf-valued models provide a formalization 
of the modified BHK semantics.
\end{remark}

\section{Semantics of the meta-logic} \label{meta-semantics}

\subsection{Generalized models} \label{generalized Tarski-style}

All the models considered so far are based on classical, moreover, two-valued interpretation of the meta-logic, 
in accordance with traditional model theory.
But the meta-logic is intuitionistic, and so is severely collapsed in such models. 

Given a logic $L$ based on a language $\L$ and given by a derivation system $\Ds$, a meta-$\L$-structure $M$ 
will be called a {\it generalized model} (resp.\ a {\it semi-generalized model}) of $L$ if it is a model 
(resp.\ a semi-standard model) of the meta-logic (see \S\ref{models}) and the meta-formula $\Ds$ is valid 
in $M$.
When $M$ is an $\L$-structure (i.e., the interpretation of the meta-logic is two-valued), this is the same
as a model in the usual sense (see \S\ref{models2}).

\subsubsection{Generalized Tarski models}\label{generalized Tarski}

Let $\D$ be a set, let $X$ and $Y$ be topological spaces, and let $f\:X\to Y$ be a continuous map.
Let $\O$ be the collection of all open sets in $X$, $\Qm$ the collection of all open sets in $Y$, and
$\ocf\:\O\to\Qm$ be defined by $\ocf(U)=\Int\big(Y\but f(X\but U)\big)$.%
\footnote{Let us note that if $X$ is compact, or more generally $f$ is a closed map, then $Y\but f(X\but U)$
is open.}
Let us interpret the intuitionistic connectives and quantifiers in $X$ in the usual way and 
the meta-connectives and the first-order meta-quantifier in $Y$ in the similar way (that is, in the way 
the corresponding connectives and the universal quantifier are interpreted in the Tarski model in $Y$ 
with domain $\D$).
The second-order meta-quantifier $\q^n$ is also interpreted similarly; namely, its interpretation $|\q^n|$
associates to a $\Hom(\D^n,\O)$-indexed family of open sets in $Y$ the interior of their intersection. 
Let $\wnf\:\Qm\to\{\Top,\Bot\}$ be defined by sending $Y$ to $\Top$ and any its proper subset to $\Bot$.

This yields a meta-structure for the language of intuitionistic logic, which will be called 
the {\it generalized Tarski meta-structure} in the map $f\:X\to Y$ with domain $\D$.

\begin{theorem} Generalized Tarski meta-structures are semi-generalized models of intuitionistic logic.
\end{theorem}

\begin{proof}
It is clear that all the laws of the Hilbert-style derivation system of intuitionistic logic
(see \S\ref{intlogic-modified}) are valid in this meta-structure (because they are valid in 
the two-valued structure given by the Tarski model in $X$ with domain $\D$), and that
all the meta-laws of the Hilbert-style formulation of the meta-logic (see \S\ref{hilbert-style meta-logic})
are also valid in it (because they are similar to the said laws of intuitionistic logic).
Also, the $\imp$-elimination meta-rule and the conditions (A) and (B) of \S\ref{models} are satisfied
for the same reasons that the {\it modus ponens} and the generalization rule of intuitionistic logic
are valid in Tarski models.
Thus it remains to check that the {\it modus ponens} and the generalization rule of intuitionistic logic
are valid in our meta-structure.

Let us check the validity of the {\it modus ponens} rule.
Given two open sets $U,V\subset X$, let $W=\Int\big(V\cup(X\but U)\big)$.
Since $\prin\fm{\alpha\land(\alpha\to\beta)\to\beta}$ is derivable in intuitionistic logic 
(or by direct reasoning), $U\cap W\subset V$.
Then $(X\but U)\cup(X\but W)=X\but(U\cap W)\supset X\but V$.
Hence $f(X\but U)\cup f(X\but W)=f\big((X\but U)\cup (X\but V)\big)\supset f(X\but V)$.
Since $\Int A\cap\Int B=\Int(A\cap B)$, we get $\ocf(U)\cap\ocf(W)\subset\ocf(V)$.

Let us check the validity of the generalization rule.
Given a family of open sets $U_x$, $x\in\D$, we need to show that
$\Int\bigcap_{x\in\D}\ocf(U_x)\subset\ocf(\Int\bigcap_{x\in\D} U_x)$.
Since $\Int\bigcap_x\Int S_x=\Int\bigcap_x S_x$ (see \S\ref{provability}), and the left hand side is open,
this is equivalent to
\[\Int\bigcap_{x\in\D}\big(Y\but f(X\but U_x)\big)\subset Y\but f\left(X\but\Int\bigcap_{x\in\D} U_x\right).\]
Since the complement to the interior is the closure of the complement, and the complement to an intersection
is the union of the complements, this is equivalent to
\[Y\but\Cl\bigcup_{x\in\D} f(X\but U_x)\subset Y\but f\left(\Cl\bigcup_{x\in\D}(X\but U_x)\right).\]
But since the image of a union is the union of the images, the latter is the same as
\[\Cl f\left(\bigcup_{x\in\D}(X\but U_x)\right)\supset f\left(\Cl\bigcup_{x\in\D}(X\but U_x)\right),\] 
which follows from the continuity of $f$.
(In fact, $f\:X\to Y$ is continuous if and only if $f(\Cl A)\subset\Cl{f(A)}$ for each $A\subset X$.)
\end{proof}

\begin{example} Let $c$ be the constant map from a topological space $X$ to a point.
It is easy to see that the generalized Tarski model of intuitionistic logic in $c$ with domain $\D$ is
nothing but the usual Tarski model in $X$ with domain $\D$.
\end{example}

\begin{example} \label{generalized classical}
If $X$ and $Y$ are topological spaces and $X$ is discrete, any map $f\:X\to Y$ is continuous.
Also, since usual Tarski models in $X$ satisfy the principle $\fm{\prin\alpha\lor\neg\alpha}$,
this principle is also valid in generalized Tarski models in $f$.
So they are in fact semi-generalized models of classical logic.
\end{example}

\begin{example} \label{interior-based example}
If $Y$ is a topological space and $Y^d$ is the same underlying set re-endowed with the discrete
topology, the previous example (Example \ref{generalized classical}) applies to the identity map $f\:Y^d\to Y$.
Let us note that $\ocf$ is interpreted by $S\mapsto\Int S$ in this case.
\end{example}

\begin{example} \label{LE'}
As a special case of the previous example (Example \ref{interior-based example}), we have semi-generalized
models of classical logic in any set $X$, with $\O=\Qm=2^Y$ and $\ocf=\id$.
\end{example}

\begin{example} There exists a semi-generalized model $W$ of classical logic with respect to a valuation field,
and principles $\prin F$ and $\prin G$ such that neither $\prin F$ nor $\prin G$ nor $\prin F\imp\prin G$ is 
valid in $W$. 

Indeed, let us consider $X=\{\mu,\nu\}$ and let
$F=\neg\fm{\big(\exists\tr x\, p(\tr x)\land\exists\tr x\,\neg p(\tr x)\big)}$ and $G=\bot$.
The usual Leibniz--Euler model $V$ in $X$ with respect to a certain valuation field $K$, constructed in 
Example \ref{deduction example}, yields a semi-generalized Leibniz--Euler model $W$ in $X$ with respect to $K$
by re-interpreting $\ocf$ and the meta-connectives and the meta-quantifiers as in Example \ref{LE'}.
Since $|\neg F|_V^\pval$ does not contain $\nu$ for each $\pval$ in $K$, but contains $\mu$ for some $\pval$ in $K$,
$|F|_V^\pval$ contains $\nu$ for each $\pval$ in $K$ and does not contain $\mu$ for some $\pval$ in $K$.
Hence $|\Rt F|_W^\pval$ also contains $\nu$ for each $\pval$ in $K$ and does not contain $\mu$ for some 
$\pval$ in $K$.
Consequently $|{\prin F}|_W^K=\{\nu\}$.
Then $|{\prin F}\imp\bot|_W^K=\{\mu\}$.
Thus neither $\prin F$ nor $\prin F\imp\bot$ is valid in $W$ with respect to $K$.
\end{example}

\subsubsection{Tarski/sheaf-valued models} \label{Tarski+sheaves}

Let us note a slightly different type of generalized models.
Let $\D$ be a set, let $X$ and $Y$ be topological spaces, and let $f\:X\to Y$ be a continuous map.
Let $\O$ be the collection of all open sets in $X$.
Let $\Qm$ the collection of sheaves on $Y$.
Let $\ocf\:\O\to\Qm$ be defined by setting $\ocf(U)$ to be the characteristic sheaf of the open set
$\Int\big(Y\but f(X\but U)\big)$.
Let us interpret the intuitionistic connectives and quantifiers in $X$ in the usual way (as in a Tarski model)
and the meta-connectives and the meta-quantifiers in $Y$ in the similar way (as in a sheaf-valued model).
Let us interpret the intuitionistic connectives and quantifiers in $X$ in the usual way (as in a Tarski model) 
and the meta-connectives and the first-order meta-quantifier in $Y$ in the similar way (as in a sheaf-valued 
model).
The second-order meta-quantifier $\q^n$ is also interpreted similarly; namely, its interpretation $|\q^n|$
associates to a $\Hom(\D^n,\O)$-indexed family of sheaves on $Y$ their product (in the category of sheaves). 
Let $\wnf\:\Qm\to\{\Top,\Bot\}$ be defined by sending $Y$ to $\Top$ and any its proper subset to $\Bot$.

This yields a meta-structure for the language of intuitionistic logic.
The arguments of the previous subsection (\S\ref{generalized Tarski}) apply without changes to show
that this is a semi-generalized model of intuitionistic logic.
Indeed, the verification of the {\it modus ponens} and generalization rules really involves only 
products of characteristic sheaves (a product of two of them and a $\D$-indexed product) which are the same
as the characteristic sheaves of the intersection.
Let us call this semi-generalized model of intuitionistic logic the {\it Tarski/sheaf-valued model}
in the map $f\:X\to Y$ with domain $\D$.

\begin{example} \label{subset/sheaf-valued example}
If $X$ is a topological space and $X^d$ is the same underlying set re-endowed with the discrete
topology, the identity map $f\:X^d\to X$ is continuous.
But since usual Tarski models in $X^d$ satisfy the principle $\fm{\prin\alpha\lor\neg\alpha}$,
this principle is also valid in Tarski/sheaf-valued models in $f$.
So they are in fact semi-generalized models of classical logic.
Let us note that $\ocf$ is interpreted by $S\mapsto\chi_{\Int S}$ in this case.
\end{example}

\subsubsection{Verificationist interpretation}\label{verificationist+}

We fix a domain of discourse $\D$.
Like in the Platonist interpretation, $\O$ is taken to be a class of propositions, consisting of 
a prescribed class of contentful (e.g.\ mathematical) primitive propositions and composite propositions, which are 
obtained inductively from the primitive ones by using contentual classical connectives and quantifiers over 
arbitrary $\D$-indexed families of propositions.
Formal classical connectives and quantifiers are interpreted straightforwardly by the contentual ones.
We also have the notion of ``truth'', defined explicitly for primitive $P\in\O$, and extended 
inductively to composite $P\in\O$ by the usual truth tables.

Now $\Qm$ is taken to be a class of problems, consisting of primitive problems of the form 
$\ocf(P)=$ ``Prove $P$'', which are the values of $\ocf\:\O\to\Qm$, and of composite 
problems, which are obtained inductively from the primitive ones by using contentual 
intuitionistic connectives and quantifiers over arbitrary $\D$-indexed and $\Hom(\D^i,\O)$-indexed 
($i=0,1,\dots$) families of problems.

The value of $\wnf\:\Qm\to\{\Top,\Bot\}$ on a problem $\Theta$ will be $\Top$ if and only if $\Theta$ 
has a solution.
To ``solve'' a primitive $\Theta=\ocf(P)$ means to prove that $P$ is true; this notion of ``solving'' is 
extended inductively to composite $\Theta\in\Qm$ by the usual BHK interpretation (see \S\ref{BHK}).
This determines a contentual interpretation of classical logic and its meta-logic, which we will call 
the {\it modified Verificationist} interpretation.

The Verificationist interpretation is very close to the sketch of an interpretation in \S\ref{contentual},
which however suggests fully classical interpretation of the meta-logic.

\begin{example} Let us record the modified Verificationist interpretations of the judgements of admissibility, 
stable admissibility and derivability of a purely logical rule $F\,/\,G$.
The tuples $\vec\phi$ and $\vec\Phi$ have the same meaning as in Example \ref{platonist example},
and $e$ is again understood to run over $\E$ as in Example \ref{platonist example2}. 
Like before, $\wn\Theta$ denotes the proposition ``The problem $\Theta$ has a solution'', and
$\oc P$ denotes the problem ``Prove the proposition $P$''.

\begin{itemize}
\item Derivability: $\forall\D,\E,\O\ \wn\forall\pval\,
\big(\forall\iass\,\oc|F|^\pval_\iass\to\forall\iass\,\oc|G|^\pval_\iass\big)$;

\item Stable admissibility: 

$\forall\vec\Phi\,\forall\D,\E,\O\ \wn
\big(\forall\pval,\iass\,\oc|F[\vec\phi/\vec\Phi]|^\pval_\iass
\To\forall\pval,\iass\,\oc|F[\vec\phi/\vec\Phi]|^\pval_\iass\big)$;

\item Admissibility:

$\forall\vec\Phi\,\big(\forall\D,\E,\O\ \wn\forall\pval,\iass\,\oc|F[\vec\phi/\vec\Phi]|^\pval_\iass
\To\forall\D,\E,\O\ \wn\forall\pval,\iass\,\oc|F[\vec\phi/\vec\Phi]|^\pval_\iass\big)$.
\end{itemize}
\end{example}

\begin{remark} \label{verificationist2}
The Verificationist interpretation can be straightforwardly generalized using 
a ``hidden parameter'' domain $\E$, so that $\O$ consists of predicates on $\E$ (similarly to the modified 
Platonist interpretation) and $\Qm$ consists of parametric problems with parameter in $\E$ (similarly to 
the modified BHK interpretation).

Let us note that this modified Verificationist semantics is formalized by the generalized models of Example 
\ref{subset/sheaf-valued example}.
\end{remark}

\subsubsection{Absolute realizability} \label{uniform realizability}

In 1945, Kleene introduced the notion of realizability for arithmetical formulas, which 
is weaker than provability in intuitionistic arithmetic (see \cite{Kl}*{\S82}, \cite{Pl3}).
A parallel notion of realizability for predicate formulas, called ``abstract realizability'', was introduced
by Plisko in the 1970s (see \cite{Pl2}, \cite{Pl3}).
However, traditionally, realizability for predicate formulas has been more extensively studied in 
an indirect way, by substituting arithmetical formulas for predicate variables, which leads to several 
inequivalent definitions of realizability for predicate formulas (see \cite{Pl3}*{\S4, \S8} for 
a comparison of these definitions with each other and \cite{VSSC}*{\S1, Assertion 1 and subsequent remarks}%
\footnote{Not found in the 2002 conference proceedings paper by the same authors with the same title.} 
concerning their relationship with abstract realizability).

All notions of realizability for predicate formulas are weaker than provability in intuitionistic logic.
Thus one can consider them as models of intuitionistic logic.
However, these are not models in the sense of the traditional, Tarski-style model theory.

We will now see that abstract realizability can in fact be regarded as a generalized model of 
intuitionistic logic.

Let $\D=\N$ and $\O=2^\N$, the set of all subsets of $\N$.
The connectives and quantifiers are interpreted as follows.
For $m,n\in\N$ let $\#(x,y)=2^m\cdot 3^n$.
Let $\frak R$ denote the set of all (partial) recursive functions $\N\supset\Dom f\xr{f}\N$, and 
for any $f\in\frak R$ let $\#f\in\N$ be the G\"odel number of $f$.%
\footnote{It is essential here that $f$ need not be total.
The set of all total recursive functions is not recursively enumerable.}
Then for any subsets $X,Y\subset\N$ and any $P\:\N\to 2^\N$ we set:
\begin{itemize}
\item $X\,|{\lor}|\,Y=\{\#(0,x)\mid x\in X\}\cup\{\#(1,x)\mid y\in Y\}$;
\item $X\,|{\land}|\,Y=\{\#(x,y)\mid x\in X,\,y\in Y\}$;
\item $X\,|{\to}|\,Y=\{\#f\mid f\in\frak R,\,X\subset\Dom f\text{ and }f(X)\subset Y\}$;
\item $|\ab|=\emptyset$;
\item $|{\exists}|(P)=\{\#(x,y)\mid y\in P(x)\}$;
\item $|{\forall}|(P)=\{\#f\mid \Dom f=\N\text{ and }f(x)\in P(x)\text{ for each }x\in\N\}$.
\end{itemize}
Let us note that $\Dom f$ may be strictly larger than $X$ in the definition of $X\,|{\to}|\,Y$.%
\footnote{This freedom is essential.
Domains of recursive functions coincide with recursively enumerable sets.}
Thus, in fact, $|\ab\to\ab|$ consists of {\it all} partial recursive functions.

It should be noted that this interpretation of the connectives is nothing but a rather computational
reading of the BHK interpretation: the set of solutions of every problem is identified with a set of 
natural numbers, and ``general method'' is interpreted as ``algorithm''.
This at once guarantees that, in the notation of \S\ref{confusion}, $\script F$ is neither injective 
nor surjective, and $\script G$ is not surjective.

Then it is natural (cf.\ the meta-clarified BHK interpretation in \S\ref{extended+} below) to set 
$\Qm=\O=2^\N$, $\ocf=\id\:\O\to\Qm$, and 
interpret the meta-connectives and the first-order meta-quantifier similarly to the corresponding 
connectives and the universal quantifier.
Thus for any subsets $X,Y\subset\N$ and any $P\:\N\to 2^\N$ and $Q_n\:\Hom(\N^n,2^\N)\to 2^\N$ we set
\begin{itemize}
\item $X\,|{\mand}|\,Y=X\,|{\land}|\,Y$;
\item $X\,|{\imp}|\,Y=X\,|{\to}|\,Y$;
\item $|\q|(P)=|{\forall}|(P)$;
\item $|\q_n|(Q_n)=\bigcap_{P\:\N^n\to 2^\N}Q_n(P)$.
\end{itemize}
Here the second-order meta-quantifier cannot be interpreted directly similarly to the first-order one,
since the set $\Hom(\N^n,2^\N)$ is uncountable, so we cannot speak of recursive functions defined on
this set (as long as ``recursive'' is understood in the original sense). 
However, a constant function would arguably qualify as ``recursive''; so one straightforward option is
to be content with constant functions; this yields the above definition of $|\q_n|$.
Finally, it is natural to define $\wnf\:\Qm\to\{\Top,\Bot\}$ by $\wnf(X)=\Top$ if and only if 
$X\ne\emptyset$.

Thus for a closed formula $\Phi$, if $\vec\gamma$ is the tuple of all problem variables that occur in 
$\Phi$, and $r_1,\dots,r_m$ are their arities, then we have $\wnf(|{\prin\Phi}|)=\Top$ if and only if 
there exists a natural number $e$ such that $e\in |\vec\gamma\mapsto\Phi|(|\vec\gamma|)$
for all tuples $|\vec\gamma|$ of functions $|\gamma_i|\:\N^{r_i}\to 2^\N$.
This is the precisely one of the two equivalent definitions (``absolute uniform realizability'') of $\Phi$ 
being abstractly realizable in the sense of Plisko \cite{Pl2}, \cite{Pl3}.

For this reason we will call the meta-structure just described the {\it abstract realizability meta-structure}.

\begin{theorem} The abstract realizability meta-structure is a generalized model of intuitionistic logic. 
\end{theorem}

\begin{proof}
It is well-known that intuitionistically valid formulas are abstractly realizable (see \cite{Pl2}, \cite{Pl3}),
so all the laws of the Hilbert-style derivation system of intuitionistic logic
(see \S\ref{intlogic-modified}) must be valid in this meta-structure.
It follows that the meta-laws of the Enderton-style formulation of the meta-logic (see 
\S\ref{enderton}) that do not involve the second-order quantifier are also valid in it.

There remain three meta-laws involving the second-order quantifier.
Let us check the validity of
$\mq{\vec z}\mq{\vec\beta}\big(\mq{\gamma}(\F\imp\G)\imp(\mq{\gamma}\F\imp\mq{\gamma}\G)\big)$.
Given an individual assignment $\iass$ and a predicate valuation $\pval$, the set $|\F\imp\G|^{\iass\pval}$ 
consists of the G\"odel numbers of all total recursive functions $f$ such that 
$f(|\F|^{\iass\pval})\subset|\G|^{\iass\pval}$.
By definition, $|\mq{\gamma}\H|^{\iass\pval}$ is the intersection of the sets $|\H|^{\iass\pval'}$ over 
all $\pval'\in V$, where $V$ consists of all predicate valuations that agree 
with $\pval$ on all predicate variables except possibly $\gamma$.
Thus $|\mq{\gamma}(\F\imp\G)|^{\iass\pval}$ consists of the G\"odel numbers of all total recursive $f$ 
such that $f(|\F|^{\iass\pval'})\subset|\G|^{\iass\pval'}$ for every $\pval'\in V$.
This is clearly a subset of the set $|\mq{\gamma}\F\imp\mq{\gamma}\G|^{\iass\pval}$ of the G\"odel numbers 
of all total recursive $f$ such that 
$f(\bigcap_{\pval'\in V}|\F|^{\iass\pval'})\subset
\bigcap_{\pval'\in V}|\G|^{\iass\pval'}$.
Thus $|\mq{\gamma}(\F\imp\G)\imp(\F\imp\mq{\gamma}\G)|^{\iass\pval}$ contains the G\"odel number of 
the identity function.
Consequently so does 
$|\mq{\vec z}\mq{\vec\beta}\big(\mq{\gamma}(\F\imp\G)\imp(\mq{\gamma}\F\imp\mq{\gamma}\G)\big)|^{\iass\pval}$.

The validity of the two remaining meta-laws involving the second-order quantifier,
$\mq{\vec z}\mq{\vec\beta}\big(\mq{\gamma}\F\imp\F[\gamma/\Phi]\big)$, provided that $\Phi$ is free for 
$\gamma$ in $\F$; and $\mq{\vec z}\mq{\vec\beta}\big(\F\imp\mq{\gamma}\F\big)$, provided that $\gamma$ 
does not occur freely in $\F$, is checked similarly, also using the G\"odel number of the identity function. 

The validity of the $\imp$-elimination meta-rule is easy: if $|\F|^{\iass\pval}$ is nonempty and there exists 
a total recursive function $f$ such that $f(|\F|^{\iass\pval})\subset|\G|^{\iass\pval}$,
then $|\G|^{\iass\pval}$ must be nonempty as well.

It remains to check the validity of the two object-level rules: the generalization rule
$\fm{\mq{\alpha}\big(\mq{\tr x}\,\alpha(\tr x)\imp\forall\tr x\,\alpha(\tr x)\big)}$ and 
the {\it modus ponens} $\fm{\mq{\alpha}\mq{\beta}(\alpha\mand\alpha\to\beta\imp\beta)}$,
or equivalently (using the exponential meta-law) 
$\fm{\mq{\alpha}\mq{\beta}\big(\alpha\to\beta\imp(\alpha\imp\beta)\big)}$.
But these are trivial, since $|{\to}|=|{\imp}|$ and $|\q|=|{\forall}|$.
\end{proof}

\subsection{Alternative semantics} \label{Frege-style}

\subsubsection{Kolmogorov's semantic consequence} \label{K-consequence}

The interpretations considered above are traditional, ``Tarski-style'' interpretations, which are based,
in particular, on Tarski's concept of semantic consequence \cite{Ta1}; and the meta-interpretations considered 
above can be called ``generalized Tarski-style''.

We will next do something different, and consider alternative, ``Kolmogorov-style'' interpretations, 
which are based,
in particular, on Kolmogorov's concept of semantic consequence, found in his problem interpretation 
of intuitionistic logic \cite{Kol}.

Kolmogorov's problem interpretation explains not only logical, but, in disguise, also meta-logical connectives 
and quantifiers by the same means of solutions of problems, and it is this coincidence of $\O$ and $\Qm$ that 
makes Kolmogorov's concept of semantic consequence particularly apposite to his context.
One might argue that the meta-logical constants should rather be interpreted by means of meta-solutions of
meta-problems, so that the issue of coincidence of $\O$ with $\Qm$ does not arise.
But what are the meta-problems, and where does one find them?
In fact, Kolmogorov's notion of semantic consequence yields a nice description of meta-problems in terms of 
problems (see \S\ref{extended+}) --- and one can hardly dream of anything simpler.
It also applies to yield a description of meta-propositions in terms of propositions (see \S\ref{universalist}).

Namely, Kolmogorov offered the following interpretation of the rules used by Heyting in his original 
axiomatization of intuitionistic logic: $\fm{p\land q\,/\,p}$, the {\it modus ponens} rule $\fm{p,p\to q\,/\,q}$ 
and the (non-structural) substitution rule --- or rather of the assumption that these are inference rules of 
the logic:%
\footnote{Kolmogorov's formulas are given here in a modern notation (replacing $\supset$ with $\to$ and
Peano's dots with the usual parentheses); also, Kolmogorov's symbol ``$\turnstile$'' is replaced 
with ``$\prin$'' since his interpretation of $\turnstile$ is precisely the same as our interpretation of $\prin$ 
(and different from our interpretation of $\turnstile$).}

\smallskip
\begin{center}
\parbox{14.7cm}{\small ``[W]e assume that we are always able (i.e.\ we possess a general method) to solve, 
for any given elementary problem functions [=formulas of zero-order intuitionistic logic regarded as 
functions of their problem variables] $p,q,r,s,\dots$, the following problems:

I. If $\prin (p\land q)$ is solved, to solve $\prin p$;

II. If $\prin p$ and $\prin (p\to q)$ are solved, to solve $\prin q$;

III. If $\prin p(a,b,c,\dots)$ is solved, to solve $\prin p(q,r,s,\dots)$.''}
\end{center}
\medskip

\begin{remark} 
It should be noted that Kolmogorov's ``If ... is solved, to solve ...'' is apparently meant to be the same
as his interpretation of object-level implication: ``$a\to b$ is the problem `assuming that a solution to $a$ 
is given, solve $b$', or, which means the same, `reduce the solution of $b$ to the solution of $a$'.''

But in our presentation of the BHK interpretation in \S\ref{BHK}, we follow Heyting's more nuanced reading of
$a\to b$ as the problem ``find a general method to transform the given solution of $a$ into a solution of $b$''.
Hence it is natural to apply this more nuanced reading here as well.
\end{remark}

It is not hard to see that the same approach applies to an arbitrary rule just as well, including those 
in first-order logic.
So here is the natural Kolmogorov-style interpretation of a rule 
\[\frac{\Gamma_1,\dots,\Gamma_k}{\Gamma_0}\]
over intuitionistic logic.
\begin{enumerate}
\item Suppose that all variables that occur freely in the formulas $\Gamma_i$ form a $p$-tuple $\vec t$ of 
individual variables and an $s$-tuple $\vec\phi$ of problem variables of arities $q_1,\dots,q_s$.
Then the rule $\Gamma_1,\dots,\Gamma_k\,/\,\Gamma_0$ is interpreted as the problem of finding a general method 
of solving a certain problem \[\Big|{\prin\Gamma_1}[\vec t/\vec T,\,\vec\phi/\vec\Phi]\ \mand\dots\mand\ 
{\prin\Gamma_k}[\vec t/\vec T,\,\vec\phi/\vec\Phi]\ \imp\ {\prin\Gamma_0}[\vec t/\vec T,\,\vec\phi/\vec\Phi]\Big|\] 
for any given $p$-tuple $\vec T$ of terms and any given $s$-tuple $\vec\Phi$ of $q_i$-formulas.
\item A solution of the problem $|{\prin\Delta_1\mand\dots\mand\prin\Delta_k\imp\prin\Delta_0}|$
is in turn a general method of transforming any given solutions of $|{\prin\Delta_1}|,\dots,|{\prin\Delta_k}|$
into a solution of $|{\prin\Delta_0}|$.
\item Suppose that all variables that occur freely in the formulas 
$\Delta_i=\Gamma_i[\vec t/\vec T,\,\vec\phi/\vec\Phi]$
form an $n$-tuple $\vec x$ of individual variables and an $m$-tuple $\vec\theta$ of problem variables of arities 
$r_1,\dots,r_m$.
Then a solution of each $|{\prin\Delta_i}|$ is a general method 
of solving the contentful problem $|\vec x,\vec\theta\mapsto\Delta_i|(\vec X,\vec\Theta)$ for any $n$-tuple 
$\vec X$ of elements of the domain and any $m$-tuple $\vec\Theta$ of contentful problems of 
arities $r_1,\dots,r_m$.
\item What is a solution of each contentful problem $|\vec x,\vec\theta\mapsto\Delta_i|(\vec X,\vec\Theta)$
is determined by the usual BHK interpretation (see \S\ref{BHK}).
\end{enumerate}

An obvious issue with this interpretation of rules is that step (1) applies what in essence is constructive 
quantification (``to find a general method'') to syntactic entities (``for any given formulas and terms''), 
thereby creating a dangerous mix of syntax and semantics.
(Perhaps this was not such an obvious issue back in 1932, when Kolmogorov published his interpretation of
rules \cite{Kol} but, for instance, Tarski's paper \cite{Ta1} on semantic consequence had not appeared yet.)

But it turns out that this mix of syntax and semantics can be disentangled.
What is important in Kolmogorov's interpretation of rules is that the formulas $\vec\Phi$ and terms $\vec T$
being quantified over are treated as functions of the problem variables $\vec\theta$ and the individual 
variables $\vec x$.
But such a functional dependence can well be understood on the semantic level.

\subsubsection{Alternative semantics}\label{Frege-definition}

Given a Tarski-style meta-interpretation $\J$ of a language $\L$ as in \S\ref{meta-formulas}, we
want to define an alternative meta-interpretation of $\L$.
We recall that $\J$ consists of the following data: an $\L$-structure $\I$, the set $\Qm$, the maps 
$\ocf\:\O\to\Qm$ and $\wnf\:\Qm\to\{\Top,\Bot\}$, elements $|\!\mand\!|$, $|\!\!\imp\!\!|$
of $\Hom(\Qm\x\Qm,\Qm)$, an element $|\q|$ of $\Hom(\Hom(\D,\Qm),\Qm)$ and an element
$|\q^n|$ of $\Hom(\Hom(\Hom(\D^n,\O),\Qm),\Qm)$ for each $n\in\N$.

The first step, which is also the crucial one, is to interpret $\Rt\:\1\to\mu$.
Its Tarski-style interpretation was simply by $\ocf\:\O\to\Qm$, which made it possible to routinely 
omit $\Rt$ altogether.
The alternative interpretation of $\Rt$ will be more sophisticated.

Given a $\lambda$-closed $(n,\vec r)$-meta-formula of the ``atomic'' form 
$\G=\vec\gamma\mapsto(\vec x\mapsto\Rt F)$, where $\vec r=(r_1,\dots,r_m)$, it can be 
interpreted by a function $||\G||\:\Fm_{r_1}\x\dots\x\Fm_{r_m}\to\Hom(\T^n,\Qm)$, where $\T=\T(\L)$ is the set 
of terms and $\Fm_i=\Fm_i(\L)$ is the set of $i$-formulas of the language $\L$.
Namely, given an $n$-tuple $\vec T\in\T^n$ and an $m$-tuple $\vec\Phi\in\Fm_{r_1}\x\dots\x\Fm_{r_m}$, we can 
find a formula $F'$, $\alpha$-equivalent to $F$ and such that $\vec T$ is free for $\vec x$ in $F$ and 
$\vec\Phi$ is free for $\vec\gamma$ in $F$.
Let $F_{\vec\Phi,\vec T}=F'[\vec\gamma/\vec\Phi,\,\vec x/\vec T]$.
The principle $\prin F_{\vec\Phi,\vec T}$ is Tarski-interpreted by a specific element 
$|{\prin F_{\vec\Phi,\vec T}}|\in\Qm$, and we set $||\G||(\vec F)(\vec T)$ to be this element.
Clearly, the $\alpha$-equivalence class of $\prin F_{\vec\Phi,\vec T}$ does not depend on the choice of $F'$,
and hence $||\G||$ is well-defined.

However, the set of terms $\T$ and the sets of $i$-formulas $\Fm_i$ are not really semantic notions, so it is 
not desirable to use them in our ``alternative semantics''.
To this end we will replace them with their semantic analogues.

(Incidentally, let us also note that the above definition of $||\G||$ does not seem to generalize in a reasonable
way to the case where $\G$ is not $\lambda$-closed, because same variables could have free occurrences in 
the original formula and also new occurrences introduced upon the substitution.)

Let $\omega$ denote, as usual, the totally ordered set $\{0,1,\dots\}$.
A {\it semantic term of degree} $k\in\omega$ is a function $T\:\D^k\to\D$, and a {\it semantic $n$-formula of 
degree} $(k,l)\in\omega\x\omega$ is a function $\Phi\:\D^k\x\O^{[l]}\to\Hom(\D^n,\O)$, where $\O^{[l]}$ denotes 
the product $\prod_{i=0}^{l-1}\Hom(\D^i,\O)^l$ of $l\x l$ factors.
(We allow for the possibility that $k\ne l$ only as a matter of convenience; it is not strictly necessary.)
A {\it semantic term} is an element $[k,T]$ of the disjoint union $|\T|:=\bigsqcup_{k\in\omega}\Hom(\D^k,\D)$,
and a {\it semantic $n$-formula} is an element $[k,l,\Phi]$ of the disjoint union 
$|\Fm_n|:=\bigsqcup_{k,l\in\omega}\Hom(\D^k\x\O^{[l]},\Hom(\D^n,\O))$.
A semantic 0-formula is also called a {\it semantic formula}.%
\footnote{Although ``semantic terms'' and ``semantic formulas'' are technically oxymorons, this 
nevertheless seems to be a reasonable terminology.}  

Given a $\lambda$-closed $(n,\vec r)$-meta-formula of the ``atomic'' form 
$\G=\vec\gamma\mapsto(\vec x\mapsto\Rt F)$, where $\vec r=(r_1,\dots,r_m)$, 
its alternative interpretation will be a map $\KK\G\KK:|\Fm_{\vec r}|\to\Hom(|\T|^n,\Qm)$, defined as follows.
Let us start from the usual Tarski-style interpretation 
$|\G|\:\prod_{i=1}^m\Hom(\D^{r_i},\O)\to\Hom(\D^n,\Qm)$.
Given an $n$-tuple $\vec T$ of semantic terms $T_i\:\D^{k_i}\to\D$, we can combine them into a map 
$T\:\D^k\to\D^n$, where $k=\max(k_1,\dots,k_n)$.
Given an $m$-tuple $\vec\Phi$ of semantic $r_i$-formulas 
$\Phi_k\:\D^{k'_i}\x\O^{[l_i]}\to\Hom(\D^{r_i},\O)$, we can combine them into a map 
$\Phi\:\D^{k'}\x\O^{[l]}\to\prod_{i=1}^m\Hom(\D^{r_i},\O)$,
where $k'=\max(k_1',\dots,k_m')$ and $l=\max(l_1,\dots,l_m)$.
Finally, $|\G|$, $T$ and $\Phi$ combine into a map $\D^{k''}\x\O^{[l]}\to\Qm$, where $k''=\max(k,k')$,
which we will denote, somewhat loosely, by $|\G|\circ(\vec T\x\vec\Phi)$.
We can apply to this map the Tarski-interpreted $k'$-fold first-order and $l\x l$-fold second-order 
meta-quantifier $|\q_{k''}^{[l]}|\:\Hom\big(\D^{k''}\x\O^{[l]},\Qm\big)\to\Qm$
and obtain an element $|\q_{k''}^{[l]}|\big(|\G|\circ(\vec T\x\vec\Phi)\big)\in\Qm$.
Thus we have described a map $\KK\G\KK:|\Fm_{\vec r}|\to\Hom(|\T|^n,\Qm)$.
Let us note that the Tarski-interpreted meta-quantifiers $|\q|$ and $|\q^n|$ are used here
to interpret $\Rt$ rather than any meta-quantifiers.

Next, we keep the Tarski-style interpretations $\KK\&\KK:=|\&|$ and $\KK\impord\KK:=|\impord|$ for 
the meta-connectives, whereas the first-order meta-quantifier $\q$ is re-interpreted 
by a function $\KK\q\KK:\Hom(|\T|,\Qm)\to\Qm$; and the $n$-ary second-order meta-quantifier $\q^n$ 
by a function $\KK\q^n\KK\:\Hom(|\Fm_n|,\Qm)\to\Qm$.

The above will be called the {\it alternative meta-interpretation} associated to the given 
Tarski-style meta-interpretation $\J$.
In the case of a Tarski-style interpretation $\J=\I_+$, we can 
define an {\it alternative interpretation} by providing explicit definitions of the functions
$\KK\q\KK$ and $\KK\q^n\KK$.
Namely, we set $\KK\q\KK(f)=\Top$ if and only $f(T)=\Top$ for all $T\in|\T|$ and $\KK\q^n\KK(f)=\Top$ 
if and only $f(\Phi)=\Top$ for all $\Phi\in|\Fm_n|$.

Given an alternative meta-interpretation, we can interpret any $\lambda$-closed 
$(n,\vec r)$-meta-formula $\G$ by a function $\KK\G\KK\:|\Fm_{\vec r}|\to\Hom(|\T|^n,\Qm)$ 
by a straightforward induction.
In particular, any $\lambda$-closed meta-formula $\F$ is interpreted by an element $\KK\F\KK\in\Qm$.
We may also write $\KK\F\KK_\J$ to emphasize the given meta-interpetation $\J$, or $\KK\F\KK_\I$
in the case of an interpretation $\J=\I_+$.

We will see below that the alternative interpretation turns out to be the same as the original Tarski-style 
interpretation for principles (\S\ref{principles-rules}), but generally different for rules
(\S\ref{Frege two-valued}).
In order to generalize this to many-valued meta-interpretations, we will need to reformulate alternative
interpretations in a different language (\S\ref{Frege-style2}).

\subsubsection{Principles and rules}\label{principles-rules}

A {\it pre-assignment} $\preass$ assigns a semantic term $\preass(x)\in|\T|$ to each individual variable $x$.
Clearly, a pre-assignment $\preass$ along with an interpretation $\I$ of function symbols yield 
an interpretation of each term $T$ by a semantic term $\KK T\KK^\I_\preass$, also denoted $\KK T\KK_\preass$.
Let us note that a usual variable assignment $\iass$ is a special case of a pre-assignment $\preass$ that 
assigns to each individual variable $x$ a semantic term $\preass(x)$ of the form $[0,T]$,
$T\:\D^0\to\D$, i.e.\ an element of $\D$.

A {\it pre-valuation} $\preval$ assigns a semantic $n$-formula $\preval(\gamma)\in|\Fm_n|$
to each $n$-ary predicate variable $\gamma$, for each $n=0,1,\dots$.
Let us note that a usual predicate valuation is a special case of a pre-valuation $\preval$ that assigns
to each $n$-ary predicate variable $\gamma$ a semantic $n$-formula 
$\preval(\gamma)$ of the form $[0,0,\Phi]$,
$\Phi\:\D^0\x\O^{(0)}\to\Hom(\D^n,\O)$, i.e.\ an element of $\Hom(\D^n,\O)$.

Clearly, a pre-valuation $\preval$ along with an interpretation $\I$ of function symbols, predicate constants,
connectives and quantifiers yield an interpretation of each closed $n$-formula $G$ by a semantic $n$-formula 
$\KK G\KK^\preval_\I$, also denoted $\KK G\KK^\preval$.
Given additionally a pre-assignment $\preass$, then also each formula $F$ is interpreted by
a semantic formula $\KK F\KK^\preval_\preass$.
Namely, if $\vec x=(x_1,\dots,x_n)$ is the tuple of all free individual variables of $F$, then
$\KK F\KK^\preval_\preass=[\max(k,k'),l,\Phi\circ T]$, $\Phi\circ T\:\D^{\max(k,k')}\x\O^{[l]}\to\O$,
is obtained by pre-composing the semantic $n$-formula $|\vec x\mapsto F|^\preval=[k',l,\Phi]$,
$\Phi\:\D^k\x\O^{[l]}\to\Hom(\D^n,\O)$
with the function $T\:\D^k\to\D^n$, $k=\max(k_1,\dots,k_n)$, obtained by combining 
the semantic terms $\preass(x_i)=[k_i,T_i]$, $T_i\:\D^{k_i}\to\D$.

Given an assignment $\iass$ and a pre-assignment $\preass$, we can define a new assignment 
$\preass[\iass]$ by $\preass[\iass](x)=f\big(\iass(\var^\0_1),\dots,\iass(\var^\0_k)\big)$, where
$\preass(x)=[k,f]$, $f\:\D^k\to\D$.
Then for each term $T$ we have $|T|_{\preass[\iass]}=f(|\var^\0_1|_\iass,\dots,|\var^\0_k|_\iass)$, 
where $\KK T\KK_\preass=[k,f]$, $f\:\D^k\to\D$.

Given an assignment $\iass$, a valuation $\pval$ and a pre-valuation $\preval$, we define a new valuation 
$\preval[\iass,\pval]$ by $\preval[\iass,\pval](\gamma)=f\Big(\iass(\var^\0_1),\dots,\iass(\var^\0_k);
\big(\pval(\var^{\0^i\too\1}_1),\dots,\pval(\var^{\0^i\too\1}_l)\big)_{i=0}^{l-1}\Big)$,
where $\preval(\gamma)=[k,l,f]$, $f\:\D^k\x\O^{[l]}\to\Hom(\D^n,\O)$.
Then for each closed $n$-formula $G$ we have
$|G|^{\preval[\iass,\pval]}=f\Big(\iass(\var^\0_1),\dots,\iass(\var^\0_k);
\big(\pval(\var^{\0^i\too\1}_1),\dots,\pval(\var^{\0^i\too\1}_l)\big)_{i=0}^{l-1}\Big)$,
where $\KK G\KK^\preval=(k,l,f)$, $f\:\D^k\x\O^{[l]}\to\Hom(\D^n,\O)$.

Consequently, given an assignment $\iass$, a valuation $\pval$, a pre-assignment $\preass$ and 
a pre-valuation $\preval$, for each formula $F$ we again have 
\[|F|^{\preval[\iass,\pval]}_{\preass[\iass]}=f\Big(\iass(\var^\0_1),\dots,\iass(\var^\0_k);
\big(\pval(\var^{\0^i\too\1}_1),\dots,\pval(\var^{\0^i\too\1}_l)\big)_{i=0}^{l-1}\Big),\]
where $\KK F\KK^\preval_\preass=(k,l,f)$, $f\:\D^k\x\O^{[l]}\to\O$.

\begin{example} \label{Frege principle}
Let us discuss the alternative interpretation (i.e.\ the alternative two-valued 
meta-interpretation) of a principle, $\prin F$.

Since $\Qm=\{\Top,\Bot\}$ and $\wnf\:\Qm\to\{\Top,\Bot\}$ is the identity map,
we may identify elements of $\Qm$ with judgements about pre-valuations and pre-assignments.
 
From the definition, we have
\[\KK{\prin F}\KK=\forall\preval,\preass\ \forall\pval,\iass\,\ocf\,
|F|_{\preass[\iass]}^{\preval[\iass,\pval]}.\]
This is clearly equivalent to the Tarski-style interpretation
\[|{\prin F}|=\forall\pval,\iass\,\ocf\,|F|_\iass^\pval.\]
\end{example}

\begin{example} \label{Frege rule}
Let us now discuss the alternative interpretation (i.e.\ the alternative two-valued 
meta-interpretation) of a rule, $F/G$.

From the definition, we have
\[\KK{F\,/\,G}\KK=\forall\preval\,(
\forall\preass\ \forall\pval,\iass\,\ocf\,|F|_{\preass[\iass]}^{\preval[\iass,\pval]}
\To\forall\preass\ \forall\pval,\iass\,\ocf\,|G|_{\preass[\iass]}^{\preval[\iass,\pval]}).\]
This clearly simplifies as 
\[\KK{F\,/\,G}\KK=\forall\preval\,(\forall\pval,\iass\,\ocf\,|F|_\iass^{\preval[\iass,\pval]}
\To\forall\pval,\iass\,\ocf\,|G|_\iass^{\preval[\iass,\pval]}).\]
The latter clearly follows from
\[|{F\,/\,G}|=\forall\preval\ \forall\pval\,(\forall\iass\,\ocf\,|F|_\iass^{\preval[\iass,\pval]}\To
\forall\iass\,\ocf\,|G|_\iass^{\preval[\iass,\pval]}),\]
which is clearly equivalent to the Tarski-style interpretation
\[|{F\,/\,G}|=\forall\pval\,(\forall\iass\,\ocf\,|F|_\iass^\pval\To
\forall\iass\,\ocf\,|G|_\iass^\pval).\]
However, we will see in Example \ref{Frege pure} below that there is no reverse implication, at least,
in a minor modification of the setup.
\end{example}

\subsubsection{Two-valued alternative models}\label{Frege two-valued}

Example \ref{Frege rule} (see also Example \ref{Frege-platonist example} below) suggests that at least 
with two-valued interpretation of the meta-logic, the alternative interpretation of the derivability of a rule 
is very similar to the interpretation of its stable admissibility (either Tarski-style or alternative, these 
are the same by Example \ref{Frege principle}). 
Let us show that despite this similarity, they do not coincide.

\begin{example} The rule $\exists \tr x\,\fm p(\tr x)\land\exists \tr x\,\neg\fm p(\tr x)\,/\,\clbot$
is stably admissible for classical logic (see Example \ref{st-adm}).
However, this rule is not alternatively-valid already in the two-valued model of classical logic
(two-valued both on the meta-level and the object-level, i.e.\ $\O=\Qm=\{\Top,\Bot\}$).
This is established in essence similarly to Example \ref{equality-st-adm}.

In more detail, we need to interpret $\fm p$ by a semantic 1-formula $\preval(\fm p)=[l,\Phi]\in|\Fm_1|$,
$\Phi\:\D^k\x\O^{[l]}\to\Hom(\D,\O)$.
Let us choose $k=1$, $l=0$ and $\Phi\:\D\to\Hom(\D,\O)$ defined by
$\Phi(y,\gamma)(x)=\Top$ if and only if $x=y$.
Then the semantic formula $\KK\exists \tr x\,\fm p(\tr x)\land\exists \tr x\,\neg\fm p(\tr x)\KK^\preval$
is $[1,0,\Psi]\in|\Fm_0|$, where $\Psi\:\D\to\O$ is the Tarski-style interpretation of the $\lambda$-closed 
1-formula $\tr y\mapsto\exists \tr x\ \,\tr x=\tr y\land\exists \tr x\ \,\neg\tr x=\tr y$ of the language
extended by the binary predicate constant $=$ with the obvious interpretation.

Therefore the alternative interpretation of the $\lambda$-closed $(0,(1))$-meta-formula
$\fm p\mapsto\Rt\big(\exists \tr x\,\fm p(\tr x)\land\exists \tr x\,\neg\fm p(\tr x)\big)\imp\bot$
is by a function $f\:|\Fm_1|\to\Qm$ whose value on $\preval(\fm p)$ is the Tarski-style interpretation of 
the $\lambda$-closed meta-formula
$(\mq{\tr y}\exists \tr x\ \,\tr x=\tr y\land\exists \tr x\ \,\neg\tr x=\tr y)\imp\bot$ of the extended language.
As long as $\D$ contains at least two elements, this value $f\big(\preval(\fm p)\big)\in\Qm$ is $\Bot$.
(see Example \ref{equality-st-adm}).
Hence the alternative interpretation of the rule in question, i.e.\ the $\lambda$-closed meta-formula
$\mq{\fm p}\Big(\Rt\big(\exists \tr x\,\fm p(\tr x)\land\exists \tr x\,\neg\fm p(\tr x)\big)\imp\bot\Big)$,
is also $\Bot$.
\end{example}

Let us call a semantic $n$-formula $(k,l,\Phi)\in|\Fm_k|$, $\Phi\:\D^k\x\O^{[l]}\to\Hom(\D^n,\O)$, {\it pure} 
if it uses the given $n+k$ elements of $\D$ only by feeding them into the given elements of 
the factors $\Hom(\D^i,\O)$ of $\O^{[l]}$; or more formally if there exists a semantic formula 
$(0,l,\Phi')\in|\Fm_0|$ such that $\Phi(\vec x,\vec\gamma)(\vec y)=\Phi'(\vec\gamma')$, where 
$\vec\gamma=(\gamma_{01},\dots,\gamma_{0l};\dots;\gamma_{l-1,1},\dots,\gamma_{l-1,l})$,
$\vec\gamma'=(\gamma_{01}\circ\phi_{01},\dots,\gamma_{0l}\circ\phi_{0l};\dots;
\gamma_{l-1,1}\circ\phi_{l-1,1},\dots,\gamma_{l-1,l}\circ\phi_{l-1,l})$ and each 
$\phi_{ij}\:\D^i\to\D^i$
replaces some of the given elements of $\D$ with some entries of the tuples $\vec x$ and $\vec y$.

Let $|\Fm_k|^p$ denote the subset of $|\Fm_k|$ consisting of pure semantic $k$-formulas.
If the language contains no function symbols and predicate constants, then we define a {\it pure alternative 
interpretation} and a {\it pure pre-valuation} by repeating the definitions of an alternative interpretation 
and a pre-valuation with $|\Fm_k|^p$ in place of $|\Fm_k|$.

\begin{remark}
One can define ``pre-valuation fields'' similarly to valuation fields (either of the two definitions in
\S\ref{valuation fields}), and check that the family $|\Fm_k|^p$ forms a pre-valuation field.
\end{remark}

\begin{example} \label{Frege pure} The rule 
$\exists \tr x\,\fm p(\tr x)\land\exists \tr x\,\neg\fm p(\tr x)\,/\,\clbot$
is not derivable in classical logic, and is not valid in the Tarski-style two-valued model of classical logic
with any non-singleton domain.
However, this rule is purely alternatively valid in the two-valued model of classical logic with every domain.
This is established in essence similarly to Example \ref{st-adm}.

In more detail, suppose that $\fm p$ is interpreted by a pure semantic 1-formula 
$\preval(\fm p)=(k,l,\Phi)\in|\Fm_1|^p$,
$\Phi\:\D^k\x\O^{[l]}\to\Hom(\D,\O)$.
Let $f_i\:\D^i\to\O$ send every input to $\Bot$, and let 
$\vec f=(f_0,\dots,f_0;\dots;f_{l-1},\dots,f_{l-1})\in\O^{[l]}$.
Since $\preval(\fm p)$ is pure, $\Phi(\vec y,\vec f)(x)$ does not depend on $\vec y$ and $x$.
Then the semantic formula $\KK\exists \tr x\,\fm p(\tr x)\land\exists \tr x\,\neg\fm p(\tr x)\KK^\preval$
is the same as $\KK\fm p(\tr x)\land\neg\fm p(\tr x)\KK^\preval_\iass$ for any individual assignment $\iass$.
But the latter is $\Bot$.
It follows that the rule in question is purely alternatively valid. 
\end{example}

\subsubsection{Platonist interpretation II (alternative)} \label{universalist}

We now present a new contentual interpretation of the meta-logic of classical logic based on the alternative 
semantics of \S\ref{Frege-definition} (two-valued, both on the meta-level and the object-level).

Like in the Platonist interpretation of \S\ref{Frege+}, a domain $\D$ is fixed, and $\O$ is taken to be 
a class of propositions, consisting of a prescribed class of contentful (e.g.\ mathematical) primitive 
propositions and of composite propositions, which are obtained inductively from the primitive ones by 
using contentual classical connectives and quantifiers over arbitrary $\D$-indexed families of propositions.
Formal classical connectives and quantifiers are interpreted straightforwardly by the contentual ones.

$\Qm$ is taken to be a larger class of propositions, consisting of those in $\O$, called ``meta-primitive'' 
(thus the function $\ocf\:\O\to\Qm$ is the inclusion) and of ``meta-composite'' propositions, obtained 
inductively from the meta-primitive ones by using contentual classical connectives and quantifiers over 
arbitrary $\D$-indexed, $\Hom(\D^i,\O)$-indexed ($i=0,1,\dots$), $|\T|$-indexed and $|\Fm_i|$-indexed 
($i=0,1,2,\dots$) families of propositions.
(We recall that $|\T|$ and $|\Fm_i|$ are defined in terms of $\D$ and $\O$, see \S\ref{Frege-definition}.)
Meta-connectives and meta-quantifiers are Tarski-interpreted by contentual classical connectives and 
$\D$-indexed and $\Hom(\D^i,\O)$-indexed quantifiers in the straightforward way; and meta-quantifiers 
are also alternatively-interpreted by contentual classical $|\T|$-indexed and $|\Fm_i|$-indexed quantifiers.

The function $\wnf\:\Qm\to\{\Top,\Bot\}$ sends a proposition $P$ to $\Top$ if and only if $P$ is true,
where ``true'' is defined explicitly for primitive $P\in\O$ and is extended inductively to composite 
$P\in\O$ and then further to all meta-composite $P\in\Qm$ by the usual truth tables.
This determines a contentual interpretation of classical logic and its meta-logic, which we will call 
the {\it alternative Platonist interpretation}.

\begin{example}\label{Frege-platonist example}
Let us record the alternative Platonist interpretations of the judgements of admissibility, 
stable admissibility and derivability of a purely logical rule $F\,/\,G$.
Let $\vec\phi$ be the tuple of all predicate variables that occur freely in $F$
or in $G$, and let $\vec r$ be the tuple of their arities.
Below $\vec\Phi$ is understood to run over $\Fm_{\vec r}$.

\begin{itemize}
\item Derivability: $\forall\D,\O\ \forall\preval\,
(\forall\pval,\iass\,|F|_\iass^{\preval[\iass,\pval]}
\To\forall\pval,\iass\,|G|_\iass^{\preval[\iass,\pval]})$;

\item Stable admissibility: $\forall\vec\Phi\,\forall\D,\O\ 
(\forall\pval,\iass\,|F[\vec\phi/\vec\Phi]|^\pval_\iass
\To\forall\pval,\iass\,|F[\vec\phi/\vec\Phi]|^\pval_\iass)$;

\item Admissibility:
$\forall\vec\Phi\,(\forall\D,\O\ \forall\pval,\iass\,|F[\vec\phi/\vec\Phi]|^\pval_\iass
\To\forall\D,\O\ \forall\pval,\iass\,|F[\vec\phi/\vec\Phi]|^\pval_\iass)$.
\end{itemize}
\end{example}

\subsection{Generalized alternative semantics} \label{Frege-style2}

Our next goal is to better understand alternative meta-interpretations, so as to be able to define
alternative meta-models.
In particular, we have seen above that principles are interpreted in the same way in Tarski-style and
in alternative two-valued semantics.
But so far we have no way to see that this is also so in the many-valued case.

To this end we will first look at an extension of the meta-logic, which besides the original copy of
the meta-logic will contain another, ``unintended'' embedding.
The point is that the straightforward, generalized-Tarski-style interpretation of the extended meta-logic
will restrict to the Kolmogorov-style interpretation on the re-embedded meta-logic (and, unsurprisingly,
to the Tarski-style interpretation on the original copy).

This extension of the meta-logic goes somewhat beyond being another special case of Paulson's formalism 
in {\tt Isabelle}'s metalogic: it employs $\lambda$-calculus with $\N$-indexed sums (i.e., countably 
infinite disjoint unions), and also assumes a generalized form of $\imp$-elimination.

\subsubsection{Extended $\lambda$-calculus}

We define an extension of the $\lambda$-calculus described in \S\ref{expressions} which involves
infinite sums and certain variable lists.

Let $\omega=\{0,1,\dots\}$ and for any $n\in\omega$ let $[n]=\{1,\dots,n\}$; thus $[0]=\emptyset$.
The extended $\lambda$-calculus has the following type forming rules:

\begin{enumerate}
\item base types are types;
\item if $\Gamma$ and $\Delta$ are types, then so is $\Gamma\too\Delta$;
\item if $n\in\omega$ and $\Gamma_i$ is a type for each $i\in [n]$, then so is $\prod_{i\in[n]}\Gamma_i$;
\item if $\Gamma_i$ is a type for each $i\in\omega$, then so is $\sum_{i\in\omega}\Gamma_i$.
\end{enumerate}

We write $\Gamma^n=\prod_{i\in[n]}\Gamma$.
We often shorten $\prod_{i\in[n]}\Gamma_i$ to $\prod_{i=1}^{n}\Gamma_i$ and 
$\sum_{i\in\N}\Gamma_i$ to $\sum_{i=1}^\infty\Gamma_i$.

While $n$-ary products as such could be defined in terms of binary products (see \S\ref{simultaneous}),
we will need them already to properly describe the expression forming rules (cf.\ Remark 
\ref{indices as variables} below.)

Constants and variables are denoted as before (see \S\ref{lambda}).
Expression forming rules are as follows, where $n\in\omega$:
\begin{enumerate}
\item a constant of type $\Gamma$ is an expression of type $\Gamma$;
\item a variable of type $\Gamma$ is an expression of type $\Gamma$;
\item (abstraction) if $T:\Delta$ on the assumption that $x$ is a variable of type $\Gamma$, then
$x\mapsto T:\Gamma\too\Delta$;
\item (application) if $F:\Delta\too\Gamma$, and $T:\Delta$, then $F(T):\Gamma$;
\item (projections) if $k\in[n]$, then $\pr_k:\prod_{i\in[n]}\Gamma_i\too\Gamma_k$;
\item (function tuples) if $F_i:\Delta\too\Gamma_i$ on the assumption that $i\in[n]$, then 
$(F_i)_{i\in[n]}:\Delta\too\prod_{i\in[n]}\Gamma_i$;%
\footnote{We use these ``function tuples'' instead of the traditional $n$-tuples in order to emphasize
symmetry between product and sum.}
\item (inclusions) if $k\in\omega$, then $\inc_k:\Gamma_k\too\sum_{i\in\omega}\Gamma_i$;
\item (definition by cases) if $F_i:\Gamma_i\too\Delta$ on the assumption that $i\in\omega$,
then $(F_i)^{i\in\omega}:\sum_{i\in\omega}\Gamma_i\too\Delta$;%
\footnote{This is yet another divergence from traditional notation of $\lambda$-calculus.}
\item (pre-tuple) there is a constant $\vec\emptyset:\prod_\emptyset$.
\end{enumerate}

Using the pre-tuple and function tuples, we can define usual tuples (at least, up to $\alpha$-equivalence): 
given a $T_i:\Gamma_i$ for each $i\in[n]$, let $(T_i)_{i\in[n]}:\prod_{i\in[n]}\Gamma_i$ denote 
$(x\mapsto T_i)_{i\in[n]}(\vec\emptyset)$,
where $x$ is a variable of type $\prod_\emptyset$ that does not occur freely in any of the $T_i$.
Thus the empty tuple $()\:\prod_\emptyset$ is the image of the pre-tuple $\vec\emptyset:\prod_\emptyset$ under 
the empty function tuple $()\:\prod_\emptyset\too\prod_\emptyset$.
(The latter has the meaning of a constant map, but secretly is also the identity map.)

\begin{remark} \label{indices as variables}
We treat the index $i$ of the variable $\var_i^\Gamma$ as a variable, rather than a meta-variable.
This means that we may say, for example, that ``$\var_i^\Gamma:\Gamma$ on the assumption that $i\in\omega$''.
It follows that $(\var_i^\Gamma)_{i\in[n]}$, which can also be written as $(\var_1^\Gamma,\dots,\var_n^\Gamma)$, 
is a valid expression of type $\Gamma^n$ on the assumption that $n\in\omega$.
Consequently,$\big(f_i\mapsto f_i(\var_1^\Gamma,\dots,\var_n^\Gamma)\big)^{i\in\omega}$ is a valid expression 
of type $\sum_{i\in\omega}(\Gamma^i\too\Gamma)\too\Gamma$.
Also, 
$\big((\var_1^{\Gamma_1},\dots,\var_n^{\Gamma_1}),\dots,(\var_1^{\Gamma_n},\dots,\var_n^{\Gamma_n})\big)$
is a valid expression of type $\prod_{i=1}^n\Gamma_i^n$ on the assumption that $n\in\omega$.
\end{remark}

\begin{example}\label{prodproduct1}
$\big(f\mapsto(\pr_i\circ f)\big)_{i\in[n]}:
\Big(\Delta\too\prod_{i\in[n]}\Gamma_i\Big)\too\prod_{i\in[n]}(\Delta\too\Gamma_i)$
is a $\lambda$-closed expression.
\end{example}

\begin{example}\label{prodproduct2}
$t\mapsto(\pr_i(t))_{i\in[n]}:\prod_{i\in[n]}(\Delta\too\Gamma_i)\too\Big(\Delta\too\prod_{i\in[n]}\Gamma_i\Big)$
is a $\lambda$-closed expression.
\end{example}

The previous two examples show that function tuples amount to a special case of usual tuples
(even though usual tuples were defined via a special case of function tuples).

Although we do not really need $n$-ary sums, they can be defined just like the infinite sums above, and
have the following relations with $n$-ary products:

\begin{example}\label{sumproduct1}
$t\mapsto\big(\pr_i(t)\big)^{i\in[n]}:
\prod_{i\in[n]}(\Gamma_i\too\Delta)\too\Big(\sum_{i\in[n]}\Gamma_i\too\Delta\Big)$
is a $\lambda$-closed expression.
\end{example}

\begin{example}\label{sumproduct2}
$f\mapsto(f\circ\inc_i)_{i\in[n]}:
\Big(\sum_{i\in[n]}\Gamma_i\too\Delta\Big)\too\prod_{i\in[n]}(\Gamma_i\too\Delta)$
is a $\lambda$-closed expression.
\end{example}

The previous two examples show that $(\cdot)^{i\in[n]}$ is in essence a special case of $(\cdot)_{i\in[n]}$.
(But technically they live in different types. Also, their comparison uses both rules, so none of the
two rules has been deduced from the other one.)

\begin{example}\label{sumproduct3}
$\Big(f_i\mapsto\big(t\mapsto f_i\circ\pr_i(t)\big)\Big)^{i\in[n]}:
\sum_{i\in[n]}(\Gamma_i\too\Delta)\too\Big(\prod_{i\in[n]}\Gamma_i\too\Delta\Big)$
is a $\lambda$-closed expression.
\end{example}

\begin{example}\label{sumproduct4}
$f_i\mapsto(\inc_i\circ f_i)^{i\in[n]}:
\sum_{i\in[n]}(\Delta\too\Gamma_i)\too\Big(\Delta\too\sum_{i\in[n]}\Gamma_i\Big)$
is a $\lambda$-closed expression.
The same goes for the infinite sums.
\end{example}

The definition of {\it free occurrence} is extended as follows:
\begin{enumerate}
\item $x$ occurs freely in $(F_i)_{i\in[n]}$ if it occurs freely in $F_i$ for some $i\in[n]$;
\item $x$ occurs freely in $(F_i)^{i\in\omega}$ if it occurs freely in $F_i$ for some $i\in\omega$.
\end{enumerate}

The definition of {\it substitution} is extended as follows:
\begin{enumerate}
\item $(F_i)_{i\in[n]}|_{x:=T}=(F_i|_{x:=T})_{i\in[n]}$;
\item $(F_i)^{i\in\omega}|_{x:=T}=(F_i|_{x:=T})^{i\in\omega}$.
\end{enumerate}

We also define terminal types and their units following \cite{CdC}:

\begin{enumerate}
\item If each $\Theta_i$ is a terminal type with unit $U_i$, then $\prod_{i\in[n]}\Theta_i$ is 
a terminal type with unit $(U_i)_{i\in[n]}$.
In particular, $\prod_\emptyset$ is a terminal type with unit $()$.
\item If $\Gamma$ is a type and $\Theta$ is a terminal type with unit $U$, then $\Gamma\to\Theta$ is 
a terminal type with unit $x\mapsto U$.
\end{enumerate}

The definition of {\it $\beta$-reduction} is extended as follows:
\begin{enumerate}
\item if $n\in[m]$, every expression of the form $\pr_n\big((F_i)_{i\in[m]}(T)\big)$ $\beta$-reduces to $F_n(T)$;
\item if $n\in\omega$, every expression of the form $(F_i)^{i\in\omega}\big(\inc_n(T)\big)$ $\beta$-reduces to $F_n(T)$.
\item if $\Theta$ is a terminal type with unit $U$, every $T:\Theta$ $\beta$-reduces to $U$.
\end{enumerate}

\begin{example} Let us see how to recover the $\beta$-reduction for usual tuples:
$\pr_n\big((T_i)_{i\in[m]}\big)$ $\beta$-reduces to $T_n$.
Indeed,
$\pr_n\big((x\mapsto T_i)_{i\in[m]}(\vec\emptyset)\big)$, where $x$ does not occur freely in any of the $T_i$,
$\beta$-reduces to $(x\mapsto T_n)(\vec\emptyset)$, which in turn $\beta$-reduces to $T_n$.
\end{example}

\begin{example} If $\Theta$ is a terminal type with unit $U$, we have $\phi:\Gamma\too(\Theta\too\Gamma)$, 
$\phi=t\mapsto(x\mapsto t)$, and $\psi:(\Theta\too\Gamma)\too\Gamma$, $\psi=f\mapsto f(U)$.
Now $\psi\big(\phi(t)\big)$ $\beta$-reduces to $(x\mapsto t)(U)$, which in turn $\beta$-reduces to $t$.
Conversely, $\phi\big(\psi(f)\big)$ $\beta$-reduces to $x\mapsto f(U)$.
This does not $\beta$-reduce to $f$, but we will now see that it $\beta\eta$-reduces to $f$.
\end{example}

The definition of {\it $\eta$-reduction} is extended as follows for $n>0$:
\begin{enumerate}
\item every expression of the form $\Big(x\mapsto\pr_i\big(F(x)\big)\Big)_{i\in[n]}$ $\eta$-reduces to $F$;
\item every expression of the form $\Big(x_i\mapsto F\big(\inc_i(x_i)\big)\Big)^{i\in\omega}$ $\eta$-reduces to $F$;
\item if $\Theta$ is a terminal type with unit $U$, for every $F:\Theta\to\Gamma$, the expression
$x\mapsto F(U)$ $\eta$-reduces to $F$;
\item if $\Theta$ is a terminal type with unit $U$, for every
$F:\Gamma\too\Theta\x\Delta$, the expression $\Big(x\mapsto U,\,x\mapsto\pr_1\big(F(x)\big)\Big)$ 
$\eta$-reduces to $F$;
\item if $\Theta$ is a terminal type with unit $U$, for every
$F:\Theta\sqcup\Gamma\too\Delta$, the expression 
$\Big(x_1\mapsto F\big(\inc_1(U)\big),\,x_2\mapsto F\big(\inc_1(x_2)\big)\Big)$ 
$\eta$-reduces to $F$.
\end{enumerate}

\begin{example} Let us see how to recover the $\eta$-reduction for usual tuples:
$\big(\pr_i(T)\big)_{i\in[n]}$ is $\beta\eta$-equivalent to $T$.
Indeed, $\big(x\mapsto\pr_i(T)\big)_{i\in[n]}(\vec\emptyset)$, where $x$ does not occur freely in $T$,
$\beta$-expands to $\Big(x\mapsto\pr_i\big((x\mapsto T)(x)\big)\Big)_{i\in[n]}(\vec\emptyset)$
which $\eta$-reduces to $(x\mapsto T)(\vec\emptyset)$,
which $\beta$-reduces to $T$.
\end{example}

\begin{remark} Let us note some immediate consequences of the $\beta$- and $\eta$-rules:
\begin{itemize}
\item $\pr_n\circ(F_i)_{i\in[m]}$ $\beta\eta$-reduces to $F_n$, as long as $n\in[m]$;
\item $(\pr_i\circ F)_{i\in[m]}$ $\eta$-reduces to $F$;
\item $(F_i)^{i\in\omega}\circ\inc_n$ $\beta\eta$-reduces to $F_n$, as long as $n\in\omega$;
\item $(F\circ\inc_i)^{i\in\omega}$ $\eta$-reduces to $F$.
\end{itemize}
\end{remark}

Here are some further consequences:

\begin{lemma} \label{composition-betaeta}
(a) $(F_i)_{i\in[n]}\circ G$ is $\beta\eta$-equivalent to $(F_i\circ G)_{i\in[n]}$;

(b) $G\circ(F_i)^{i\in\omega}$ is $\beta\eta$-equivalent to $(G\circ F_i)^{i\in\omega}$.
\end{lemma}

\begin{proof} By the previous remark, $(F_i\circ G)_{i\in[n]}$ $\beta\eta$-expands to 
$\big(\pr_i\circ(F_j)_{j\in[n]}\circ G\big)_{i\in[n]}$, which in turn $\eta$-reduces to
$(F_j)_{j\in[n]}\circ G$.

Similarly, $(G\circ F_i)^{i\in\omega}$ $\beta\eta$-expands to 
$\big(G\circ (F_j)^{j\in\omega}\circ\inc_i\big)^{i\in\omega}$, 
which in turn $\eta$-reduces to $G\circ (F_j)^{j\in\omega}$.
\end{proof}

\begin{remark} \label{sum-product}
Let $\Gamma=\sum_{i\in\omega}\Gamma_i$ and $\Delta=\sum_{i\in\omega}\Delta_i$.
Suppose that $F_{ij}:\Gamma_i\x\Delta_j\too\Theta$ for each $i,j\in I$.
We will define $(F_{ij})^{(i,j)\in\omega\x\omega}:\Gamma\x\Delta\too\Theta$.

We will use $\frak F$ and $\frak G$ defined in Example \ref{exp-lambda}.
We have $\frak F(F_{ij}):\Gamma_i\too(\Delta_j\too\Theta)$.
Then $\big(\frak F(F_{ij})\big)^{i\in\omega}:\Gamma\too(\Delta_j\too\Theta)$.
Let $\frak H=\frak F\circ\frak T\circ\frak G$, where 
$\frak T:(\Gamma\x\Delta\too\Theta)\too(\Delta\x\Gamma\too\Theta)$ is defined by
$f\mapsto\big(q\mapsto f(\pr_2q,\pr_1q)\big)$.
Then $\frak H\big(\frak F(F_{ij})\big)^{i\in\omega}:\Delta_j\too(\Gamma\too\Theta)$.
Hence $\Big(\frak H\big(\frak F(F_{ij})\big)^{i\in\omega}\Big)^{j\in\omega}:\Delta\too(\Gamma\too\Theta)$.
Finally, let 
\[(F_{ij})^{(i,j)\in\omega\x\omega}=\frak G\Big(\frak H\big(\frak F(F_{ij})\big)^{i\in\omega}\Big)^{j\in\omega}:
\Gamma\x\Delta\too\Theta.\]
More generally, given $F_{k_1,\dots,k_n}\:\Gamma_{1,k_1}\x\dots\x\Gamma_{n,k_n}\too\Theta$ for each
$k_1,\dots,k_n\in I$, one can similarly define 
\[(F_{k_1,\dots,k_n})^{(k_1,\dots,k_n)\in\omega^n}:
\sum_{i\in\omega}\Gamma_{1,i}\x\dots\x\sum_{i\in\omega}\Gamma_{n,i}\too\Theta.\]
\end{remark}

\subsubsection{Extended meta-logic} \label{extended meta-logic}

We use the notation introduced in \S\ref{Frege-definition}.
Given a $k\in\omega$ (i.e., a nonnegative integer), we denote by $\1^{[k]}$ 
the product $\prod_{i=0}^{k-1}(\0^i\too\1)^k$ of $k\x k$ factors.
We will also write $\var_{[k]}^\tau=(\var_1^\tau,\dots,\var_k^\tau):\tau^k$ and 
$\var_{[[k]]}^\1=(\var_{[k]}^{\0^0\too\1},\dots,\var_{[k]}^{\0^{k-1}\too\1})):\1^{[k]}$
(see Remark \ref{indices as variables}).

It is easy to see that every term $T$ of the language can be obtained as a $\beta$-reduction of 
$T_k(\var_{[k]}^\0)$ for some $\lambda$-closed expression $T_k:\0^k\too\0$, where $k$ is 
the maximal number such that $\var_k^\0$ occurs in $T$.
Similarly, every $n$-formula $F$ of the language can be obtained as a $\beta$-reduction of 
$F_k(\var_{[k]}^\0,\var_{[[k]]}^\1)$ for some $\lambda$-closed $F_k:\0^k\x\1^{[k]}\too(\0^n\too\1)$,
where $k$ is the maximal number such that either $\var_k^\0$ or $\var_k^{\0^{l-1}\too\1}$ or 
$\var_l^{\0^{k-1}\too\1}$ occurs in $F$ for some $l$.

Let $\0^*=\sum_{k\in\omega}(\0^k\too\0)$ and 
$\1_n^*=\sum_{k\in\omega}\big(\0^k\x\1^{[k]}\too(\0^n\too\1)\big)$.
Let us define {\it term evaluation} $\frak E:\0^*\too\0$ by $\frak E=(\frak E_k)^{k\in\omega}$, where 
$\frak E_k:(\0^k\too\0)\too\0$ is defined by $t_k\mapsto t_k(\var_{[k]}^\0)$.
(See Remark \ref{indices as variables}.)
Then every term $T$ of the language can also be obtained as a $\beta$-reduction of $\frak E(T^*)$ for some 
$\lambda$-closed $T^*:\0^*$.
Indeed, let $T^*=\inc_k(T_k)$, where $T_k$ is as above; then $\frak E\big(\inc_k(T_k)\big)$ $\beta$-reduces 
to $\frak E_k(T_k)$, which in turn $\beta$-reduces to $T_k(\var_{[k]}^\0)$.

Similarly, we have {\it $n$-formula evaluation} $\frak E^n:\1^*_n\too(\0^n\too\1)$,
$\frak E^n=(\frak E^n_k)^{k\in\omega}$, where 
$\frak E^n_k:\big(\0^k\x\1^{[k]}\too(\0^n\too\1)\big)\too(\0^n\too\1)$ is defined by
$\phi_k\mapsto\phi_k(\var_{[k]}^\0,\var_{[[k]]}^\1)$.
Every $n$-formula $F$ of the language can also be obtained as a $\beta$-reduction of $\frak E^n(F^*)$ 
for some $\lambda$-closed $F^*:\1^*_n$.
Namely, $F^*=\inc_k(F_k)$, where $F_k$ is as above.
Then $\frak E^n\big(\inc_k(F_k)\big)$ $\beta$-reduces to $\frak E^n_k(F_k)$, which in turn 
$\beta$-reduces to $F_k(\var_{[k]}^\0,\var_{[[k]}^\1)$.

The variables of type $\0^*$ will be called {\it term variables}, and the variables of type $\1^*_n$ will be 
called {\it $n$-formula variables} (or {\it formula variables} when $n=0$).
Corresponding to these two sorts of variables, there are two new types of meta-quantifiers: 
\begin{itemize}
\item {\it the term meta-quantifier} $\q^*:(\0^*\too\m)\too\m$;
\item {\it the $k$-formula meta-quantifier} $\q^*_n:(\1^*_n\too\m)\too\m$.
\end{itemize}
In line with previous notation, we write $\mq{t}\F$ to mean $\q^*(t\mapsto\F)$ where $t$ is
a term variable of an $n$-formula variable and $\F$ is a $\lambda$-expression of type $\m$.
We also write $\mq{\vec t}\F$ to mean $\mq{t_1}\dots\mq{t_n}\F$, where $\vec t=(t_1,\dots,t_n)$ is 
a tuple of term variables and $\mq{\vec\phi}\F$ to mean $\mq{\phi_1}\dots\mq{\phi_m}\F$, where 
$\vec\phi=(\phi_1,\dots,\phi_m)$ is a tuple of $r_i$-formula variables.

We define a {\it term scheme} (resp.\ a {\it term pre-scheme}) as an expression of type $\0$ 
(resp.\ $\0^*$) defined using function application, projections, inclusions, tuples and definition by cases;
function symbols of the language; individual and term variables and abstraction over them.
We define a {\it $k$-formula scheme} (resp.\ a {\it $k$-formula pre-scheme}) as an expression of type 
$\0^k\too\1$ (resp.\ $\1^*_k$) defined using all of the above, plus: predicate constants, 
connectives and quantifiers of the language; predicate and $n$-formula variables and abstraction over them.
Finally, we define a {\it meta-formula scheme} as an expression of type $\m$ defined using all of 
the above, plus $\Rt$, the two meta-connectives and the four sorts of meta-quantifiers.
In particular, terms schemes may involve $\frak E$, and $k$-formula schemes and meta-formula schemes
may involve $\frak E$ and each $\frak E^n$.

The inference meta-rules of \S\ref{meta-rules} are now applied to term schemes, $k$-formula schemes and
meta-formula schemes in place of terms, $k$-formulas and meta-formulas, and are augmented by the following
new meta-rules, where $\F$, $\F_k$, $\G_k$ are meta-formula schemes, $t$ is a term variable, $T$ is 
a term pre-scheme, $\phi$ is an $n$-formula variable and $\Phi$ is an $n$-formula pre-scheme.

\begin{enumerate}
\item (meta-generalization)

$\Dfrac{\begin{matrix}\vdots\\ \F\end{matrix}}{\mq{t}\F}$,
provided that $t$ does not occur freely in any of the assumptions;
\bigskip

$\Dfrac{\begin{matrix}\vdots\\ \F\end{matrix}}{\mq{\phi}\F}$,
provided that $\phi$ does not occur freely in any of the assumptions.
\bigskip

\item (meta-specialization) 

$\Dfrac{\begin{matrix}\vdots\\ \mq{t}\F\end{matrix}}{\F[t/T]}$, 
provided that $T$ is free for $t$ in $\F$;
\bigskip

$\Dfrac{\begin{matrix}\vdots\\ \mq{\phi}\F\end{matrix}}{\F[\phi/\Phi]}$,
provided that $\Phi$ is free for $\phi$ in $\F$.
\bigskip

\item (generalized $\imp$-elimination)

$\Dfrac{\begin{matrix}\vdots\\ (t_k\mapsto\F_k)^{k\in\omega}(t)\end{matrix}\qquad
\begin{matrix}\vdots\\ \F_k\imp\G_k\text{ for each $k\in\omega$}\end{matrix}}
{(t_k\mapsto\G_k)^{k\in\omega}(t)}$
\bigskip

$\Dfrac{\begin{matrix}\vdots\\ (\phi_k\mapsto\F_k)^{k\in\omega}(\phi)\end{matrix}\qquad
\begin{matrix}\vdots\\ \F_k\imp\G_k\text{ for each $k\in\omega$}\end{matrix}}
{(\phi_k\mapsto\G_k)^{k\in\omega}(\phi)}$
\end{enumerate}
\bigskip

Let us note that each of the latter two meta-rules has infinitely many premisses.

\begin{remark}
Here is a brief sketch of an equivalent formalism.
Instead of meta-quantifiers over term variables and $k$-formula variables, one could use their pre-limit
versions (i.e.\ meta-quantifiers over variables of types $\0^k\too\0$ and
$\0^l\x\1^{[k]}\too(\0^k\too\1)$) along with a meta-quantifier over variables running over $\omega$.
All three sorts of meta-quantifiers come with natural specialization and generalization rules.
Moreover, the first two types of meta-quantifiers subsume the original meta-quantifiers over individual
and predicate variables (and similarly for the corresponding rules).
Thus only three types of meta-quantifiers are needed, in contrast to the four types in the approach used
above; and additionally the generalized $\imp$-elimination meta-rule is not needed.

So in some sense the alternative formalism is simpler; also it is closer to dependent type theory.
On the other hand, the above formalism has precisely those four meta-quantifiers that we actually care 
about as primitive notions, and also avoids some excessive layers of indices by focusing on limit,
rather than pre-limit types.   
\end{remark}

\subsubsection{Enderton-style formulation} \label{enderton2}

The extended meta-logic also admits a straightforward Enderton-style formulation. 

Suppose that $\F$, $\F_k$, $\G$, $\G_k$ and $\H$ are meta-formula schemes, $x$ is an individual variable, 
$T$ is a term scheme, $\gamma$ is an $n$-ary predicate variable, $\Phi$ is an $n$-formula scheme, 
$t$ is a term variable, $W$ is a term pre-scheme, $\phi$ is an $n$-formula variable and $\Xi$ is 
an $n$-formula pre-scheme.
Further suppose that $\vec z$ is a tuple of individual variables, $\vec\beta$ is tuple of predicate variables,
$\vec s$ is a tuple of term variables and $\vec\psi$ is a tuple of $i$-formula variables.

\begin{enumerate}
\item $\mq{\vec z}\mq{\vec\beta}\mq{\vec s}\mq{\vec\psi}\big(\F\imp\G\big)$, 
if $\F$ is $\alpha$-equivalent to $\G$;
\smallskip

\item $\mq{\vec z}\mq{\vec\beta}\mq{\vec s}\mq{\vec\psi}\big(\F\imp\F\big)$;
\medskip

\item
$\mq{\vec z}\mq{\vec\beta}\mq{\vec s}\mq{\vec\psi}\big((\F\imp\G)\mand(\G\imp\H)\Imp(\F\imp\H)\big)$;
\medskip

\item
$\mq{\vec z}\mq{\vec\beta}\mq{\vec s}\mq{\vec\psi}
\Big(\big((\F\mand\G)\imp\H\big)\Iff\big(\F\imp(\G\imp\H)\big)\Big)$;
\medskip

\item
$\mq{\vec z}\mq{\vec\beta}\mq{\vec s}\mq{\vec\psi}\big(\F\mand\G\imp\F\big)$ \ and \ 
$\mq{\vec z}\mq{\vec\beta}\mq{\vec s}\mq{\vec\psi}\big(\F\mand\G\imp\G\big)$;
\smallskip

\item
$\mq{\vec z}\mq{\vec\beta}\mq{\vec s}\mq{\vec\psi}\big((\H\imp\F)\mand(\H\imp\G)\Imp(\H\imp\F\mand\G)\big)$;
\medskip

\item $\mq{\vec z}\mq{\vec\beta}\mq{\vec s}\mq{\vec\psi}\big(\mq{x}\F\imp\F[x/T]\big)$,
provided that $T$ is free for $x$ in $\F$;
\medskip

\item $\mq{\vec z}\mq{\vec\beta}\mq{\vec s}\mq{\vec\psi}\big(\mq{\gamma}\F\imp\F[\gamma/\Phi]\big)$,
provided that $\Phi$ is free for $\gamma$ in $\F$;
\medskip

\item $\mq{\vec z}\mq{\vec\beta}\mq{\vec s}\mq{\vec\psi}\big(\mq{t}\F\imp\F[t/W]\big)$,
provided that $W$ is free for $t$ in $\F$;
\medskip

\item $\mq{\vec z}\mq{\vec\beta}\mq{\vec s}\mq{\vec\psi}\big(\mq{\phi}\F\imp\F[\phi/\Xi]\big)$,
provided that $\Xi$ is free for $\phi$ in $\F$;
\medskip

\item $\mq{\vec z}\mq{\vec\beta}\mq{\vec s}\mq{\vec\psi}
\big(\mq{x}(\F\imp\G)\Imp(\mq{x}\F\imp\mq{x}\G)\big)$;
\medskip

\item $\mq{\vec z}\mq{\vec\beta}\mq{\vec s}\mq{\vec\psi}
\big(\mq{\gamma}(\F\imp\G)\Imp(\mq{\gamma}\F\imp\mq{\gamma}\G)\big)$;
\medskip

\item $\mq{\vec z}\mq{\vec\beta}\mq{\vec s}\mq{\vec\psi}
\big(\mq{t}(\F\imp\G)\Imp(\mq{t}\F\imp\mq{t}\G)\big)$;
\medskip

\item 
$\mq{\vec z}\mq{\vec\beta}\mq{\vec s}\mq{\vec\psi}
\big(\mq{\phi}(\F\imp\G)\Imp(\mq{\phi}\F\imp\mq{\phi}\G)\big)$;
\medskip

\item $\mq{\vec z}\mq{\vec\beta}\mq{\vec s}\mq{\vec\psi}\big(\F\imp\mq{x}\F\big)$,
provided that $x$ does not occur freely in $\F$;
\medskip

\item $\mq{\vec z}\mq{\vec\beta}\mq{\vec s}\mq{\vec\psi}\big(\F\imp\mq{\gamma}\F\big)$,
provided that $\gamma$ does not occur freely in $\F$;
\medskip

\item $\mq{\vec z}\mq{\vec\beta}\mq{\vec s}\mq{\vec\psi}\big(\F\imp\mq{t}\F\big)$,
provided that $t$ does not occur freely in $\F$;
\medskip

\item $\mq{\vec z}\mq{\vec\beta}\mq{\vec s}\mq{\vec\psi}\big(\F\imp\mq{\phi}\F\big)$,
provided that $\phi$ does not occur freely in $\F$;
\medskip

\item $\Dfrac{\F,\ \F\imp\G}{\G}$;
\bigskip

\item $\Dfrac{(t_k\mapsto\F_k)^{k\in\omega}(t),\quad
\F_k\imp\G_k\text{ for each $k\in\omega$}}
{(t_k\mapsto\G_k)^{k\in\omega}(t)}$;
\bigskip

\item $\Dfrac{(\phi_k\mapsto\F_k)^{k\in\omega}(\phi),\quad
\F_k\imp\G_k\text{ for each $k\in\omega$}}
{(\phi_k\mapsto\G_k)^{k\in\omega}(\phi)}$.
\end{enumerate}

\subsubsection{Principles and rules}

\begin{proposition} \label{mq-commute}
Let $\F(x)$ be a meta-formula scheme, where $x$ is an individual variable and no individual variable 
occurs freely in $\F$, and let $t$ be a term variable that does not occur freely in $\F$.
Then the meta-formula scheme 
\[\mq{x}\F(x)\Iff\mq{t}\Big(t_k\mapsto\mq{\var_{[k]}^\0}\F\big(\frak E_k(t_k)\big)\Big)^{k\in\omega}(t)\]
is deducible.
\end{proposition}

Let us note that the $\beta$-reduction of $\F\big(\frak E_k(t_k)\big)$ may require renaming of bound variables;
in fact, it {\it must} require renaming of bound variables if bound variables ever occur in $\F$, since 
{\it each} individual variable occurs freely in some $\frak E_k$.

\begin{proof} By the $\imp$-introduction meta-rule, it suffices to deduce the left hand side from 
the right hand side, and conversely.

Suppose that $x=\var_m^\0$.
By the meta-specialization for term variables, from the right hand side we obtain
\[
\Big(t_k\mapsto\mq{\var_{[k]}^\0}\F\big(\frak E_k(t_k)\big)\Big)^{k\in\omega}\big(\inc_m(\pr_m)\big),
\tag{$*$}
\]
where 
$\inc_m\:(\0^m\too\0)\too\0^*$ is the inclusion of the $m$th summand of $\0^*=\sum_{k\in\omega}(\0^k\too\0)$ 
and $\pr_m\:\0^m\too\0$ is the projection onto the $m$th factor.
The expression ($*$) $\beta$-reduces to 
$\Big(t_m\mapsto\mq{\var_{[m]}^\0}\F\big(\frak E_m(t_m)\big)\Big)(\pr_m)$, which in turn
$\beta$-reduces to $\mq{\var_{[m]}^\0}\F\big(\frak E_m(\pr_m)\big)$.
Here $\frak E_m(\pr_m)$ $\beta$-reduces to $\pr_m(\var_{[m]}^\0)$, and hence to $\var_m^\0=x$.
Thus we have deduced $\mq{x}\F(x)$.

Conversely, by the meta-specialization for individual variables, from $\mq{x}\F(x)$ we obtain
$\F\big(\frak E(t)\big)$, which $\beta$-expands to $\big(\F\circ\frak E\big)(t)$.
Here $\F\circ\frak E$ is $\beta\eta$-equivalent to $\big(\F\circ\frak E_k\big)^{k\in\omega}$ 
by Lemma \ref{composition-betaeta}.
Now $\F\circ\frak E_k$ is the same as $t_k\mapsto\F\big(\frak E_k(t_k)\big)$,
thus we have deduced $\Big(t_k\mapsto\F\big(\frak E_k(t_k)\big)\Big)^{k\in\omega}(t)$.
On the other hand, by the meta-generalization for individual variables and $\imp$-introduction
$\F\big(\frak E_k(t_k)\big)\imp\mq{\var_{[k]}^\0}\F\big(\frak E_k(t_k)\big)$ is deducible 
for each $k\in\omega$.
Hence by the generalized $\imp$-elimination we get
$\Big(t_k\mapsto\mq{\var_{[k]}^\0}\F\big(\frak E_k(t_k)\big)\Big)^{k\in\omega}(t)$.
Finally, by the meta-generalization for term variables, from the latter meta-formula scheme we get
$\mq{t}\Big(t_k\mapsto\mq{\var_{[k]}^\0}\F\big(\frak E_k(t_k)\big)\Big)^{k\in\omega}(t)$.
\end{proof}

Similarly one can handle second-order quantification:

\begin{proposition} \label{mq-commute2}
Let $\F(\gamma)$ be a meta-formula scheme, where $\gamma$ is an $r$-ary predicate 
variable and no individual or predicate variables occur freely in $\F$,
and let $\vec\phi$ be an $r$-formula variable that does not occur freely in $\F$.
Then the meta-formula scheme 
\[\mq{\gamma}\F(\gamma)\Iff
\mq{\phi}\Big(\phi_k\mapsto\mq{\var_{[k]}^\0}\mq{\var_{[[k]]}^\1}
\F\big(\frak E^r_k(\phi_k)\big)\Big)^{k\in\omega}(\phi)\]
is deducible.
\end{proposition}

Let us also formulate the multivariable case, using the notation of Remark \ref{sum-product}:

\begin{proposition} \label{mq-commute3}
Let $\F(\vec x,\vec\gamma)$ be a meta-formula scheme, where $\vec x$ is an $n$-tuple of 
individual variables, $\vec\gamma$ is an $m$-tuple of predicate 
variables of arities $(r_1,\dots,r_m)=\vec r$ and no individual or predicate variables occur freely in $\F$,
and let $\vec t$ be an $n$-tuple of term variables that do not occur freely in $\F$ and $\vec\phi$ be 
an $m$-tuple of $r_i$-formula variables that do not occur freely in $\F$.
Then the meta-formula scheme \[\mq{\vec x}\mq{\vec\gamma}\F(\vec x,\vec\gamma)\Iff
\mq{\vec t}\mq{\vec\phi}\Big(t_{\vec k},\phi_{\vec l}\mapsto
\mq{\var_{[\vec k,\vec l]}^\0}\mq{\var_{[[\vec l]]}^\1}
\F\big(\frak E_{\vec k}(t_{\vec k}),\frak E^{\vec r}_{\vec l}
(\phi_{\vec l})\big)\Big)^{(\vec k,\vec l)\in\omega^n\x\omega^m}(\vec t,\vec\phi)\]
is deducible.
\end{proposition}

Here for $\vec k=(k_1,\dots,k_n)$ we write $[\vec k]=[\max(k_1,\dots,k_n)]$, $[[\vec k]]=[[\max(k_1,\dots,k_n)]]$,
$t_{\vec k}=(t_{k_1},\dots,t_{k_n})$ and $\phi_{\vec k}=(\phi_{k_1},\dots,\phi_{k_n})$.
The evaluations $\frak E_{\vec k}:\prod_{i=1}^n(\0^{k_i}\too\0)\too\0^n$ and 
$\frak E^{\vec r}_{\vec l}:\prod_{i=1}^m\big(\0^{l_i}\x\1^{[l_i]}\too(\0^{r_i}\too\1)\big)\too
\prod_{i=1}^m(\0^{r_i}\too\1)$, where $\vec l=(l_1,\dots,l_m)$,
are defined respectively by $t_{\vec k}\mapsto\big(\frak E_{k_1}(t_{k_1}),\dots,\frak E_{k_n}(t_{k_n})\big)$ and
$\phi_{\vec l}\mapsto\big(\frak E^{r_1}_{l_1}(\phi_{l_1}),\dots,\frak E^{r_m}_{l_m}(\phi_{l_m})\big)$.

\subsubsection{Semantics of the extended meta-logic} \label{extended semantics}

A {\it co-assignment} $\preass$ assigns a semantic term $\preass(x)\in|\T|$ to each term variable $x$.
A {\it co-valuation} $\preval$ assigns a semantic $k$-formula $\preval(\gamma)\in|\Fm_k|$
to each $k$-formula variable $\gamma$, for each $k=0,1,\dots$.

Given a co-assignment $\preass$ and an assignment $\iass$ along with an interpretation $\I$ of function 
symbols, we clearly get an interpretation of each term scheme $T$ by an element $|T|^\I_{\iass\preass}\in\O$,
which does not depend on $\iass$ (resp.\ on $\preass$) if $T$ contains no free occurrences of individual
(resp.\ term) variables.

Given additionally a co-valuation $\preval$ and a valuation $\pval$ along with an interpretation $\I$ of 
predicate constants, connectives and quantifiers, we clearly get an interpretation of each formula scheme $F$ 
by an element $|F|_\I^{\pval\preval\iass\preass}\in\O$, which does not depend on $\iass$ (resp.\ on $\preass$)
if $F$ contains no free occurrences of individual (resp.\ term) variables and does not depend on $\pval$ 
(resp.\ on $\preval$) if $F$ contains no free occurrences of predicate (resp.\ $k$-formula) variables.

Given additionally also an interpretation $\J$ of meta-connectives and meta-quantifiers, we clearly also get
an interpretation of each meta-formula scheme $\F$ by an element $|\F|_\J^{\pval\preval\iass\preass}\in\Qm$,
with the same dependence properties.
In particular, a $\lambda$-closed meta-formula scheme $\F$ is interpreted by an element $|\F|_\J\in\Qm$.

We can now generalize the definitions of \S\ref{models} to the case of the extended meta-logic.
An absolute meta-rule is said to be {\it satisfied} in a meta-structure $\J$ if for every its instance 
\[\Dfrac{\F_1,\dots,\F_n}{\G}\]
and every individual assignment $\iass$, co-assignment $\preass$, predicate valuation $\pval$, 
and co-valuation $\preval$, if
$\wnf(|\F_i|^{\iass\preass\pval\preval}_\J)=\Top$ for each $i$, then 
$\wnf(|\G|^{\iass\preass\pval\preval}_\J)=\Top$.
In particular, a meta-law is satisfied in $\J$ if and only $\wnf(|\G|^{\iass\preass\pval\preval}_\J)=\Top$
for every its instance $\G$ and arbitrary $\iass$, $\preass$, $\pval$ and $\preval$.

A meta-interpretation $\J$ of the extended meta-logic is called a {\it model of the extended meta-logic}
if it satisfies all meta-rules of the Enderton-style formulation of the extended meta-logic (see 
\S\ref{enderton2}).

Now we define the {\it Kolmogorov translation} $\script F$ of $\lambda$-closed meta-formulas into 
$\lambda$-closed meta-formula schemes:
\begin{itemize}
\item $\script F$ sends each individual variable $\var_i^\0$ to the corresponding term variable $\var_i^{\0^*}$,
and each $k$-ary predicate variable $\var_i^{\0^k\to\1}$ to the corresponding $k$-formula variable 
$\var_i^{\1^*_k}$.
\item $\script F$ sends both meta-connectives to themselves, the first-order meta-quantifier to the
term meta-quantifier, and the $k$-ary second-order meta-quantifier to the $k$-formula meta-quantifier. 
\item If $\G$ is a $\lambda$-closed atomic $(n,\vec r)$-meta-formula $\vec x,\vec\gamma\mapsto\Rt F$, 
where $F$ is a formula and $\vec r=(r_1,\dots,r_m)$, then its Kolmogorov-image $\script F(\G)$ is 
the following $\lambda$-closed expression of type $(\0^*)^n\x\1^*_{r_1}\x\dots\x\1^*_{r_m}\too\m$:
\[\Big(t_{\vec k},\phi_{\vec l}\mapsto
\mq{\var_{[\vec k,\vec l]}^\0}\mq{\var_{[[\vec l]]}^\1}
\Rt F\big(\frak E_{\vec k}(t_{\vec k}),\frak E^{\vec r}_{\vec l}
(\phi_{\vec l})\big)\Big)^{(\vec k,\vec l)\in\omega^n\x\omega^m}.\]
\end{itemize}

If $\F$ is a $\lambda$-closed meta-formula, it is easy to see that $|\script F(\F)|_\J$ is
precisely the alternative meta-interpretation $\KK\F\KK_\J$ as defined in \S\ref{Frege-definition}.

Proposition \ref{mq-commute3} now implies that Examples \ref{Frege principle} and \ref{Frege rule} 
extend to the case of many-valued interpretations of the meta-logic:
 
\begin{theorem} \label{principles persist}
If $F$, $G$ be formulas and $\J$ a model of the extended meta-logic.
Then 

(a) $\KK{\prin F}\KK_\J=|{\prin F}|_\J$.

(b) $\wnf(|F\,/\,G|_\J)=\Top$ implies $\wnf(\KK F\,/\,G\KK_\J)=\Top$.
\end{theorem}

An {\it alternative model of the meta-logic} is defined by applying a model of the extended meta-logic
to the Kolmogorov translation from meta-logic to extended meta-logic.
Given a logic $L$ based on a language $\L$ and given by a derivation system $\Ds$, a meta-$\L$-structure $M$ 
is called a {\it generalized alternative model} of $L$ if it is an alternative model of the meta-logic and
the meta-formula $\Ds$ is valid in $M$.

By Theorem \ref{principles persist}, a principle is valid in a generalized alternative model if and only if
it is valid in the corresponding generalized model; this clearly remains so if valuation fields or
pre-valuation fields are used.
But by Example \ref{Frege pure}, this is not the case for rules, at least if pre-valuation fields are used.

\subsubsection{Meta-clarified BHK interpretation}\label{extended+}

This will be based on the alternative semantics of \S\ref{Frege-definition} and \S\ref{extended semantics} 
with many-valued meta-interpretation of the meta-logic.

A domain $\D$ is fixed, and $\O$ is taken to be a class of problems, consisting of a prescribed class 
of contentful (e.g.\ mathematical) primitive problems and of composite problems, which 
are obtained inductively from the primitive ones by using contentual intuitionistic connectives and 
quantifiers over arbitrary $\D$-indexed families of parametric problems.
Formal intuitionistic connectives and quantifiers are interpreted straightforwardly by the contentual ones.

$\Qm$ is taken to be a larger class of problems, consisting of those in $\O$, called ``meta-primitive'' 
(thus the function $\ocf\:\O\to\Qm$ is the inclusion) and of ``meta-composite'' problems obtained inductively 
from the meta-primitive ones by using contentual intuitionistic connectives and quantifiers over arbitrary 
$\D$-indexed, $\Hom(\D^i,\O)$-indexed ($i=0,1,\dots$), $|\T|$-indexed and $|\Fm_i|$-indexed ($i=0,1,2,\dots$) 
families of problems.
(We recall that $|\T|$ and $|\Fm_i|$ are defined in terms of $\D$ and $\O$, see \S\ref{Frege-definition}.)
Meta-connectives and meta-quantifiers are Tarski-interpreted by contentual intuitionistic connectives and 
$\D$-indexed and $\Hom(\D^i,\O)$-indexed quantifiers in the straightforward way; and 
meta-quantifiers are also alternatively-interpreted by contentual intuitionistic $|\T|$-indexed and 
$|\Fm_i|$-indexed quantifiers.

The function $\wnf\:\Qm\to\{\Top,\Bot\}$ sends a problem $\Theta$ to $\Top$ if and only if there exists 
a solution of $\Theta$.
What is meant by solving a problem $\Theta$ is defined explicitly for primitive $\Theta\in\O$ and is 
extended inductively to composite $\Theta\in\O$ and then further to all meta-composite $\Theta\in\Qm$ 
by the usual BHK-interpretation (see \S\ref{BHK}).
This determines a contentual interpretation of intuitionistic logic and its meta-logic, which we will call 
the {\it meta-clarified BHK} interpretation.

\begin{example} Let us record the meta-clarified BHK interpretations of the judgements of admissibility, 
stable admissibility and derivability of a purely logical rule $\Gamma\,/\,\Delta$.
The tuples $\vec\phi$ and $\vec\Phi$ have the same meaning as in Example (\ref{platonist example}).
Like before, $\wn\Theta$ denotes the proposition ``The problem $\Theta$ has a solution''.

\begin{itemize}
\item Derivability: $\forall\D,\O\ \wn\forall\preval\,
(\forall\pval,\iass\,|\Gamma|_\iass^{\preval[\iass,\pval]}
\To\forall\pval,\iass\,|\Delta|_\iass^{\preval[\iass,\pval]})$;

\item Stable admissibility: $\forall\vec\Phi\,\forall\D,\O\ 
\wn(\forall\pval,\iass\,|\Gamma[\vec\phi/\vec\Phi]|^\pval_\iass
\To\forall\pval,\iass\,|\Delta[\vec\phi/\vec\Phi]|^\pval_\iass)$;

\item Admissibility:
$\forall\vec\Phi\,(\forall\D,\O\ \wn\forall\pval,\iass\,|\Gamma[\vec\phi/\vec\Phi]|^\pval_\iass
\To\forall\D,\O\ \wn\forall\pval,\iass\,|\Delta[\vec\phi/\vec\Phi]|^\pval_\iass)$.
\end{itemize}
\end{example}

\begin{remark}\label{verificationist3}
Let us note that there is an obvious variation of the Verificationist interpretation with alternative 
interpretation of the meta-level, similar to the one in the meta-clarified BHK interpretation. 
\end{remark}

\end{document}